\documentclass[a4paper, 10pt]{article}
\usepackage{amsmath}
\numberwithin{equation}{section}
\usepackage{amssymb,esint,hyperref}
\usepackage{amscd}
\usepackage{xspace}
\usepackage{fancyhdr}
\usepackage{color}
\usepackage{verbatim}
\usepackage{graphicx}
\usepackage{cite}
\usepackage{stmaryrd}
\usepackage{bbm}
\usepackage{xcolor}
\usepackage{mathrsfs}
\usepackage{srcltx} 

\usepackage{marginnote}
\let\oldmarginpar\marginpar
\renewcommand\marginpar[1]{\-\oldmarginpar[\raggedleft\footnotesize #1]%
	{\raggedright\footnotesize\color{red} #1}} 
\newtheorem{theorem}{Theorem}[section]
\newtheorem{corollary}[theorem]{Corollary}

\newtheorem{definition}[theorem]{Definition}

\newtheorem{lemma}[theorem]{Lemma}

\newtheorem{proposition}[theorem]{Proposition}
\newtheorem{remark}[theorem]{Remark}

\def\da{\mathsf{Data}}
\def\B{\mathsf{B}}
\def\A{\mathsf{A}}
\def\M{\mathsf{M}}

\def\E{\mathsf{E}}
\def\F{\mathsf{F}}
\def\G{\mathsf{G}}

\def\K{\mathsf{K}}

\def\N{\mathsf{N}}

\newcommand{\eps}{\varepsilon}

\newenvironment{proof}[1][Proof]{\textbf{#1.} }{\hfill\rule{0.5em}{0.5em}}
{\catcode`\@=11\global\let\AddToReset=\@addtoreset
	\AddToReset{equation}{section}
	
	\AddToReset{theorem}{section}

	\title{Schauder-type Estimates and Well-posedness for Nonlocal Quasilinear Evolution Equations in Fluid Dynamics 
    \footnote{This article forms Part I of a series originating from the comprehensive preprint arXiv:2407.05313, focusing on establishing Schauder-type estimates and their application to the well-posedness of certain free boundary problems.}}
	\author{
		{{\bf Ke Chen,\thanks{E-mail address: k1chen@polyu.edu.hk, Department of Applied Mathematics, The Hong Kong Polytechnic University, Kowloon, Hong Kong, PR China.}
				~~Ruilin Hu\thanks{E-mail address: rh2488@bath.ac.uk, Department of Mathematical Sciences, University of Bath, Bath BA2 7AY, UK.}
				~~Quoc-Hung Nguyen\thanks{E-mail address: qhnguyen@amss.ac.cn, Academy of Mathematics and Systems Science, Chinese Academy of Sciences, Beijing, 100190, China.}}}}
	
\begin{document}
		\maketitle
        \begin{abstract}
We establish Schauder-type estimates for linear parabolic systems driven by variable-coefficient nonlocal pseudo-differential operators of order $s>0$.
These estimates are formulated in critical time-weighted H\"older/Besov-type spaces and are tailored to quasilinear equations at scaling-critical regularity. 

A key ingredient is a kernel-adapted freezing-coefficient method. After freezing the coefficients at a reference point, we derive explicit representation formulas through the corresponding fundamental kernels and then evaluate the resulting bounds at the physical point. This avoids treating the coefficient variation as a separate lower-order perturbation and yields robust control of the residual terms within the leading-order dynamics.

As an application, we obtain a general well-posedness framework for a class of nonlocal quasilinear parabolic equations in critical spaces. In particular, we prove critical local and, in suitable regimes, global well-posedness for the Muskat equation with surface tension and for the two- and three-dimensional Peskin problems with nonlinear elastic tension. These results provide a unified critical framework for distinct nonlocal evolution equations arising in fluid dynamics and related areas.
\end{abstract}

		\section{Introduction}\label{secintro}
		The general form of partial differential equations in fluid dynamics can be expressed as a system of conservation laws that account for transport, diffusion, and sources of mass, momentum, and energy. A prototypical formulation synthesizes four essential components: first-order temporal evolution, nonlinear convective transport, second-order diffusive effects, and external forcing terms. This comprehensive structure can be mathematically represented as:
		\begin{align*}
			\partial_t \mathbf{U} +\nabla\cdot F_{inviscid}(\mathbf{U})-\nabla\cdot F_{viscous}(\mathbf{\nabla U})=S,
		\end{align*}
		where the constituent elements are defined as
		\begin{itemize}
			\item $\mathbf{U}$: State vector of conserved quantities (e.g.,mass, momentum, and energy).
			\item $F_{inviscid}$: Inviscid flux tensor (e.g., convective and pressure terms).
			\item $F_{viscous}$: Viscous flux tensor (e.g., stress and heat diffusion terms, dependent on $\nabla\mathbf{U}$).
			\item $S$: Source terms (e.g., gravitational forces, heat sources).
		\end{itemize}
		While this general framework has been extensively employed in fluid modeling, critical challenges persist in establishing rigorous well-posedness, particularly for systems where parabolic dissipation mechanisms interact with nonlinear convective phenomena. The current investigation addresses this fundamental issue through a focused study of quasilinear parabolic systems with the following canonical structure:
		\begin{align}\label{quasi}
			\partial_t u(t,x)+\mathcal{A}_s[t,x,u,\nabla u,\cdots, \nabla^l u](u)(t,x)=\mathcal{N}(t,x,u),\quad\quad\quad (t,x)\in \mathbb{R}^+\times\mathbb{R}^d,
		\end{align}
		where $l\in\mathbb{N}, s\in\mathbb{R}^+$. For any fixed $t\in\mathbb{R}^+, x_0\in\mathbb{R}^d$, the operator $$\mathcal{A}_s[t,x_0,u(t,x_0),\cdots,\nabla^lu(t,x_0)]$$ is a
		linear differential operator of order $s$ with $s>l$, and $\mathcal{N}$ represents nonlinear forcing terms containing derivatives of $u$ of order strictly less than $s$. The operator  $\mathcal{A}_s[t,x,u,\nabla u,\cdots, \nabla^l u](u)(t,x)$ can be equivalently represented as a pseudo-differential operator. 
		This hierarchical differentiation structure ensures the viscous dissipation mechanism governs the highest-order dynamics.

		In this paper, we develop a new analytical framework for investigating the well-posedness of quasilinear partial differential equations of the form \eqref{quasi}. Central to this framework is a novel methodology for deriving Schauder-type a priori estimates applicable to general linear parabolic operators—a cornerstone result underpinning the establishment of regularity properties for this class of equations. Subsequently, we apply these estimates to establish local and global well-posedness results for several nonlocal and nonlinear evolution equations arising in fluid dynamics.
		
		The cornerstone of our approach lies in establishing sharp Schauder-type estimates for linear parabolic operators. As background and motivation, we recall that classical Schauder theory provides fundamental Hölder-norm bounds for solutions to linear elliptic equations in terms of the data. For instance, the well-known result \cite{GTbook} for the Poisson equation $\Delta u = \operatorname{div} F$ in $\mathbb{R}^d$ yields:
		\begin{align*}
			\|\nabla u\|_{\dot{C}^{a}} \lesssim\left(\sup _{k \in \mathbb{N}} \frac{1}{|a-k|}\right)\|F\|_{\dot{C}^{a}},\quad \quad \forall a \in (0, \infty)  \backslash \mathbb{N}.
		\end{align*}
		This extends via the freezing coefficient method \cite{Evans98} to variable-coefficient elliptic systems $\operatorname{div}(\A(x) \nabla u) = \operatorname{div} F$ (with uniformly elliptic $\A$):
		\begin{align*}
			\|\nabla u\|_{\dot{C}^{a}} \lesssim\left(\sup _{k \in \mathbb{N}} \frac{1}{|a-k|}\right)\left(\|F\|_{\dot{C
				}^{a}}+\left(1+\|\A\|_{\dot{C}^{a}}\right)\|u\|_{L^{\infty}}\right),\ \  \forall a \in(0, \infty) \backslash \mathbb{N}.
		\end{align*}
		Building on these classical elliptic estimates, we establish their parabolic counterparts through spacetime scaling analysis. A prototype is the heat equation
		\begin{equation}\label{heat}
			\begin{aligned}
				&\partial_t u-\Delta u=\operatorname{div} F ,\ \ \text{in} \ (0,\infty)\times\mathbb{R}^d ,\\
				&u|_{t=0}=u_0,
			\end{aligned}
		\end{equation}
		for which we have the sharp regularity estimate
		\begin{align}
			&\sup _{t>0}(\|u(t)\|_{L^{\infty}}+ t^{\frac{1+a}{2}}\|\nabla u(t)\|_{\dot{C}^{a}} )\lesssim \left\|u_{0}\right\|_{L^{\infty}}+\frac{1}{a(1-a)}\sup _{t>0} t^{\frac{1+a}{2}}\|F(t)\|_{\dot{C}^{a}},\ \ \quad\forall a\in(0,1).\label{heatb}
		\end{align}
		The estimate \eqref{heatb} can be extended to equations with smooth coefficients. 
		Consider
		\begin{equation}\label{heatvar}
			\begin{aligned}
				&\partial_t u-\operatorname{div}(\A(x)\nabla u) =\operatorname{div}(F),\ \ \text{in} \ (0,\infty)\times \mathbb{R}^d,\\
				&u|_{t=0}=u_0,
			\end{aligned}
		\end{equation}
		where the coefficient function $\A$ is assumed to be H\"{o}lder continuous. By the fundamental observation that equations with H\"{o}lder continuous coefficients can be treated locally as a 
		perturbation of constant coefficient equations, we can use the classical freezing coefficient method to reduce \eqref{heatvar} to \eqref{heat}. However, this approach inherently yields estimates that are only valid locally in time:
		\begin{align}
			&\sup _{t\in[0,T]}(\|u(t)\|_{L^{\infty}}+ t^{\frac{1+a}{2}}\|\nabla u(t)\|_{\dot{C}^{a}} )\lesssim \left\|u_{0}\right\|_{L^{\infty}}+\frac{1}{a(1-a)}\sup _{t\in[0,T]} t^{\frac{1+a}{2}}\|F(t)\|_{\dot{C}^{a}},\ \ \quad\forall a\in(0,1).\label{heatloc}
		\end{align}
		Here, the time horizon $T>0$ depends critically on $\|\A\|_{\dot{C}^{a}}$, reflecting the intrinsic dependence of parabolic regularity on coefficient smoothness. The necessity for a finite $T$ arises from the need to control lower-order terms generated during the coefficient freezing process.
		
		To extend Schauder theory beyond classical local operators and accommodate critical applications in modern fluid dynamics, particularly those involving nonlocal effects or system coupling, we develop a unified methodology to derive sharp  priori estimates for general linear parabolic systems, mirroring the global/local regularity characteristics (\eqref{heat} and \eqref{heatloc}) seen in the classical heat equation case.
		\subsection{Schauder-type estimates for linear parabolic systems}	\label{secsch}
	Let $T\in(0,+\infty]$. We investigate the non-local parabolic system governing vector-valued functions $u:(0,T)\times \mathbb{R}^d\to \mathbb{R}^N$,
		\begin{equation}\label{eqpara}
			\begin{aligned}
				&	\partial_{t} u(t, x)+\mathcal{L}_{s} u(t, x)=f(t, x) \quad \text { in } (0, T)\times\mathbb{R}^{d},\\
				&	u|_{t=0}=u_0,
			\end{aligned} 
		\end{equation}
		with parameter $s>0$. 
		
		To establish precise mathematical foundations, we formalize the following parameter specifications and operator conditions:\vspace{0.2cm}\\
		1. {\bf Regularity Parameters:} We fix two parameters $0<b< s$. The parameter $s$ denotes the order of the underlying linear operator. The choice of
		$b$ is motivated by its role as the critical regularity index in the target quasilinear equation.\vspace{0.2cm}\\
		For the proof concerning Schauder-type estimates of the linear system \eqref{eqpara}, we additionally select parameters $m,\kappa $ such that  $\kappa,\kappa+b,\kappa+b-s\notin\mathbb{N}$ and  \begin{align}
			\label{defpara}
			m\in\mathbb{N}_+, \ \ \ \ s-b<\kappa<s.
		\end{align}
     For simplicity, denote $\kappa_0=b-s+\kappa$. \vspace{0.2cm}\\
		2. {\bf Forcing Terms:} The external input $f:(0,T)\times \mathbb{R}^d\to \mathbb{R}^N$ is known force term satisfying $\da_T(f)<\infty$, where
		\begin{align*}
			\da_T(f)=  \sup _{t \in[0,T]} \left(t^{\frac{\kappa}{s}}\|f(t)\|_{\dot{C}^{\kappa_0}(\mathbb{R}^d)}+t^{\frac{m+\kappa}{s}}\|f(t)\|_{\dot{C}^{m+\kappa_0}(\mathbb{R}^d)}\right).
		\end{align*}
		3. {\bf Principal Operator:} The pseudo-differential operator $\mathcal{L}_{s}$  of order $s>0$ is defined by
		\begin{align}\label{defop}
			&\mathcal{L}_su(t,x)=(2\pi)^{-\frac{d}{2}}\int_{\mathbb{R}^d}\A(t,x,\xi)\hat u(t,\xi)e^{ix\cdot\xi}d\xi,\quad\quad\quad
		\end{align}
		where $\hat u(t,\xi)$  denotes the Fourier transform of $u(t,x)$ with respect to the spatial variable $x$, $\A(t,x,\xi)\in \mathbb{R}^{N\times N}$ is a  matrix satisfying 
		\begin{equation}\label{condop}
			\begin{aligned}
				&\quad\A(t,x,\xi)\geq c_0|\xi|^s\mathrm{Id},\\
				& \sum_ {l\leq d+s+2 } {|\xi|^{l-s}}{\left|\nabla^{l}_\xi \A(t,x,\xi)\right| }\leq c_1, \ \ \ \forall  \xi \neq 0, (t,x)\in(0,T)\times\mathbb{R}^d,
			\end{aligned} 
		\end{equation}
		for some constants $0<c_0<c_1$. Here $\mathrm{Id}$ is the $N\times N$ identity matrix, and for any  $Q_1, Q_2\in\mathbb{R}^{N\times N}$, we write  $Q_1\geq Q_2$ if $Q_1-Q_2$ is coercive. 
        The operator $\mathcal{L}_s u(t,x)$ can also be written in the shorthand notation:
\begin{equation*}
\mathcal{L}_s u(t,x) = \A(t,x,\nabla) u(t,x).
\end{equation*}
        Moreover, we have the following regularity assumption for $\A$:
		\begin{equation}\label{defnormA}
			\|\A\|_{B_T}=\sup_{t\in[0,T]}\sup_{\xi\in\mathbb{R}^d}\sum_{l\leq d+s+2}\left(t^{\frac{\eta}{s}}\frac{\|\nabla_\xi^{l}\A(t,\cdot,\xi)\|_{\dot C^{\eta}}}{|\xi|^{s-l}}+t^{\frac{m+\kappa}{s}}\frac{\|\nabla_\xi^l\A(t,\cdot,\xi)\|_{\dot C^{m+\kappa}}}{|\xi|^{s-l}}\right)<\infty,
		\end{equation}
        where $\eta$ is a fixed small parameter that satisfies 
        \begin{equation}	\label{defparab}
            		 \begin{aligned}
	&	0<\eta\leq \frac{\min\{s-b,\kappa+b-s-[\kappa+b-s],\kappa+b-[\kappa+b]\}}{4},\\&
	 0<\eta\leq \frac{b-[b]}{4}\quad~\text{if}~\quad b\notin \mathbb{N},
	\end{aligned}
        \end{equation}
        where $[b]:=\max\{n: n\in\mathbb{N},n\leq b\}$. \\
        4. {\bf Scaling critical time weighted H\"{o}lder norm:} Define  the following semi-norms:
		\begin{equation}\label{nor}
			\begin{aligned}
				&\|u\|_{T}=	\sup_{t\in[0,T]} \left(\|u(t)\|_{\dot C^{b}}+t^\frac{m+\kappa }{s}\|u(t)\|_{\dot C^{m+\kappa +b}}\right),\\
				&\|u\|_{T,*}=	\sup_{t\in[0,T]} \left(t^{\frac{\eta}{s}}\|u(t)\|_{\dot C^{b+\eta}}+t^\frac{m+\kappa }{s}\|u(t)\|_{\dot C^{m+\kappa +b}}\right),
			\end{aligned}
		\end{equation}
		where $\eta$ is a fixed parameter that satisfies \eqref{defparab}. By interpolation inequality, we have $\|u\|_{T,*}\lesssim \|u\|_{T}$. 
        
		\vspace{0.2cm}
		We now present the Schauder-type estimate for the parabolic system \eqref{eqpara}. 
		\begin{theorem}\label{lemmain}
	Let $T>0$ and  $\|\A\|_{B_T}<\infty$. Suppose $u$ is a solution to the Cauchy problem \eqref{eqpara} in $C((0,T],\dot C^{b}_x(\mathbb{R}^d))\cap L^\infty_{loc}((0,T],\dot C^{m+\kappa +b}_x(\mathbb{R}^d))$ with $\A(t,x,\xi)$ satisfying \eqref{condop} and \eqref{defnormA}, then the following estimates hold:
            		\begin{align}\label{main3}
				&\|u\|_T\lesssim \|u_0\|_{\dot C^b}+  \da_T(f)+\|u\|_{T,*}\|\A\|_{B_T},\\
             &\|u\|_{T,*}\lesssim \|u_0\|_{\dot B^b_{\infty,\infty}}+  \da_T(f)+\|u\|_{T,*}\|\A\|_{B_T},   \label{main4}
			\end{align}
            where the implicit constants  depend only on $d,s,m,\kappa,b,\eta,c_0,c_1$.
            In particular, there exists $\eps_0>0$ independent of $T$ such that if $\|\A\|_{B_T}\leq \varepsilon_0$, then 
            \begin{align*}
				&\|u\|_T\lesssim \|u_0\|_{\dot C^b}+  \da_T(f),\\
             &\|u\|_{T,*}\lesssim \|u_0\|_{\dot B^b_{\infty,\infty}}+  \da_T(f).
			\end{align*}
		\end{theorem}
        \begin{remark}\label{extmthm}
If $\|\A\|_{B_T}<\infty$ and $\|\A\|_{B_{T'}}$ is small for some $0<T'<T$, the existence of solution to the Cauchy problem \eqref{eqpara} in $[0,T]$ follows from the standard compactness method, see the proof of Proposition \ref{propb=0} for details.
        \end{remark}
		\begin{remark}
			Theorem \ref{lemmain} can extend to $u:(0,+\infty)\times\mathcal{M}\to \mathbb{R}^N$, where $\mathcal{M}$ is a smooth manifold without boundary, see Theorem \ref{thmmani}.
		\end{remark}
		\subsection{Applications to quasilinear evolution equations}
		Theorems \ref{lemmain} establishes broad applicability to parabolic systems. To demonstrate the core strategy for proving local and global well-posedness, we consider a general quasilinear evolution system for $U: [0,T]\times\mathbb{R}^d \to \mathbb{R}^N$:
		\begin{equation}\label{besovgen}
			\partial_t U(t,x) + A(\nabla U(t,x), \nabla)U(t,x) = \mathcal{N}(U)(t,x),
		\end{equation}
		 the symbol $A$ satisfies $A(Z, \xi) \sim_{Z} |\xi|^s$ for some $s > 1$ and all $Z \in \mathbb{R}^{d \times N}$. We assume that the Lipschitz space constitutes a critical space for this problem, which corresponds to $b=1$ in section \ref{secsch}.
		
		Fix parameters $m,\kappa,\eta$ satisfying \eqref{defpara}. For $T>0$, we define the following time-dependent critical semi-norms:
        \begin{equation}\label{musdefnorm}
		\begin{aligned}
			\|f\|_{T} &= \sup_{t \in [0,T]} \left( \|\nabla f(t)\|_{L^\infty} + t^{\frac{m+\kappa}{s}} \|\nabla f(t)\|_{\dot C^{m+\kappa}} \right), \\
			\|f\|_{T,*} &= \sup_{t \in [0,T]} \left( t^{\frac{\eta}{s}} \|\nabla f(t)\|_{\dot C^\eta} + t^{\frac{m+\kappa}{s}} \|\nabla f(t)\|_{\dot C^{m+\kappa}} \right),
		\end{aligned}
        \end{equation}
        which coincides with the definition \eqref{nor} with $b=1$.\vspace{0.2cm}\\        
		The nonlinear term $\mathcal{N}(U)$ and the coefficient $A(Z,\xi)$ are subject to the following assumptions:\vspace{0.1cm}\\
		{\bf Regularity of the nonlinear term.}
		There exist $\epsilon_0, c_1 > 0$ such that for any $T>0$, and any function $U: [0,T]\times\mathbb{R}^d \to \mathbb{R}^N$, it holds
		\begin{equation}\label{besovno}
			\begin{split}
			\da_T(\mathcal{N}(U)):=	&  \sup_{t \in [0,T]} \Bigl( t^{\frac{\kappa}{s}} \|\mathcal{N}(U)(t)\|_{\dot C^{\kappa-s+1}} + t^{\frac{m+\kappa}{s}} \|\mathcal{N}(U)(t)\|_{\dot C^{m+\kappa-s+1}} \Bigr) \\
				\lesssim& \|U\|_{T,*}^{1+\epsilon_0} (1 + \|\nabla U\|_{L^\infty_T L^\infty_x} + \|U\|_{T,*})^{c_1}.
			\end{split}
		\end{equation}
		{\bf Regularity of coefficients.} 
		Assume $A(Z,\xi)$ is smooth in $Z$, and there exists $c_2>1$ such that 
		\begin{align*}
			A(Z,\xi)\geq (1+|Z|)^{-c_2}|\xi|^s\mathrm{Id},\quad\quad \ \ \sum_{l\leq d+s+2} |\xi|^{l-s}|\nabla_\xi^lA(Z,\xi)|\leq (1+|Z|)^{c_2}.
		\end{align*}
		Moreover, the operator symbol $\A(t,x,\xi) := A(\nabla U(t,x),\xi)$ satisfies the condition \eqref{condop}, and there exists a constant $c_3>0$ such that
		\begin{align}\label{ABT}
			\|\A\|_{B_T} \lesssim (1 + \|\nabla U\|_{L^\infty_T L^\infty_x})^{c_3} \|U\|_{T,*},
		\end{align}
		where $\|\cdot\|_{B_T}$ is defined in \eqref{defnormA}. \vspace{0.2cm}\\
		Under this analytical framework, we establish the following a priori estimate.
		\begin{proposition}\label{prop111}
			There exist $\varepsilon_0>0$ and $c=c(c_1,c_2,c_3)>0$ such that for any $\eps\in(0,\eps_0)$, if the initial data 
			$U_0$ satisfies 
			\begin{align}\label{bdini}
				(1 + \|\nabla U_0\|_{L^\infty})^c \|\nabla (U_0 - \Phi)\|_{\dot B^0_{\infty,\infty}}\leq \varepsilon,
			\end{align}
            for some  $\Phi \in C^{m+4}(\mathbb{R}^d)$, 
			 then there exists $T_0=T_0(\|\Phi\|_{C^{m+s+1}})$ such that for any solution $U$ to \eqref{besovgen} with $\|U\|_{T_0,*}<\infty$, it holds
            \begin{equation}\label{Ubes}
				\begin{aligned}
					&\|\nabla U\|_{L^\infty_{T_0}L^\infty}\leq C(\eps+ \|\nabla U_0\|_{L^\infty}),\\
					&\|U-\Phi\|_{T_0,*}\leq C  \eps.
				\end{aligned}
			\end{equation}
			In particular, if $\Phi\equiv 0$, then \eqref{Ubes} holds globally $(T_0=\infty)$.
		\end{proposition}
		The Proposition \ref{prop111} establishes a priori estimate to the quasilinear system \eqref{besovgen}. To obtain the existence of solution, we use contraction mapping theorem by defining a map $\mathcal{S}:V\to U$, where $U$ is the solution to the linear system 
		\begin{align*}
			\partial_t U(t,x) + A(\nabla \Phi(x), \nabla)U(t,x) = \mathcal{N}(V)(t,x)+F(V,\Phi)(t,x),
		\end{align*}
		where $F(V,\Phi)= A(\nabla \Phi, \nabla)V- A(\nabla V, \nabla)V$. It is straightforward to check that a fixed point of $\mathcal{S}$ is a solution to the original system \eqref{besovgen}. Under the condition 
		\begin{align*}
			(1 + \|\nabla U_0\|_{L^\infty})^c \|\nabla (U_0 - \Phi)\|_{L^\infty} \ll 1,
		\end{align*}
		for some $c\geq 1$, and provided that the nonlinear operator $\mathcal{N}$ is Lipschitz continuous in the relevant function space (i.e., 
$\|\mathcal{N}(U_1)-\mathcal{N}(U_2)\| \leq \frac{1}{2} \|U_1-U_2\|$), 
we prove that the map $\mathcal{S}$ has a unique fixed point in a suitable set. This implies the existence of a solution with initial data satisfying \eqref{bdini} by the compactness method. 

 We also note that the critical rough-data viewpoint adopted here is related in spirit to our earlier works on the one-dimensional compressible Navier--Stokes system \cite{CHN-CNS-local,CHHN-CNS-global}. Although the underlying equations are of a different nature, both settings rely on time-weighted critical estimates and on a linear-to-nonlinear framework adapted to low-regularity data.

		Building on the Schauder-type estimates presented in Theorem \ref{lemmain}, we prove well-posedness theory to the following non-local quasilinear equations arising from fluid dynamics, for which the linearized equation has the form \eqref{eqpara}.\vspace{0.2cm}\\
		{\bf The 2D Muskat equation with surface tension}\\ 
		The Muskat equation, first introduced in the pioneering works of Darcy \cite{DF96} and Muskat \cite{Mus34}, describes the motion of two incompressible and immiscible fluids in a porous medium. We consider the particular case where the fluids have equal viscosities and assume that the flows are two-dimensional. The free boundary separating the fluids
		is described by the evolution equation
		\begin{equation}\label{eqmst}
			\begin{aligned}
				&\partial_t f(x)=\frac{\mathbf{k}}{2\pi\mu}\mathrm{P.V.} \int_{\mathbb{R}}\frac{1+\partial_x f(x)\Delta_\alpha f(x)}{\langle\Delta_\alpha f(x)\rangle^2}\partial_x\left(\sigma'\kappa(f)- \varrho_0 f\right)(x-\alpha)\frac{d\alpha}{\alpha},\\
				&f|_{t=0}=f_0,
			\end{aligned}
		\end{equation}
		where  $\mathbf{k}, \mu, \sigma'$ are positive constants that denote the permeability of the
		homogeneous porous medium,  the viscosity of the fluids and the surface tension
		coefficient respectively, $\kappa(f)=\frac{\partial_x^2f}{\langle \partial_xf \rangle^3}$ denoting the curvature of the graph $\{y=f(t,x)\}$, and 
		\begin{align}\label{defvr0}
			\varrho_0=\mathbf{g}(\rho_--\rho_+),
		\end{align}
		where $\mathbf{g}$ denotes the Earth’s gravity and $\rho_+,  \rho_-$ denote the density of the upper fluid and lower fluid.   We remark that  $\varrho_0>0$ corresponds to the stable regime (heavier fluid below), and $\varrho_0<0$ corresponds to the unstable regime  (heavier fluid up).
		Moreover, to shorten the notation we denote $\langle A\rangle=(1+A^2)^\frac{1}{2}$, and 
		$$
		\delta_\alpha f(x)=f(x)-f(x-\alpha),\quad\quad \ \ \Delta_\alpha f(x)=\frac{\delta_\alpha f(x)}{\alpha}.
		$$
		We refer readers to \cite{1Matioc2018} for a proof of equivalence of \eqref{eqmst} to the classical formulation of the Muskat problem. 
		
		The Muskat problem has attracted much interest in the last decades. From a mathematical perspective, the Muskat problem is governed by a nonlinear degenerate parabolic equation of fractional order, which encapsulates the interplay of diffusion, convection, and gravitational effects on the evolving interface. The equation's degenerate nature and nonlocal structure pose significant analytical challenges, making it a rich subject for both theoretical and computational studies in applied mathematics.
 Over the years, it has attracted considerable attention from both the analytical and numerical communities, and has been shown to exhibit a rich variety of behaviors, including turnover of the interface leading to loss of regularity, as established by Castro {\em et al.} \cite{RT2012, CCFG}, stability shifting between stable and unstable regimes, as shown by C\'ordoba, G\'omez-Serrano, and Zlato\v{s} \cite{CGZ2015, CGZ2017}, and splash singularities in the one-phase setting \cite{CCFG2016}. We also mention the work of Shi \cite{Shi:regularity-muskat,Shi2024}, which proves analyticity for solutions whose interfaces have already turned over. For broader background on the Muskat problem, we refer the reader to the surveys by Gancedo and by Granero-Belinch\'{o}n--Lazar \cite{Gancedo:survey,GBL2020}.

As for well-posedness, substantial progress has been made in the subcritical regime in recent years. Local-in-time well-posedness of the Cauchy problem in subcritical Sobolev spaces, as well as global existence for small initial data, is now fairly well understood. Early results in high-regularity Sobolev spaces were obtained by Yi \cite{Yi}, Ambrose \cite{Amb04,Amb07}, and Caflisch--Howison--Siegel \cite{SCH04}, the latter also identifying ill-posedness in certain unstable regimes. Later, D. C\'ordoba and Gancedo \cite{CG07,CG09} established well-posedness in the infinite-depth setting without viscosity jump in the space \(H^{d+2}(\mathbb{R}^d)\), for \(d=1,2\). This result was subsequently extended, together with A. C\'ordoba, to include viscosity jump and non-graphical interfaces \cite{CCG2011, CCG2013}. A significant step toward lower regularity was achieved by Cheng, Granero-Belinch\'on, and Shkoller \cite{CGS2016}, who reduced the regularity threshold in two dimensions to \(H^2\), without relying on the contour equation and therefore allowing for more general domains. Since then, a number of works have been devoted to pushing the theory toward barely subcritical spaces.

In the constant-viscosity case, Constantin, Gancedo, Shvydkoy, and Vicol \cite{CGSV} constructed solutions for initial data in \(W^{2,p}\), \(p\in(1,\infty]\). Matioc obtained local well-posedness results for data in \(H^2\) and \(H^{3/2+\epsilon}\) \cite{Matio2018, 1Matioc2018}, while Abels and Matioc \cite{Am} treated the subcritical \(L^p\)-Sobolev range. Alazard and Lazar \cite{Alazard-Lazar} further extended the theory to allow data outside the \(L^2\)-framework. Most recently, H. Q. Nguyen and Pausader \cite{1HuyNguyen2020} proved well-posedness for the full \(d\)-dimensional Muskat problem with or without bottom and with or without viscosity jump in \(H^s(\mathbb{R}^d)\) for all \(s>d/2+1\).

We now turn to the Muskat problem at critical regularity. The first results in this direction concerned small-data solutions in the Wiener algebra \(\mathcal{L}_{1,1}\), which consists of functions whose first derivative is integrable in Fourier space. Constantin, C\'ordoba, Gancedo, and Strain \cite{1Constantin2010} established the existence of global small-data solutions without viscosity jump, and later, together with Piazza, extended this result to the three-dimensional setting \cite{CCG2016}. This was further improved by Gancedo, Garc\'ia-Ju\'arez, Patel, and Strain \cite{GGPS}, who constructed global small-data strong solutions allowing viscosity jump in both two and three dimensions.

There has been substantial recent progress in the critical Sobolev space \(\dot{H}^{3/2}\). The first important result is due to C\'ordoba and Lazar \cite{Cordoba-Lazar-H3/2}, who, although working in the subcritical space \(\dot{H}^{5/2}\), derived a priori estimates at the critical level. This was later extended to the three-dimensional setting by Gancedo and Lazar \cite{1Gancedo2020Global}. The first fully critical result was obtained by Alazard and the third author, who constructed two-dimensional solutions with initial data in \(H^{3/2}\cap \dot{W}^{1,\infty}\) \cite{Alazard2020}, together with the log-subcritical work \cite{AN2021}, where unbounded slopes are allowed. Depending on the size of the data, these results yield either local-in-time well-posedness for large data or global-in-time well-posedness for small data. In subsequent work, they removed the \(L^\infty\)-assumption altogether and obtained well-posedness directly in \(H^{3/2}\) \cite{AN2020endpoint}. They later extended this analysis to the three-dimensional setting in \(\dot{H}^{2}\cap W^{1,\infty}\) \cite{TH4}.

Further critical small-data results include the work of  Cameron, who first studied well-posedness in \(\dot{W}^{1,\infty}\cap L^2\) in two dimensions \cite{Cameron2019}, and later in three dimensions under merely sublinear growth at infinity \cite{Cameron2020}. A notable feature of Cameron's results is that they allow ``medium-sized'' initial data, in the sense that the slope is assumed to be bounded by \(1\), rather than by a perturbative small constant \(\eps\). H. Q. Nguyen \cite{HuyNguyen2021} established well-posedness in the Besov space \(\dot{B}_{\infty,1}^1\), which is embedded in the critical space \(\dot{W}^{1,\infty}\). For genuinely large data, global regularity is not expected in general, due to the possible formation of singularities. For $d=1$ and under the assumption that the initial interface is monotone, Deng--Lei--Lin \cite{DengLei2017} constructed global weak solutions. The local well-posedness in the critical H\"older space $C^1$ was established by the first and the third author and Xu \cite{KeC1}. In the more recent work \cite{CCHNX}, the local existence of classical solutions was established in regimes where the product of maximal and minimal slopes is allowed to be large for any dimension. Moreover,
in \cite{GGNP}, Garc\'ia-Ju\'arez, G\'omez-Serrano, H. Q.  Nguyen and Pausader established the existence of small self-similar solutions emanating from exact corner singularities, which become instantaneously regularized. We also refer to \cite{GGHP} for the desingularization of small moving corners.

		The introduction of surface tension ($\sigma' > 0$) fundamentally alters this paradigm. It is well known that the regularizing effect of surface tension bypasses the Rayleigh-Taylor
		stability condition required for well-posedness of free boundary problems in the absence of surface tension. The well-posedness of the Muskat problem with surface tension effects has been investigated in periodic geometries in \cite{Escher11,Escher12,Escher18}. For unbounded settings, H. Q. Nguyen \cite{HNguyen} established local existence for arbitrary initial data in the largest $L^2$-based subcritical Sobolev spaces $H^{s}(\mathbb{R}^d)$ ($s>1+\frac{d}{2}$), covering general settings including single/two-phase flows, arbitrary viscosity jumps, and all dimensions while permitting unbounded initial curvature.
		Subsequently, $L^p$-extensions were achieved by A.-V. Matioc and B.-V. Matioc \cite{MATIOC2022} through quasilinear parabolic theory, proving local well-posedness and instantaneous smoothing in subcritical $W_p^s(\mathbb{R})$ for $p \in (1,2]$ and $s \in (1+1/p, 2)$. More recently, Lazar \cite{Lazar24} proved a global well-posedness result for sufficiently regular initial data that are small in a critical norm, while possibly being arbitrarily large in Lipschitz. The proof relies on a new reformulation of the equation and on important cancellation mechanisms that allow one to close the key a priori estimates at the critical level.  Note that for the case $\varrho_0=0$, the equation \eqref{eqmst} is invariant under the scaling transform 
		$$
		f_\lambda(t,x)=\lambda^{-1}f(\lambda ^3t,\lambda x),\ \ \ \forall \lambda>0.
		$$
		Hence, $\dot W^{1,\infty}(\mathbb{R})$ and $\dot H^{\frac{3}{2}}(\mathbb{R})$ are two examples of scaling critical spaces.

		In this paper, we prove the first result on local and global well-posedness of the Muskat equation with surface tension in the largest critical space $\dot B^{1}_{\infty,\infty}$. 	For brevity, fix $m\in\mathbb{N}_+,m\geq 3$ and $0<3-\kappa\ll 1$. Denote 
		\begin{equation}\label{normmst}
			\begin{aligned}
				&	\|h\|_{T}:=\sup_{t\in[0,T]}(\| \partial_xh(t)\|_{L^\infty}+t^\frac{m+\kappa}{3}\| \partial_xh(t)\|_{\dot C^{m+\kappa}}),\\
				&\|h\|_{T,*}:=\sup_{t\in[0,T]}(t^\frac{1}{15}\| \partial_xh(t)\|_{\dot C^\frac{1}{5}}+t^\frac{m+\kappa}{3}\| \partial_xh(t)\|_{\dot C^{m+\kappa}}),\\
				&\|h\|_{X_T}:=\|h\|_{L^\infty_TL^\infty}+\|h\|_T.
			\end{aligned}
		\end{equation}
		The main result is the following.
		\begin{theorem}\label{thmMusbes}
			There exists $\varepsilon_0>0$ such that for any $0<\eps<\varepsilon_0$, the following statements hold.	~\\
			i) ($ \varrho_0=0$) For any initial data $f_0 \in \dot W^{1,\infty}$ satisfying the smallness condition 
			\begin{align}\label{bound1}
				(1+\| f_0\|_{\dot W^{1,\infty}})^{4(m+5)}\|f_0\|_{\dot B^{1}_{\infty,\infty}}\leq \varepsilon,
			\end{align} the equation \eqref{eqmst} admits a unique global solution $f$ in the class 
			$$\mathcal{X}:=\{f\in L^\infty_{loc}((0,\infty),C^{m+\kappa}(\mathbb{R})): \|f\|_{\infty}\leq C(\eps+\|f_0\|_{\dot W^{1,\infty}}), \|f\|_{\infty,*}\leq C \varepsilon\}.$$
			Moreover, the solution exhibits stability: for any two solutions $f,g\in\mathcal{X}$ corresponding to initial data $f_0, g_0$ satisfying \eqref{bound1}, it holds
			\begin{align*}
				\sup_{t>0}\|(f-g)(t)\|_{\dot W^{1,\infty}}\leq C\|f_0-g_0\|_{\dot W^{1,\infty}}.
			\end{align*}
			ii) ($ \varrho_0\in\mathbb{R}$) For any initial data $f_0 \in W^{1,\infty}$, if  
			\begin{align}\label{bound2}
				(1+\|f_0\|_{ W^{1,\infty}})^{4(m+5)}\| f_0-\phi\|_{\dot B^{1}_{\infty,\infty}}\leq \varepsilon,
			\end{align}
         holds for some    $\phi \in C^{m+5}(\mathbb{R})$, 
			then   there exists $T>0$ such that \eqref{eqmst} admits a unique solution  $f$ in the class 
			$$\mathcal{X}_{T,\phi}:=\{f\in L^\infty_{loc}((0,T),C^{m+\kappa}(\mathbb{R})): \|f-\phi\|_{X_T}\leq C(\eps+\|f_0\|_{W^{1,\infty}}), \|f\|_{T,*}\leq C \varepsilon\}.$$
			Moreover, we have the stability property: For any two solutions $f,g\in\mathcal{X}_{T,\phi}$ corresponding to initial data $f_0, g_0$ satisfying \eqref{bound2}, it holds
			\begin{align*}
				\sup_{t\in[0,T]}\|(f-g)(t)\|_{W^{1,\infty}}\leq C\|f_0-g_0\|_{W^{1,\infty}}.
			\end{align*}
		\end{theorem}
		\begin{remark}
			When $\varrho_0 < 0$, our approach fails to establish global existence. This limitation stems from the fact that $\Lambda^3 f$ fails to dominate $\Lambda f$ at low frequencies, which is consistent with Rayleigh-Taylor instability mechanisms \cite{CCG2011}.
		\end{remark}
		\begin{remark}
			Theorem \ref{thmMusbes} yields the following consequences:
            \begin{itemize}
                \item i) Global well-posedness for small Lipschitz data when $\varrho_0 = 0$;
                \item  ii) Local well-posedness for arbitrary $C^1$ data when $\varrho_0 \in \mathbb{R}$.
            \end{itemize}
            These follow from the observations that $\|f_0\|_{\dot{W}^{1,\infty}} \ll 1$ ensures condition \eqref{bound1}, while $f_0 \in C^1(\mathbb{R})$ guarantees the existence of a function $\phi$ satisfying \eqref{bound2}.
		\end{remark}
    
To study the well-posedness of the problem \eqref{eqmst}, we  rewrite \eqref{eqmst} in a quasilinear parabolic form where the surface tension produces the principal dissipative term. More precisely, after decomposing the nonlinear integral operator, we write the equation as
\begin{align}\label{musrefo}
\partial_t f+\frac{\Lambda^3 f}{\langle \partial_x f\rangle^3}=N[f],
\end{align}
where $N[f]$ consists only of lower-order nonlinear terms. Thus, the principal part is given by a third-order nonlocal quasilinear parabolic operator whose coefficients depend on $\partial_x f$. Hence, the equation fits into the framework of Theorem \ref{lemmain} and Proposition \ref{prop111}.

The proof is then divided into three steps. First, we establish a relaxed Lipschitz version of the theorem, namely Proposition \ref{propLip}, where the Besov smallness assumptions are replaced by stronger Lipschitz smallness assumptions. In this regime, the coefficient $\langle \partial_x f\rangle^{-3}$ is easily controlled, and the nonlinear remainder $N[f]$ can be estimated perturbatively in the time-weighted H\"older norms associated with the scaling of the equation. Applying the Schauder-type estimate for the linearized operator then yields the required a priori bounds.

In the second step, we pass from the Lipschitz-type statement to the Besov-type result of Theorem \ref{thmMusbes}. The key point here is that the a priori estimates only require smallness of the critical norm measuring the deviation from a smooth profile, while the solution itself becomes instantly smoother for positive time. One therefore constructs approximate solutions with smooth initial data, applies the estimates obtained in the first step uniformly, and then passes to the limit by a compactness argument. This yields existence both in the global case $\varrho_0=0$ and in the local case $\varrho_0\in\mathbb{R}$, depending on whether one takes $\Phi\equiv 0$ or a general smooth reference profile $\Phi$.

Finally, one proves a stability estimate for the difference of two solutions in $W^{1,\infty}$. This is done by writing the equation for the difference and exploiting again the fact that the highest-order part is parabolic and coercive, while all commutator and remainder terms are lower order. The resulting estimate gives continuous dependence on the initial data and, in particular, uniqueness of the solution.

The distinction between the global case $\varrho_0=0$ and the local case $\varrho_0\in\mathbb{R}$ comes from the low-frequency behavior of the linear part. When $\varrho_0=0$, the leading cubic dissipation $\Lambda^3$ controls the dynamics globally in time under the small critical assumption. By contrast, for general $\varrho_0$, and especially when $\varrho_0<0$, the lower-order gravitational contribution cannot in general be dominated at low frequencies, so the argument yields only local well-posedness.

A number of difficulties arise in the proof of Theorem \ref{thmMusbes}. The first one is that, although surface tension yields a third order dissipative term, the equation remains fully nonlocal and quasilinear: after rewriting \eqref{eqmst} in the form
\eqref{musrefo}, 
the coefficient of the principal operator still depends on the solution through \(\partial_x f\), while the remainder \(N[f]\) consists of singular nonlocal terms that must be shown to be genuinely lower order. Thus one has to separate the cubic parabolic part from the nonlinear remainders in a way compatible with the time-weighted H\"older norms used in the Schauder estimates. A second difficulty is that our goal is to work in the critical space \(\dot B^1_{\infty,\infty}\), which is weaker than \(W^{1,\infty}\). In particular, the Besov smallness does not directly provide the Lipschitz control needed to run a standard fixed-point argument, so one must first prove a relaxed Lipschitz version of the theorem and then recover the critical Besov result by approximation and compactness. Finally, when \(\varrho_0\neq 0\), the lower-order gravitational contribution affects the low-frequency behavior of the equation. In particular, for \(\varrho_0<0\), the cubic dissipation \(\Lambda^3\) no longer dominates the order-one term \(\Lambda\) at low frequencies, which explains why the present argument yields only local well-posedness in the general case and does not provide global existence in the unstable regime.

		Our methodology is not applicable to the Muskat equation in the absence of surface tension (i.e., $\sigma'=0$ with $\varrho_0>0$). The primary difficulty stems from an inability to transfer derivatives from the term $\partial_x f(x)-\Delta_\alpha f(x)$ to lower-order components. This obstruction precludes closing the a priori estimates under Lipschitz initial data regularity. We note that this limitation can be circumvented using alternative approaches, as demonstrated in \cite{KeC1}. For analogous reasons concerning derivative transfer, our framework cannot be directly applied to the Muskat equation with different viscosities; further discussion is provided in Section \ref{secfu}. \vspace{0.3cm}\\
		\textbf{The Peskin Problem in 2D and 3D}\\
		We consider the problem of an immersed elastic structure interacting with a Stokesian fluid: specifically, a one-dimensional elastic string immersed in a two-dimensional Stokes fluid, or a two-dimensional elastic membrane immersed in a three-dimensional Stokes fluid. This fundamental model, known as the 2D/3D Peskin problem, originated in the work of Peskin~\cite{PeskinFlow1972,Peskin1972Thesis} for simulating blood flow around heart valves. The Immersed Boundary Method (IBM) developed for this purpose has since become a cornerstone technique for fluid-structure interaction (FSI) problems, finding extensive applications across physics, biology, and medical science~\cite{PeskinImmersed2002,MittalImmersed2005,HouNumerical2012}.
		
		Despite its physical significance and practical utility, rigorous mathematical analysis of the Peskin problem remains challenging. The model's inherent nonlinearity, non-local character, and the presence of singular forces arising from the immersed boundary pose significant difficulties. Unlike fluid interface problems where surface tension dominates, the elastic force density here depends critically on the stretching of the structure, governed by its parametrization. This dependence introduces an additional layer of complexity that complicates analysis.
		
		We begin with the 2D formulation. Let $\Gamma(t)$ be a simple closed curve partitioning $\mathbb{R}^2$ into an interior domain $\Omega^+$ and an exterior domain $\Omega^- = \mathbb{R}^2\backslash\Omega^+$. The curve is parameterized by $X(t,x) = (X_1(t,x), X_2(t,x))$, where $x \in \mathbb{S} := \mathbb{R}/(2\pi\mathbb{Z})$ is the material coordinate and $t \geq 0$ denotes time. The position $X(t,x)$ moves with the local fluid velocity. The elastic force density on the boundary is given by:
		\begin{equation*}
			F(X) = \partial_x \left( \mathcal{T}(|\partial_x X|) \tau(X) \right),
		\end{equation*}
		where $\tau(X) = \frac{\partial_x X}{|\partial_x X|}$ is the unit tangent vector and $\mathcal{T}:\mathbb{R}^+\to \mathbb{R}^+$ is a smooth elastic tension function satisfying 
	\begin{equation}\label{conten}
	       \begin{aligned}
	&	\inf_{r\in [0,\lambda]}\min\left\{\mathcal{T}'(r),\frac{\mathcal{T}(r)}{r}\right\}\geq 	\mathfrak{c}(\lambda)>0,\\
       & \sup_{r\in [0,\lambda]}|\mathcal{T}^{(k)}(r)|\lesssim_k \mathfrak{C}(\lambda)<\infty,\ \quad\quad\ \ \forall \lambda>0,\ \ k\in\mathbb{N}.
		\end{aligned}
	\end{equation}
     Denoting the fluid velocity by $u$ and pressure by $p$, the coupled Peskin system is:
		\begin{align*}
			\begin{cases}
				-\Delta u =-\nabla p~~~~~&\text{in} ~\mathbb{R}^2\backslash\Gamma(t),  \\
				\operatorname{div} u =0 ~~~~\quad\quad&\text{in} ~\mathbb{R}^2\backslash\Gamma(t),\\
				\llbracket u \rrbracket =0 ~~~~\quad\quad&\text{on} ~\Gamma(t),\\
				\llbracket\left(\nabla u+\nabla u^{\top}-p \mathrm{Id}\right) n\rrbracket=\frac{F(X)}{|\partial_x X|}~~~&\text{on}~\Gamma(t),\\
				\partial_tX=u~~~~\quad\quad&\text{on}~ \Gamma(t).
			\end{cases}	
		\end{align*}
		Here $n$ is the outward unit normal to the free boundary $\Gamma(t)$ and $\llbracket\cdot\rrbracket$ denotes the jump across $\Gamma$:
		\[\llbracket U\rrbracket=U^+-U^-,\]
		where $U^{\pm}$ denotes the limiting value of $U$ evaluated on $\Gamma$ from the $\Omega^\pm$ side.
		
		Equivalently, the 2D Peskin problem can be expressed as a contour evolution equation for $X:[0,T]\times \mathbb{S}\to \mathbb{R}^2$ (\cite{LinTongSolvability2019,MoriWell2019}):
		\begin{equation}
			\label{peskin}
			\begin{aligned}
				\partial_{t} X(t,x) &=\int_{\mathbb{S}} {G}(X(t,x)-X(t,\sigma)) F(X(t,\sigma)) \mathrm{d} \sigma, \\
				{G}(Z)&=\frac{1}{4\pi}\left(-\log |Z|         \mathrm{Id}+\frac{Z\otimes Z}{|Z|^2}\right),\quad \quad Z=(z_1,z_2)\in \mathbb{R}^2\backslash\{0\},
			\end{aligned}
		\end{equation}
		where ${G}$ is the fundamental solution of the 2D Stokes problem, and $d\sigma$ denotes the standard measure on the circle.
		
		For the 3D case, let $\Gamma(t)$ denote a closed elastic membrane enclosing a simply connected domain $\Omega$ filled with Stokes fluid. The membrane is parameterized by a mapping $$X(t,\widehat{\boldsymbol{x}}) = (X_1, X_2, X_3) \in \mathbb{R}^3,$$
        where $\widehat{\boldsymbol{x}} \in \mathbb{S}^2$ is the material (Lagrangian) coordinate.   Here, $\mathbb{S}^2$ denotes the 2-dimensional unit sphere embedded in 3-dimensional Euclidean space, defined by
$$\mathbb{S}^2 = \left\{ (x_1, x_2, x_3) \in \mathbb{R}^3 \,:\, x_1^2 + x_2^2 + x_3^2 = 1 \right\}.$$
 The elastic force density has an analogous form:
		\begin{equation*}
			F(X)=\nabla_{\mathbb{S}^2}\left(\mathcal{T}\left(\left|\nabla_{\mathbb{S}^2} X\right|\right) \frac{\nabla_{\mathbb{S}^2} X}{\left|\nabla_{\mathbb{S}^2}X\right|}\right),
		\end{equation*}
		where $\nabla_{\mathbb{S}^2}$ denotes the surface gradient operator on the unit sphere, and $\mathcal{T}$ satisfies \eqref{conten}. The 3D model then can be reformulated as \cite{3Dpeskin}
		\begin{equation}\label{eqpes3d}
			\begin{aligned}
				&	\partial _tX(\widehat{\boldsymbol{x}}) =\int_{\mathbb{S}^2} G(X(\widehat{\boldsymbol{x}})-X(\widehat{\boldsymbol{y}})) \nabla_{\mathbb{S}^2} \cdot\left(\mathcal{T}\left(\left|\nabla_{\mathbb{S}^2} X(\widehat{\boldsymbol{y}})\right|\right) \frac{\nabla_{\mathbb{S}^2} X(\widehat{\boldsymbol{y}})}{\left|\nabla_{\mathbb{S}^2}X(\widehat{\boldsymbol{y}})\right|}\right) d \mu_{\mathbb{S}^2}(\widehat{\boldsymbol{y}}), \\
				&	\left.X(\widehat{\boldsymbol{x}})\right|_{t=0} =X_0(\widehat{\boldsymbol{x}}),
			\end{aligned}
		\end{equation}
		where $d\mu_{\mathbb{S}^2}$ is the standard measure on the unit sphere, $G(\widehat{\boldsymbol{x}})$ is the 3D Stokeslet tensor
		
		\begin{equation}\label{3dpesdefG}
		    G(\widehat{\boldsymbol{x}})=\frac{1}{8 \pi}\left(\frac{1}{|\widehat{\boldsymbol{x}}|} \mathrm{Id}_3+\frac{\widehat{\boldsymbol{x}} \otimes \widehat{\boldsymbol{x}}}{|\widehat{\boldsymbol{x}}|^3}\right).
		\end{equation}
		
		A key observation is that both the 2D system \eqref{peskin} and the 3D system \eqref{eqpes3d} are invariant under the scaling transformation $X_\lambda(t,x) = \lambda^{-1} X(\lambda t, \lambda x)$ for any $\lambda > 0$. This scaling invariance identifies the critical spaces $\dot{B}^1_{\infty,\infty}$, ${\rm BMO}^1$, and $\dot{W}^{1,\infty}$ as natural candidates to study the well-posedness of the Cauchy problem.
		
		The breakthrough in the analytical study of the Peskin problem was initiated in \cite{LinTongSolvability2019,MoriWell2019}. 
		Lin and Tong \cite{LinTongSolvability2019} proved the local well-posedness for arbitrary $H^\frac{5}{2}$ data. Their proof relies on energy arguments and an application of the Schauder fixed point theorem. They also proved the global existence result and exponential decay towards equilibrium when the initial configuration is sufficiently close to the equilibrium. Tong \cite{TongRegularized2021} also established global well-posedness of a regularized Peskin problem and proved the convergence as the regularization parameter diminishes. Mori, Rodenberg and Spirn \cite{MoriWell2019} established a local well-posedness result for initial data in $C^{1,\gamma}$ with $\gamma>0$ (see also \cite{Rodenberg2018}).
		These spaces are subcritical under the scaling of the Peskin problem. For the well-posedness in critical spaces, Garcia-Juarez, Mori and Strain \cite{GarciaViscosityContrast2020} proved the global well-posedness result with initial data in the Wiener space $\mathcal{F}^{1,1}$ and sufficiently close to the stationary states. Their result holds for two fluids with different viscosities. Gancedo, Belinch\'{o}n and Scrobogna \cite{GancedoGlobal2020} considered a toy model of the Peskin problem and proved a global existence result in the critical Lipschitz space. In the work \cite{KN} of the first and third authors, the local and global well-posedness of the Peskin problem (for Hookean material) are established in critical Besov space $\dot B^1_{\infty,\infty}$. The new ingredient in \cite{KN} is the construction of a new norm based on the structure of the nonlinear terms of the equation.   More recently, Cameron and Strain \cite{Camepeskin} considered the problem with fully nonlinear tension, and proved local well-posedness in the critical Besov space $\dot B^\frac{3}{2}_{2,1}$. The global well-posedness and  asymptotic stability for small perturbations near circle solutions in the critical space $\dot W^{1,\infty}\cap \dot B^{\frac{3}{2}}_{2,\infty}$ are established in \cite{GHpeskin}. For 3D Peskin problem, Garcia-Juarez, Kuo, Mori and Strain  \cite{3Dpeskin} derived a boundary integral formulation and also proved local well-posedness in H\"{o}lder spaces $C^{1,\gamma}$, which offers important insights for our work.
		
		Investigations on related models have expanded the research landscape. Li \cite{Li2020Stability} explored the coupled bending-stretching dynamics. Tong \cite{Tong24} considered a simplified framework involving an infinitely long elastic string undergoing purely tangential stretching. Notably, this work established the existence of globally defined weak solutions within the energy space 
		$H^1$, eliminating any requirement for amplitude restrictions. Progressing beyond this foundation,  Tong and Wei \cite{TongWei} demonstrated global well-posedness for the Peskin system in any subcritical Hölder spaces $C^{1,\gamma}$, requiring only mild geometric constraints on normal deviations from equilibrium configurations.  More recently, Tong and Wei \cite{TW2026} established well-posedness of the Peskin problem in the Navier--Stokes case for critical $C^1$ configuration.
		
		In Sections \ref{secpeskin} and \ref{secpes3d}, we consider general nonlinear elastic laws, corresponding to the fully nonlinear Peskin problem in 2D and 3D, respectively. We remark that linear elasticity leads to a semilinear system, whereas a nonlinear tension law renders the system quasilinear. By applying Theorem \ref{lemmain}, we establish the local well-posedness of the systems \eqref{peskin} and \eqref{eqpes3d} in the critical Besov/Lipschitz space. 
        
       The definition of solutions requires that the interface remain non-degenerate and free of self-intersections. To this end, we introduce the arc-chord condition. For the 2D Peskin problem \eqref{peskin}, define
		\begin{equation*}
			\mathbf{\Theta}_Y=\sup_{x_1,x_2\in \mathbb{S},x_1\neq x_2}\frac{|x_1-x_2|}{|Y(x_1)-Y(x_2)|},
		\end{equation*}
		where $|x_1-x_2|=\inf_{k\in \mathbb{Z}}|x_1-x_2-2k\pi|$ is the distance between $x_1$ and $x_2$ on the torus. For simplicity, we denote
		\begin{equation*}
			\mathbf{\Theta}_X(t)=	\sup_{\tau\in[0,t]}\mathbf{\Theta}_{X(\tau)},\quad \mathbf{\Theta}_0=	\mathbf{\Theta}_{X_0}.
		\end{equation*}
		We say that  the initial configuration $X_0$ satisfies the arc-chord condition if $\mathbf{\Theta}_0 < \infty$. 
        
       Analogously, for the 3D Peskin problem \eqref{eqpes3d}, we impose the arc-chord condition:
       \begin{align}\label{well3D}
			\mathbf{\Theta}_X:=\sup_{\widehat{\boldsymbol{x}},\widehat{\boldsymbol{y}}\in\mathbb{S}^2,\widehat{\boldsymbol{x}}\neq\widehat{\boldsymbol{y}}}\frac{|\widehat{\boldsymbol{x}}-\widehat{\boldsymbol{y}}|}{|{X}(\widehat{\boldsymbol{x}})-{X}(\widehat{\boldsymbol{y}})|}<+\infty,
		\end{align}
		where $|\widehat{\boldsymbol{x}}-\widehat{\boldsymbol{y}}|$ is the standard Euclidean distance in  $\mathbb{R}^3$. The notations $\mathbf{\Theta}(t)$ and $\mathbf{\Theta}_0$ are used in the same manner as in the 2D case.
		\begin{remark}
			The arc-chord conditions are essential, since otherwise the integral on the right hand side of \eqref{peskin} and \eqref{eqpes3d} are not well-defined. For example, in the 3D case, \eqref{well3D} ensures that 
			\begin{equation*}
				|G(X(\widehat{\boldsymbol{x}})-X(\widehat{\boldsymbol{y}}))|\lesssim \frac{1}{|\widehat{\boldsymbol{x}}-\widehat{\boldsymbol{y}}|}.
			\end{equation*}
			Note that the arc-chord conditions imply that the interface does not have some pathological behavior, like self-intersect, and ensures that the operator is non-degenerate, which enables us to establish well-posedness in Lipschitz space.
		\end{remark}
	Despite the nonlinear and nonlocal nature of the system, a central difficulty arises from parametrization. Unlike the Muskat problem, where the interfacial geometry dominates, here the stretching induced by the parametrization plays a decisive role in the evolution dynamics, necessitating carefully designed schemes to accurately capture the structure of the underlying equations. A further challenge is the propagation of the arc-chord condition: we will show that the solution remains a small perturbation of the initial configuration and that the quantity $\mathbf{\Theta}_X(t)$ remains uniformly controlled by the initial data. The main results are summarized in the following theorems; in particular, for the 2D Peskin problem \eqref{peskin}, we have:
		\begin{theorem}\label{thmPesB} 
			Fix $m\in\mathbb{N}^+$.
			For any  initial data  $X_0$ satisfying $$z_0:=\|X_0\|_{\dot W^{1,\infty}}+\mathbf{\Theta}_0<\infty,$$ there exist $\varepsilon_0>0$ and a function $\mathfrak{M}: [0,\infty)\to [1,\infty)$ depending on the tension $\mathcal{T}$ and the parameter $m$,  such that for any $0<\varepsilon\leq \varepsilon_0$, if  
        \begin{align}\label{pesconbe}
				 \mathfrak{M}(z_0)\| X_0-\Phi\|_{\dot B^{1}_{\infty,\infty}}\leq \varepsilon, 
			\end{align} 
			holds for some smooth function  $\Phi\in C^\infty(\mathbb{S})$, then the 2D Peskin problem \eqref{peskin} admits a unique solution  $X$ in the class
            \begin{align*}
				&\mathcal{Z}:=\left\{X(t,x):(0,T)\times\mathbb{S}\rightarrow \mathbb{R}^2:\sup_{t\in[0,T]}(\|\partial_x X(t)\|_{L^\infty}+t^{m+\kappa}\|\partial_xX(t)\|_{ \dot C^{m+\kappa}})\leq C(\eps+ \|X_0\|_{W^{1,\infty}}),\right.\\
               & \quad\quad\quad\quad\quad\left.\sup_{t\in[0,T]}	(t^\eta\|\partial_x X(t)\|_{\dot C^{\eta}}+t^{m}\|\partial_x X(t)\|_{\dot C^{m}})\leq C \varepsilon,\mathbf{\Theta}_X(T)\leq 10\mathbf{\Theta}_0\right\},
			\end{align*} 
            for some $C>1$, $T=T(\varepsilon_0,\|\Phi\|_{C^{m+3}})>0$, and $\kappa\in(\frac{1}{2},1), \eta\in(0,\frac{1}{2})$ are fixed parameters. Moreover, the solution is stable in the following sense: for any two solutions $X,Y\in\mathcal{Z}$ corresponding to initial data $X_0,Y_0$ satisfying \eqref{pesconbe}, we have
            \begin{equation*}
                \sup_{t\in[0,T]}\|(X-Y)(t)\|_{\dot W^{1,\infty}}\leq C\|X_0-Y_0\|_{\dot W^{1,\infty}}.
            \end{equation*}
		\end{theorem}
		Note that $(C^\infty)^{\dot B^1_{\infty,\infty}}\backslash \dot W^{1,\infty}\neq\emptyset$, see \cite[V.4.3.1]{Steinbook} and \cite{DLN14}. It is well-known that  $C^1\subsetneqq VMO^1\subsetneqq \dot W^{1,\infty}\cap (C^\infty)^{\dot B^1_{\infty,\infty}}$.
		As we remarked before, here we require $X_0\in \dot W^{1,\infty}$ to control the variable coefficients. 
		
		A new challenge stems from propagating the arc-chord condition. While proving it is straightforward for $C^1$ initial data, which ensures the smallness of the deviation from the well-behaved profile $\Phi$ in the Lipschitz norm, this guarantee breaks down when considering the initial data satisfying  \eqref{pesconbe}. This is because the Lipschitz norm cannot be controlled by the Besov norm $\dot B^{1}_{\infty,\infty}$. To overcome this difficulty, we adapt the method from \cite{KN}. Specifically, we control the arc-chord constant using a non-endpoint quantity and leverage the structure of the evolution equation. Define the quantity
		\begin{equation*}
			Q_h(T)=\sup_{t\in[0,T]}\sup_{\alpha,s}\frac{|\alpha|^\varepsilon}{t^\varepsilon}\left|\frac{1}{|\Delta_\alpha h(t,x)|}-\frac{1}{|\Delta_\alpha h(0,x)|}\right|,
		\end{equation*}
		where $\Delta_\alpha h(x)=\frac{\delta_\alpha h(x)}{\alpha}$, and $\varepsilon$ is a small positive constant. By  \cite[Lemma 2.8]{KN}, if $Q_X(T)+\sup_{t\in(0,T)}t^{a}\|\partial_xX\|_{\dot C^{a}}$
        is finite for some $a>0$, then $X$ satisfies the arc-chord condition on $[0,T]$.
		
	Recently, a remarkable work by Garc\'{\i}a-Ju\'{a}rez and Haziot \cite{GHpeskin} established global well-posedness and asymptotic stability for the 2D Peskin problem with fully nonlinear tension, for small perturbations around circular solutions in the critical space $\dot W^{1,\infty}\cap B^{3/2}_{2,\infty}$. Their analysis is based on a careful linearization around the circle and a Fourier-mode decomposition that separates the symmetry-induced non-decaying modes from the dissipative ones. By combining Duhamel’s formula with a fixed-point argument, they further proved that the solution converges exponentially to a translated and rotated disk.
    In Section \ref{secglo}, we extend the result to small perturbations in the critical space $\dot B^1_{\infty,\infty}$.\\

		For the 3D Peskin problem \eqref{eqpes3d}, we have the following result.
		\begin{theorem}\label{thmPes3d}
			Fix  $m\in\mathbb{N}^+$. 	For any  initial data  $X_0$ satisfying $$z_0:=\|X_0\|_{ W^{1,\infty}(\mathbb{S}^2)}+\mathbf{\Theta}_0<\infty,$$ there exist $\varepsilon_0>0$ and a function $\mathfrak{M}: [0,\infty)\to [1,\infty)$ depending on the tension $\mathcal{T}$ and the parameter $m$,  such that for any $0<\varepsilon\leq \varepsilon_0$, if  
			\begin{align}\label{con3dpes}
			\mathfrak{M}(z_0)	\|X_0- \Phi\|_{ W^{1,\infty}}\leq \varepsilon,
			\end{align}
			holds for some smooth function $\Phi\in C^\infty(\mathbb{S}^2)$,
			then the 3D Peskin problem \eqref{eqpes3d} has a unique solution ${X}$ in the class
			\begin{align*}
				\mathcal{Z}_T:=&\big\{Y(t,\widehat{\boldsymbol{x}}):(0,T)\times\mathbb{S}^2\rightarrow \mathbb{R}^3:\mathbf{\Theta}_Y(T)\leq 2\mathbf{\Theta}_0,\\
				&\quad\|Y\|_{Z_T}:=\sup_{t\in(0,T)}(\| Y(t)\|_{W^{1,\infty}}+t^{m+\kappa}\|Y(t)\|_{ C^{m+1+\kappa}})\leq C \|X_0\|_{W^{1,\infty}}\big\}
			\end{align*}
			for some universal $C>1$, $T=T(\varepsilon_0,\|\Phi\|_{C^{m+3}})>0$, and $\kappa\in(0,1)$ is a fixed parameter. 	Moreover, the solution is stable in the following sense: for any two solutions $X,Y\in \mathcal{Z}_T$ corresponding to initial data $X_0, Y_0$ satisfying \eqref{con3dpes}, it holds
			\begin{align*}
				\sup_{t\in[0,T]}(\|(X-Y)(t)\|_{W^{1,\infty}}+t^{m+\kappa}\|(X-Y)(t)\|_{ C^{m+1+\kappa}})\leq C\|X_0-Y_0\|_{ W^{1,\infty}}.
                \end{align*}
		\end{theorem}
		The H\"{o}lder norm on $\mathbb{S}^2$ is defined in \eqref{normsph}. Note that for the Peskin system defined on a manifold, the presence of lower-order terms is inevitable (see Proposition \ref{propb=0} and Theorem \ref{thmmani}). Consequently, the result cannot be directly extended to the Besov-type case in a parallel manner.
		
	We briefly overview our reformulation of the 2D/3D Peskin system into the form \eqref{eqpara}. In the 2D case, for small displacements ($|\alpha| \ll 1$), we approximate $\delta_\alpha X(s) \approx \alpha \cdot X'(s)$, which leads to the key asymptotic relation:
		\begin{equation*}
			\frac{\delta_\alpha X(s)\cdot X'(s-\alpha)}{|\delta_\alpha X(s)|^2} \sim \frac{1}{2}\cot\left(\frac{\alpha}{2}\right),
		\end{equation*}
		motivating the introduction of the Hilbert transform $\mathcal{H}$ as in \cite{KN}. This yields the operator formulation
		\begin{equation*}
			\partial_t\partial_x X + \frac{1}{4}\partial_x\mathcal{H}\left(\mathbf{T}(|\partial_x X|)\partial_x X\right) = \partial_x\mathcal{N},
		\end{equation*}
		where $\mathbf{T}(|b|) = \frac{\mathcal{T}(|b|)}{|b|},$ and $\mathcal{N}$ denote terms of lower differential order. Using the identity $\partial_x\mathcal{H} = \Lambda = \Lambda^\frac{1}{2}\Lambda^\frac{1}{2}$, we decompose the quasilinear operator as
		\begin{align*}
			\Lambda\left(\mathbf{T}(|\partial_x X|)\partial_x X\right) &= \Lambda^\frac{1}{2}\left(\Lambda^\frac{1}{2}(\mathbf{T}(|\partial_x X|))\partial_x X + \mathbf{T}(|\partial_x X|)\Lambda^\frac{1}{2}\partial_x X\right) + \mathcal{R}_1 \nonumber \\
			&= \Lambda^\frac{1}{2}\left(\mathsf{A}(\partial_x X)\Lambda^\frac{1}{2}\partial_x X\right) + \mathcal{R}_2,
		\end{align*}
		where $\mathcal{R}_1, \mathcal{R}_2$ are lower-order remainder terms, and the coefficient matrix $\mathsf{A}$ is defined by:
		\begin{align*}
			\A(b)=\frac{1}{4}\frac{\mathcal{T}(|b|)}{|b|}\left(\mathrm{Id}-\frac{b\otimes b}{|b|^2}\right)+\frac{1}{4}\mathcal{T}'(|b|)\frac{b\otimes b}{|b|^2}.
		\end{align*}
		Under the condition \eqref{conten}, the matrix $\mathsf{A}$ is positive definite for all $b$. Approximating $X$ by a smooth function $\Phi$, we obtain:
		\begin{equation*}
			\partial_t\partial_x X + \A(\partial_x\Phi)\Lambda\partial_x X = \mathcal{R},
		\end{equation*}
		where $\mathcal{R}$ collects lower-order terms. 
		
		From the above reformulation, we observe that the form of the 2D Peskin equation heavily relies on the explicit structure of the Hilbert transform $\mathcal{H}$. However, in the 3D case, such a representation is unavailable, necessitating a different approach to parametrization. The increased dimensionality introduces additional geometric complexity, requiring tools from differential geometry, such as surface gradients, curvature tensors, and divergence operators on manifolds, to properly describe elasticity forces and interfacial dynamics. In equation \eqref{eqpes3d}, we  choose a smooth function $\Phi$ in $\mathbb{S}^2$ which approximates $F_0$ in the $C^1$ norm. By approximating $F$ with $\Phi$, we can linearize the system \eqref{eqpes3d} near the smooth function $\Phi$ as
		\begin{equation*}
			\partial_t(F-\Phi)(t,\widehat{\boldsymbol{x}}) + \int_{\mathbb{S}^2} G(\Phi(\widehat{\boldsymbol{x}})-\Phi(\widehat{\boldsymbol{y}})) \nabla_{\mathbb{S}^2}\cdot\left(J(\nabla_{\mathbb{S}^2}\Phi) \nabla_{\mathbb{S}^2}(F-\Phi)\right)(t,\widehat{\boldsymbol{y}}) d\mu_{\mathbb{S}^2}(\widehat{\boldsymbol{y}}) = N(F,\Phi)(t,\widehat{\boldsymbol{x}}),
		\end{equation*}
		where $N(F,\Phi)$ represents lower-order terms, and $J$ defined by \eqref{3dpesdefJ}, is positive definite under the same condition \eqref{conten} used in 2D. Following \cite{3Dpeskin}, we employ stereographic projection for localization. As in the proof of Theorem \ref{thmmani}, we select smooth cut-off functions $\{\chi_i\}$ on $\mathbb{S}^2$ such that $F-\Phi = \sum_{i}\chi_i(F-\Phi)$. Each component can be mapped from $\mathbb{S}^2$ to $\mathbb{R}^2$ by stereographic projection, defined by $h_i$, and $\Phi$ is mapped to $\phi_i$. Applying change of variables yields
		\begin{equation*}
			\partial_th_i(t,\theta)+\int_{\mathbb{R}^2}G(\phi_i(\theta)-\phi_i(\eta))\nabla\cdot\tilde{J}(\eta,\nabla\phi_i)\nabla h_i(t,\eta)d\eta=\bar{N}_i(t,\theta),
		\end{equation*}
		where $\bar{N}_i$ is lower order term, and $\tilde{J}$ defined in \eqref{deftJ}, is coercive on $\mathbb{R}^2$. Furthermore, we can approximate $\phi_i(\theta)-\phi_i(\eta)$ by $(\theta-\eta)\cdot\nabla\phi_i$, and derive the formula
		\begin{equation*}
			\partial_t h_i + \mathcal{L}_{\mathbb{R}^2}^i h_i = \tilde{N}_i,
		\end{equation*}
		where $\mathcal{L}_{\mathbb{R}^2}^i$ (defined in \eqref{3dpesdefop}) satisfies the conditions of Theorem \ref{lemmain}, and $\tilde{N}_i$ represents lower-order term. 
		
		Through the above reformulation and by approximating the solution with $\Phi$, we observe that the quasilinear system becomes essentially semilinear. A key observation is that the dominant operator is a first-order pseudo-differential operator with a positive symbol. This structural property allows us to apply Theorem \ref{lemmain} to establish local well-posedness for both the 2D and 3D Peskin systems. Furthermore, the structure of steady solutions is explicitly characterized in the 2D case by \cite{LinTongSolvability2019,GHpeskin}, allowing us to establish global well-posedness for the 2D Peskin system through the kernel estimates derived in \cite{GHpeskin}. In contrast, the 3D Peskin problem presents significant challenges due to the inherent geometric complexity and the absence of analytical tools analogous to the Hilbert transform, making it difficult to construct steady solutions for general configurations.

		\subsection{Some extensions}\label{secfu}
		1. {\bf Initial-boundary value problem of parabolic equations.}  
		\vspace{0.2cm}\\
		The analysis of quasilinear parabolic systems in bounded domains presents substantial challenges, primarily arising from potential boundary singularities. A canonical example can be found in three-dimensional Navier-Stokes theory: as demonstrated in \cite{ChangKang,sesv}, the Dirichlet boundary value problem for 3D Navier-Stokes equations lacks $L^\infty_tC^\alpha$ solutions for certain $\alpha>0$. This fundamental limitation originates from the boundary irregularity of the Leray projection operator $\mathbb{P}=\mathrm{Id}-\nabla\Delta_{D}^{-1}\operatorname{div}$, where $\Delta_{D}$ denotes the Neumann Laplacian.
		
		Another paradigm illustrating these analytical difficulties is the surface quasi-geostrophic (SQG) equation in confined domains. This notoriously challenging problem has stimulated significant research activity \cite{CI17,CI16,CI20,Ignatova}, particularly regarding solution regularity. Although the initial finite-time blow-up conjecture proposed in \cite{CMT} was subsequently disproven through rigorous analysis in \cite{Cor98} and numerical simulations in \cite{CLSTW}, the question of global regularity for SQG in bounded domain remained open until the groundbreaking work of Constantin, Ignatova, and the third author \cite{Constantin2023}. Their seminal result established smooth solutions for the critical dissipative SQG equation by ingeniously exploiting inherent nonlinear structures.
		
		The Schauder-type estimates developed in this context have proven particularly effective for establishing local/global well-posedness of various quasilinear parabolic equations. Nevertheless, extending these methodologies to initial-boundary value problems for general quasilinear systems continues to pose significant challenges, underscoring the need for novel analytical approaches.	\vspace{0.2cm}\\
		2. {\bf The Muskat problem in general setting.}
	\vspace{0.2cm}\\		
		In this paper, we present a rigorous mathematical framework for analyzing the two-dimensional Muskat problem under the special scenario where the two immiscible fluids share identical viscosity coefficients. This particular configuration permits a remarkable structural reduction -- the governing equations collapse into a boundary integral formulation that completely characterizes the interface evolution. However, this structural simplification disappears when viscosity contrast is introduced, necessitating fundamentally different analytical approaches.
		
		The complexity inherent in viscosity-mismatched systems fundamentally stems from the intricate and inseparable nonlinear coupling between interfacial kinematics and the hydrodynamic behavior of bulk phases. At the heart of this coupling lies the Neumann--Dirichlet operator, which establishes the non-local relationship governing interfacial stress (Neumann boundary condition) and velocity continuity (Dirichlet boundary condition) across the dynamically evolving interface. For comprehensive treatments of this formalism, we direct readers to \cite{1HuyNguyen2020} for reformulation of the general Muskat problem through this operator framework, and to \cite{ABZ} demonstrating its successful implementation in water wave dynamics modeling. When the viscosity ratio deviates from unity, the mathematical description requires concurrent resolution of:
		(i) moving boundary conditions dictating interface motion, and
		(ii) Stokes flow constraints with viscosity-dependent stress coupling.
		
		The strong cross-coupling mechanism fundamentally alters the operator structure, precluding direct application of Schauder-type estimates to the general case. Crucially, the interface evolution equation contains implicit forcing terms dependent on Darcy-system solutions within the bulk domains, where the Neumann-Dirichlet operator plays a dual role. This interdependence requires establishing refined elliptic regularity estimates across the evolving interface, particularly regarding transmission conditions for pressure-velocity fields. The development of such estimates through customized potential theory approaches constitutes a primary objective of the ongoing research program.	\vspace{0.2cm}\\
		3. {\bf Kinetic theory}
			\vspace{0.2cm}\\
	The extension of Schauder-type estimates to kinetic equations, in particular to those involving the nonlocal collision dynamics of the Boltzmann and Landau equations \cite{CV}, opens a new analytical framework for the study of well-posedness. Adapting the Schauder methodology developed herein to the kinetic setting necessitates substantial refinements in order to simultaneously capture the anisotropic structure of the linearized operator and the influence of the free transport dynamics.

Moreover, the strategy introduced in Section \ref{secglo} for establishing global well-posedness and asymptotic behavior is sufficiently robust to apply to the Boltzmann equation. Building upon the core ideas of the present work, we establish the global well-posedness of the fractional Fokker–Planck equation (FFPE) in the largest critical Besov space; see \cite{KHN2025}. Further developments on regularity and hypoelliptic estimates for the FFPE are available in \cite{RAkin,FB2002,NL2012,YMCX2007}, while results concerning Besov initial data can be found in \cite{HZ2024}.
 The framework proposed in this paper admits further extensions to anisotropic hypoelliptic settings, yielding sharp well-posedness results for the Boltzmann and Landau equations with very soft potentials, see \cite{CNY}.
		
		\subsection{Organization of article}
		The remainder of this paper is organized as follows. In Section \ref{secproof}, we establish the main Schauder-type estimates for the linear system \eqref{eqpara}. These estimates are then applied to investigate the critical well-posedness theory for three distinct problems: the 2D Muskat equation with surface tension \eqref{eqmst} in Section \ref{intomus1}, the 2D Peskin problem \eqref{peskin} in Section \ref{secpeskin}, and the 3D Peskin problem \eqref{eqpes3d} in Section \ref{secpes3d}.
		\section{Schauder-type estimate}\label{secproof}
		
		\subsection{Notations}
		Throughout this paper, we use the following notations:
		\begin{itemize}
			\item \textit{Fourier Transform.} For a function $f \in L^1(\mathbb{R}^d)$, we define its Fourier transform and inverse Fourier transform as
			\begin{equation*}
				\begin{aligned}
					&\mathcal{F}(f)(\xi)=\frac{1}{(2\pi)^{\frac{d}{2}}}\int_{\mathbb{R}^d}f(x)e^{-ix\cdot \xi}dx,\\
					&\mathcal{F}^{-1}(g)(x)=\frac{1}{(2\pi)^{\frac{d}{2}}}\int_{\mathbb{R}^d}g(\xi)e^{ix\cdot \xi}d\xi.
				\end{aligned}
			\end{equation*}
			We also write $\hat f(\xi)=\mathcal{F}(f)(\xi)$ for short. 
			It is easy to check $\mathcal{F}\mathcal{F}^{-1}(f)=\mathcal{F}^{-1}\mathcal{F}(f)=f$.
			Moreover, 
			\begin{align*}
				\mathcal{F}(f\ast g)(\xi)=(2\pi)^\frac{d}{2}\hat f(\xi)\hat g(\xi).
			\end{align*}
			
			\item \textit{Finite Difference Operators.}
			\begin{align*}
				&\delta_\alpha f(x)=f(x)-f(x-\alpha),\\&
				\Delta_\alpha f(x)=\frac{\delta_\alpha f(x)}{\alpha}~~\text{in } ~\mathbb{R},\quad\quad\quad\Delta_\alpha f(x)=\frac{\delta_\alpha f(x)}{|\alpha|}~~\text{in} ~\mathbb{R}^l, ~l\geq 2.
			\end{align*}
			The finite difference operators appearing in this paper are all in the spatial variable, \textit{ i.e.} for $h(t,x)$ defined in $[0,T]\times \mathbb{R}^d$, we denote $\delta_\alpha h(t,x)=(\delta_\alpha h(t,\cdot))(x)$. Specifically, for domain $\Omega\subset \mathbb{R}^d$, we define the finite difference in $\Omega$ as
			\begin{equation*}
				\delta_\alpha f(x)=f(x)-f(x-\alpha), \quad \text{for}\ x,x-\alpha\in\Omega.
			\end{equation*}
			\item \textit{Fractional Laplacian operators.}
			\begin{equation}\label{deffracla}
				\Lambda=(-\Delta)^{\frac{1}{2}},\quad\quad\quad\Lambda^a=(-\Delta)^{\frac{a}{2}},\quad \text{in}\ \mathbb{R}^d\ \text{or}\ \mathbb{T}^d.
			\end{equation}
			\item \textit{Gradient Operator.} For a vector-valued function $f = (f^1, \dots, f^n) \in C^1(\mathbb{R}^d; \mathbb{R}^n)$, the gradient $ \nabla f $ is the $d \times n$ matrix given by
			$$
			(\nabla f)_{ij}=\partial_{i}f^j.
			$$
			\item \textit{Multi-Index Derivatives.}
			For any function $f$, and any $\beta=(\beta_1,...\beta_d)\in\mathbb{N}^d$, we denote $\partial_x^\beta f=\partial_1^{\beta_1}...\partial_d^{\beta_d}f$. With a slight abuse of notation, for $m\in\mathbb{N}$, we denote $\nabla^mf=(\partial^\beta_x f)_{|\beta|=m}$, where $|\beta|=\sum_{i=1}^d\beta_i$.
			\item \textit{Frobenius Inner Product and Matrix Norm.}  For any matrix $\M, \N\in \mathbb{R}^{d\times d}$, the Frobenius inner product of $
			\M:\N$ and the induced matrix norm are defined by
			\begin{align*}\label{matrxnorm}
				\M:\N=\sum_{i,j=1}^d\M_{ij}\N_{ij},\quad\quad\quad |\M|=\sqrt{\M:\M}.
			\end{align*}
			\item \textit{Friedrichs Mollifier.} We denote $\rho_\eps$ the standard Friedrichs mollifier with parameter $\eps>0$.
			\item \textit{Japanese Bracket.} For $r\in\mathbb{R}^d$, we denote the bracket $\langle r\rangle=\sqrt{1+|r|^2}$.
			\item \textit{Integer Part.} For $s\in\mathbb{R}^+$, denote $[s]=\max\{n:n\in\mathbb{N},n\leq s\}$ the integer part of $s$.
			\item \textit{Time-Space Norms.}  For $f:[0,T]\times \mathbb{R}^d\to \mathbb{R}^m$, we denote $\|f\|_{L^\infty(I;X)}=\sup_{t\in I}\|f(t)\|_X$, where $I\subset[0,T]$ and $X$ is a Banach space equipped with norm $\|\cdot\|_X$. Specifically, when $I=[0,T]$, we simply write $\|\cdot\|_{L^\infty_TX}$.
			\item \textit{Indicator Function.} For a measurable set $A \subset \mathbb{R}^d$, define  
			\begin{equation*}\label{indicaf}
				\begin{aligned}
					\mathbf{1}_A(x):=\begin{cases}
						1,\quad\quad x\in A,\\
						0,\quad\quad x\notin A.
					\end{cases}
				\end{aligned}
			\end{equation*}
		\end{itemize}

		\subsection{Proof of Theorem \ref{lemmain}}	
		When dealing with equations with constant coefficients, it is straightforward to derive the equation in terms of Fourier transform and find the fundamental solution. The proof of Theorem \ref{lemmain} is divided into three steps. First, we fix the coefficient, and transform the differential equation into an integral equation by Fourier method. Then we give some elementary estimates of the fundamental solution. Finally, we estimate the solution and show that the effect of  variable coefficients can be ignored provided $\|\A\|_{B_T}$ small enough.  \vspace{0.3cm}\\
		\textbf{Step 1: Transform the differential equation to an integral equation.}\\
		Fix $x_0\in\mathbb{R}^d$, we rewrite \eqref{eqpara} as 
		\begin{equation}\label{eqpafi}
			\begin{aligned}
				&\partial_tu(t,x)+(2\pi)^{-\frac{d}{2}}\int _{\mathbb{R}^d}\A(t,x_0,\xi)\hat{u}(t,\xi)e^{ix\cdot\xi}d\xi=f(t,x)+R_{x_0}[u](t,x),
			\end{aligned}
		\end{equation}
		with 
		\begin{align}\label{defRx0}
			R_{x_0}[u](t,x)=(2\pi)^{-\frac{d}{2}}\int _{\mathbb{R}^d}(\A(t,x_0,\xi)-\A(t,x,\xi))\hat{u}(t,\xi)e^{ix\cdot\xi} d\xi.
		\end{align}
		We remark that if the coefficient matrix $\A(t,x,\xi)$ is independent of $x$, then $R_{x_0}[u](t,x)\equiv 0$.\vspace{0.2cm}\\
		Let $\K_{
			x_0}(t,\tau,x)$ be the fundamental solution of the adjoint system, satisfying 
		\begin{equation}\label{defk0}
			\begin{aligned}
				&-\partial_\tau \K_{x_0}(t,\tau,x)+(2\pi)^{-\frac{d}{2}}\int _{\mathbb{R}^d}\hat \K_{x_0}(t,\tau,\xi)\A(\tau,x_0,\xi)e^{ix\cdot \xi}d\xi=0, \ \ (\tau,x) \in (0,t)\times \mathbb{R}^d,\\
				&\lim_{\tau\to t^-}\K_{x_0}(t,\tau,x)=\delta(x)\mathrm{Id},\ \ \ \ \ x\in\mathbb{R}^d.
			\end{aligned}
		\end{equation}
		Here $\mathrm{Id}$ is the $N\times N$ identity matrix,  $\hat \K_{x_0}(t,\tau,\xi)$ is the Fourier transform of  $\K_{
			x_0}(t,\tau,x)$ in $x$-variable, and $\delta(x)$ is the Dirac delta function.
		Then we transform the differential equation \eqref{eqpafi} to the following form,
		\begin{align}
			u(t,x)=&\int_{\mathbb{R}^d} \K_{x_0}(t,0,x-z)u_0(z)dz+\int_0^t \int _{\mathbb{R}^d}\K_{x_0}(t,\tau,x-z)f(\tau,z)dzd\tau\nonumber \\
			&+\int_0^t\int_{\mathbb{R}^d} \K_{x_0}(t,\tau,x-z)  R_{x_0}[u](\tau,z)dzd\tau\nonumber\\
			:=	&u_{L,x_0}(t,x)+u_{N,x_0}(t,x)+u_{R,{x_0}}(t,x),\label{uform}
		\end{align}
		which holds for any $x_0\in\mathbb{R}^d$. Note that when estimating the H\"{o}lder norm of $u$, we will first take derivatives in the formula \eqref{uform}, and then fix $x_0=x$ such that the remainder term $u_{R,x_0}$ can be absorbed by the principal terms, see \eqref{holu}. Hence our method does not require the regularity of $\K_{x_0}$ with respect to $x_0$. It is essential to note the non-commutativity:
		\begin{equation*}
			\big(\nabla^n_x(u_{L,x_0}(t,x)+u_{N,x_0}(t,x)+u_{R,{x_0}}(t,x))\big)|_{x_0=x}\neq \nabla^n_x\big((u_{L,x_0}(t,x)+u_{N,x_0}(t,x)+u_{R,{x_0}}(t,x))|_{x_0=x} \big).
		\end{equation*}
	~~\\
		{\bf Step 2: Properties of the fundamental solution.}\\
		To estimate the kernel $\K_{x_0}$ and its derivatives, we first prove the following lemma.
		\begin{lemma}\label{esthtker}
			Let $\lambda,\eps >0$.	Suppose  $\F(\xi)$ satisfies
			\begin{equation}\label{condb}
				|\nabla^n_\xi\F(\xi)|\leq C|\xi|^{\sigma-n}\exp(-\lambda^\eps|\xi|^\eps),\quad \forall n\leq d+[\sigma]+1,
			\end{equation}
			and define
			\begin{equation*}
				\G(x)=\int_{\mathbb{R}^d}\F(\xi)e^{ix\cdot\xi}d\xi.
			\end{equation*}
			If $\sigma>-d$, then it holds
			\begin{align}\label{leme}
				|\nabla^k\G(x)|\lesssim \lambda^{-d-k-\sigma}\left(1+\frac{|x|}{\lambda}
                \right)^{-(d+\sigma)},\ \ \ \forall\ k\geq 0.
			\end{align}
			Moreover, if $\sigma>0$, then for any $k\in\mathbb{N}$,  $0\leq\sigma'<\min\{\sigma,1\}$, and for any $\alpha\in\mathbb{R}^d$,
			\begin{align}\label{leme1}
				\int_{\mathbb{R}^d}|\delta_\alpha\nabla^k \G(x)||x|^{\sigma'}dx\lesssim \frac{|\alpha|^{\sigma'}}{\lambda^{k+\sigma}}\min\left\{1,\frac{|\alpha|}{\lambda}\right\}^{1-\sigma'}.
			\end{align}
		\end{lemma}
		\begin{proof}
			By performing a dilation with $\lambda$, we can assume without loss of generality that $\lambda=1$. Using the condition \eqref{condb}, it is straightforward to obtain the estimate in \eqref{leme} for $\lambda=1$ by standard techniques in Fourier analysis.
		Then we estimate \eqref{leme1}. For any $\alpha\in\mathbb{R}^d$, if $|\alpha|\geq 1$, then
			\begin{equation}\label{d111}
				\begin{aligned}
					\int_{\mathbb{R}^d}  |\delta_\alpha \nabla^k\G(x)||x|^{\sigma'}dx
					&  \lesssim \int_{\mathbb{R}^d} \left(\frac{1}{(|x|+1)^{d+\sigma}}+\frac{1}{(|x-\alpha|+1)^{d+\sigma}}\right)|x|^{\sigma'}dx\\
					&\lesssim \int_{\mathbb{R}^d} \frac{|x|^{\sigma'}+|\alpha|^{\sigma'}}{(|x|+1)^{d+\sigma}}dx\\
					&\lesssim|\alpha|^{\sigma'},
				\end{aligned}
			\end{equation}
			provided $0\leq \sigma'<\sigma$. Moreover, if $|\alpha|<1$,	it holds
			\begin{align*}
				|\delta_\alpha \nabla^k\G(x)|\lesssim |\alpha|\int_0^1|\nabla^{k+1}\G(x-\tau\alpha)|d\tau.
			\end{align*}
			Then we obtain 
			\begin{align*}
				\int_{\mathbb{R}^d}  |\delta_\alpha \nabla^k\G(x)||x|^{\sigma'}dx 
				\lesssim&\ |\alpha|\int_0^1  \int_{\mathbb{R}^d} \frac{|x|^{\sigma'}}{(|x|+1)^{d+1+\sigma}}dxd\tau\lesssim |\alpha|.
			\end{align*}
			Combining this with \eqref{d111} yields \eqref{leme1}.
			This completes the proof of Lemma \ref{esthtker}.
		\end{proof}\\
		\begin{remark}
			Lemma \ref{esthtker} is available for a family of functions $\{F_{\lambda}(\xi)\}_{\lambda>0}$, where in some cases, the functions $F_{\lambda}(\xi)$ cannot be presented as $F(\lambda,\xi)$. For example, in Corollary \ref{lemfourierK}, we apply  Lemma \ref{esthtker} to a family of functions  $F_\lambda(\xi)=\hat \K_{x_0}(t_1,t_2,\xi)$ with $\lambda=(t_1-t_2)^\frac{1}{s}$.
		\end{remark}
		\begin{lemma}\label{lemadk}
			Let $\A(t,x,\xi)\in\mathbb{R}^{N\times N}$ satisfy \eqref{condop} with $0<c_0<1<c_1$, and let  $\K_{x_0}(t,\tau,x)$ be the kernel defined in \eqref{defk0}.
			Then for the Fourier transform $\hat{\K}_{x_0}(t,\tau,\xi)$, the following estimates hold:
			\begin{align}
				&\left|\hat{\K}_{x_0}(t,\tau,\xi) \right|\lesssim e^{-c_0(t-\tau)|\xi|^s},	 \label{expA}\\
				&\left|\nabla_{\xi}^l\hat{\K}_{x_0}(t,\tau,\xi) \right|\lesssim  c_0^{-l}c_1^l (t-\tau)|\xi|^{{s}-l}e^{-\frac{c_0}{2}(t-\tau)|\xi|^s},\quad\forall l\in\mathbb{N}_+,  l\leq d+s+2.\label{derA}
			\end{align}
		\end{lemma}
		\begin{proof}
			Taking the Fourier transform of \eqref{defk0} yields
			\begin{equation*}
				\begin{aligned}
					&-\partial_\tau\hat{\K}_{x_0}(t,\tau,\xi)+\hat{\K}_{x_0}(t,\tau,\xi)\A(\tau,x_0,\xi)=0,\ \ \tau\in(0,t),\\
					&\lim_{\tau\rightarrow t^-}\hat{\K}_{x_0}(t,\tau,\xi)=\mathrm{Id}.
				\end{aligned}
			\end{equation*}
			Note that $\hat{\K}_{x_0}(t,\tau,\xi)$ is real-valued because $\A(t,x,\xi)\in\mathbb{R}^{N\times N}$.
			Performing the change of variable  $\hat {\mathcal{K}}_{x_0}(t,\tau,\xi)=\hat{\K}_{x_0}(t,t-\tau,\xi)$, we transform the backward evolution system to a forward one:
			\begin{equation}\label{forwK}
				\begin{aligned}
					&\partial_\tau \hat {\mathcal{K}}_{x_0}(t,\tau,\xi)+\hat {\mathcal{K}}_{x_0}(t,\tau,\xi)\A(\tau,x_0,\xi)=0,\ \ \tau\in(0,t),\\
					&\lim_{\tau\to 0^+}\hat {\mathcal{K}}_{x_0}(t,\tau,\xi)=\mathrm{Id}.
				\end{aligned}
			\end{equation}
			By the coercive condition \eqref{condop}, we obtain

            \begin{equation}\label{KKK}
                \begin{aligned}
				&\partial_\tau |\hat {\mathcal{K}}_{x_0}(t,\tau,\xi)|^2+2 c_0|\xi|^s|\hat {\mathcal{K}}_{x_0}(t,\tau,\xi)|^2\leq 0,\\
				&\lim_{\tau\to 0^+} |\hat {\mathcal{K}}_{x_0}(t,\tau,\xi)|^2=N.
			\end{aligned}
            \end{equation}
			Applying Gronwall's inequality yields
			\begin{align}\label{huaK}
				|\hat {\mathcal{K}}_{x_0}(t,\tau,\xi)|^2\leq N e^{-2c_0\tau|\xi|^s},\ \ \ \forall \tau\in(0,t).
			\end{align}
			This implies 
			\begin{align*}
				|\hat \K_{x_0}(t,\tau,\xi)|=|\hat {\mathcal{K}}_{x_0}(t,t-\tau,\xi)|\lesssim e^{-c_0(t-\tau)|\xi|^s},\ \ \ \forall \tau\in(0,t),
			\end{align*}
			completing the proof of \eqref{expA}.\vspace{0,3cm}\\
			For the estimates of derivatives, we proceed by induction. We first prove that \eqref{derA} holds for $l=1$. Differentiating \eqref{forwK} gives
			\begin{align*}
				&\partial_\tau \nabla_\xi\hat {\mathcal{K}}_{x_0}(t,\tau,\xi)+\nabla_\xi\hat {\mathcal{K}}_{x_0}(t,\tau,\xi)\A(\tau,x_0,\xi)=-\hat {\mathcal{K}}_{x_0}(t,\tau,\xi)\nabla_\xi\A(\tau,x_0,\xi),\ \ \tau\in(0,t),\\
				&\lim_{\tau\to 0^+}\nabla_\xi\hat {\mathcal{K}}_{x_0}(t,\tau,\xi)=0.
			\end{align*}
			Taking Frobenius inner product with $\nabla_\xi\hat {\mathcal{K}}_{x_0}(t,\tau,\xi)$, and using \eqref{huaK}, we deduce 
			\begin{align*}
				&\partial_\tau |\nabla_\xi\hat {\mathcal{K}}_{x_0}(t,\tau,\xi)|+ c_0|\xi|^s|\nabla_\xi\hat {\mathcal{K}}_{x_0}(t,\tau,\xi)|\leq |\nabla_\xi\A(\tau,x_0,\xi)||\hat {\mathcal{K}}_{x_0}(t,\tau,\xi)|\leq c_1|\xi|^{s-1}e^{-c_0\tau|\xi|^s}.
			\end{align*}
		Combining this with \eqref{KKK} implies 
			\begin{align*}
				|\nabla_\xi\hat {\mathcal{K}}_{x_0}(t,\tau,\xi)|\lesssim c_0^{-1}c_1|\xi|^{s-1}\tau e^{-c_0\tau|\xi|^s},\ \ \ \tau\in(0,t).
			\end{align*}
			From this we can recover the estimate of $\nabla_\xi\hat \K_{x_0}$, which leads to \eqref{derA} in the case $l=1$.
            
			For $1<l\leq d+s+2$, denote $\nabla_{\xi}^l\hat{\mathcal{K}}_{x_0}(t,\tau,\xi)=\hat{\mathcal{K}}_{x_0}^{l}(t,\tau,\xi)$. Assume that the cases for $1,2,\cdots,l-1$ have been proved. The equation of $\hat{\mathcal{K}}_{x_0}^{l}(t,\tau,\xi)$ reads
			\begin{equation*}
				\begin{aligned}
					&\partial_{\tau}\hat{\mathcal{K}}_{x_0}^{l}(t,\tau,\xi)+\hat{\mathcal{K}}_{x_0}^{l}(t,\tau,\xi)\A(\tau,x_0,\xi)=R^l(t,\tau,\xi),\\
					&\lim_{\tau\rightarrow 0_+}\hat{\mathcal{K}}_{x_0}^{l}(t,\tau,\xi)=0,
				\end{aligned}
			\end{equation*}
			with the remainder term satisfying
			\begin{equation*}
				\left|R^l(t,\tau,\xi)\right|\lesssim \sum_{i=1}^l\left|\nabla_\xi^i\A(\tau,\xi)\right|\left|\hat{\mathcal{K}}_{x_0}^{l-i}(t,\tau,\xi)\right|\lesssim c_0^{-(l-1)}c_1^l\tau|\xi|^{2s-l}e^{-\frac{c_0}{2}\tau|\xi|^s}.
			\end{equation*}
			Taking Frobenius inner product, and applying \eqref{condop} gives
			\begin{equation*}
				\begin{aligned}
					&\partial_\tau|\hat{\mathcal{K}}_{x_0}^{l}(t,\tau,\xi)|+c_0|\xi|^s|\hat{\mathcal{K}}_{x_0}^{l}(t,\tau,\xi)|\leq \left|R^l(t,\tau,\xi) \right|.
				\end{aligned}
			\end{equation*}
			This implies 
			\begin{align*}
				\partial_\tau\left(e^{c_0\tau|\xi|^s}|\hat{\mathcal{K}}_{x_0}^{l}(t,\tau,\xi)|\right)\lesssim e^{c_0\tau|\xi|^s}\left|R^l(t,\tau,\xi) \right|\lesssim c_0^{-(l-1)}c_1^l\tau|\xi|^{2s-l}e^{\frac{c_0}{2}\tau|\xi|^s}.
			\end{align*}
			Integrating this inequality in $\tau$, and using the fundamental inequality 
			\begin{align*}
				\int_0^a z e^{bz}dz\lesssim \frac{ae^{ba}}{b},\ \ \ \forall a,b>0,
			\end{align*}
			we deduce that
			\begin{equation*}
				|\hat{\mathcal{K}}_{x_0}^{l}(t,\tau,\xi)|\lesssim c_0^{-l}c_1^l \tau|\xi|^{s-l}e^{-\frac{c_0}{2}\tau|\xi|^s}.
			\end{equation*}
			Thus, we have proved
			\begin{equation*}
				\left|\nabla_{\xi}^l\hat{\mathcal{K}}_{x_0}(t,\tau,\xi) \right|\lesssim c_0^{-l}c_1^l \tau|\xi|^{{s}-l}e^{-\frac{c_0}{2}\tau|\xi|^s},\quad\forall l\in\mathbb{N}_+, l\leq d+s+2.
			\end{equation*}
			This completes the proof of the lemma.
		\end{proof}\vspace{0.2cm}\\
		Combining Lemma \ref{esthtker} with Lemma \ref{lemadk}, we obtain the following estimates of the kernel $\K_{x_0}$.
		\begin{corollary}\label{lemfourierK}It is straightforward to verify that for any $n\in\mathbb{N}_+, n\leq d+s+2$, 
			\begin{equation*}
				|\nabla_{\xi}^n\hat{\K}_{x_0}(t_1,t_2,\xi)|\lesssim c_0^{-n}c_1^n (t_1-t_2)|\xi|^{s-n} \exp\big(-\frac{c_0}{2}(t_1-t_2)|\xi|^s\big).
			\end{equation*}
			Applying Lemma \ref{esthtker} with $\lambda=(\frac{c_0}{2}(t_1-t_2))^\frac{1}{s}$, we obtain the point-wise estimates for $\K(t_1,t_2,x)$:
			\begin{align}
		&\sup_{x_0\in\mathbb{R}^d}|\nabla_x^l\K_{x_0}(t_1,t_2,x)|\lesssim  (t_1-t_2)^{-\frac{d+l}{s}}\left(1+\frac{|x|}{(t_1-t_2)^\frac{1}{s}}\right)^{-(d+l+s)},\ \ \forall\, l\in\mathbb{N}.\label{ptKd1}
			\end{align}
            And 	for any $\alpha\neq 0$, $0\leq \sigma<\min\{1,l+s\}$,
			\begin{align}\label{delKL1}
				\sup_{x_0\in\mathbb{R}^d}	\|\nabla _x^l\delta_{\alpha}\K_{x_0}(t_1,t_2)|\cdot|^{\sigma}\|_{L^1_x}\lesssim |\alpha|^{\sigma}\min\left\{1,\frac{|\alpha|}{(t_1-t_2)^{\frac{1}{s}}}\right\}^{1-\sigma} (t_1-t_2)^{-\frac{l}{s}},\ \ \ \ \forall \,l\in\mathbb{N}.
			\end{align}  
			As a direct result of \eqref{ptKd1} and \eqref{ptKd1}, we obtain 
			\begin{align}\label{kerL1}
				\sup_{x_0\in\mathbb{R}^d}\|\nabla _x^l\K_{x_0}(t_1,t_2)|\cdot|^{\sigma}\|_{L^1_x}\lesssim (t_1-t_2)^{-\frac{l-\sigma}{s}},\ \ \forall \,l\in\mathbb{N}.
			\end{align} 
            The implicit constants in \eqref{ptKd1}- \eqref{kerL1} depend on $c_0$ and $c_1$. Specifically, we have 
            \begin{align}\label{KL1}
            \sup_{x_0\in\mathbb{R}^d}\|\K_{x_0}(t_1,t_2)\|_{L^1}\leq C,
            \end{align}
            with $C>0$ that is independent of $c_0,c_1$.
		\end{corollary}
      We proceed by returning to the proof of Theorem \ref{lemmain}.\\
		\textbf{Step 3: A priori estimate.}\\
		By \eqref{uform}, we have for any $\alpha \in \mathbb{R}^d\backslash\{0\}$, and any $n\in\mathbb{N}$,
		\begin{equation}\label{duform}
			\begin{aligned}
				& \nabla^n u(t,x)=\nabla ^nu_{L,x_0}(t,x)+\nabla ^nu_{N,x_0}(t,x)+\nabla ^nu_{R,x_0}(t,x),\\
				&  \delta_\alpha \nabla^n u(t,x)=\delta_\alpha \nabla ^nu_{L,x_0}(t,x)+\delta_\alpha \nabla ^nu_{N,x_0}(t,x)+\delta_\alpha \nabla ^nu_{R,x_0}(t,x),
			\end{aligned}
		\end{equation}
		hold for any fixed $x_0\in\mathbb{R}^d$. Fixing $x_0=x$  in \eqref{duform} yields
		\begin{align*}
			\|\nabla^nu(t)\|_{L^\infty}\leq \|\nabla^nu_{L,x_0}(t,x)|_{x_0=x}\|_{L^\infty_x}+\|\nabla^nu_{N,x_0}(t,x)|_{x_0=x}\|_{L^\infty_x}+\|\nabla^nu_{R,x_0}(t,x)|_{x_0=x}\|_{L^\infty_x},
		\end{align*}
		and 
		\begin{equation*}
			\begin{aligned}
				\|\delta_\alpha \nabla^n u(t)\|_{L^\infty}&\leq \|\big(\delta_\alpha \nabla^nu_{L,x_0}(t,x)\big)|_{x_0=x}\|_{L^\infty_x}+\|\big(\delta_\alpha \nabla^nu_{N,x_0}(t,x)\big)|_{x_0=x}\|_{L^\infty_x}\\
				&\quad\quad+\|\big(\delta_\alpha \nabla^nu_{R,x_0}(t,x)\big)|_{x_0=x}\|_{L^\infty_x}.
			\end{aligned}
		\end{equation*}
		Thus, for any $\vartheta\geq 0$, we obtain 
        \begin{equation}\label{holu}
        \begin{aligned}
			\|u(t)\|_{\dot C^\vartheta}\leq \sup_{x_0\in\mathbb{R}^d}&\|u_{L,x_0}(t)\|_{\dot C^\vartheta}+\sup_{x_0\in\mathbb{R}^d}\|u_{N,x_0}(t)\|_{\dot C^\vartheta}\\
			&+\begin{cases}
				&\|\nabla^\vartheta u_{R,x_0}(t,x)|_{x_0=x}\|_{L^\infty_x},\ \ \ \vartheta\in\mathbb{N},\\
				&\displaystyle\sup_{\alpha\neq 0} \frac{\|\big(\delta_\alpha \nabla^{[\vartheta]}u_{R,x_0}(t,x)\big)|_{x_0=x}\|_{L^\infty_x}}{|\alpha|^{\vartheta-[\vartheta]}},\ \ \ \vartheta\notin\mathbb{N}.
			\end{cases}
		\end{aligned}
        \end{equation}
		We first prove the following lemma to estimate $u_{L,x_0}$ and $u_{N,x_0}$.
		\begin{lemma}\label{lemuLN}
			Let $T>0$. For any  $g:\mathbb{R}^d\to \mathbb{R}$, $h:[0,T]\times\mathbb{R}^d\to \mathbb{R}$, and let  $\K_{x_0}(t,\tau,x)$ be the kernel defined in \eqref{defk0}. Define 
			\begin{align*}
				&\mathcal{Q}_{x_0}[g](t,x)=\int_{\mathbb{R}^d}\K_{x_0}(t,0,x-y)g(y)dy,\\
				&\mathcal{W}_{x_0}[h](t,x)=\int_0^t\int_{\mathbb{R}^d}\K_{x_0}(t,\tau,x-y)h(\tau,y)dyd\tau.
			\end{align*}
			Then the following estimates hold.\\
			i) For any $\sigma_1\geq 0$, $\sigma_2>0$, it holds
			\begin{align}\label{newestuL}
				&\sup_{x_0\in\mathbb{R}^d}\sup_{t>0} \|\mathcal{Q}_{x_0}[g](t)\|_{\dot C^{\sigma_1}}\lesssim \|g\|_{\dot C^{\sigma_1}},\\
				&\sup_{x_0\in\mathbb{R}^d}\sup_{t>0}t^{\frac{\sigma_2}{s}}\|\mathcal{Q}_{x_0}[g](t)\|_{\dot C^{\sigma_1+\sigma_2}}\lesssim _{\sigma_2}\|g\|_{\dot B^{\sigma_1}_{\infty,\infty}},\label{newestuL1}
			\end{align}
			ii) For any $\sigma_3,\sigma_4$ satisfying $0<\sigma_3<s$, $\sigma_3+\sigma_4,\sigma_3+\sigma_4-s\notin\mathbb{N}$  and $\sigma_4>s-\sigma_3$, and  any $n\in\mathbb{N}$, it holds
			\begin{equation}\label{wv}
				\begin{aligned}
					\sum_{a\in\{0,n+\sigma_3\}}\sup_{x_0\in\mathbb{R}^d}\sup_{t\in[0,T]}t^\frac{a}{s}&(\|\mathcal{W}_{x_0}[h](t)\|_{\dot C^{\sigma_4+a}}+\|\Lambda^{\sigma_4}\mathcal{W}_{x_0}[h](t)\|_{\dot C^{a}})\\
					&\lesssim \sup_{t\in[0,T]}\left(t^{\frac{\sigma_3}{s}}\|h(t)\|_{\dot C^{\sigma_3+\sigma_4-s}}+t^{\frac{n+\sigma_3}{s}}\|h(t)\|_{\dot C^{n+\sigma_3+\sigma_4-s}}\right).
				\end{aligned}
			\end{equation}
		\end{lemma}
		\begin{proof}
			By \eqref{delKL1}, \eqref{kerL1}, and  Young's inequality, we obtain 
			\begin{equation*}
				\begin{aligned}
					&\|\mathcal{Q}_{x_0}[g](t)\|_{\dot C^{\sigma_1}}\lesssim \|\K_{x_0}(t,0)\|_{L^1}\|g\|_{\dot C^{\sigma_1}}\lesssim \|g\|_{\dot C^{\sigma_1}}.
				\end{aligned}
			\end{equation*}
			This implies \eqref{newestuL}. The estimate \eqref{newestuL1} follows from Proposition \ref{normequ}.
			
			Then we prove \eqref{wv}. Denote $\tilde \sigma=\sigma_3+\sigma_4-s\notin\mathbb{N}$. Since $\sigma_4>\tilde \sigma$, 
			\begin{equation*}
				\begin{aligned}
					&\Lambda^{\sigma_4}\mathcal{W}_{x_0}[h](t,x)=	\int_0^t \int_{\mathbb{R}^{d}} \Lambda^{\sigma_4-[\tilde \sigma]}\K_{x_0}(t,\tau,x-z) (\Lambda^{[\tilde \sigma]}h(\tau,z)-\Lambda^{[\tilde \sigma]}h(\tau,x))dzd\tau.
				\end{aligned}
			\end{equation*}
			Applying Corollary \ref{lemfourierK}, we obtain 
			\begin{equation}\label{www}
				\begin{aligned}
					\left|\Lambda^{\sigma_4} \mathcal{W}_{x_0}[h](t,x)\right|
					&\lesssim \int_0^t\int_{\mathbb{R}^d}|\Lambda^{\sigma_4-[\tilde \sigma]} K_{x_0}(t,\tau,x-z)||x-z|^{\tilde \sigma-[\tilde \sigma]}dz\|h(\tau)\|_{\dot{C}^{\tilde \sigma}}d\tau\\
					&\lesssim \int_0^t(t-\tau)^{-\frac{s-\sigma_3}{s}}\tau^{-\frac{\sigma_3}{s}}d\tau \sup _{t \in[0, T]} t^{\frac{\sigma_3}{s}}\|h(t)\|_{\dot{C}^{\tilde \sigma}}\\
					&\lesssim \sup _{t \in[0, T]} t^{\frac{\sigma_3}{s}}\|h(t)\|_{\dot{C}^{\tilde \sigma}}.
				\end{aligned}
			\end{equation}
			Note that for $\sigma_4\notin\mathbb{N}$, there holds
            \begin{equation*}
                \|f\|_{\dot C^{\sigma_4}}=\|f\|_{\dot B_{\infty,\infty}^{\sigma_4}}\lesssim \|\Lambda^{\sigma_4}f\|_{L^\infty},
            \end{equation*}
            and the estimate \eqref{www} still holds if $\Lambda^{\sigma_4}$ is replaced by $\nabla^{\sigma_4}$ for $\sigma_4\in\mathbb{N}$. Hence, we obtain 
			\begin{align}\label{LuN1}
				\sup_{x_0\in\mathbb{R}^d}( \|\mathcal{W}_{x_0}[h](t)\|_{\dot C^{\sigma_4}}+\|\Lambda^{\sigma_4}\mathcal{W}_{x_0}[h](t)\|_{L^\infty})\lesssim \sup_{t\in[0,T]}t^{\frac{\sigma_3}{s}}\|h(t)\|_{\dot C^{\tilde \sigma}}.
			\end{align}
			Then we estimate the higher order norm. Assume $n+\sigma_3+\sigma_4=\tilde n+\gamma$, where $\tilde n\in\mathbb{N}$, $\gamma\in(0,1)$. Let $\tilde \gamma\in(\max\{\gamma-s,0\},\gamma)$. Using integration by parts, we obtain
            \begin{equation*}
				\begin{aligned}
					&   |\delta_\alpha \nabla^{\tilde n}\mathcal{W}_{x_0}[h](t,x)|\\
					&\lesssim 
					\int _0^\frac{t}{2} \int_{\mathbb{R}^d}|\delta_\alpha \nabla^{\tilde n}\Lambda^{-(\tilde\sigma-\tilde \gamma)}\K_{x_0}(t,\tau,x-z)||\Lambda^{\tilde\sigma-\tilde \gamma}h(\tau,z)-\Lambda^{\tilde\sigma-\tilde \gamma}h(\tau,x)|dzd\tau\\
					&\quad\quad  +\int _\frac{t}{2}^t\int_{\mathbb{R}^d} |\delta_\alpha \nabla^{\tilde n-n}\Lambda^{-(\tilde\sigma-\tilde \gamma)}\K_{x_0}(t,\tau,x-z)||\nabla^{n}\Lambda^{\tilde\sigma-\tilde \gamma} h(\tau,z)-\nabla^{n}\Lambda^{\tilde\sigma-\tilde \gamma}h(\tau,x)|dzd\tau\\
					&\lesssim \int _0^\frac{t}{2} \int_{\mathbb{R}^d}|\delta_\alpha \nabla^{\tilde n}\Lambda^{-(\tilde\sigma-\tilde \gamma)}\K_{x_0}(t,\tau,x-z)||x-z|^{\tilde \gamma}dz\,\|h(\tau)\|_{\dot C^{\tilde \sigma}}d\tau\\
					&\quad\quad  +\int _\frac{t}{2}^t \int_{\mathbb{R}^d}|\delta_\alpha \nabla^{\tilde n-n}\Lambda^{-(\tilde \sigma-\tilde \gamma)}\K_{x_0}(t,\tau,x-z)||x-z|^{\tilde \gamma}dz\,\|h(\tau)\|_{\dot C^{n+\tilde \sigma}}d\tau\\
                    &:=\mathcal{I}_1+\mathcal{I}_2.
				\end{aligned}
			\end{equation*}
	From Corollary \ref{lemfourierK}, we obtain 
    \begin{align*}
    &\int_{\mathbb{R}^d}|\delta_\alpha \nabla^{\tilde n}\Lambda^{-(\tilde\sigma-\tilde \gamma)}\K_{x_0}(t,\tau,x-z)||x-z|^{\tilde \gamma}dz \lesssim |\alpha|^\gamma (t-\tau)^{-\frac{n+s}{s}},\\
    &\int_{\mathbb{R}^d}|\delta_\alpha \nabla^{\tilde n-n}\Lambda^{-(\tilde\sigma-\tilde \gamma)}\K_{x_0}(t,\tau,x-z)||x-z|^{\tilde \gamma}dz \lesssim |\alpha|^{\tilde \gamma} \min\left\{1,\frac{|\alpha|}{(t-\tau)^{\frac{1}{s}}}\right\}^{1-\tilde \gamma}(t-\tau)^{-\frac{s-(\gamma-\tilde \gamma)}{s}}.
    \end{align*} 
    It follows that 
    \begin{align*}
    \mathcal{I}_1&\lesssim |\alpha|^\gamma t^{-\frac{n+s}{s}}\int _0^\frac{t}{2} \tau^{-\frac{\sigma_3}{s}} d\tau \,\Big(\sup_{t\in[0,T]}t^{\frac{\sigma_3}{s}}\|h(t)\|_{\dot C^{\tilde \sigma}} \Big)\\
    &\lesssim |\alpha|^\gamma t^{-\frac{n+\sigma_3}{s}} \,\Big(\sup_{t\in[0,T]}t^{\frac{\sigma_3}{s}}\|h(t)\|_{\dot C^{\tilde \sigma}} \Big).
    \end{align*}
    For $\mathcal{I}_2$, we have 
    \begin{align*}
    \mathcal{I}_2&\lesssim |\alpha|^{\tilde \gamma} t^{-\frac{n+\sigma_3}{s}}\int_{0}^\frac{t}{2} \min\left\{1,\frac{|\alpha|}{\tau^{\frac{1}{s}}}\right\}^{1-\tilde \gamma}\tau^{-\frac{s-(\gamma-\tilde \gamma)}{s}} d\tau\  \Big( \sup_{t\in[0,T]}t^{\frac{n+\sigma_3 }{s}}\|h(t)\|_{\dot C^{n+\tilde \sigma}}\Big)\\
    &\lesssim |\alpha|^{ \gamma} t^{-\frac{n+\sigma_3}{s}} \Big( \sup_{t\in[0,T]}t^{\frac{n+\sigma_3 }{s}}\|h(t)\|_{\dot C^{n+\tilde \sigma}}\Big).
\end{align*}
    Then we obtain 
    \begin{align*}
     |\delta_\alpha \nabla^{\tilde n}\mathcal{W}_{x_0}[h](t,x)|	& \lesssim |\alpha|^{\gamma} \sup_{t\in[0,T]}(t^{\frac{\sigma_3}{s}}\|h(t)\|_{\dot C^{\tilde \sigma}}+t^{\frac{n+\sigma_3 }{s}}\|h(t)\|_{\dot C^{n+\tilde \sigma}}).
    \end{align*}
    Hence, for any $t\in[0,T]$, 
			\begin{align*}
				\sup_{x_0\in\mathbb{R}^d}\sup_{t\in[0,T]}t^\frac{n+\sigma_3}{s}\|\mathcal{W}_{x_0}[h](t)\|_{\dot C^{n+\sigma_3+\sigma_4}}\lesssim \sup_{t\in[0,T]}\left(t^{\frac{\sigma_3}{s}}\|h(t)\|_{\dot C^{\tilde \sigma}}+t^{\frac{n+\sigma_3}{s}}\|h(t)\|_{\dot C^{n+\tilde \sigma}}\right).
			\end{align*}
            Combining this with \eqref{LuN1} yields \eqref{wv}.
			This completes the proof of the lemma.
		\end{proof}

		By Lemma \ref{lemuLN}, we obtain 
		\begin{equation}\label{uLB}
			\begin{aligned}
				& \sup_{x_0\in\mathbb{R}^d} \sup_{t\in[0,T]} (\|u_{L,x_0}(t)\|_{\dot C^b}+t^\frac{m+\kappa}{s}\|u_{L,x_0}(t)\|_{\dot C^{b+m+\kappa}})\lesssim  \|u_0\|_{\dot C^b},\\
				& \sup_{x_0\in\mathbb{R}^d} \sup_{t\in[0,T]} (t^\frac{\eta}{s}\|u_{L,x_0}(t)\|_{\dot C^{b+\eta}}+t^\frac{m+\kappa}{s}\|u_{L,x_0}(t)\|_{\dot C^{b+m+\kappa}})\lesssim \|u_0\|_{\dot B_{\infty,\infty}^b},
			\end{aligned}
		\end{equation}
		and 
		\begin{align}\label{mainestuN1B}
			\sup_{x_0\in\mathbb{R}^d} \sup_{t\in[0,T]} (\|u_{N,x_0}(t)\|_{\dot C^b}+t^\frac{m+\kappa}{s}\|u_{N,x_0}(t)\|_{\dot C^{b+m+\kappa}})\lesssim  \da_T(f).
		\end{align}
	It remains to estimate $$u_{R,{x_0}}(t,x)= -\int_0^t \int_{\mathbb{R}^d}\K_{x_0}(t,\tau,x-z)R_{x_0}[u](\tau,z)dzd\tau,$$
			where $R_{x_0}[u]$ is defined in \eqref{defRx0}.
            By Lemma \ref{intp} and \eqref{defnormA}, for any $x_0,x\in\mathbb{R}^d$, we have
		\begin{equation}\label{ptRx}
			\begin{aligned}
				\vert R_{x_0}[u](t,x)\vert\lesssim t^{-\frac{\eta}{s}}|x_0-x|^{\eta}\|\A\|_{B_T}\|u\|_{\dot C^{s-\eps}}^{\frac{1}{2}}\|u\|_{\dot C^{s+\eps}}^{\frac{1}{2}}\lesssim |x_0-x|^{\eta}t^{-\frac{\eta+s-b}{s}}\|\A\|_{B_T}\|u\|_{T,*},
			\end{aligned}
		\end{equation}
		for any $s-\eps\geq b+\eta$. 
		For the case $b\in\mathbb{N}_+$, by \eqref{ptRx} and Corollary \ref{lemfourierK}, we obtain 
       \begin{equation}\label{newestuRint}
			\begin{aligned}
        	\vert \nabla^bu_{R,x_0}(t,x)|_{x_0=x}\vert&\lesssim \int_0^t\int_{\mathbb{R}^d}|\nabla ^b\K_{x_0}(t,\tau,x-z)||R_{x_0}[u](\tau,z)|dzd\tau\Big|_{x_0=x}\\
            &\lesssim \int_0^t\tau^{-\frac{\eta+s-b}{s}}\int_{\mathbb{R}^d}|\nabla ^b\K_{x_0}(t,\tau,x-z)||x-z|^\eta dz d\tau\|\A\|_{B_T}\|u\|_{T,*}\\
            &\lesssim \int_0^t\tau^{-\frac{\eta+s-b}{s}}\tau^{-\frac{b-\eta}{s}}d\tau \|\A\|_{B_T}\|u\|_{T,*}\lesssim\|\A\|_{B_T}\|u\|_{T,*}.
			\end{aligned}
		\end{equation}
        If $b\notin\mathbb{N}$,  we have 
		\begin{equation}\label{uRBend}
			\begin{aligned}
				\vert \delta_\alpha \nabla^{[b]}u_{R,x_0}|_{x_0=x}\vert&\lesssim \int_0^t\int_{\mathbb{R}^d}|\delta_\alpha \nabla ^{[b]}\K_{x_0}(t,\tau,x-z)||R_{x_0}[u](\tau,z)|dzd\tau\Big|_{x_0=x}\\
                 &\lesssim \int_0^t\tau^{-\frac{\eta+s-b}{s}}\int_{\mathbb{R}^d}|\delta_\alpha\nabla ^{[b]}\K_{x_0}(t,\tau,x-z)||x-z|^\eta dz d\tau\,\|\A\|_{B_T}\,\|u\|_{T,*}\\
                &\lesssim |\alpha|^{b-[b]}\|\A\|_{B_T}\|u\|_{T,*},
			\end{aligned}
		\end{equation}
        provided $\eta\leq b-[b]$.
        
		To estimate higher-order derivatives, we split the time integral in $u_{R,x_0}$ over the intervals $[0, \frac{t}{2}]$ and $[\frac{t}{2}, t]$, denoting the resulting contributions as $u_{R,x_0}^1$ and $u_{R,x_0}^2$, respectively.  For $u_{R,x_0}^1$, by a similar proof as \eqref{uRBend}, we obtain 
		\begin{equation}\label{ur1}
			\begin{aligned}
			\vert \nabla^{m+[\kappa+b]}u_{R,x_0}^1(t,x)|_{x_0=x}\vert	\lesssim &\left| \int_0^{\frac{t}{2}}\int_{\mathbb{R}^d}\delta_\alpha\nabla^{m+[\kappa+b]} \K_{x_0}(t,\tau,x-z)R_{x_0}[u](\tau,z)dzd\tau\Big|_{x_0=x}\right|\\
            \lesssim& |\alpha|^{\kappa+b-[\kappa+b]}t^{-\frac{m+\kappa}{s}}\|\A\|_{B_T}\|u\|_{T,*},
			\end{aligned}
		\end{equation}
        provided $\eta\leq \kappa+b-[\kappa+b]$.

To estimate $u_{R,x_0}^2$, we need to deal with higher order derivative of $R_{x_0}[u]$. 
Let $n\in\mathbb{N}$, $\sigma\in(0,1)$ such that $n+\sigma\leq  b+m+\kappa-s$. 
Using the inequality 
\begin{align*}
|\delta_\alpha  (fg)(x)-f(x)\delta_\alpha  g(x)|\leq \|g\|_{L^\infty} \|\delta_{-\alpha} f\|_{L^\infty}+\|f\|_{L^\infty}\|  \delta_\alpha g\|_{L^\infty},
\end{align*}
we obtain 
\begin{equation*}
		\begin{aligned}
			\vert \delta_\alpha \nabla^nR_{x_0}[u](t,x)-R_{x_0}[\delta_\alpha \nabla^n u](t,x)\vert\lesssim   
	\min\{(|\alpha|t^{-\frac{1}{s}})^\eta, |\alpha|t^{-\frac{1}{s}}\}t^{-\frac{n+s-b} {s}}\|\A\|_{B_T}\|u\|_{T,*}.
		\end{aligned}
\end{equation*}
Combining this with \eqref{ptRx} yields
\begin{align*}
	\vert \delta_\alpha \nabla^nR_{x_0}[u](t,x)\vert\lesssim  &	\vert R_{x_0}[\delta_\alpha \nabla^n u](t,x)\vert
	+	\min\{(|\alpha|t^{-\frac{1}{s}})^\eta, |\alpha|t^{-\frac{1}{s}}\}t^{-\frac{n+s-b} {s}}\|\A\|_{B_T}\|u\|_{T,*}\\
    \lesssim &(|\alpha|^{\sigma} |x-x_0|^\eta t^{-\frac{\eta+\sigma}{s}}+\min\{(|\alpha|t^{-\frac{1}{s}})^\eta, |\alpha|t^{-\frac{1}{s}}\})t^{-\frac{n+s-b} {s}}\|\A\|_{B_T}\|u\|_{T,*}.
\end{align*}
The first term vanishes if we take $x_0=x$. This indicates that the derivatives are not carried solely by $u$, part of them are absorbed by the coefficient $\A$. This mechanism is essential for avoiding the time singularity as $\tau$ approaches $t$. Furthermore, for double finite difference operator $\delta_\alpha \delta_\beta$, and any $\mu_1,\mu_2\in[0,1]$ with $n+\max\{\mu_1,\mu_2\}\leq b+m+\kappa-s$ and $\min\{\mu_1,\mu_2\}\geq \eta$, it holds
\begin{align*}
|\delta_\alpha \delta_\beta\nabla^nR_{x_0}[u](t,x) ||_{x_0=x}\lesssim  t^{-\frac{n+s-b+\mu_1+\mu_2}{s}}  |\alpha|^{\mu_1}|\beta|^{\mu_2}\|\A\|_{B_T}\|u\|_{T,*}.
\end{align*}
Denote $\bar s=s\mathbf{1}_{s\in\mathbb{N}}+([s]+1)\mathbf{1}_{s\notin\mathbb{N}}$. From this, we obtain 
\begin{align*}
&|\delta_\alpha\nabla^{m+[\kappa+b]-\bar s}R_{x_0}(t,z)-\delta_\alpha\nabla^{m+[\kappa+b]-\bar s}R_{x_0}(t,x)| \Big|_{x_0=x}\\
&\quad\quad\quad\lesssim t^{-\frac{m+\kappa+s-\bar s+\tilde \eta}{s}} |\alpha|^{\kappa+b-[\kappa+b]}|x-z|^{\tilde \eta}\|\A\|_{B_T}\|u\|_{T,*},
\end{align*}
where $\tilde \eta=\eta\mathbf{1}_{s\in\mathbb{N}}+\min\{(\kappa+b-s-[\kappa+b]+\bar s),1\}\mathbf{1}_{s\notin\mathbb{N}}$.
Using integration by parts, we have 
    \begin{equation}\label{Rhi}
        \begin{aligned}
       & |\delta_\alpha \nabla^{m+[\kappa+b]} u_{R,x_0}^2(t,x)|_{x_0=x}|\\&\lesssim \int_{\frac{t}{2}}^t \int_{\mathbb{R}^d}|\nabla^{\bar s}\K_{x_0}(t,\tau,x-z)||\delta_\alpha \nabla^{m+[\kappa+b]-\bar s}R_{x_0}(\tau,z)-\delta_\alpha\nabla^{m+[\kappa+b]-\bar s}R_{x_0}(\tau,x)|dzd\tau\Big|_{x_0=x}\\
        &\lesssim [\alpha]^{\kappa+b-[\kappa+b]} t^{-\frac{m+\kappa+s-\bar s+\tilde\eta}{s}}\int_{\frac{t}{2}}^t \int_{\mathbb{R}^d}|\nabla^{\bar s}\K_{x_0}(t,\tau,z)| |z|^{\tilde \eta}dzd\tau\Big|_{x_0=x}\|\A\|_{B_T}\|u\|_{T,*}.
        \end{aligned}
        \end{equation}
         From Corollary  \ref{lemfourierK}, we obtain 
        \begin{align*}
        \int_{\mathbb{R}^d}|\nabla^{\bar s}\K_{x_0}(t,\tau,z)||z|^{\tilde\eta}dz\lesssim (t-\tau)^{-\frac{\bar s-\tilde \eta}{s}}.
        \end{align*}
        We have $\bar s-\tilde \eta<s$. 
      Integrating in time and 
      combining with \eqref{Rhi},  we deduce that 
        \begin{align}\label{newestuRd}
          & |\delta_\alpha \nabla^{m+[\kappa+b]} u_{R,x_0}^2(t,x)|_{x_0=x}|\lesssim |\alpha|^{\kappa+b-[\kappa+b]} t^{-\frac{m+\kappa}{s}}\|\A\|_{B_T}\|u\|_{T,*}.
    \end{align}
		Combining \eqref{newestuRint}, \eqref{uRBend}, \eqref{ur1} and \eqref{newestuRd}, we get
		\begin{equation}\label{mainestuR}
			\begin{aligned}
				&\sup_{t\in[0,T]}\left(\|\nabla^bu_{R,x_0}(t)|_{x_0=x}\|_{L^\infty}+t^{\frac{m+\kappa}{s}}\frac{\|\delta_\alpha \nabla^{m+[\kappa+b]}u_{R,x_0}(t)|_{x_0=x}\|_{L^\infty}}{|\alpha|^{\kappa+b-[\kappa+b]}}\right)\leq C\|\A\|_{B_T}\|u\|_{T,*},\quad b\in\mathbb{N}_+,\\
				&\sup_{t\in[0,T]}\left(\frac{\|\nabla^{[b]}u_{R,x_0}(t)|_{x_0=x}\|_{L^\infty}}{|\alpha|^{b-[b]}}+t^{\frac{m+\kappa}{s}}\frac{\|\delta_\alpha \nabla^{m+[\kappa+b]}u_{R,x_0}(t)|_{x_0=x}\|_{L^\infty}}{|\alpha|^{\kappa+b-[\kappa+b]}}\right)\leq C\|\A\|_{B_T}\|u\|_{T,*},\quad b\notin\mathbb{N}_+.
			\end{aligned}
		\end{equation}
	We conclude from \eqref{holu}, \eqref{uLB}, \eqref{mainestuN1B} and \eqref{mainestuR} that 
		\begin{equation*}\label{mainestuB}
			\|u\|_{T}\lesssim \|u_0\|_{\dot C^b}+\da_T(f)+\|\A\|_{B_T}\|u\|_{T,*},
		\end{equation*}
		and the non-endpoint result
		\begin{equation*}\label{mainestuBnonend}
			\|u\|_{T,*}\lesssim \|u_0\|_{\dot B_{\infty,\infty}^b}+\da_T(f)+\|\A\|_{B_T}\|u\|_{T,*}.
		\end{equation*}
	This completes the proof of Theorem \ref{lemmain}. \vspace{0.3cm}
		\begin{corollary}\label{remlinf}
		  In the same setting as Theorem \ref{lemmain}, the proof of Theorem \ref{lemmain} also yields the following estimates:
            	\begin{align}\label{main31}
				&\sup_{t\in[0,T]} (\|\Lambda^bu(t)\|_{L^\infty}+t^\frac{m+\kappa }{s}\|\Lambda^bu(t)\|_{\dot C^{m+\kappa }})\lesssim \|\Lambda^b u_0\|_{L^\infty}+  \da_T(f)+\|\A\|_{B_T}\|u\|_{T,*},\\
             &\sup_{t\in[0,T]} (t^{\frac{\eta}{s}}\|\Lambda^bu(t)\|_{\dot C^{\eta}}+t^\frac{m+\kappa}{s}\|\Lambda^bu(t)\|_{\dot C^{m+\kappa}})\lesssim \|u_0\|_{\dot B^b_{\infty,\infty}}+  \da_T(f)+\|\A\|_{B_T}\|u\|_{T,*}.   \label{main41}
			\end{align}
		\end{corollary}
        \begin{remark}\label{rmkcons}
            The implicit constants in \eqref{main3} and \eqref{main31} can be controlled by  $\frac{(c_0^{-1}c_1)^{m+s+2}}{\kappa_0(s-\kappa)}$, the implicit constants in \eqref{main4} and \eqref{main41} can be controlled by  $\frac{(c_0^{-1}c_1)^{m+s+2}}{\eta\kappa_0(s-\kappa)}$,. Furthermore, when only $\|u\|_{\dot C^b}$ or $\|\Lambda^b u\|_{L^\infty}$ is considered (i.e., without higher regularity term), the estimates
\begin{align*}
    &\sup_{t \in [0,T]} \|u(t)\|_{\dot C^b} \lesssim \|u_0\|_{\dot C^b} + \da_T(f) +\|\A\|_{B_T}\|u\|_{T,*},\\
    &  \sup_{t \in [0,T]} \|\Lambda^b u(t)\|_{L^\infty} \lesssim \|\Lambda^bu_0\|_{L^\infty} + \da_T(f) +\|\A\|_{B_T}\|u\|_{T,*},
\end{align*}
hold with implicit constants independent of $c_0,c_1$, see \eqref{KL1}. This is natural since no smoothing effect is involved in this case, and the bound does not rely on the lower or upper bound of the coefficient $\A(t,x,\xi)$.
        \end{remark}
		\begin{remark}
			In contrast to conventional approaches in existing literature, our methodology avoids the introduction of spatial truncations during coefficient freezing. Instead, we fix the spatial coordinate $x_0$
			in the frozen coefficients and explicitly construct the fundamental solution (Green's function) of the linearized equation. This enables us to derive a closed-form representation formula for the solution. By differentiating this formula and subsequently setting 
			$x_0=x$, we rigorously demonstrate that the coefficient freezing procedure only generates residual terms of lower differential order. This approach eliminates artificial truncation errors while preserving the intrinsic regularity structure of the problem.
		\end{remark}
		
		Finally, as an extension to Theorem \ref{lemmain}, we establish the following proposition for the endpoint case $b=0$, provided that the coefficient $\A$ is regular in $x$  and the forcing term possesses a derivative structure. Specifically, consider the system
		\begin{equation}\label{eqpara1}
			\begin{aligned}
				&\partial_{t} u(t, x)+\mathcal{L}_{s} u(t, x)=\mathcal{P}_{\gamma} f(t, x)+g(t,x),\ \ \ (t,x)\in(0,T)\times \mathbb{R}^d,\\
				&u|_{t=0}=u_0,
			\end{aligned}    
		\end{equation}
		where  $\mathcal{P}_{\gamma}$ is a differential operator defined by 
		\begin{align}\label{defPg}
			\mathcal{P}_{\gamma} f(t,x)=\int_{\mathbb{R}^d} \B(\xi)\hat f(t,\xi)e^{ix\cdot\xi}d\xi,\ \ \ \ 0<\gamma\leq s,
		\end{align}
		where the symbol $\B(\xi)\in \mathbb{R}^{N\times N}$ satisfies the growth condition
		\begin{align*}
			\left|\nabla^{j}_\xi \B(\xi)\right| \lesssim |\xi|^{\gamma-j},\ \ \ \forall 0\leq j\leq d+\kappa+2.
		\end{align*}
		The operator $\mathcal{L}_s$ is defined by \eqref{defop} with $\A(t,x,\xi)$ satisfying \eqref{condop}. Moreover,  suppose that there exists $\M>1$ such that 
		\begin{align}\label{smA}
			&\sum_ {\substack{j\leq m+\kappa+2\\ l\leq d+s+2 }} {|\xi|^{l-s}}{\left|\nabla _x^j\nabla^{l}_\xi \A(t,x,\xi)\right| }\leq \M,
			\ \ \forall  \xi \neq 0, (t,x)\in[0,T]\times\mathbb{R}^d.
		\end{align}
		We define the following quantity for the force terms,
		\begin{align*}
			&\da_T^{m,\kappa,\gamma}(f,g)=\sup_{t\in[0,T]}\left(t^{\frac{\kappa}{s}}\|f(t)\|_{\dot C^{\kappa_{\gamma}}}+t^{\frac{m+\kappa}{s}}\|f(t)\|_{\dot C^{m+\kappa_{\gamma}}} \right)+\|g\|_{L_T^1L^\infty}+\sup_{t\in[0,T]}t^{\frac{m}{s}+1}\|g(t)\|_{\dot C^{m}},
		\end{align*}
		with $\kappa_{\gamma}=\kappa+\gamma-s$. We have the following Schauder-type estimate.
		\begin{proposition}\label{propb=0}
			For  any $T>0$, consider the Cauchy problem \eqref{eqpara1} with $\A(t,x,\xi)$ satisfying \eqref{condop} and \eqref{smA}, if $\|u_0\|_{ L^\infty}+\da_T^{m,\kappa,\gamma}(f,g)<\infty$,  then there exists a unique solution $u\in C((0,T],L^\infty_x(\mathbb{R}^d))\cap L^\infty_{loc}((0,T],C^{m}_x(\mathbb{R}^d))$, and the following estimates hold:\\
			i)
			\begin{equation}\label{main2}
				\begin{aligned}
					&\sup_{t\in[0,T]} (\|u(t)\|_{L^\infty}+t^\frac{m+\kappa}{s}\|u(t)\|_{\dot C^{m+\kappa}})\lesssim  e^{CT\log (T+2)}\left(\|u_0\|_{L^\infty}+  \da_T^{m,\kappa,\gamma}(f,g)\right),
				\end{aligned}
			\end{equation}
			where $C>0$ is a constant depending only on $\M,m, s,c_0$.\vspace{0.1cm}\\
			ii) If $\A(t,x,\xi)$ is independent of $x$, we have the improved estimate:
			\begin{equation}\label{main1}
				\begin{aligned}
					&\sup_{t\in[0,T]} (\|u(t)\|_{L^\infty}+t^\frac{m+\kappa }{s}\|u(t)\|_{\dot C^{m+\kappa}})\lesssim\|u_0\|_{L^\infty}+
					\da_T^{m,\kappa,\gamma}(f,g). 
				\end{aligned}
			\end{equation}
		\end{proposition}
		\begin{remark}\label{rmk1} ~\\
			i) The implicit constants in \eqref{main2} and \eqref{main1} can be controlled by  $\frac{(c_0^{-1}c_1)^{m+s+2}}{\kappa_\gamma(s-\kappa)}$. Furthermore, when only the \(L^\infty\) norm is considered (i.e., without the Hölder regularity term), the estimate
$$
    \sup_{t \in [0,T]} \|u(t)\|_{L^\infty} \lesssim e^{C T \log (T+2)} \Big( \|u_0\|_{L^\infty} + \da_T^{m,\kappa,\gamma}(f,g) \Big)
$$
holds with implicit constants independent of $c_0,c_1$, see \eqref{KL1}. \vspace{0.1cm}\\
			ii) The estimate \eqref{main2} can be applied to obtain local well-posedness of some quasilinear equations with initial data satisfying some ``continuity condition" (which indicates that the variable coefficient can be approximated by smooth function). Moreover, if $\A(t,x,\xi)$ is independent of $x$, the estimate \eqref{main1} holds, where the right-hand side does not have an exponential growth in time. This will be applied to obtain global well-posedness of some quasilinear equations with small initial data (which indicates that the variable coefficient can be approximated by a constant).\vspace{0.1cm}\\
			iii) The estimates in Proposition \ref{propb=0} are optimal, 
			being scaling-invariant with respect to \eqref{eqpara1}.\vspace{0.1cm}\\
			iv) We can extend the estimates in Proposition \ref{propb=0} to ultra-parabolic equations, where the operator $\mathcal{L}_s$ is hypoelliptic and the system \eqref{eqpara1} also exhibits smoothing effect (see \cite{Hormand,wz1,wz2}).
		\end{remark}
		\begin{proof}
			The proof of Proposition \ref{propb=0} consists of the following two main parts. First we establish a priori estimate for \eqref{eqpara1}, where uniqueness follows, then we prove the existence  of solution via compactness method.\\
			\textbf{Step 1: A priori estimate.}\\
			Assume $u$ is a smooth solution to \eqref{eqpara1}. 
            Denote $$
            \|u\|_{X_T}=\sup_{t\in[0,T]} (\|u(t)\|_{L^\infty}+t^\frac{m+\kappa}{s}\|u(t)\|_{\dot C^{m+\kappa}}).
            $$
          We recall the decomposition of the solution $u=u_{L,x_0}+u_{N,x_0}+u_{R,x_0}$ as in \eqref{uform}, where $u_{N,x_0}=u_{N,x_0}^1+u_{N,x_0}^2$, with
			\begin{equation*}
			\begin{aligned}	&u_{N,x_0}^1=\int_0^t\int_{\mathbb{R}^d}\K_{x_0}(t,\tau,x-z)\,\mathcal{P}_\gamma f(\tau,z)\,dzd\tau,\\
            &u_{N,x_0}^2=\int_0^t\int_{\mathbb{R}^d}\K_{x_0}(t,\tau,x-z)\, g(\tau,z)\, dzd\tau.
            \end{aligned}
			\end{equation*}
			By Lemma \ref{lemuLN}, we  obtain
			\begin{equation}\label{uLuN1}
				\begin{aligned}
					&\|u_{L,x_0}(t,x)|_{x_0=x}\|_{L^\infty}+\sup_{\alpha} t^{\frac{m+\kappa }{s}}\frac{\|\delta_\alpha\nabla^{m+[\kappa]} u_{L,x_0}(t,x)|_{x_0=x}\|_{L^\infty}}{|\alpha|^{\kappa-[\kappa]}}\lesssim \|u_0\|_{L^\infty},\\
					&\|u_{N,x_0}^1(t,x)|_{x_0=x}\|_{L^\infty}+\sup_{\alpha} t^{\frac{m+\kappa }{s}}\frac{\|\delta_\alpha\nabla^{m+[\kappa]} u_{N,x_0}^1(t,x)|_{x_0=x}\|_{L^\infty}}{|\alpha|^{\kappa-[\kappa]}}\lesssim\da_T^{m,\kappa,\gamma}(f,0).\\
				\end{aligned}
			\end{equation}
			On the other hand, for $u_{N,x_0}^2$, by \eqref{kerL1}, it is easy to check that 
			\begin{align*}
				|u_{N,x_0}^2(t,x)|\lesssim \int_0^t \|\K_{x_0}(t,\tau)\|_{L^1}\|g(\tau)\|_{L^\infty}d\tau\lesssim \|g\|_{L^1_TL^\infty}.
			\end{align*}
			For H\"{o}lder norm, by \eqref{delKL1} we have
			\begin{equation*}
				\begin{aligned}
					&|\delta_\alpha\nabla^{m+[\kappa]} u_{N,x_0}^2(t,x)|\\
					&\lesssim\int_0^\frac{t}{2}\int _{\mathbb{R}^d} \left|\delta_\alpha\nabla^{m+[\kappa]}\K_{x_0}(t,\tau,x-z)\right| |g(\tau,z)|dzd\tau+\int_\frac{t}{2}^t\int _{\mathbb{R}^d} \left|\delta_\alpha\nabla^{[\kappa]} \K_{x_0}(t,\tau,x-z)\right| \left|\nabla^mg(\tau,z)\right|dzd\tau\\
					&\lesssim \int_0^\frac{t}{2}\frac{1}{(t-\tau)^{\frac{m+[\kappa]}{s}}}\min\left\{1,\frac{|\alpha|}{(t-\tau)^{\frac{1}{s}}}\right\}\|g(\tau)\|_{L^\infty}d\tau+\int_{\frac{t}{2}}^{t}\frac{1}{(t-\tau)^{\frac{[\kappa]}{s}}}\min\left\{1,\frac{|\alpha|}{(t-\tau)^{\frac{1}{s}}}\right\}\|g(\tau)\|_{\dot C^m}d\tau\\
					&\lesssim |\alpha|^{\kappa-[\kappa]} t^{-\frac{m+\kappa }{s}}\left(\|g\|_{L^1_TL^\infty}+\sup_{\tau'\in[0,T]}\tau'^{\frac{m}{s}+1}\|g(\tau')\|_{\dot C^m}\right).
				\end{aligned}
			\end{equation*}
			Hence, 
			\begin{equation}\label{uN2}
				\begin{aligned}
					\|u_{N,x_0}^2(t,x)|_{x_0=x}\|_{L^\infty}+\sup_{\alpha} t^{\frac{m+\kappa }{s}}\frac{\|\delta_\alpha\nabla^{m+[\kappa]} u_{N,x_0}^2(t,x)|_{x_0=x}\|_{L^\infty}}{|\alpha|^{\kappa-[\kappa]}}\lesssim \da_T^{m,\kappa,\gamma}(0,g).
				\end{aligned}
			\end{equation}
			If $\A(t,x,\xi)=\A(t,\xi)$, then $u_{R,x_0}\equiv 0$. The estimate \eqref{main1} follows from \eqref{uLuN1} and  \eqref{uN2}. 
\vspace{0.1cm}

            In the case where $\A(t,x,\xi)$ varies with $x$, it remains to consider $u_{R,x_0}$. Recall $R_{x_0}[u]$ defined in \eqref{defRx0}.
			Let $0<\sigma<\min\{1,s\}$.
            Note that
\begin{align}\label{decomR}
R_{x_0}[u](t,x)=\Lambda^\sigma R^1_{x_0}[u](t,x)+R^2[u](t,x),
\end{align}
where
\begin{align*}
R^1_{x_0}[u](t,x)
&=(2\pi)^{-\frac d2}\int_{\mathbb R^d}|\xi|^{-\sigma}
\bigl(\A(t,x_0,\xi)-\A(t,x,\xi)\bigr)\widehat u(t,\xi)e^{ix\cdot\xi}\,d\xi,\\
R^2[u](t,x)
&=C_{d,\sigma}\int_{\mathbb R^d}\int_{\mathbb R^d}
|\xi|^{-\sigma}\bigl(\A(t,x,\xi)-\A(t,x-z,\xi)\bigr)
\widehat{\tau_z u}(t,\xi)e^{ix\cdot\xi}\,
\frac{d\xi\,dz}{|z|^{d+\sigma}},
\end{align*}
with $\tau_z u(x)=u(x-z)$.
We briefly justify this decomposition. For a function
$h(x,y):\mathbb R^d\times\mathbb R^d\to\mathbb R$, denote by
$\Lambda_1$ and $\Lambda_2$ the fractional Laplacian acting on the
$x$- and $y$-variables, respectively. Let
$
h_1:=\Lambda_1^{-\sigma}h.
$
Then
$
h=\Lambda_1^\sigma h_1.$
Using the identity
\[
h_1(x,x)-h_1(x-z,x)
=
\bigl(h_1(x,x)-h_1(x-z,x-z)\bigr)
+\bigl(h_1(x-z,x-z)-h_1(x-z,x)\bigr),
\]
we obtain
\begin{align}\label{decomh}
h(x,x)=\Lambda_1^\sigma h_1(\cdot,x)(x)
=\Lambda_x^\sigma\bigl(h_1(\cdot,\cdot)\bigr)(x)+h_R(x),
\end{align}
where the commutator term $h_R$ is given by
\begin{align*}
h_R(x)
&:=\Lambda_1^\sigma h_1(\cdot,x)(x)-\Lambda_x^\sigma\bigl(h_1(\cdot,\cdot)\bigr)(x)\bigr)\\
&=C_{d,\sigma}\int_{\mathbb R^d}
\bigl(h_1(x-z,x-z)-h_1(x-z,x)\bigr)\frac{dz}{|z|^{d+\sigma}}.
\end{align*}
Now take
\[
h(x,y)
=(2\pi)^{-\frac d2}\int_{\mathbb R^d}
\bigl(\A(t,x_0,\xi)-\A(t,y,\xi)\bigr)\widehat u(t,\xi)e^{ix\cdot\xi}\,d\xi.
\]
Then
\[
R_{x_0}[u](t,x)=h(x,x),
\qquad
R^1_{x_0}[u](t,x)=h_1(x,x),
\qquad
R^2[u](t,x)=h_R(x),
\]
and thus \eqref{decomR} follows from \eqref{decomh}.

         From \eqref{decomR} and \eqref{uform}, we have 
			\begin{align*}
				u_{R,{x_0}}(t,x)&= -\int_0^{t}\int_{\mathbb{R}^d}\Lambda^\sigma\K_{x_0}(t,\tau,x-z)R_{x_0}^1[u](\tau,z)dzd\tau-\int_0^{t}\int_{\mathbb{R}^d}\K_{x_0}(t,\tau,x-z)R^2[u](\tau,z)dzd\tau\\
				&:=  u_{R,{x_0}}^1(t,x)+  u_{R,{x_0}}^2(t,x).
			\end{align*}
            Following \eqref{uN2}, we obtain 
            \begin{align*}
            \|u_{R,x_0}^2(t,x)|_{x_0=x}\|_{L^\infty}+&\sup_{\alpha} t^{\frac{m+\kappa }{s}}\frac{\|\delta_\alpha\nabla^{m+[\kappa]} u_{R,x_0}^2(t,x)|_{x_0=x}\|_{L^\infty}}{|\alpha|^{\kappa-[\kappa]}} \lesssim \da_T^{m,\kappa,\gamma}(0,R^2[u]).
            \end{align*}
            By Lemma \ref{intp}, we have
            \begin{align*}
            &\|\nabla^k_x R^2[u](t)\|_{L^\infty}\lesssim \M t^{-\frac{k+s-\sigma}{s}}\|u\|_{X_T},\ \ \ \forall \ 0\leq k\leq m.
            \end{align*}
            This implies $\da_T^{m,\kappa,\gamma}(0,R^2[u])\lesssim \M T^\frac{\sigma}{s}\|u\|_{X_T}$. Hence, we obtain 
            \begin{align}\label{ur22}
              \|u_{R,x_0}^2(t,x)|_{x_0=x}\|_{L^\infty}+&\sup_{\alpha} t^{\frac{m+\kappa }{s}}\frac{\|\delta_\alpha\nabla^{m+[\kappa]} u_{R,x_0}^2(t,x)|_{x_0=x}\|_{L^\infty}}{|\alpha|^{\kappa-[\kappa]}} \lesssim \M T^\frac{\sigma}{s}\|u\|_{X_T}.
            \end{align}
	Then we estimate $u_{R,x_0}^1$.		By the Lemma \ref{intp}, we have for any $\eta'\in[0,1]$,
			\begin{equation}\label{mainestRx1}
				\begin{aligned}
					&\vert R_{x_0}^1[u](t,x)\vert\lesssim  \M|x_0-x|^{\eta'}\|u(t)\|_{\dot C^{s-\sigma-\eps}}^{\frac{1}{2}}\|u(t)\|_{\dot C^{s-\sigma+\eps}}^{\frac{1}{2}}\lesssim \M |x_0-x|^{\eta'} t^{-\frac{s-\sigma}{s}}\|u\|_{X_T},
				\end{aligned}
			\end{equation}
			for some $\eps$ small enough. Similarly, for any $0\leq k\leq m$ and $\eta'\in[0,1]$, we obtain 
			\begin{equation*}\label{estRx0der}
				\begin{aligned}
					&|\nabla^{k}_x R_{x_0}^1[u](t,x)|\lesssim  \M(|x_0-x|^{\eta'}+t^{\frac{1}{s}})t^{-\frac{k+s-\sigma}{s}}\|u\|_{X_T}.
				\end{aligned}
			\end{equation*}
		By \eqref{mainestRx1} an Lemma \ref{esthtker}, 
			\begin{align*}
				\vert u_{R,x_0}^1(t,x)|_{x_0=x}\vert&\lesssim \M 
                \int_0^t \tau^{-\frac{s-\sigma}{s}}\int_{\mathbb{R}^d}|\Lambda^\sigma \K_{x_0}(t,\tau,x-z)||x-z|^{\eta'} dzd\tau \|u\|_{X_T}
                \\
                &\lesssim  \M\int_0^t(t-\tau)^{-\frac{\sigma-\eta'}{s}}\tau^{-\frac{s-\sigma}{s}}d\tau \|u\|_{X_T}\lesssim \M T^{\frac{\eta'}{s}}\|u\|_{X_T},
			\end{align*}
            for any $0<\eta'<\sigma$.
		Then we estimate higher order derivatives. 
		 By taking $\sigma\in(s-\kappa,s-[\kappa])$, and decomposing the integral interval into $[0,\frac{t}{2}]\cup[\frac{t}{2},t]$, we obtain 
			\begin{align*}
				\vert \delta_\alpha\nabla^{m+[\kappa]} u_{R,x_0}^{1}(t,x)\vert&\lesssim \left|\int_0^{\frac{t}{2}}\int_{\mathbb{R}^d}\delta_\alpha\nabla^{m+[\kappa]}\Lambda_z^\sigma\K_{x_0}(t,\tau,x-z)R_{x_0}^1(\tau,z)dzd\tau\right|\\
                &\quad+\left|\int_{\frac{t}{2}}^t\int_{\mathbb{R}^d}\delta_\alpha\nabla^{[\kappa]}\Lambda_z^\sigma\K_{x_0}(t,\tau,x-z)\nabla^mR_{x_0}^1(\tau,z)dzd\tau\right|\\
                &\lesssim (|\alpha|^{\kappa-[\kappa]}+|x_0-x|^{\kappa-[\kappa]})T^{\frac{\kappa-[\kappa]}{s}}t^{-\frac{m+\kappa}{s}}\|u\|_{X_T}.
			\end{align*}
			So we have proved
            \begin{align*}\label{ur11}
              \|u_{R,x_0}^1(t,x)|_{x_0=x}\|_{L^\infty}+&\sup_{\alpha} t^{\frac{m+\kappa }{s}}\frac{\|\delta_\alpha\nabla^{m+[\kappa]} u_{R,x_0}^1(t,x)|_{x_0=x}\|_{L^\infty}}{|\alpha|^{\kappa-[\kappa]}} \lesssim \M T^\frac{\eta'}{s}\|u\|_{X_T}.
            \end{align*}
            Combining this with \eqref{ur22} yields
			\begin{equation*}\label{uRf}
				\begin{aligned}
					&\sup_{t\in[0,T]}\left( \|u_{R,x_0}(t)|_{x_0=x}\|_{L^\infty}+ t^\frac{m+\kappa }{s}\sup_\alpha\frac{\|\left(\delta_\alpha\nabla^{m+[\kappa]}u_{R,x_0}(t)\right)|_{x_0=x}\|_{L^\infty}}{|\alpha|^{\kappa-[\kappa]}} \right) \leq C \M T^\frac{\eta'}{s}\|u\|_{X_T},
				\end{aligned}
			\end{equation*}
			provided $0<\eta'<\sigma<\kappa-[\kappa]$. By taking $T$ small enough such that $C \M T^{\frac{\sigma}{s}}<\frac{1}{100}$, from \eqref{holu}, \eqref{uLuN1} and \eqref{uN2}
			we conclude that \begin{equation}\label{mainestuL}
				\begin{aligned}
					&\sup_{t\in[0,T]}\left( \|u(t)\|_{L^\infty}+ t^\frac{m+\kappa }{s}\sup_\alpha\frac{\|\delta_\alpha\nabla^{m+[\kappa]}u(t)\|_{L^\infty}}{|\alpha|^{\kappa-[\kappa]}} \right) \lesssim  \M\left(\|u_0\|_{L^\infty}+\da_T^{m,\kappa,\gamma}(f,g)\right).
				\end{aligned}
			\end{equation}
			From \eqref{mainestuL}, if we start from $\frac{T_0}{2}$, it holds 
			\begin{align*}
				\sup_{t\in[\frac{T_0}{2},\frac{3T_0}{2}]}(\|u(t)\|_{L^\infty}+(t-T_0/2)^\frac{m+\kappa}{s}\|u(t)\|_{\dot C^{m+\kappa}})\leq& C \M \left(\|u(T_0/2)\|_{L^\infty}+\da_{3T_0/2}^{m,\kappa,\gamma}(f,g)\right).
			\end{align*}
			In particular, for $t\in[T_0,3T_0/2]$, one has $t-T_0/2\sim t$, then we obtain 
			\begin{align*}
				\|u\|_{X_{3T_0/2}}\leq C  \M\left(\|u(T_0/2)\|_{L^\infty}+\da_{3T_0/2}^{m,\kappa,\gamma}(f,g)\right)\leq C^2 \M^2(\|u_0\|_{L^\infty}+\da^{m,\kappa,\gamma}_{3T_0/2}(f,g)).
			\end{align*}
			Repeating the procedure $n-1$ times, it follows that 
			\begin{align*}
				\|u\|_{X_{(n+1)T_0/2}}\leq C^n \M^n(n!)^\frac{m+\kappa}{s}(\|u_0\|_{L^\infty}+\da^{m,\kappa,\gamma}_{(n+1)T_0/2}(f,g)).
			\end{align*}
			Take $n=[\frac{2T}{T_0}]+1$, then by the Stirling formula, we derive
			\begin{align*}
				\|u\|_{X_T}\lesssim e^{\tilde CT\log (T+2)}(\|u_0\|_{L^\infty}+\da_T^{m,\kappa,\gamma}(f,g)),\ \ \ \forall T>0,
			\end{align*}
			which implies \eqref{main2}. This a priori estimate also ensures uniqueness by taking $u_0=f=g=0$.\\
			\textbf{Step 2: Existence of solution.}\\
			To prove the existence, we use the compactness method. Consider a regularized approximation system, for which the Cauchy problem is easily studied, and whose solutions are expected to converge to solutions of the original system \eqref{eqpara1}. More precisely, let $\varepsilon_1,\varepsilon_2\in(0,1)$, we  consider for $u:(0,\infty)\times \mathbb R^d\rightarrow\mathbb R^N$, \begin{equation*}
				\begin{aligned}
					&  \partial_tu(t,x)+\mathcal{L}_su(t,x)=(1-\chi_{\varepsilon_1}(t))(\mathcal{P}_\gamma f(t,x)+g(t,x))\quad \text{in}\ (0,\infty)\times\mathbb{R}^d,\\
					& u|_{t=0}=u_0\ast \rho_{\varepsilon_2}.
				\end{aligned}
			\end{equation*}
			Here $\rho_{\eps_2}$ is the standard mollifier in $\mathbb{R}^d$, $\chi_{\eps_1}$ is a smooth temporal cutoff function that satisfies $\mathbf{1}_{[0,\eps_1]}\leq \chi_{\eps_1}\leq \mathbf{1}_{[0,2\eps_1]}$ and the force term is supported away from time $0$ and belongs to $C_c^\infty((0,\infty);C^\infty(\mathbb{R}^d))$. By the classical theory of parabolic equations, there exists a unique classical solution $u_\eps=u_{\eps_1,\eps_2}\in C([0,\infty),L^\infty(\mathbb{R}^d))\cap L^\infty_{loc}((0,\infty),C^{m+\kappa }(\mathbb{R}^d))$, see Remark \ref{semi}. Moreover, for any $T>0$, we have  
			\begin{align}\label{hox}
				&\sup_{t\in[0,T]} (\|u_\eps(t)\|_{L^\infty}+t^\frac{m+\kappa }{s}\|u_\eps(t)\|_{\dot C^{m+\kappa }})\leq Ce^{\tilde CT\log (T+2)}(\|u_0\|_{L^\infty}+ \da_T^{m,\kappa,\gamma}(f,g)).
			\end{align}
            By interpolation inequality, we obtain 
           \begin{equation}\label{hot}
                \begin{aligned}
            \sup_{t\in[0,T]}t\|\partial_t u_\eps(t)\|_{L^\infty}&\leq     \sup_{t\in[0,T]}t(\|\mathcal{L}_s u_\eps(t)\|_{L^\infty}+ \|\mathcal{P}_\gamma f(t)+g(t)\|_{L^\infty})\\
            &\leq
            Ce^{\tilde CT\log (T+2)}(\|u_0\|_{L^\infty}+ \da_T^{m,\kappa,\gamma}(f,g)).
            \end{aligned}
           \end{equation}
            The estimates \eqref{hox} and \eqref{hot} guarantee the compactness and convergence of the sequence $\{u_\eps\}$.
		 Taking $\varepsilon_1,\eps_2\to 0$, we obtain that $u_{\varepsilon}$ converges to $u\in C((0,T],L^\infty(\mathbb{R}^d))\cap L^\infty_{loc}((0,T],C^{m +\kappa-\varepsilon'}(\mathbb{R}^d))$ for any $\varepsilon'>0$, which is a solution to the original system \eqref{eqpara1} and satisfies the estimate \eqref{hox}. This completes the proof of the theorem.\vspace{0.2cm}
		\end{proof}
	
		We now prove Proposition \ref{prop111}, which develops a general framework for deriving a priori estimates by applying Theorem \ref{lemmain} and Proposition \ref{propb=0} to generic quasilinear systems. This methodology is subsequently applied to several distinct model equations in Sections \ref{intomus1}–\ref{secpes3d}.\vspace{0.2cm}\\
		\begin{proof}[Proof of Proposition \ref{prop111}] Recalling the norms $\|\cdot\|_T,\|\cdot\|_{T,*}$ defined in \eqref{musdefnorm}, we first give a priori estimates of $\|\nabla U\|_{L^\infty_TL^\infty}$ and $\|U-\Phi\|_{T,*}$.\\
			{\bf Step 1: Estimate of $\|\nabla U\|_{L^\infty_TL^\infty}$}.\\ 
			Here we assume $U(t,x)$ is a regular solution. From \eqref{besovgen}, we have  
			\begin{equation*}
				\partial_t  U(t,x) + A(\nabla \Phi(x), \nabla) U(t,x) = \mathcal{N}(U)(t,x)+\mathcal{R}(U,\Phi)(t,x),
			\end{equation*}
			where 
			\begin{align*}
				&\mathcal{R}(U,\Phi)(t,x)=A(\nabla \Phi(x), \nabla)U(t,x)-A(\nabla U(t,x), \nabla)U(t,x).
			\end{align*}
			Applying Theorem \ref{lemmain} and \eqref{besovno}, we obtain 
			\begin{equation}\label{UT}
				\begin{aligned}
					\|\nabla U\|_{L^\infty_TL^\infty}&\lesssim \|\nabla U_0\|_{L^\infty}+\da_T(\mathcal{N}(U)+\mathcal{R}(U,\Phi))+\|\A\|_{B_T}\|U\|_{T,*}\\
					&\lesssim \|\nabla U_0\|_{L^\infty}+\|U\|_{T,*}^{1+\epsilon_0} (1 + \|\nabla U\|_{L^\infty_T L^\infty_x} + \|U\|_{T,*})^{c_1 }\\
					&\quad\quad+(1 + \|\nabla U\|_{L^\infty_T L^\infty_x} + \|\nabla \Phi\|_{L^\infty_T L^\infty_x})^{c_2}(\|U-\Phi\|_T\|U\|_{T,*}+T^\frac{1}{s}\|\Phi\|_{C^{m+2}}\|U\|_{T,*}).
				\end{aligned}
			\end{equation}
			{\bf Step 2: Estimate of $\|U-\Phi\|_{T,*}$}.\\ 
			Rewrite \eqref{besovgen} as 
			\begin{align*}
				\partial_t (U-\Phi)+A(\nabla U,\nabla)(U-\Phi)=\mathcal{N}(U)-A(\nabla U,\nabla)\Phi.
			\end{align*}
			By Theorem \ref{lemmain} and the estimates \eqref{besovno}, \eqref{ABT}, we obtain 
			\begin{equation}\label{bes11}
				\begin{aligned}
					&||U- \Phi||_{T,*}\\
                    &\lesssim (1+\|\nabla U\|_{L^\infty_TL^\infty_x})^{c_2(m+s)}\left(\|U_0-\Phi\|_{\dot B^1_{\infty,\infty}}+\da_T(\mathcal{N}(U)-A(\nabla U,\nabla)\Phi)+\|\A\|_{B_T}\|U-\Phi\|_{T,*}\right)\\
                    &\lesssim
					(1+||\nabla U||_{L^\infty_TL^\infty_x}+||U||_{T,*})^{\frac{c}{2}} (\|U_0-\Phi\|_{\dot B^1_{\infty,\infty}}+||U||_{T,*}^{1+\epsilon_0}+T^\frac{1}{s}\|\Phi\|_{C^{m+s}}),
				\end{aligned}
			\end{equation}
			with $c=2(c_1+c_2(m+s)+c_3+1)$.
			Now we denote
			\begin{align*}
				B(T):=\|U-\Phi\|_{T,*},\ \ \ \ M(T):=\|\nabla U\|_{L^\infty}.
			\end{align*}
			Note that 
			\begin{align*}
				\| U\|_{T,*}\leq \| U-\Phi\|_{T,*}+\|\Phi\|_{T,*}\lesssim B(T)+T^\frac{\eta}{s}\|\Phi\|_{C^{m+s+1}}.
			\end{align*}
			Hence, we write \eqref{UT} and \eqref{bes11} as 
			\begin{align*}
				&M(T)\lesssim \|\nabla U_0\|_{L^\infty}+(B(T)+T^\frac{\eta}{s}\|\Phi\|_{C^{m+s+1}})(1+M(T)+B(T)+T^\frac{1}{s}\|\Phi\|_{C^{m+s+1}})^{\frac{c}{2}},\\
				&B(T)\lesssim (\|U_0-\Phi\|_{\dot B^1_{\infty,\infty}}+B(T)^{1+\eps_0}+T^\frac{1}{s}\|\Phi\|_{C^{m+s+1}})(1+M(T)+B(T)+T^\frac{1}{s}\|\Phi\|_{C^{m+s+1}})^{\frac{c}{2}}.
			\end{align*}
			It is easy to check that under the assumption $$(1+||\nabla U_0||_{L^\infty})^c\|\nabla(U_0-\Phi)\|_{\dot B_{\infty,\infty}^0}\leq \eps \ll 1,$$ one can take
			$T_0=T_0(\|\Phi\|_{C^{m+s+1}})$ small enough such that
			\begin{align*}
				&M(T_0)\leq C(\eps+ \|\nabla U_0\|_{L^\infty}),\\
				&B(T_0)\leq C\eps.
			\end{align*}
			Moreover, if $\Phi\equiv 0$, then 
			\begin{align*}
				&M(T)\lesssim \|\nabla U_0\|_{L^\infty}+B(T)(1+M(T)+B(T))^{\frac{c}{2}},\\
				&B(T)\lesssim (\|U_0\|_{\dot B^1_{\infty,\infty}}+B(T)^{1+\eps_0})(1+M(T)+B(T))^{\frac{c}{2}},\quad\quad\quad \forall T>0.
			\end{align*}
			By the bootstrap method, we obtain 
			\begin{align*}
				&M(T)\leq C(\eps+ \|\nabla U_0\|_{L^\infty}),\\
				&B(T)\leq C\eps,\ \ \ \forall\ T>0.
			\end{align*}
			This completes the proof of Proposition \ref{prop111}. 
		\end{proof}
		\subsection{Schauder type estimates for equations on manifolds}
		The extension of Theorem \ref{lemmain} to manifolds without boundary can be established as follows.  Let $\mathcal{M}$ be a smooth, closed manifold equipped with:
		\begin{itemize}
			\item A finite atlas $\{\mathcal{U}_i, \phi_i \}_{i=1}^n$ satisfying $\tilde{\mathcal{U}}_i\Subset\mathcal{U}_i$ for each $i$, where $\{ \tilde{\mathcal{U}}_i\}_{i=1}^n$ forms an open cover of $\mathcal{M}$.
			\item Local homeomorphisms $\phi_i:\mathcal{U}_i\rightarrow\mathbb{R}^d$.
			\item A partition of unity $\{\chi_i\}$ subordinate to $\mathcal{U}_i$ with $\text{Supp}(\chi_i)\subset\mathcal{U}_i$, $\chi_i|_{\tilde{\mathcal{U}}_i}=1$.
		\end{itemize} 
		Define the H\"{o}lder norm on $\mathcal{M}$ through localization:
		\begin{equation}\label{mfdhd}
			\|u\|_{ C^{m}(\mathcal{M})}:=\sum_{i=1}^n\|(u\chi_i)\circ\phi_i^{-1}\|_{ C^m(\mathbb{R}^d)}.
		\end{equation}
		We denote $\mathcal{L}_{\mathcal{M}}^s$ a Pseudo-differential operator on the manifold $\mathcal{M}$, which is defined via local coordinates:
		\begin{equation}\label{mfdcdl}
			(\chi_i\mathcal{L}_{\mathcal{M}}^su)\circ\phi_i^{-1}=\mathcal{L}_{\mathbb{R}^d}^{s,i}((\chi_iu)\circ\phi_i^{-1})+R_i(u),
		\end{equation}
		where $\mathcal{L}_{\mathbb{R}^d}^{s,i}$, $i=1,\cdots,n$ are Pseudo-differential operators in $\mathbb{R}^d$ as defined in \eqref{defop} and \eqref{condop}, and the remainder term $R_i$ satisfies
		\begin{equation}\label{mtmfcd1}
			\begin{aligned}
				\|R_i(u)\|_{L^\infty}\lesssim C_{\mathcal{M}}\|u\|_{C^{s-\zeta_0}(\mathcal{M})},\quad\|R_i(u)\|_{\dot C^{k+\alpha}}\lesssim C_{\mathcal{M}}\|u\|_{C^{s+k+\alpha-\zeta_0}(\mathcal{M})},\quad \forall k\leq m,\alpha\in[0,1),
			\end{aligned}
		\end{equation}
		for some $\zeta_0>0$. 
		Suppose that 
		\begin{equation}\label{gardineq}
			(\mathcal{L}_{\mathcal{M}}^su,u)\geq c_0\|u\|_{H^{s_0}(\mathcal{M})}^2-c_0^{-1}\|u\|_{L^2(\mathcal{M})}^2,\quad\forall u\in C^\infty(\mathcal{M}),
		\end{equation}
		for some $s_0\in (0,\frac{s}{2}]$ and $c_0\in (0,1).$ Here $(\cdot,\cdot)$ is the inner product on $\mathcal{M}$. 
		
		Consider the following system on $\mathcal{M}$,
		\begin{equation}\label{maineqM}
			\begin{aligned}
				&\partial_tu(t,x)+\mathcal{L}_{\mathcal{M}}^su(t,x)=G(t,x),\quad\text{in}\ (0,\infty)\times \mathcal{M},\\
				&u|_{t=0}=u_0,
			\end{aligned}
		\end{equation}
		with a known force term $G$.
		The existence and uniqueness of solutions to \eqref{maineqM} follow from standard parabolic theory combined with the G\aa{r}ding's inequality \eqref{gardineq}, adapting the techniques in \cite{JJS}. See also \cite{Polden,CM12,QTLG14}.
		\begin{remark}\label{semi}
		We discuss  an iterative method to solve \eqref{maineqM}. The solution is given by 
		\begin{align}\label{forsolma}
			u(t)=S(t,0)u_0+\int_0^t S(t,\tau)G(\tau)d\tau,
		\end{align}
		where 
		\begin{equation*}
			S(t,\tau)=\lim_{n\rightarrow \infty}\Pi_{k=1}^n\left(\mathrm{Id}+\frac{t-\tau}{n}\mathcal{L}^s_{\mathcal{M}}\left(\tau+\frac{k(t-\tau)}{n}\right)\right)^{-1}. 
		\end{equation*}
		To see this, we divide the time interval $[0,T]$ into $n$ parts, and denote $\Delta t=\frac{t-\tau}{n}$, $t_j=j\Delta t$, then consider the approximate system 
		\begin{equation*}
			\begin{aligned}
				\frac{u_{j+1}-u_j}{\Delta t}+\mathcal{L}^s_{\mathcal{M}}(t_j)u_{j+1}=G_j,\quad j=0,1,\cdots,n-1,
			\end{aligned}
		\end{equation*}
		which is equivalent to 
		$$           \frac{u_{j+1}}{\Delta t}+\mathcal{L}^s_{\mathcal{M}}(t_j)u_{j+1}=\frac{u_{j}}{\Delta t}+G_j,\quad j=0,1,\cdots,n-1.$$        Using $\frac{1}{\Delta t}\gg 1$ and   the G\aa{rding} inequality \eqref{gardineq}, we can solve inductively
		\begin{align*}
			u_{j+1}=S_ju_j+\Delta t S_j(G_j),
		\end{align*}
		where 
		\begin{align*}
			S_j=\left(\mathrm{Id}+\Delta t \mathcal{L}^s_{\mathcal{M}}(t_j)\right)^{-1}.
		\end{align*}
		This implies that 
		\begin{align*}
			u_n=\prod_{j=1}^nS_ju_0+\Delta t\sum_{j=0}^n\prod_{k=j}^nS_k(G_j).
		\end{align*}
		Taking $n\to\infty$ yields the formula \eqref{forsolma}.\vspace{0.3cm}
		\end{remark}
		
		We prove the following a priori estimate for the solution to the evolution system \eqref{maineqM}, establishing a counterpart of Theorem \ref{lemmain} on the manifold.
		\begin{theorem}\label{thmmani}
			Let $0<T<\infty$. Suppose that $u\in C([0,T],L^\infty(\mathcal{M}))\cap L_{loc}^\infty((0,T], C^{m+\kappa}(\mathcal{M}))$ is a solution to \eqref{maineqM} with $\mathcal{L}_\mathcal{M}^s$ satisfying \eqref{mfdcdl}.
			 If $u_0\in C^n(\mathcal{M})$ with $s-\kappa\leq n\leq s$, and $R_i$ satisfies \eqref{mtmfcd1} with $k\leq m+n$,
			then
			\begin{equation}\label{pehmfd}
				\sup_{t\in [0,T]}(\|u(t)\|_{C^n(\mathcal{M})}+t^{\frac{m+\kappa}{s}}\|u(t)\|_{C^{m+n+\kappa}(\mathcal{M})})\lesssim e^{CT\log(T+2)}( \|u_0\|_{C^n}+\da^{n,m,\kappa}_{T,\mathcal{M}}(G)),
			\end{equation}
			where  
			\begin{align*}
				\da^{n,m,\kappa}_{T,\mathcal{M}}(g):=  &\sup _{t \in[0,T]}\left(t^{\frac{\kappa}{s}}\|g(t)\|_{  C^{\kappa-s+n}(\mathcal{M})}
				+t^{\frac{m+\kappa}{s}}\|g(t)\|_{  C^{m+\kappa-s+n}(\mathcal{M})}\right).
			\end{align*}
		\end{theorem}
		\begin{proof}
			Denote $v_j=(u\chi_j)\circ\phi_{j}^{-1}$. By \eqref{mfdcdl}, we write the equation of $v_j$ as 
			\begin{equation*}
				\begin{aligned}
					&\partial_tv_j(t,x)+\mathcal{L}_{\mathbb{R}^d}^{s,j}v_j(t,x)=G_j(t,x)+R_j(t,x),\quad\text{in}\ (0,\infty)\times\mathbb{R}^d,\\
					&v_j|_{t=0}=(u_0\chi_j)\circ\phi_j^{-1},
				\end{aligned}
			\end{equation*}
			where $G_j=(G\chi_j)\circ\phi_j^{-1}$, and
			\begin{align*}
				R_j(t,x)=\mathcal{L}_{\mathbb{R}^2}^{s,j}(v_j)-(\chi_j\mathcal{L}_{\mathcal{M}}^su)\circ\phi_j^{-1}.
			\end{align*}
            		Denote $v_j^{n}:=\nabla^n_xv_j$, the equation of $v_j^{n}$ can be written as
			\begin{equation*}
				\begin{aligned}
					&\partial_tv_j^{n}(t,x)+\mathcal{L}_{\mathbb{R}^d}^{s,j}v_j^{n}(t,x)=\nabla^n(G_j(t,x)+R_j(t,x))+R_j^n,\quad\text{in}\ (0,\infty)\times\mathbb{R}^d\\
					&v_j^{n}|_{t=0}=\nabla^n((u_0\chi_j)\circ\phi_j^{-1}),
				\end{aligned}
			\end{equation*}
			where
			\begin{equation*}
				R_j^n=\mathcal{L}_{\mathbb{R}^d}^{s,j}v_j^{n}(t,x)-\nabla^n(\mathcal{L}_{\mathbb{R}^d}^{s,j}v_j).
			\end{equation*}
			For simplicity, denote 
			\begin{align*}
				&\|f\|_{X_T}:=  \sup_{t\in[0,T]}\left(\|f(t)\|_{L^\infty}+t^{\frac{m+\kappa }{s}}\|f(t)\|_{C^{m+\kappa }}\right),\\
				&\|f\|_{X_{T,n}}:=  \sup_{t\in[0,T]}\left(\|f(t)\|_{C^n}+t^{\frac{m+\kappa }{s}}\|f(t)\|_{C^{m+n+\kappa }}\right),
			\end{align*}
			where we drop the domain of the function, which can be $\mathcal{M}$ or $\mathbb{R}^d$, whenever it is clear from context.\\
            Let $T_0\in (0,\min\{T,1\}]$. 		It follows from Proposition \ref{propb=0}  that 
			\begin{equation*}
				\begin{aligned}
					&\|v_j^{n}\|_{X_{T_0}}\lesssim\|v_j^{n}(0)\|_{L^\infty}+\da^{n,m,\kappa}_{T_0,\mathcal{M}}(G_j)+\da^{n,m,\kappa}_{T_0,\mathcal{M}}(R_j)+\|R_j^n\|_{L_{T_0}^1L^\infty}+\sup_{t\in[0,T_0]}t^{\frac{m}{s}+1}\|R_j^n(t)\|_{\dot C^{m}}.
				\end{aligned}
			\end{equation*}
            Note that by definition \eqref{mfdhd}, 
			\begin{equation*}\label{h1}
				\sum_{j=1}^\ell\|v_j^n\|_{X_{T}}\sim\sum_{j=1}^\ell\|v_j\|_{X_{T,n}}\sim \|u\|_{X_{T,n}},\quad\sum_{j=1}^\ell\|v_j^n(0)\|_{L^\infty}\sim\sum_{j=1}^\ell\|v_j(0)\|_{C^n}\sim \|u_0\|_{C^n(\mathcal{M})}.
			\end{equation*}
            		By the facts that
                    	\begin{equation*}\label{h2}
		\sum_{j=1}^\ell\da^{n,m,\kappa}_{T,\mathcal{M}}(G_j)\sim \da^{n,m,\kappa}_{T,\mathcal{M}}(G),
			\end{equation*}
            and \begin{equation*}\label{mfdrt2}
				\|R_j^n\|_{L_{T_0}^1L^\infty}+\sup_{t\in[0,T_0]}t^{\frac{m}{s}+1}\|R_j^n(t)\|_{\dot C^{m}}+	\da^{n,m,\kappa}_{T_0,\mathcal{M}}(R_j)\lesssim T_0^{\frac{\zeta_0}{s}}\|v_j\|_{X_{T_0,n}},
			\end{equation*}
	 where we use \eqref{mtmfcd1} and the constants depends on the domain. We obtain
			\begin{equation*}
				\|u\|_{X_{T_0,n}}\lesssim \|u_0\|_{C^n(\mathcal{M})}+\da^{n,m,\kappa}_{T,\mathcal{M}}(G)+T_0^{\frac{\zeta_0}{s}}\|u\|_{X_{T_0,n}}.
			\end{equation*}
			By a similar argument in the proof of Proposition \ref{propb=0}, we deduce 
			\begin{equation*}
				\|u\|_{X_{T,n}}\lesssim e^{C'T\log(T+2)}\left(\|u_0\|_{C^n(\mathcal{M})}+\da^{n,m,\kappa}_{T,\mathcal{M}}(G)\right),\ \ \ \forall T>0.
			\end{equation*}
			This completes the proof of the theorem.
		\end{proof}
		

		\section{The 2D Muskat equation with surface tension }
		\label{intomus1}
		In this section, we apply Theorem \ref{lemmain} to study the 2D Muskat equation with surface tension, given by \eqref{eqmst}. To do this, we first reformulate the nonlinear term.
        
		For simplicity, fix the constants $\mathbf{k}=2, \mu=1$ and $\sigma'=1$. 
		We start by noting that the following holds:
		\begin{align*}
			&	\frac{1}{\pi}\mathrm{P.V.}\int_{\mathbb{R}}\frac{1+\partial_x f(x)\Delta_\alpha f(x)}{\langle\Delta_\alpha f(x)\rangle^2}\partial_x\left(\kappa(f)- \varrho_0 f\right)(x-\alpha)\frac{d\alpha}{\alpha} = - \frac{\Lambda^3f(x)}{\langle\partial_x f(x)\rangle^3}+\N[f](x),
		\end{align*}
		where $\N[f]=\N_1[f]+\N_2[f]+\varrho_0\N_3[f]$, with the following definitions: 
		\begin{equation}\label{defNst}
			\begin{aligned}
				&\N_{1}[f](x)=\frac{1}{\pi} \int_{\mathbb{R}}\left(\frac{ \Delta_{\alpha} f(x)(\partial_{x} f(x)-\Delta_{\alpha} f(x))}{\left\langle\Delta_{\alpha} f(x)\right\rangle^{2}}\right) \partial_{x}\left(\frac{\partial_{x}^{2} f(x-\alpha)}{\left\langle\partial_{x} f(x-\alpha)\right\rangle^{3}}\right) \frac{d \alpha}{\alpha},\\
				&\N_2[f](x)=-\frac{1}{\pi} \int_{\mathbb{R}}\partial_x^2f(x-\alpha)\left(\frac{1}{\left\langle\partial_{x} f(x-\alpha)\right\rangle^{3}}-\frac{1}{\left\langle\partial_{x} f(x)\right\rangle^{3}}\right)\frac{d\alpha}{\alpha^2},\\
				&\N_3[f](x)=-\frac{1}{\pi} \int_{\mathbb{R}}\left(\frac{ \Delta_{\alpha} f(x)(\partial_{x} f(x)-\Delta_{\alpha} f(x))}{\left\langle\Delta_{\alpha} f(x)\right\rangle^{2}}\right) \partial_{x}f(x-\alpha)\frac{d \alpha}{\alpha}-\Lambda f.
			\end{aligned}
		\end{equation}
		Thus, the 2D Muskat equation with surface tension \eqref{eqmst} can be rewritten as:
		\begin{equation}\label{eqmusre}
			\begin{aligned}
				&\partial_t f(x)+\frac{\Lambda^3f(x)}{\langle \partial_x f(x)\rangle^3}=\N[f](x),\\
				&f|_{t=0}=f_0.
			\end{aligned}
		\end{equation}
		The proof of Theorem \ref{thmMusbes} is organized in three stages:
		\begin{enumerate}
			\item {\bf Relaxed version:} As a preliminary step, we establish Proposition \ref{propLip}, in which \textit{the Besov norms in \eqref{bound1} and \eqref{bound2} are replaced by Lipschitz norms}.

			\item {\bf Existence theory:} 
			For initial data satisfying \eqref{bound1} or \eqref{bound2}, we establish a priori estimates and construct solutions via compactness arguments.
			
			\item {\bf Stability analysis:} 
			We derive continuous dependence on initial data for the obtained solutions, which yields the uniqueness of the solution.
		\end{enumerate}
		We begin with the following relaxed version of Theorem \ref{thmMusbes}.
		\begin{proposition}\label{propLip}
			There exist $\varepsilon_0,C>0$ such that the following statements hold.	~\\
			i) ($ \varrho_0=0$) For any initial data $f_0 \in \dot W^{1,\infty}$ satisfying the smallness condition 
			\begin{align}\label{bound1L}
				\|f_0\|_{\dot W^{1,\infty}}\leq \varepsilon_0,
			\end{align} the equation \eqref{eqmst} admits a unique global solution $f$ in the class 
			$$\{f\in L^\infty_{loc}((0,\infty),C^{m+\kappa}(\mathbb{R})): \|f\|_{\infty}\leq C\|f_0\|_{W^{1,\infty}}, \|f\|_{\infty,*}\leq C \varepsilon_0\},$$
            with $\|\cdot\|_{T}$, $\|\cdot\|_{T,*}$ defined in \eqref{normmst}.\\
			ii) ($ \varrho_0\in\mathbb{R}$) For any initial data $f_0 \in W^{1,\infty}$, if
			\begin{align}\label{stcon}
				(1+\|f_0\|_{ W^{1,\infty}})^{2(m+5)}\| f_0-\phi\|_{\dot W^{1,\infty}}\leq \varepsilon_0,
			\end{align}
            holds for some $\phi \in C^{m+5}(\mathbb{R})$, 
			then   there exists $T>0$ such that \eqref{eqmst} admits a unique solution  $f$ in the class 
			$$\{f\in L^\infty_{loc}((0,T),C^{m+\kappa}(\mathbb{R})): \|f\|_{X_T}\leq C\|f_0\|_{W^{1,\infty}}, \|f\|_{T,*}\leq C \varepsilon_0\}.$$
		\end{proposition}
		\begin{remark}
			For conciseness, we term Theorem \ref{thmMusbes} the \textit{Besov-type result} and Proposition \ref{propLip} the \textit{Lipschitz-type result}. Notably, the Lipschitz-type result holds independent ofterest for studying well-posedness in certain quasilinear parabolic equations where the Besov-type approach is inapplicable due to structural constraints in nonlinear terms. 
		\end{remark}

		Now we start the proof of Proposition \ref{propLip} which relies on the contraction mapping theorem.\\
		\begin{proof}
			We first prove \textit{i)}. Suppose $\varrho_0=0$ and $\|f_0\|_{\dot W^{1,\infty}}\leq \varepsilon_0$. Consider the set 
			\begin{align*}
				\mathcal{X}^\sigma=\left\{g\in L^\infty_t\dot W^{1,\infty}:g|_{t=0}=f_0,\|g\|_{\infty}\leq \sigma\right\},
			\end{align*}
			where $\sigma>0$ is a constant that will be fixed later. For $g\in \mathcal{X}^\sigma$, define a map $\mathcal{S}g=f$, where $f$ is the solution to the equation 
			\begin{equation*}
				\begin{aligned}
					&\partial_t f+{\Lambda^3f(x)}=\N[g](x)+G[g_1,0],\\
					&f|_{t=0}=f_0,
				\end{aligned}
			\end{equation*}
			where
			\begin{align}\label{defffg}
				\G[g_1,g_2]=	\left(\frac{1}{\langle\partial_xg_2\rangle^3}-\frac{1}{\langle\partial_xg_1\rangle^3}\right)\Lambda^3 g_1.
			\end{align}
			In the following, we prove that $\mathcal{S}$ is a contraction map from $	\mathcal{X}^\sigma$ to itself. Therefore, there exists a unique fixed point in $	\mathcal{X}^\sigma$, which is a global solution to  \eqref{eqmst}.
            
			First, we prove $\mathcal{S}:\mathcal{X}^\sigma\rightarrow \mathcal{X}^\sigma$. Indeed, we apply Proposition \ref{propb=0} to obtain that 
			\begin{align*}
				\|f\|_\infty\lesssim &\|\partial_x f_0\|_{L^\infty}+\da_\infty(\N[g])+\da_\infty(\G[g,0]).
			\end{align*}
			where 
\begin{align}\label{da}\da_T(F):=\sup_{t\in(0,T)}\left(t^\frac{\kappa}{3}\|F(t)\|_{\dot C^{\kappa-2}}+t^\frac{m+\kappa }{3}\|F(t)\|_{\dot C^{m+\kappa -2}}\right).\end{align}
			It follows from Lemma \ref{nonst} that 
			\begin{align*}
				\da_\infty(\N[g])\lesssim \|g\|_\infty^2(1+\|g\|_\infty)^{2m+5}.
			\end{align*}
			Applying Lemma \ref{GGG} with $f_1=g$, $f_2\equiv 0$, and $\phi\equiv 0$ gives
			\begin{align*}
				\da_\infty(\G[g,0])\lesssim \|g\|_\infty^2(1+\|g\|_\infty)^{2m+5}.
			\end{align*}
			Hence,
			\begin{align*}
				\|f\|_\infty\leq C_0\|\partial_x f_0\|_{L^\infty}+C_0\|g\|_\infty^2(1+\|g\|_\infty)^{2m+5},
			\end{align*}
			where the constant $C_0$ depends only on $m$. 	Take $\varepsilon_0\leq\frac{1}{100(2C_0+1)^{2(m+5)}}$ and $\sigma=2C_0 \varepsilon_0$. Then for any $g\in \mathcal{X}^{\sigma}$, we obtain
			\begin{align*}
				\|f\|_{\infty}\leq C_0\varepsilon_0+4C_0^3\varepsilon_0^2(2C_0 \varepsilon_0+1)^{2m+5}\leq \sigma,
			\end{align*}
			provided \eqref{bound1L} holds. This implies $f=\mathcal{S}g\in \mathcal{X}^{\sigma}$.
            
			Next, we prove that $\|\mathcal{S}g_1-\mathcal{S}g_2\|_\infty\leq\frac{1}{2}\|g_1-g_2\|_\infty$ for $g_1,g_2\in \mathcal{X}^\sigma$ with $\sigma$ small enough. 
			Let $f_1=\mathcal{S}g_1, f_2=\mathcal{S}g_2$. Then 
			\begin{align*}
				&\partial_t (f_1-f_2)(x)+{\Lambda^3(f_1-f_2)(x)}=(\N[g_1]-\N[g_2])(x)+(\G[g_1,0]-\G[g_2,0])(x),\\
				&(f_1-f_2)|_{t=0}=0.
			\end{align*}
			Applying Theorem \ref{lemmain} yields
			\begin{align*}
				\|f_1-f_2\|_\infty\lesssim &\da_\infty(\N[g_1]-\N[g_2])+\da_\infty
				(\G[g_1,0]-\G[g_2,0]).
			\end{align*}
			Lemma \ref{nonst} and Lemma \ref{GGG} imply
			\begin{align*}
				&   \da_\infty(\N[g_1]-\N[g_2])+\da_\infty
				(\G[g_1,0]-\G[g_2,0])\\
				&\quad\quad\quad\lesssim \|g_1-g_2\|_\infty\|(g_1,g_2)\|_{\infty}(1+\|(g_1,g_2)\|_{\infty})^{2m+5}.
			\end{align*}
			Thus, we obtain 
			\begin{align*}
				\|f_1-f_2\|_\infty&\leq C_1\|g_1-g_2\|_\infty\|(g_1,g_2)\|_{\infty}(1+\|(g_1,g_2)\|_{\infty})^{2m+5}\\
                &\leq C_1\sigma(1+2\sigma)^{2m+5}\|g_1-g_2\|_\infty\leq \frac{1}{2}\|g_1-g_2\|_\infty,
			\end{align*}
			provided $C_1\sigma(1+2\sigma)^{2m+5}\leq \frac{1}{2}$.
			Thus, there exists a unique $f\in \mathcal{X}^\sigma$ such that $\mathcal{S}f=f$. This completes the proof of \textit{i)}.
			\vspace{0.3cm}\\
			Now we prove \textit{ii)}.            When $\varrho_0 \neq 0$, the equation \eqref{eqmst} is not homogeneous. We need to control the lower-order norm $\|h\|_{L^\infty_TL^\infty}$.  
			For this, we define
			\begin{align*}\label{musnol}
				&\|h\|_{X_T}:=\|h\|_{L^\infty_TL^\infty}+\|h\|_T.
			\end{align*}
			Consider the set
			$$
			\mathcal{X}_{T,\phi}^\sigma=\{g\in L^\infty_T\dot W^{1,\infty}:g|_{t=0}=f_0,\|g-\phi\|_{X_T}\leq \sigma\}.
			$$ 
			For any $g\in	\mathcal{X}_{T,\phi}^\sigma$, where $T,\sigma$ will be fixed later, define a map $\tilde {\mathcal{S}}g=f$, where $f$ is the solution to the equation 
			\begin{equation}\label{mapstt}
				\begin{aligned}
					&\partial_t f+\frac{\Lambda^3f}{\langle\partial_x\phi\rangle^3}=\N[g]+\G[g,\phi],\\
					&f|_{t=0}=f_0,
				\end{aligned}
			\end{equation}
			where $\G[g,\phi]$ is defined in \eqref{defffg}.
            
			In the following, we prove that $\tilde{\mathcal{S}}$ is a contraction map from $	\mathcal{X}_{T,\phi}^\sigma$ to itself. Hence, there exists a unique fixed point in $	\mathcal{X}_{T,\phi}^\sigma$, which is a solution to \eqref{eqmst}. Indeed, denoting $\tilde f(t,x)=f(t,x)-\phi(x)$, $\tilde g(t,x)=g(t,x)-\phi(x)$, we obtain the equation
			\begin{align*}
				&	\partial_t \tilde f+\frac{\Lambda^3\tilde f}{\langle\partial_x\phi\rangle^3}= \N[g]+
				\G[g,\phi]-\frac{\Lambda^3\phi}{\langle\partial_x\phi\rangle^3},\\
				&\tilde f|_{t=0}=f_0-\phi.
			\end{align*}
			Applying Theorem \ref{lemmain}, we obtain for any $T>0$, 
			\begin{equation}\label{stm}
				\begin{aligned}
					\|\tilde f\|_{T}
					\lesssim&(1+\|\partial_x\phi\|_{L^\infty})^{m+5}\left(\|\partial_x(f_0-\phi)\|_{L^\infty}+	\da_T(\N[g]+\G[g,\phi]-\frac{\Lambda^3\phi}{\langle\partial_x\phi\rangle^3})+\|\tilde f\|_T\|\phi\|_{T,*}\right),
				\end{aligned}
			\end{equation}
          where the additional factor $(1+\|\partial_x \phi\|_{L^\infty})^{m+5}$ comes from the underlying lower bound of the coefficient $\frac{1}{\langle\partial_x \phi\rangle^3}$ in \eqref{eqmusre}, see Remark \ref{rmkcons}.
			Lemma \ref{nonst} yields
			\begin{equation}\label{st1}
				\begin{aligned}
					&	\da_T(\N[g])\lesssim (\| g\|_{T,*}^2+T^\frac{2}{3}\|g\|_{X_T})(1+\| g\|_{X_T})^{2m+5}.
				\end{aligned}
			\end{equation}
			Note that $\|\partial_x\phi\|_{L^\infty}\lesssim \|\partial_x f_0\|_{L^\infty}$, and 
			\begin{align}\label{esTsm}
			\|\phi\|_{T,*}\lesssim T^\frac{1}{15}\|\phi\|_{C^{m+5}},\quad\quad\quad 	\| g\|_{T,*}\lesssim \|\tilde g\|_{X_T}+T^\frac{1}{15}\|\phi\|_{C^{m+5}}.
			\end{align}
			Applying Lemma \ref{GGG} with $f_1=g$, $f_2=\phi$ yields
			\begin{align}
				\da_T(\G[g,\phi])
				\lesssim \|\tilde g\|_{T}(\|g\|_{T,*}+\|\tilde g\|_T)&(1+\|(g,\phi)\|_{T})^{m+5}.\label{st2}
			\end{align}
            Finally, by the smoothness of $\phi$, it is easy to check that 
            \begin{align}\label{re}
            \da_T\left(\frac{\Lambda^3\phi}{\langle\partial_x\phi\rangle^3}\right)\lesssim T^\frac{1}{15}\|\phi\|_{C^{m+4}}(1+\|\phi\|_{T})^{m+4}.
            \end{align}
			We conclude from \eqref{stm}, \eqref{st1}, \eqref{st2} and \eqref{re} that 
			\begin{equation}\label{stfin1}
				\begin{aligned}
					\|\tilde f\|_{T}\lesssim  
					&(1+\|\partial_xf_0\|_{L^\infty})^{m+5}\\
                    &\times\left(\|\partial_x (f_0-\phi)\|_{L^\infty}+(\| g\|_{T,*}^2+T^\frac{1}{15}\|\phi\|_{C^{m+3}}+\|\tilde g\|_{X_T}^2)(1+\| (f,g,\phi)\|_{X_T})^{2(m+5)}\right).
				\end{aligned}
			\end{equation}
			On the other hand, integrating \eqref{mapstt} in time, thanks to \eqref{eee1} and Remark \ref{musLinfty}, we obtain
			\begin{align*}
				|f(t,x)-f_0(x)|&\leq \int_0^t (\|\Lambda^3f(\tau)\|_{L^\infty}+\|\N[g](\tau)\|_{L^\infty}+\|\G[g,\phi](\tau)\|_{L^\infty})d\tau\\
				&\leq \tilde C_0t^\frac{1}{3} (1+\|f\|_{X_T}+\|(g,\phi)\|_{X_T})^5,\ \ \ \forall t\in[0,T].
			\end{align*}
			From this we obtain
			\begin{align*}
				\|\tilde f\|_{L^\infty_TL^\infty}\leq \|f_0-\phi\|_{L^\infty}+T^\frac{1}{3}(1+\|f\|_{X_T}+\|(g,\phi)\|_{X_T})^5.
			\end{align*}
			Combining this with \eqref{stfin1} and \eqref{esTsm},
			we derive 
			\begin{align*}
				\|\tilde f\|_{X_T}\leq &\tilde C_0(1+\|\partial_xf_0\|_{L^\infty})^{m+5}\\
                &\times\left(\|\partial_x( f_0-\phi)\|_{L^\infty}+\left(\|\tilde g\|_{X_T}+T^\frac{1}{20}(1+\|(\tilde f,\tilde g,\phi)\|_{X_T})\right)^2\left(1+\|(\tilde f,\tilde g,\phi)\|_{X_T}\right)^{2(m+5)}\right).
			\end{align*}
			Here the constant $\tilde C_0$ depends only on $m$ and $\varrho_0$. 
			Let  
			\begin{align}\label{sma}
				\eps_0\leq(10
            +\tilde{C}_0)^{-10(m+d+5)} ,\quad \sigma'=2\tilde C_0\varepsilon_0, \quad\quad\text{and}\quad 0<T<\left(\frac{\varepsilon_0}{1+\tilde C_0+\|\phi\|_{C^{m+5}}}\right)^{(m+10)^2}.
			\end{align} If \eqref{stcon} holds, then for any $g\in  \mathcal{X}_{T,\phi}^{\sigma'}$, we have
			\begin{align*}
				\|\tilde f\|_{X_T}&\leq \tilde C_0\varepsilon_0+\tilde C_0(2\tilde C_0\varepsilon_0+T^\frac{1}{20})\|f_0\|_{W^{1,\infty}})^2(1+\|f_0\|_{W^{1,\infty}})^{3(m+5)}\\
				&\leq 2\tilde C_0\varepsilon_0,
			\end{align*}
			hence $f\in  \mathcal{X}_{T,\phi}^{\sigma'}$.\vspace{0.1cm}
			
			In the following, we do the contraction estimates. Consider $g_1,g_2\in \mathcal{X}^{\sigma'}_{T,\phi}$. Denote $\mathbf{g}=g_1-g_2$, $\mathbf{f}=f_1-f_2=\tilde {\mathcal{S}}g_1-\tilde {\mathcal{S}}g_2$. We have 
			\begin{equation}\label{mapst1}
				\begin{aligned}
					&\partial_t \mathbf{f}+\frac{\Lambda^3\mathbf{f}}{\langle\partial_x\phi\rangle^3}=\N[g_1]-\N[g_2]+\G[g_1,\phi]-\G[g_2,\phi],\\
					&\mathbf{f}|_{t=0}=0.
				\end{aligned}
			\end{equation}
			Applying Theorem \ref{lemmain} yields
			\begin{equation}\label{ft1}
				\begin{aligned}
					\|\mathbf{f}\|_T\lesssim (1+\|\partial_xf_0\|_{L
                    ^\infty})^{m+5}\left(\da_T(\N[g_1]-\N[g_2])+\da_T(\G[g_1,\phi]-\G[g_2,\phi])\right).
				\end{aligned}
			\end{equation}
			From Lemma \ref{nonst}, we obtain 
			\begin{align*}
				\da_T(\N[g_1]-\N[g_2])\lesssim (1+\|\partial_xf_0\|_{L
                    ^\infty})^{m+5}\|\mathbf{g}\|_{X_T}(\|(g_1,g_2)\|_{T,*}+T^\frac{1}{10})(1+\|(g_1,g_2)\|_{X_T})^{2(m+5)}.
			\end{align*}
			Moreover, by Lemma \ref{GGG},
			\begin{align*}
				&\da_T(\G[g_1,\phi]-\G[g_2,\phi])\lesssim \|\mathbf{g}\|_{T}(\|g_1\|_{T,*}+\|g_2-\phi\|_T)(1+\|(g_1,g_2,\phi)\|_{T})^{m+5}.
			\end{align*}
			Combining this with \eqref{ft1}, we obtain
			\begin{equation}   \label{haha1}
				\begin{aligned}
					\|\mathbf{f}\|_T\lesssim & \|\mathbf{g}\|_{X_T} (\|(g_1-\phi,g_2-\phi)\|_T+T^\frac{1}{15}\|\phi\|_{C^{m+5}})(1+\|(g_1,g_2,\phi)\|_{X_T})^{2(m+5)}.
				\end{aligned}
			\end{equation}
			Finally, to control $\|\mathbf{f}\|_{X_T}$, it remains to estimate $\|\mathbf{f}\|_{L^\infty_TL^\infty}$. Integrating \eqref{mapst1} in time, and applying \eqref{eee1} and Remark \ref{musLinfty}, we obtain 
			\begin{align*}
				\|\mathbf{f}\|_{L^\infty_TL^\infty} &\leq \int_0^T \left(\|\Lambda^3 \mathbf{f}(\tau)\|_{L^\infty}+\|(\N[g_1]-\N[g_2])(\tau)\|_{L^\infty}+\|(\G[g_1,\phi]-
				\G[g_2,\phi])(\tau)\|_{L^\infty}\right)d\tau\\
				&\lesssim T^\frac{1}{3}(\|\mathbf{f}\|_T+\|\mathbf{g}\|_T)(1+\|(g_1,g_2,\phi)\|_{X_T})^{2(m+5)}.
			\end{align*}
			Combining this with \eqref{haha1} yields
			\begin{align*}
				\|\mathbf{f}\|_{X_T}\lesssim & \|\mathbf{g}\|_{X_T} (\|(g_1-\phi,g_2-\phi)\|_T+T^\frac{1}{15}(1+\|\phi\|_{C^{m+5}}))(1+\|(g_1,g_2,\phi)\|_{X_T}+\|\partial_x f_0\|_{L^\infty})^{3(m+5)}\\
				&+T^\frac{1}{3}\|\mathbf{f}\|_{T}(1+\|\phi\|_{C^{m+2}}+\|(g_1,g_2)\|_{X_T}+\|\partial_x f_0\|_{L^\infty})^{3(m+5)}.
			\end{align*}
			If $g_1,g_2\in  \mathcal{X}_{T,\phi}^{\sigma'}$, one has $\|(g_1-\phi,g_2-\phi)\|_{X_T}\leq 2 \sigma'$ and  $\|(f_1-\phi,f_2-\phi)\|_{X_T}\leq 2 \sigma'$. Combining this with \eqref{stcon}  yields
			\begin{align*}
				\|\mathbf{f}\|_{X_T}\leq & \tilde C_1(\|\mathbf{f}\|_{X_T}+\|\mathbf{g}\|_{X_T}) (2\sigma'+{T}^\frac{1}{100}\|\phi\|_{C^{m+5}})(1+2\sigma'+{T}^\frac{1}{100}\|\phi\|_{C^{m+5}}+T^{\frac{1}{3}})^{3(m+5)}\\
				:=&\Theta(T,\sigma',\|\phi\|_{C^{m+5}})(\|\mathbf{f}\|_{X_T}+\|\mathbf{g}\|_{X_T}).
			\end{align*}
			Taking $\varepsilon_0,\sigma', T$ defined in \eqref{sma} such that $\Theta(T,\sigma',\|\phi\|_{C^{m+5}})\leq \frac{1}{100}$, then
			\begin{align*}
				\|\mathbf{f}\|_{T}\leq \frac{1}{2}\|\mathbf{g}\|_{X_T}.
			\end{align*}
			Hence $\tilde{\mathcal{S}}:\mathcal{X}_{T,\phi}^{\sigma'}\to \mathcal{X}_{T,\phi}^{\sigma'}$ is a contraction map. This completes the proof of the proposition.
		\end{proof}
            \begin{remark}
            The non-endpoint norm $\|f\|_{T,*}$ provides the desired smallness to proceed the fixed point argument and obtain the local solution under the condition \eqref{stcon}. More precisely, as we can see in the proof of Proposition \ref{propLip} \text{ii)},  the contraction mapping theorem is performed in the set with center $\phi$. In this case, $\|f\|_T$ is merely bounded, but the non-endpoint norm
		\begin{align*}
			\|f\|_{T,*}\leq \|f-\phi\|_{T,*}+\|\phi\|_{T,*}\lesssim \|f-\phi\|_{T,*}+T^\frac{1}{15}\|
			\phi\|_{C^{m+3}},
		\end{align*}
		is small because of the smallness of $\|f-\phi\|_{T,*}$ and $T$. 
        \end{remark}
		\begin{proof}[Proof of Theorem \ref{thmMusbes}]
			We first prove \textit{i)} under the condition that $\varrho_0=0$ and \eqref{bound1} holds.\\ 
            {\bf 1: Construction of approximating sequence.}\\
           For $\vartheta\in(0,1)$,  define $f_{0,\vartheta}=f_0\ast \rho_\vartheta$, then by Proposition \ref{propLip}, there exists $T_{\vartheta}=T(\|f_0\|_{\dot W^{1,\infty}},\vartheta)$, such that \eqref{eqmst} admits a unique solution denoted by $f_\vartheta$ in $[0,T_\vartheta]$ with initial data $f_{0,\vartheta}$. Note that if \eqref{bound1} holds for $f_0$, it also holds for $f_{0,\vartheta}$ for any $\vartheta\in(0,1)$.\\
			{\bf 2: Control of $\|f_{\vartheta}\|_{\dot W^{1,\infty}}$.} \\   
			Applying Theorem \ref{lemmain} to \eqref{eqmusre} with $(s,b,\eta)=(3,1,\frac{1}{5})$ and $\A(t,x,\xi)=\frac{|\xi|^3}{\langle \partial_x f_\vartheta(t,x)\rangle^3}$, we have 
			\begin{align*}
				\|f_{\vartheta}\|_{L_{T_{\vartheta}}^\infty\dot W^{1,\infty}}\lesssim &\|f_{0,\vartheta}\|_{\dot W^{1,\infty}}+\da_{T_{\vartheta}}(\N[f_{\vartheta}])+\|f_\vartheta\|_{T,*}\|\A\|_{B_T},
			\end{align*}
			where $\da_{T_{\vartheta}}$, $\|\cdot \|_{B_T}$ are defined in \eqref{da} and \eqref{defparab}, respectively. 
			It follows from Lemma \ref{nonst} that 
			\begin{align}\label{esNf}
				\da_{T_{\vartheta}}(\N[f_{\vartheta}])\lesssim \|f_{\vartheta}\|_{T_{\vartheta},*}^2(1+\|f_{\vartheta}\|_{T_{\vartheta}})^{2m+5}.
			\end{align}
			Moreover, by definition, it is easy to obtain 
            \begin{align*}
            \|\A\|_{B_T}\lesssim \|f_{\vartheta}\|_{T_{\vartheta},*}(1+\|f\|_{T_{\vartheta}})^{2m+5}.
            \end{align*}
			Hence,
			\begin{align*}
				\|f_{\vartheta}\|_{L_{T_{\vartheta}}^\infty\dot W^{1,\infty}}\leq C_0\|f_{0,\vartheta}\|_{\dot W^{1,\infty}}+C_0\|f_{\vartheta}\|_{T_{\vartheta},*}^2(1+\|f\|_{T_{\vartheta}})^{2m+5},
			\end{align*}
			where the constant $C_0$ depends only on $m$. \\
			{\bf 3: Control of $\|f_{\vartheta}\|_{T_{\vartheta},*}$.}\\
			Applying Theorem \ref{lemmain} to \eqref{eqmusre}, we obtain 
			\begin{align}\label{fstar}
				\|f_{\vartheta}\|_{T_{\vartheta},*}\lesssim &(1+\|\partial_x f_{\vartheta}\|_{L^\infty_{T_{\vartheta}}L^\infty})^{m+5}\left(\|f_{0,{\vartheta}}\|_{\dot B^1_{\infty,\infty}}+\da_{T_{\vartheta}}(\N[f_{\vartheta}])+\|f_{\vartheta}\|_{T_{\vartheta},*}^2\right).
			\end{align}
            Remark that the additional factor $(1+\|\partial_x f_{\vartheta}\|_{L^\infty_{T_{\vartheta}}L^\infty})^{m+5}$ comes from the underlying lower bound of the coefficient $\frac{1}{\langle\partial_x f\rangle^3}$ in \eqref{eqmusre}, see Remark \ref{rmkcons}.
			Combining \eqref{fstar} with \eqref{esNf}, we obtain 
			\begin{align*}
				&    \|f_{\vartheta}\|_{T_{\vartheta},*}\leq C_0(1+\|\partial_x f_{\vartheta}\|_{L^\infty_{T_{\vartheta}}L^\infty})^{m+5}\left(\|f_{0,\vartheta}\|_{\dot B^1_{\infty,\infty}}+\|f_{\vartheta}\|_{T_{\vartheta},*}^2(1+\|f_{\vartheta}\|_{T_{\vartheta}})^{2m+5}\right),
			\end{align*}
			where the constant $C_0$ depends only on $m,\kappa,\eta$. \\
			{\bf 4: Existence.}\\
			Take $\eps_0:=\frac{1}{(10+C_0)^{10(m+5)}}$ and $\eps\in(0,\eps_0).$ Suppose that the initial data satisfies the quantitative bound 
			\begin{align*}
				(1+\|f_{0,\vartheta}\|_{\dot W^{1,\infty}})^{4(m+5)}\|f_{0,\vartheta}\|_{\dot B^1_{\infty,\infty}}\leq \eps.
			\end{align*}
	There exist $T_\vartheta>0$ and a unique solution $f_\vartheta$ in $[0,T_\vartheta]$		which satisfies 
			\begin{align}\label{uniformestmus}
				\|f_{\vartheta}\|_{T_{\vartheta}}\leq 4C_0(\eps+\|f_{0,\vartheta}\|_{\dot W^{1,\infty}}),\ \ \ \ \|f_{\vartheta}\|_{T_{\vartheta},*}\leq 10 C_0(1+10C_0\|f_{0,\vartheta}\|_{\dot W^{1,\infty}})^{2m+5}\|f_{0,\vartheta}\|_{\dot B^1_{\infty,\infty}}.
			\end{align}
		Furthermore, for the time derivatives, by \eqref{eqmusre} and \eqref{uniformestmus}, we obtain 
            \begin{equation*}
                \sup_{t\in[0,T_\vartheta]}t^{\frac{2}{3}}\|\partial_tf_{\vartheta}(t)\|_{L^\infty}\leq C\|f_0\|_{\dot W^{1,\infty}}<\infty.
            \end{equation*}
           By a bootstrap argument and the existence argument in Proposition \ref{propb=0}, we can pass to the limit $\vartheta\rightarrow 0$, and the sequence will converge to $f$, which is a solution to \eqref{eqmusre} in $[0,\infty)$. By the a priori estimates above, $f$ satisfies
           \begin{equation}\label{bdmus}
            \begin{aligned}
            &\|f\|_{\infty}\leq C(\eps+\|f_0\|_{\dot W^{1,\infty}}),\\
            &\|f\|_{\infty,*}\leq 10 C_0(1+10C_0\|f_0\|_{\dot W^{1,\infty}})^{2m+5}\|f_0\|_{\dot B^1_{\infty,\infty}}\leq C\eps.
        \end{aligned}
           \end{equation}
			{\bf 5: Stability and uniqueness.}\\
			Suppose $f,g$ are solutions to \eqref{eqmusre} in $[0,T]$ with initial data $f_0,g_0\in\dot W^{1,\infty}$ which satisfy the quantitative bound \eqref{bound1}. 
			By making a difference of the equations, we obtain 
			\begin{equation}\label{eqdif}
				\begin{aligned}
					&\partial_t (f-g)+\frac{\Lambda^3(f-g)}{\langle\partial_x f\rangle^3}=\N[f](x)-\N[g](x)-\G[g,f],\\
					&(f-g)|_{t=0}=f_0-g_0.
				\end{aligned}
			\end{equation}
			Applying Theorem \ref{lemmain}, we have 
			\begin{align*}
				\|f-g\|_T\lesssim (1+\|f\|_{L^\infty_T\dot W^{1,\infty}})^{m+5}(\|f_0-g_0\|_{\dot W^{1,\infty}}+\da_T(\N[f]-\N[g]-\G[g,f])+\|f-g\|_{T,*}\|f\|_{T,*}).
			\end{align*}
			Lemma \ref{nonst} and Lemma \ref{GGG} yield
			\begin{align*}
				&\da_T(\N[f]-\N[g])+\da_T(\G[g,f])\lesssim \|f-g\|_T \|(f,g)\|_{T,*}(1+\|(f,g)\|_{T})^{2m+5}.
			\end{align*}
			Thus, we obtain 
			\begin{align*}
				\|f-g\|_T
				&\leq C(1+\|f\|_{L^\infty_T\dot W^{1,\infty}})^{m+5}(\|f_0-g_0\|_{\dot W^{1,\infty}}+\|f-g\|_T \|(f,g)\|_{T,*}(1+\|(f,g)\|_{T})^{2m+5}).
			\end{align*}
			The condition \eqref{bound1} with $\eps\leq (10+C)^{-100}$, together with  \eqref{bdmus} yields $$C(1+\|f\|_{L^\infty_T\dot W^{1,\infty}})^{m+5} \|(f,g)\|_{T,*}(1+\|(f,g)\|_{T})^{2m+5}\leq \frac{1}{10}.$$ This implies 
			\begin{align}\label{stabglomus}
				\|f-g\|_T
				\lesssim \|f_0-g_0\|_{\dot W^{1,\infty}}.
			\end{align}
			Moreover, let $f,g$ to be two solutions to \eqref{eqmst} with the same initial data $f_0$, then \eqref{stabglomus} infers $f=g$, which ensures the uniqueness of the solution.\vspace{0.2cm}
			
			Now we prove \textit{ii)}. When $\varrho_0 \neq 0$, the equation \eqref{eqmst} is not homogeneous. We need to control the lower-order norm $\|h\|_{L^\infty_TL^\infty}$. 
			So we work on $\|\cdot\|_{X_T}$ in this setting. \\
            {\bf 1: Construction of approximating sequence.}\\
            To show the existence of solution, we use the standard compactness argument. We will construct the approximating solutions as follows. Denote $f_{0,\vartheta}=f_0\ast \rho_\vartheta$, then there exist $T_{0,\vartheta}=T_0(\|f_0\|_{\dot W^{1,\infty}},\vartheta)>0$, and a unique solution $f_\vartheta$ to \eqref{eqmst} in $[0,T_{0,\vartheta}]$ with initial data $f_{0,\vartheta}$. Note that \eqref{bound2} is equivalent to 
            \begin{equation*}
                (1+\|f_0\|_{\dot W^{1,\infty}})^{4(m+5)}\liminf_{\ell\rightarrow 0}\| f_0-f_{0,\ell}\|_{\dot B^{1}_{\infty,\infty}}\leq C\varepsilon,
            \end{equation*}
            for some universal $C$, so we can take a subsequence $\{f_{0,\vartheta}\}$ such that $\phi=f_{0,\ell}$ for some fixed $\ell$, and for any $\vartheta<\ell$ small, it holds
            \begin{equation*}
                \begin{aligned}
                &\| f_{0,\vartheta}\|_{\dot W^{1,\infty}}\leq C\|f_0\|_{\dot W^{1,\infty}},\\
                    &(1+\|f_0\|_{\dot W^{1,\infty}})^{4(m+5)}\| f_{0,\vartheta}-f_{0,\ell}\|_{\dot B^{1}_{\infty,\infty}}\leq 2C\varepsilon.
                \end{aligned}
            \end{equation*}
            {\bf 2: Control of $\|f_\vartheta\|_{W^{1,\infty}}$.}\\
			Rewrite \eqref{eqmusre} as 
			\begin{align*}
				&	\partial_t  f_\vartheta+\frac{\Lambda^3 f_\vartheta}{\langle\partial_x\phi\rangle^3}= \N[f_\vartheta]+
				\G[f_\vartheta,\phi],\\
				&f_\vartheta|_{t=0}= f_{0,\vartheta},
			\end{align*}
			with $\phi=f_{0,\ell}$ fixed in step 1. Applying Theorem \ref{lemmain}, we obtain the following a priori estimates. 
			\begin{equation*}\label{stm1}
				\begin{aligned}
					\| f_\vartheta\|_{L_{T_{0,\vartheta}}^\infty\dot W^{1,\infty}}
					\lesssim&\|f_{0,\vartheta}\|_{\dot W^{1,\infty}}+	\da_{T_{0,\vartheta}}(\N[f_\vartheta]+\G[f_\vartheta,\phi]).
				\end{aligned}
			\end{equation*}
			Lemma \ref{nonst} and Lemma \ref{GGG} yield
			\begin{equation*}\label{st11}
				\begin{aligned}
					&	\da_{T_{0,\vartheta}}( \N[f_\vartheta])+\da_{T_{0,\vartheta}}(\G[f_\vartheta,\phi])\\
                    &\quad\quad\lesssim (\| f_\vartheta\|_{{T_{0,\vartheta}},*}^2+T_{0,\vartheta}^\frac{1}{10}\|f_\vartheta\|_{X_{T_{0,\vartheta}}}+\|f_\vartheta-\phi\|_{T_{0,\vartheta}}\|f_\vartheta\|_{{T_{0,\vartheta}},*})(1+\| f_\vartheta\|_{X_{T_{0,\vartheta}}}+\|\phi\|_{T_{0,\vartheta}})^{2m+5}.
				\end{aligned}
			\end{equation*}
			Then we obtain
			\begin{equation}\label{stfin}
				\begin{aligned}
					&\|f_\vartheta\|_{L_{T_{0,\vartheta}}^\infty\dot W^{1,\infty}}\lesssim  
					\|\partial_xf_{0,\vartheta}\|_{\dot W^{1,\infty}}\\
                    &\quad\quad\quad\quad+(\| f_\vartheta\|_{{T_{0,\vartheta}},*}^2+T_{0,\vartheta}^\frac{1}{10}\|f_\vartheta\|_{X_{T_{0,\vartheta}}}+\|f_\vartheta-\phi\|_{T_{0,\vartheta}}\|f_\vartheta\|_{{T_{0,\vartheta}},*})(1+\| f_\vartheta\|_{X_{T_{0,\vartheta}}}+\|\phi\|_{T_{0,\vartheta}})^{2m+5}.
				\end{aligned}
			\end{equation}
			On the other hand, integrating \eqref{mapstt} in time, together with \eqref{eee1} and Remark \ref{musLinfty}, we can get
			\begin{align*}
				|f_\vartheta(t,x)-f_{0,\vartheta}(x)|&\lesssim \int_0^t (\|\Lambda^3f_\vartheta(\tau)\|_{L^\infty}+\|\N[f_\vartheta](\tau)\|_{L^\infty}+\|\G[f_\vartheta,\phi](\tau)\|_{L^\infty})d\tau\\
				&\leq \tilde C_0t^\frac{1}{3} (1+\|f_\vartheta\|_{X_{T_{0,\vartheta}}}+\|(f_\vartheta,\phi)\|_{X_{T_{0,\vartheta}}})^5,\ \ \ \forall t\in[0,T_{0,\vartheta}].
			\end{align*}
			From this we obtain
			\begin{align*}
				\|f_\vartheta\|_{L^\infty_{T_{0,\vartheta}}L_x^\infty}\leq \|f_{0,\vartheta}\|_{L^\infty}+\tilde C_0T_{0,\vartheta}^\frac{1}{3}(1+\|(f_\vartheta,\phi)\|_{X_{T_{0,\vartheta}}})^5.
			\end{align*}
			Combining this with \eqref{stfin}, 
			we derive 
			\begin{equation}\label{tFXt}
            \begin{aligned}
				&\| f_\vartheta\|_{L^\infty_{T_{0,\vartheta}}W^{1,\infty}}\leq \tilde C_0\|f_{0,\vartheta}\|_{W^{1,\infty}}\\
                &\quad+\tilde C_0\left(\|f_\vartheta\|_{{T_{0,\vartheta}},*}(\|f_\vartheta\|_{X_{T_{0,\vartheta}}}+\|\phi\|_{X_{T_{0,\vartheta}}})+T_{0,\vartheta}^\frac{1}{10}\right)\left(1+\|f_\vartheta\|_{X_{T_{0,\vartheta}}}+\|\phi\|_{X_{T_{0,\vartheta}}}\right)^{2m+5}.
			\end{aligned}
            \end{equation}
			Here the constant $\tilde C_0$ depends only on $d$, $\kappa$, $m$ and $\varrho_0$. \\
			{\bf 3: Control of $\|f-\phi\|_{{T_{0,\vartheta}},*}$.}\\  
			Applying Theorem \ref{lemmain} to the equation 
			\begin{equation*}
				\begin{aligned}
					&\partial_t (f_\vartheta-\phi)(x)+\frac{\Lambda^3(f_\vartheta-\phi)(x)}{\langle \partial_x f_\vartheta(x)\rangle^3}=\N[f_\vartheta](x)-\frac{\Lambda^3\phi(x)}{\langle \partial_x f_\vartheta(x)\rangle^3},\\
					&(f_\vartheta-\phi)|_{t=0}=f_{0,\vartheta}-\phi,
				\end{aligned}
			\end{equation*}
			we obtain 
			\begin{align*}
				&\|f_\vartheta-\phi\|_{{T_{0,\vartheta}},*}\lesssim (1+\|\partial_x f_\vartheta\|_{L^\infty_{T_{0,\vartheta}}L^\infty})^{m+5}\\
                &\quad\quad\times\left(\|f_{0,\vartheta}-\phi\|_{\dot B^1_{\infty,\infty}}+\da_{T_{0,\vartheta}}(\N[f_\vartheta]-\frac{\Lambda^3\phi}{\langle\partial_x f_\vartheta\rangle^3})+\|f_\vartheta\|_{{T_{0,\vartheta}},*}\|f_\vartheta-\phi\|_{{T_{0,\vartheta}},*}\right),
			\end{align*}
			Combining this with Lemma \ref{nonst} yields
			\begin{align}
				\|f_\vartheta-\phi\|_{{T_{0,\vartheta}},*}\leq &\tilde C_0(1+\|f_\vartheta\|_{X_{T_{0,\vartheta}}})^{m+5}\|f_{0,\vartheta}-\phi\|_{\dot B^1_{\infty,\infty}}\nonumber\\
				&+\tilde C_0(1+\|f_\vartheta\|_{X_{T_{0,\vartheta}}})^{3(m+5)}\left(\|f_\vartheta-\phi\|_{{T_{0,\vartheta}},*}+T_{0,\vartheta}^\frac{1}{20}(1+\|(f_\vartheta,\phi)\|_{X_{T_{0,\vartheta}}})\right)^2.\label{st22}
			\end{align}
			{\bf 4: Existence.}\\
			Denote 
			\begin{align*}
			\varepsilon_1=10\tilde C_0(1+10\tilde C_0\|f_0\|_{W^{1,\infty}})^{m+5}\|f_0-\phi\|_{\dot B^1_{\infty,\infty}}, \quad\quad\quad	M_1=4\tilde C_0(\eps_1+\|f_0\|_{W^{1,\infty}}).\end{align*}
			Take $\eps_0=(10+\tilde C_0)^{-10(m+5)}$. For any $\eps\in(0,\eps_0)$, if the initial data satisfy the quantitative bound 
			\begin{align*}
				(1+\|f_0\|_{W^{1,\infty}})^{4(m+5)}\|f_0-\phi\|_{\dot B^1_{\infty,\infty}}\leq \eps, 
			\end{align*} then $(1+M_1)^{3(m+5)}\varepsilon_1\leq \frac{M_1}{10}$.
			Define
			\begin{align*}
				T_{1,\vartheta}:=\sup\{T: \|f_\vartheta\|_{X_{T}}\leq M_1,   \|f_\vartheta-\phi\|_{T,*}\leq \varepsilon_1\}.
			\end{align*}
			We claim that 
			\begin{align}\label{lowbdmus}
				T_{1,\vartheta}\geq T_2:=\left(\frac{\varepsilon_1}{1+M_1+\|\phi\|_{C^{m+5}}}\right)^{1000(m+5)},\quad\forall\vartheta\in(0,1).
			\end{align}
			Note that since $\phi=f_{0,\ell}$ is independent of $\vartheta$, \eqref{lowbdmus} gives a uniform bound of $T_{1,\vartheta}$ for all $\vartheta$. We prove \eqref{lowbdmus} by contradiction. If the claim is not true, then $T_{1,\vartheta}<T_2$ for some $\vartheta$. We have $$
			\|\phi\|_{X_{T_1}}\leq M_1,\ \ \|\phi\|_{T_1,*}\leq \varepsilon_1.
			$$
			It follows from \eqref{tFXt} and \eqref{st22} that 
			\begin{equation*}\label{priestmus}
				\begin{aligned}
					& \|f_\vartheta\|_{L^\infty_{T_{0,\vartheta}}W^{1,\infty}}\leq \frac{M_1}{4}+\tilde C_0\left(\varepsilon_1+T_{1,\vartheta}^\frac{1}{10}\right)\left(1+  2M_1\right)^{2(m+5)}\leq \frac{M_1}{2},\\
					&     \|f_\vartheta-\phi\|_{T_{1,\vartheta},*}\leq \frac{\varepsilon_1}{10}
					+\tilde C_0\left(\varepsilon_1+T_{1,\vartheta}^\frac{1}{20}(1+2M_1)\right)^2\left(1+  2M_1\right)^{3(m+5)}\leq \frac{\varepsilon_1}{2}.
				\end{aligned}
			\end{equation*}
			This contradicts the definition of $T_{1,\vartheta}$. Thus, we have proved the following uniform local-in-time a priori estimates
			$$\|f_\vartheta\|_{X_{T_2}}\leq M_1, \quad\quad  \|f_\vartheta-\phi\|_{T_2,*}\leq \varepsilon_1.$$
			Recall that $f_{0,\vartheta}=f_0\ast \rho_{\vartheta}$, $f_\vartheta$ is the unique solution to \eqref{eqmst} in $[0,T_{0,\vartheta}]$ with initial data $f_{0,\vartheta}$, and 
			\begin{equation*}
				\sup_{t\in[0,T_{0,\vartheta}]}\left(\|f_{\vartheta}(t)\|_{W^{1,\infty}}+t^{\frac{m+\kappa}{s}}\|\partial_xf_\vartheta(t)\|_{\dot C^{m+\kappa}} \right)\leq C\|f_0\|_{ W^{1,\infty}}.
			\end{equation*}
  Using the equation \eqref{eqmusre} and by interpolation inequality, we obtain 
\begin{align*}
\sup_{t\in[0,T_{0,\vartheta}]}t^{\frac{2}{3}}\|\partial_t f_\vartheta
(t)\|_{L^\infty}\lesssim C\|f_0\|_{ W^{1,\infty}}.
\end{align*}
           Applying the existence argument in Proposition \ref{propb=0}, we can see that $\{f_\vartheta\}$  has a subsequence that converges to a function $f$ on $[0,T_2]$, which is a solution to \eqref{eqmst} with initial data $f_0$, and satisfying the a priori estimate
			\begin{equation*}
				\|f\|_{T_2}\leq 4\tilde{C}_0(\eps+\|f_0\|_{W^{1,\infty}}),\quad \|f-\phi\|_{T_2,*}\leq  10\tilde C_0(1+10\tilde C_0\|f_0\|_{W^{1,\infty}})^{m+5}\|f_0-\phi\|_{\dot B^1_{\infty,\infty}}.
			\end{equation*}
			{\bf 5: Stability and uniqueness.} \\
			Suppose $f,g$ are solutions to \eqref{eqmst} in $[0,T]$ with initial data $f_0,g_0\in\dot W^{1,\infty}$ that satisfy the quantitative bound \eqref{bound2}. Applying Theorem \ref{lemmain} to the equation \eqref{eqdif}, we obtain 
			\begin{align*}
				\|f-g\|_{T}\lesssim C(1+\|f\|_{L^\infty_T\dot W^{1,\infty}})^{m+5}\left(\|f_0-g_0\|_{\dot W^{1,\infty}}+\da_{T}(\N[f]-\N[g]-\G[g,f])+\|f-g\|_{T,*}\|f\|_{T,*}\right).
			\end{align*}
			Applying Lemma \ref{nonst} and \ref{GGG} again, we obtain
			\begin{align*}
				\da_{T}(\N[f]-\N[g])&+\da_{T}(\G[g,f])\\
				&\lesssim \|f-g\|_{X_{T}}(\|(f,g)\|_{T,*}+T^\frac{1}{10})(1+\|(f,g)\|_{T})^{2m+5}.
			\end{align*}
			This yields 
            \begin{equation}\label{sta1}
            		\begin{aligned}
				\|f-g\|_{T}\lesssim &C(1+\|f\|_{L^\infty_T\dot W^{1,\infty}})^{m+5}\\
                &\quad\quad\times\left(\|f_0-g_0\|_{\dot W^{1,\infty}}+\|f-g\|_{X_{T}}(\|(f,g)\|_{T,*}+T^\frac{1}{10})(1+\|(f,g)\|_{T})^{2m+5}\right).
			\end{aligned}
            \end{equation}
			To estimate the $L^\infty$ norm, we integrate \eqref{eqdif} in time, which yields
			\begin{align*}
				&\sup_{t\in[0,T]}\sup_x|(f-g)(t,x)-(f_0-g_0)(x)|\\
				&\quad\quad\lesssim \int_0^{T}(\|\Lambda^3(f-g)(\tau)\|_{L^\infty}+\|(\N[f]-\N[g])(\tau)\|_{L^\infty}+\|\G[f,\phi](\tau)\|_{L^\infty} )d\tau\\
				&\quad\quad\lesssim T^\frac{1}{3}\|f-g\|_{X_{T}}(1+\|(f,g)\|_{X_{T}})^{2m+5}.
			\end{align*}
			This implies 
			\begin{align*}
				\|f-g\|_{L^\infty_{T}L^\infty}\lesssim \|f_0-g_0\|_{L^\infty}+T^\frac{1}{3}\|f-g\|_{X_{T}}(1+\|(f,g)\|_{X_{T}})^{2m+5}.
			\end{align*}
			Combining this with \eqref{sta1}, we obtain 
			\begin{align*}
				\|f-g\|_{X_{T}}\leq C(1+\|f\|_{L^\infty_T\dot W^{1,\infty}})^{m+5}\left(\|f_0-g_0\|_{W^{1,\infty}}+\|f-g\|_{X_{T}}(\|(f,g)\|_{T,*}+T^\frac{1}{10})(1+\|(f,g)\|_{T})^{2m+5}\right).
			\end{align*}
			By taking $T,\varepsilon_0$ small enough, from \eqref{bound2} we have 
			\begin{align*}
				C(\|(f,g)\|_{T,*}+T^\frac{1}{10})(1+\|(f,g)\|_{T})^{3m+5}\leq \frac{1}{10}.
			\end{align*}
			Thus, we obtain 
			\begin{equation}\label{stabmus}
				\|f-g\|_{X_{T}}\leq 2C\|f_0-g_0\|_{W^{1,\infty}}.
			\end{equation}
			Moreover, if $f,g$ are two solutions to \eqref{eqmst} with the same initial data $f_0$, then \eqref{stabmus} gives $f=g$, which infers the uniqueness of the solution.
		\end{proof}

		\section{The Peskin problem in 2D}\label{secpeskin}
		In this section, we study the 2D Peskin problem \eqref{eqpes} and prove Theorem \ref{thmPesB}. We begin by reformulating \eqref{eqpes} in Section \ref{secre}, distinguishing between the dominant quasilinear part and the lower-order nonlinear terms. Next, in Section \ref{secpp}, we apply Theorem \ref{lemmain} and a contraction mapping argument to the reformulated equation to establish local well-posedness. Finally, in Section \ref{secglo}, we provide a brief discussion on the global well-posedness and asymptotic behavior of solutions for initial data close to equilibrium circle configurations.
		\subsection{Reformulation}\label{secre}
		
		For simplicity, denote $\mathbf{T}(\lambda)=\frac{\mathcal{T}(\lambda)}{\lambda}$.		The Peskin problem \eqref{peskin} can be written as 
		\begin{equation}\label{eqpesk}
			\begin{aligned}
				\partial_t X(x)&=\int _{\mathbb{S}} G(\delta_\alpha X(x))\partial_x\left({\mathbf{T}(|\partial_x X|)}\partial_x X\right)(x-\alpha)d\alpha\\
				&=-\int _{\mathbb{S}} \partial_\alpha G(\delta_\alpha X(x))\delta_\alpha\left({\mathbf{T}(|\partial_x X|)}\partial_xX\right)d\alpha:=\N(X(x)),
			\end{aligned}
		\end{equation}
		where the second line follows from integration by parts. We denote $\tilde\alpha=(\frac{1}{2}\cot(\frac{\alpha}{2}))^{-1}$. Note that
        \begin{equation*}
            \frac{1}{2}\cot(\frac{\alpha}{2})=\frac{1}{\alpha}+\sum_{n=1}^{\infty}\left(\frac{1}{\alpha+2n\pi}+\frac{1}{\alpha-2n\pi} \right).
        \end{equation*}
        Hence, for any $h(\alpha):\mathbb{S}\to\mathbb{R}$, by using the periodicity  $h(\alpha)=h(\alpha+2n\pi)$ for any $n\in\mathbb{Z}$, one has
		\begin{align}\label{1perio}
        \int_{\mathbb{R}}h(\alpha)\frac{d\alpha}{\alpha}=\int_{-\pi}^\pi h(\alpha)\frac{d\alpha}{\tilde\alpha}.
		\end{align}
		Following \cite{KN}, we can reformulate the equation as 
		\begin{align}\label{eqpes}
			\partial_t X(x)+\frac{1}{4}\mathcal{H}({\mathbf{T}(|\partial_x X|)}\partial_x X)(x)=\mathcal{N}(X(x)),
		\end{align}
		where $\mathcal{H}$ is the Hilbert transform on the torus, which is defined by
		$$
		\mathcal{H}f(x)=\frac{1}{2\pi}\int_{\mathbb{S}}f(x-\alpha)\cot\left(\frac{\alpha}{2}\right)d\alpha,
		$$
		and the nonlinear term reads
		\begin{equation}\label{defnonpes}
			\begin{aligned}
				\mathcal{N}(X)
				&=\frac{1}{4\pi}\int_{\mathbb{R}}  \frac{\tilde\Delta_\alpha X\cdot \tilde{\E}^\alpha X}{|\tilde\Delta_\alpha X|^2}\delta_\alpha ({\mathbf{T}(|\partial_x X|)}\partial_x X)\frac{d\alpha}{\alpha}\\
				&\quad\quad\quad\quad-\frac{1}{4\pi}\int_{\mathbb{R}}\frac{ \tilde{\E}^\alpha X\otimes \tilde\Delta_\alpha X+\tilde\Delta_\alpha X\otimes \tilde{\E}^\alpha X}{|\tilde\Delta_\alpha X|^2} \delta_\alpha ({\mathbf{T}(|\partial_x X|)}\partial_x X)\frac{d\alpha}{\alpha}\\
				&\quad\quad\quad\quad+\frac{1}{2\pi}\int_{\mathbb{R}} \frac{\tilde\Delta_\alpha X\otimes \tilde\Delta_\alpha X}{|\tilde\Delta_\alpha X|^4}\left(\tilde\Delta_\alpha X\cdot \tilde{\E}^\alpha X\right) \delta_\alpha ({\mathbf{T}(|\partial_x X|)}\partial_x X)\frac{d\alpha}{\alpha}.
			\end{aligned}
		\end{equation}
		Here we denote 
		\begin{align}\label{notapes}
			\tilde\Delta_\alpha X(x)=\frac{\delta_\alpha X(x)}{\tilde \alpha},\quad \tilde{\E}^\alpha X(x)=X'(x-\alpha)-\tilde\Delta_\alpha X(x),\quad X'=\partial_xX, \quad \tilde \alpha =\left(\frac{1}{2}\cot \left(\frac{\alpha}{2}\right)\right)^{-1}. 
		\end{align}
        Taking $\Lambda^\frac{1}{2}\mathcal{H}$ to \eqref{eqpes}, we obtain 
       \begin{align}\label{eqonede}
        \partial_t \Lambda^\frac{1}{2}\mathcal{H}X(x)+\Lambda^\frac{1}{2}({\mathbf{T}(|\partial_x X|)}\partial_x X)=\Lambda^\frac{1}{2}\mathcal{H}\mathcal{N}(X).
        \end{align}
		It is easy to check that 
		\begin{equation}\label{pesopdecop}
			\begin{aligned}
				\Lambda^\frac{1}{2}({\mathbf{T}(|\partial_x X|)}\partial_x X)&=\Lambda^\frac{1}{2}(\mathbf{T}(|\partial_x X|))\partial_x X+\mathbf{T}(|\partial_x X|)\Lambda^\frac{1}{2}\partial_x X-\frac{1}{\pi}\int_{\mathbb{R}} \delta_\alpha (\mathbf{T}(|\partial_x X|))\delta_\alpha \partial_x X\frac{d\alpha}{|\alpha|^{\frac{3}{2}}}\\
				&=\A(\partial_x X)\Lambda^\frac{1}{2} \partial_x X+ \mathsf{M}(\partial_xX),
			\end{aligned}
		\end{equation}
		where we denote 
		\begin{align}
			&\A(\partial_x X)=\frac{1}{4}\mathbf{T}(|\partial_x X|)\mathrm{Id}+\frac{\mathbf{T}'(|\partial_x X|)}{4}\frac{\partial_x X\otimes \partial_x X}{|\partial_x X|},\nonumber\\
			&\mathsf{M}(\partial_xX)=-\frac{1}{\pi}\int_{\mathbb{R}} \delta_\alpha (\mathbf{T}(|\partial_x X|))\delta_\alpha \partial_x X\frac{d\alpha}{|\alpha|^{\frac{3}{2}}}+\frac{1}{\pi} \partial_x X\int_{\mathbb{R}} \delta_\alpha (\mathbf{T}(|\partial_x X|))-\delta_\alpha \partial_xX\cdot \nabla \big(\mathbf{T}(|\partial_x X|)\big)\frac{d\alpha}{|\alpha|^{\frac{3}{2}}}\label{defpeR}.
		\end{align}
        	Hence, we rewrite \eqref{eqonede} as
		\begin{align}\label{pesre}
		 \partial_t \Lambda^\frac{1}{2}\mathcal{H}X+\A(\partial_x X)\Lambda^\frac{1}{2} \partial_x X=\Lambda^\frac{1}{2}\mathcal{H}\mathcal{N}(X)-\mathsf{M}(\partial_xX).
		\end{align}
		Note that 
		\begin{align*}
			\A(b)=\frac{1}{4}\frac{\mathcal{T}(|b|)}{|b|}\left(\mathrm{Id}-\frac{b\otimes b}{|b|^2}\right)+\frac{1}{4}\mathcal{T}'(|b|)\frac{b\otimes b}{|b|^2}.
		\end{align*}
	From \eqref{conten}, we obtain
        \begin{align}\label{Auplo}
    \A(b)\geq \frac{1}{4} \mathfrak{c}(|b|)\mathrm{Id},\ \ \ \ |\nabla_b^n\A(b)|\lesssim_n \mathfrak{C}(|b|)(1+|b|^{-n-1}).
        \end{align}
	
		\subsection{
			Proof of Theorem \ref{thmPesB}}\label{secpp}
		The proof of Theorem \ref{thmPesB} is organized in the following stages.
		\begin{enumerate}
			\item {\bf Relaxed version:} 
			We prove Proposition \ref{thmPesLip} as a preliminary case, replacing the Besov norms in \eqref{pesconbe} by Lipschitz norms.
			
			\item {\bf Existence theory:} 
			For initial data satisfying \eqref{pesconbe}, we establish a priori estimates and construct solutions via compactness arguments.
			
			\item {\bf Stability analysis:} 
			We derive continuous dependence on initial data for the obtained solutions.
		\end{enumerate}
		First, we will prove a relaxed version of Theorem \ref{thmPesB}.
		\begin{proposition}\label{thmPesLip}
			Fix  $m\in\mathbb{N}^+$ and $\kappa\in (0,1)$.
			For any  initial data $X_0\in \dot W^{1,\infty}(\mathbb{S})$ with $\mathbf{\Theta}_0<\infty$, there exists $\varepsilon_0=\varepsilon_0(\mathbf{\Theta}_0,\|\partial_xX_0\|_{L
				^\infty})>0$  such that if  
			\begin{align}\label{pescon}
				\| X_0-\Phi\|_{\dot W^{1,\infty}}\leq \varepsilon_0, 
			\end{align} 
			holds for some smooth function  $\Phi\in C^\infty(\mathbb{S})$, then there exists $T=T(\varepsilon_0,\|\Phi\|_{C^{m+3}})>0$ such that the 2D Peskin problem \eqref{peskin} admits a unique solution  $X$ in 
            \begin{equation*}\label{defnorpes}
            		\begin{aligned}
				\mathcal{X}_T:&=\big\{ Y(t,x): (0,T)\times \mathbb{S}\to \mathbb{R}^2:\mathbf{\Theta}_Y(T)\leq 2\mathbf{\Theta}_0,,\\&\quad\quad\quad\|Y\|_{T}:=\sup_{t\in(0,T)}	(\|\partial_xY(t)\|_{L^\infty}+t^{m+\kappa}\|\partial_x Y(t)\|_{\dot C^{m+\kappa}})\leq 2 \|\partial_x X_0\|_{L^\infty}\big\}.
			\end{aligned}
            \end{equation*}
		\end{proposition} 
		\begin{proof}
			We apply Theorem \ref{lemmain} along with the nonlinear estimates from Lemmas \ref{lemnonpes} and \ref{lempesR} to \eqref{pesre} to establish well-posedness.
			
			Rewrite \eqref{pesre} as  
			\begin{align}\label{reeqPes2d}
				\partial_t\Lambda^\frac{1}{2}\mathcal{H} X+\A(\partial_x \Phi) \Lambda(\Lambda^\frac{1}{2}\mathcal{H} X)=\Lambda^\frac{1}{2}\mathcal{H}\mathcal{N}(X)+ F(\Phi,X),
			\end{align}
			where 
			\begin{align*}
				&F(\Phi,X)=\mathsf{M}(\partial_xX)+(\A(\partial_x \Phi)-\A(\partial_x X)) \Lambda(\Lambda^\frac{1}{2}\mathcal{H} X).
			\end{align*}
			Define   
			\begin{equation}\label{normstar}
				\begin{aligned}
					\|h\|_{T,*}=\sup_{t\in[0,T]}(t^\frac{1}{10}\|\partial_xh(t)\|_{\dot C^\frac{1}{10}}+t^{m+\kappa}\|\partial_xh(t)\|_{\dot C^{m+\kappa}}).
				\end{aligned}
			\end{equation}
			Let  $\sigma,T>0$. Consider the set 
			\begin{align}\label{defset}
				\mathcal{X}_{T,\Phi}^\sigma=\left\{Y\in L_T^\infty\dot W^{1,\infty}: Y|_{t=0}=X_0, \ \|Y-\Phi\|_{T}\leq \sigma\right\}.
			\end{align}
			By definition and \eqref{pescon}, it is easy to check that $|\partial_x\Phi(x)|$ has the following upper and lower bound
			\begin{equation}
				\label{esphi}
				\begin{aligned}
					&|\partial_x\Phi(x)|\leq \|X_0-\Phi\|_{\dot W^{1,\infty}}+\|X_0\|_{{\dot W^{1,\infty}}}\leq \varepsilon_0+\|X_0\|_{\dot W^{1,\infty}}\leq \frac{3}{2}\|X_0\|_{\dot W^{1,\infty}},\\
					&|\partial_x\Phi(x)|\geq |\partial_x X_0(x)|-\|X_0-\Phi\|_{\dot W^{1,\infty}}\geq \mathbf{\Theta}_0^{-1}-\varepsilon_0\geq  \frac{1}{2\mathbf{\Theta}_0},\quad\quad \forall\ x\in\mathbb{S},
				\end{aligned}
			\end{equation}
			provided $\varepsilon_0\leq \min\{\frac{1}{100\mathbf{\Theta}_0},\frac{1}{2}\|X_0\|_{\dot W^{1,\infty}}\}$. 
           Taking $$\sigma\leq \min\{\frac{1}{100\mathbf{\Theta}_0},\frac{1}{2}\|X_0\|_{\dot W^{1,\infty}}\},$$  it is straightforward to check that 
			for any  $Y\in\mathcal{X}^\sigma_{T,\Phi}$, it holds $\|Y-\Phi\|_{T}\leq \sigma$ and
			\begin{align}\label{wes}
				\inf_{\alpha,x}|\Delta_\alpha Y(x)|&\geq 	\inf_{\alpha,x}|\Delta_\alpha X_0(x)|-\|\partial_x(X_0-\Phi)\|_{L^\infty}-\|\partial_x (\Phi-Y)\|_{L^\infty}\nonumber\\
				&\geq \frac{1}{\mathbf{\Theta}_0}-\frac{5}{4}\varepsilon_0-\sigma\geq  \frac{1}{2\mathbf{\Theta}_0},
			\end{align}
		Hence, we obtain  
			\begin{align}\label{arc}
			\mathbf{\Theta}_Y(T)\leq2\mathbf{\Theta}_0,\ \ \  	\|\partial_x Y\|_{L^\infty_TL^\infty}\leq 2\|\partial_x X_0\|_{L^\infty},\ \ \forall Y\in\mathcal{X}^\sigma_{T,\Phi}.
			\end{align}
			Combining this with \eqref{Auplo}, we obtain 
			\begin{align*}
            &\inf_{\tau\in[0,1]}\A(\tau\partial_x Y+(1-\tau)\partial_x\Phi)\geq \frac{1}{4}\mathfrak{c}(2\|\partial_x X_0\|_{L^\infty})\mathrm{Id},\\
			&	\sup_{\tau\in[0,1]}|(\nabla^n\A)(\tau\partial_x Y+(1-\tau)\partial_x\Phi)|\lesssim \mathfrak{C}(2\|\partial_x X_0\|_{L^\infty})(1+10\mathbf{\Theta}_0)^{n+1},\ \ \forall n\in\mathbb{N},
			\end{align*}
		where $\mathfrak{c}(\lambda),\mathfrak{C}(\lambda)$ are defined in \eqref{conten}. Denote 
        \begin{equation}\label{defM0}
              \begin{aligned}
             & \mathfrak{M}_T(Y):=\left(1+ (\mathfrak{c}(\|\partial_x Y\|_{L^\infty_TL^\infty}))^{-1}+\mathfrak{C}(\|\partial_x Y\|_{L^\infty_TL^\infty})(1+\mathbf{\Theta}_Y(T))^{m+1}\right)^{m+2},\\
      &  \mathfrak{M}_0:=\left(1+ (\mathfrak{c}(1+4\|\partial_x X_0\|_{L^\infty}))^{-1}+\mathfrak{C}(1+4\|\partial_x X_0\|_{L^\infty})(1+10\mathbf{\Theta}_0)^{m+1}\right)^{m+2}.
        \end{aligned}
        \end{equation}
        Then we have 
        \begin{align}\label{controlM}
         \mathfrak{M}_T(Y)\lesssim  \mathfrak{M}_0,\ \ \ \quad\quad\forall Y\in \mathcal{X}_{T,\Phi}^\sigma.
        \end{align}
			For any $Y\in \mathcal{X}_{T,\Phi}^\sigma$, where $T, \sigma$ will be fixed later, we define a map $\mathcal{S}Y=X$, where $X$ solves the Cauchy problem
			\begin{align*}
				&	\partial_t\Lambda^\frac{1}{2}\mathcal{H} X+\A(\partial_x \Phi) \Lambda(\Lambda^\frac{1}{2}\mathcal{H} X)=\Lambda^\frac{1}{2}\mathcal{H}\mathcal{N}(Y)+ F(\Phi,Y),\\
				&\Lambda^\frac{1}{2}\mathcal{H} X|_{t=0}=\Lambda^\frac{1}{2}\mathcal{H} X_0.
			\end{align*}
			\textbf{Step 1:} We first prove that there exist $\sigma,T>0$ such that $\mathcal{S}$ maps $\mathcal{X}_{T,\Phi}^\sigma$ to itself.
			Denote $\tilde X=X-\Phi$, 	  one has 
			\begin{align}\label{peseqr}
				\partial_t \Lambda^\frac{1}{2}\mathcal{H} \tilde  X+\A(\partial_x\Phi) \Lambda (\Lambda^\frac{1}{2}\mathcal{H} \tilde  X)=\Lambda^\frac{1}{2}\mathcal{H} \mathcal{N}(Y)+ F(\Phi,Y)-\A(\partial_x\Phi) \Lambda^\frac{1}{2}\partial_x\Phi.
			\end{align}
			Applying Corollary \ref{remlinf} to \eqref{peseqr} with $b=\frac{1}{2}$, we obtain that  there exists $T_0=T_0(\|\Phi\|_{C^{m+3}})>0$ such that  for any $0<T\leq T_0$,
			\begin{equation}
		\begin{aligned}\label{pesma}
					\|\tilde X\|_T&\lesssim \mathfrak{M}_0\left(\|\tilde X_0\|_{\dot W^{1,\infty}}+\da_{T}(\Lambda^\frac{1}{2}\mathcal{H} \mathcal{N}(Y)+F(\Phi,Y)+\A(\partial_x\Phi) \Lambda^\frac{1}{2} \partial_x\Phi)\right),
				\end{aligned}
			\end{equation}
            where 
            \begin{align*}
            \da_T(f)= \sup_{t\in[0,T]}(t^\kappa\|f(t)\|_{\dot C^{\kappa-\frac{1}{2}}}+t^{m+\kappa}\|f(t)\|_{\dot C^{m+\kappa-\frac{1}{2}}}).
            \end{align*}
            By Lemma \ref{lemnonpes} and \eqref{controlM}, we have
			\begin{equation}	\label{pesN}
				\begin{aligned}
				\da_{T}(\Lambda^\frac{1}{2}\mathcal{H}\mathcal{N}(Y))
					\lesssim &\mathfrak{M}_0(\|\tilde Y\|_{T}+T^\frac{1}{10}\|\Phi\|_{C^{m+2}})^2(1+\|Y\|_T)^{2(m+5)},
				\end{aligned}
			\end{equation}
			where $\tilde Y=Y-\Phi$. Here we also used the fact that $\| Y\|_{T,*}\leq \|\tilde Y\|_{T,*}+\|\Phi\|_{T,*}\lesssim \|\tilde Y\|_{T}+T^\frac{1}{10}\|\Phi\|_{C^{m+2}}$.\vspace{0.1cm}\\
			Then we estimate $F(\Phi,Y)$. We first deal with $\mathsf{M}(\partial_xY)$.
			By Lemma \ref{lempesR} and \eqref{controlM},  we have for any $0<T<1$,
			\begin{align*}
				\da_{T}(\M(\partial_x Y))
				\lesssim \mathfrak{M}_0\|\tilde Y\|_T(\|\tilde Y\|_{T}+T^\frac{1}{10}\|\Phi\|_{C^{m+2}})(1+\|\tilde Y\|_T+\|\Phi\|_T)^{m+5}.
			\end{align*}
			From Lemma \ref{maininterpo}, it is straightforward to obtain 
			\begin{equation*}\label{FPY}
				\begin{aligned}
				\da_T((\A(\partial_x \Phi)-\A(\partial_x Y)) \Lambda(\Lambda^\frac{1}{2}\mathcal{H} Y))\lesssim \mathfrak{M}_0(\|\tilde Y\|_{T}+T^\frac{1}{10}\|\Phi\|_{C^{m+3}})^2(1+\|\tilde Y\|_{T}+\|\Phi\|_{T})^{m+5}.
				\end{aligned}
			\end{equation*}
			Hence, we deduce that   
			\begin{align*}
			\da_T(F(\Phi,Y))\lesssim \mathfrak{M}_0(\|\tilde Y\|_{T}+T^\frac{1}{10}\|\Phi\|_{C^{m+3}})^2(1+\|\tilde Y\|_{T}+\|\Phi\|_{T})^{m+5}.
			\end{align*}
			It remains to estimate the lower order term $\A(\partial_x\Phi) \Lambda^\frac{1}{2}\partial_x\Phi $.
			It is easy to check that 
			\begin{equation}\label{peslow}
				\begin{aligned}
					\da_T(\A(\partial_x\Phi) \Lambda^\frac{1}{2}\partial_x\Phi )\lesssim \mathfrak{M}_0T^{\kappa}(1+\|\Phi\|_{C^{m+3}})^{m+3}.
				\end{aligned}
			\end{equation}
			Combining \eqref{pesN}-\eqref{peslow} with \eqref{pesma},  we obtain for any $0<T\leq T_0$,
			\begin{equation}\label{maespes}
				\begin{aligned}
					\|\tilde X\|_T\leq& C_0\mathfrak{M}_0\|\tilde X_0\|_{\dot W^{1,\infty}}+C_0\mathfrak{M}_0^2(\|\tilde Y\|_{T}+T^\frac{1}{10}\|\Phi\|_{C^{m+3}})^2(1+\|\tilde Y\|_{T}+\|\partial_x\Phi\|_{T})^{2(m+5)}\\
					&+C_0\mathfrak{M}_0^2T^\kappa(1+\|\Phi\|_{C^{m+3}})^{2(m+5)}.
				\end{aligned}
			\end{equation}
			Note that by \eqref{esphi},
			\begin{align}\label{phiT}
				\|\Phi \|_T\leq \|\partial_x \Phi\|_{L^\infty}+T^{m+\kappa}\|\partial_x \Phi\|_{\dot C^{m+\kappa}}\leq 1+\frac{3}{2}\|X_0\|_{\dot W^{1,\infty}}+T^{m+\kappa}\|\Phi\|_{\dot C^{m+2}}.
			\end{align}
			Take $0<\varepsilon_0<\frac{1}{100}(C_0\mathfrak{M}_0^2+\|X_0\|_{\dot W^{1,\infty}}+1)^{-m-5}$, $\sigma=2C_0\mathfrak{M}_0\varepsilon_0$, and $$T_1=\min\left\{T_0,\left(\frac{\varepsilon_0}{10+C_0+\|\Phi\|_{C^{m+3}}}\right)^{100m}\right\}.$$  Then \eqref{maespes} leads to 
			\begin{align*}
				\|\tilde X\|_{T_1}&\leq C_0\mathfrak{M}_0\varepsilon_0+2C_0^2\mathfrak{M}_0^2\varepsilon_0 (2C_0\mathfrak{M}_0\varepsilon_0+T_1^\frac{1}{10}\|\Phi\|_{C^{m+3}})(1+\|X_0\|_{\dot W^{1,\infty}}+T^\frac{1}{10}\|\Phi\|_{C^{m+3}})^{2(m+5)}\\
				&\quad+C_0\mathfrak{M}_0^2T_1^\kappa(1+\|\Phi\|_{C^{m+3}})^{2(m+5)}\leq \sigma.
			\end{align*}
			Hence $\mathcal{S}$ maps $\mathcal{X}_{T_1,\Phi}^\sigma$ to itself. \vspace{0.1cm}\\
			\textbf{Step 2:} We prove that $\mathcal{S}$ is a contraction map. Consider $Y_1,Y_2\in \mathcal{X}_{T_1,\Phi}^\sigma$. Denote $\mathbf{Y}=Y_1-Y_2$, and $\mathbf{X}=X_1-X_2=\mathcal{S}Y_1-\mathcal{S}Y_2$. Then we have 
			\begin{align*}
				\partial_t \Lambda^\frac{1}{2}\mathcal{H}\mathbf{X}+\A(\partial_x\Phi) \Lambda (\Lambda^\frac{1}{2}\mathcal{H}\mathbf{X})=\Lambda^\frac{1}{2}\mathcal{H}(\mathcal{N}(Y_1)-\mathcal{N}(Y_2))+ F(\Phi,Y_1)-F(\Phi,Y_2).
			\end{align*}
			Applying Theorem \ref{lemmain}, we have for any $0<T\leq T_0$,
			\begin{align*}
				\|\mathbf{X}\|_{T}\lesssim& \mathfrak{M}_0\left(\da_T(\Lambda^\frac{1}{2}\mathcal{H}(\mathcal{N}(Y_1)-\mathcal{N}(Y_2)))+\da_T(F(\Phi,Y_1)-F(\Phi,Y_2))\right).
			\end{align*}
			From Lemma \ref{lemnonpes} and \ref{lempesR}, we obtain 
            \begin{equation}\label{diffor}
            \begin{aligned}
				\da_T(\Lambda^\frac{1}{2}\mathcal{H}(\mathcal{N}(Y_1)-\mathcal{N}(Y_2)))&+\da_T(F(\Phi,Y_1)-F(\Phi,Y_2))\\
				&\lesssim \mathfrak{M}_0 \|\mathbf{Y}\|_{T}\|({Y}_1, {Y}_2)\|_{T,*}(1+\|(Y_1,Y_2)\|_T)^{2(m+1)}.
			\end{aligned}
            \end{equation}
			Hence, by \eqref{phiT}, 
			\begin{align*}
				\|\mathbf{X}\|_{T_1}\leq& C_1 \mathfrak{M}_0 ^2\|\mathbf{Y}\|_{T_1}(\|(\tilde Y_1,\tilde Y_2)\|_{T_1}+T_1^\frac{1}{10}\|\Phi\|_{C^{m+3}})\\
                &\times(1+\| (\tilde Y_1,\tilde Y_2)\|_{T_1}+\|X_0\|_{\dot W^{1,\infty}}+T^\frac{1}{2}\|\Phi\|_{C^{m+3}})^{2(m+5)}\\
				\leq& 4C_1\mathfrak{M}_0 ^2\sigma(1+4\sigma+T_1^\frac{1}{10}\|\Phi\|_{C^{m+3}})^{2(m+5)}\|\mathbf{Y}\|_{T_1},
			\end{align*}
			where $\tilde{Y}_i=Y_i-\Phi$, and  we used the fact that $\|\tilde  Y_1\|_{T_1},\|\tilde  Y_2\|_{T_1}\leq \sigma$ and  $\mathbf{\Theta}_{Y_1}(T_1),\mathbf{\Theta}_{Y_2}(T_1)\leq 2\mathbf{\Theta}_0$. We have 
			\begin{align*}
				4C_1\mathfrak{M}_0 ^2\sigma(1+4\sigma+T_1^\frac{1}{10}\|\Phi\|_{C^{m+3}})^{2(m+5)}\leq \frac{1}{2}
			\end{align*}
			by taking $\varepsilon_0=\frac{1}{100}(C_0\mathfrak{M}_0 ^2+C_1\mathfrak{M}_0 ^2+\|X_0\|_{\dot W^{1,\infty}}+1)^{-2(m+10)}$ and the corresponding $T$ small enough. This implies $\mathcal{S}:\mathcal{X}_{T_1,\Phi}^\sigma \to \mathcal{X}_{T_1,\Phi}^\sigma$ is a contraction map. Hence, there exists a unique $X\in \mathcal{X}_{T_1,\Phi}^\sigma$ such that $X=\mathcal{S}X$, which indicates $X$ is a solution to \eqref{eqpes}. Moreover, by \eqref{arc} we have 
			$\mathbf{\Theta}_X(T_1)\leq 2\mathbf{\Theta}_0$. This completes the proof of the Proposition \ref{thmPesLip}.
		\end{proof}\vspace{0.3cm}\\
		\begin{proof}[Proof of Theorem \ref{thmPesB}] \\
            \textbf{Step 1: Construction of approximating sequence.}\\
            To show the existence of solution, we use the standard compactness argument. We will construct the approximating solutions as follows. Denote $X_{0,\vartheta}=X_0\ast \rho_\vartheta$, and denote $X_{\vartheta}$ to be the solution of \eqref{peskin} with initial data $X_{0,\vartheta}$ on $[0,T_\vartheta]$ with $T_{\vartheta}=T(\|X_0\|_{\dot W^{1,\infty}},\vartheta)$. Note that \eqref{pesconbe} implies 
            \begin{equation*}
                \mathfrak{M}(z_0)\liminf_{\ell\rightarrow 0}\| X_0-X_{0,\ell}\|_{\dot B^{1}_{\infty,\infty}}\leq \frac{3}{2}\varepsilon.
            \end{equation*}
        We fix an $\ell\ll 1$ and let $\Phi=X_{0,\ell}$. There exists a subsequence $\{X_{0,\vartheta}\}$,  $\vartheta<\ell$,  such that 
            \begin{align*}
               & \| X_{0,\vartheta}\|_{\dot W^{1,\infty}}\leq 2\|X_0\|_{\dot W^{1,\infty}},\\
                & \mathfrak{M}(z_0)\| X_{0,\vartheta}-X_{0,\ell}\|_{\dot B^{1}_{\infty,\infty}}\leq 2\varepsilon.
            \end{align*}
            Furthermore, we can take $\vartheta$ small enough such that for any $\alpha,x\in \mathbb{S}$, $\alpha\neq 0$,
			\begin{align*}
				\frac{|\delta_\alpha X_{0,\vartheta}(x)|}{|\alpha|}\geq \frac{1}{2\mathbf{\Theta}_0}.
			\end{align*}
			This implies that 
			\begin{equation}\label{acmod}
			\mathbf{\Theta}(X_{0,\vartheta})\leq 2 \mathbf{\Theta}_0.
			\end{equation}
            \textbf{Step 2: Control of $\|X_\vartheta\|_{L^\infty_{T_{\vartheta}}\dot W^{1,\infty}}$.}\\
			Applying Corollary \ref{remlinf} to \eqref{pesre} with $b=\frac{1}{2}$, we obtain 
			\begin{equation}\label{pes1.1}
				\begin{aligned}
					\|X_\vartheta\|_{\dot W^{1,\infty}}\lesssim \| X_{0,\vartheta}\|_{\dot W^{1,\infty}}+\da_{T_\vartheta}(\Lambda^\frac{1}{2}\mathcal{H}\mathcal{N}(X_\vartheta)-\M(\partial_xX_\vartheta))+\|X_\vartheta\|_{T,*}^2,
				\end{aligned}
			\end{equation}
			where we denote $\Phi=X_{0,\ell}$ as fixed in step 1. It follows from Lemma \ref{lemnonpes} and Lemma \ref{lempesR} that
			\begin{equation}\label{pes1.2}
				\begin{aligned}
					&\da_{T_\vartheta}(\Lambda^\frac{1}{2}\mathcal{H}\mathcal{N}(X_\vartheta)-\M(\partial_xX_\vartheta))\lesssim\mathfrak{M}_0\| X_\vartheta\|_{T_\vartheta,*}^2 (1+\| X_\vartheta\|_{T_\vartheta})^{2(m+5)}.
				\end{aligned}
			\end{equation}
		 The estimates \eqref{pes1.1} and \eqref{pes1.2} give that
			\begin{equation}\label{estpesL}
				\begin{aligned}
					\|X_\vartheta\|_{L^\infty_{T_{\vartheta}}\dot W^{1,\infty}}&\leq C_0 \left(\|X_{0,\vartheta}\|_{\dot W^{1,\infty}}+\mathfrak{M}_0( \|  \tilde {X}_\vartheta\|_{T_\vartheta,*}+T_\vartheta^\frac{1}{10}\| \Phi\|_{C^{m+3}})^2 (1+\| X_\vartheta\|_{T_\vartheta})^{2(m+5)}\right).
				\end{aligned}
			\end{equation}
			\textbf{Step 3: Control of $\|X_\vartheta\|_{T_\vartheta,*}$.}\\
		From \eqref{pesre}, we have for $\tilde X_\vartheta=X_\vartheta-\Phi$,
        \begin{align*}
         \partial_t \Lambda^\frac{1}{2}\mathcal{H}\tilde X_\vartheta(x)+\A(\partial_x X_\vartheta)\Lambda^\frac{1}{2} \partial_x \tilde  X_\vartheta=\Lambda^\frac{1}{2}\mathcal{H}\mathcal{N}(X_\vartheta)-\mathsf{M}(\partial_xX_\vartheta)-\A(\partial_x X_\vartheta)\Lambda^\frac{1}{2} \partial_x \Phi.
        \end{align*}
It follows from	 Corollary \ref{remlinf} and \eqref{pes1.2} that 
			\begin{equation}\label{esTs}
				\begin{aligned}
					\|\tilde {X}_\vartheta\|_{T_\vartheta,*}&\lesssim\mathfrak{M}_0\left(\|\tilde X_{0,\vartheta}\|_{\dot B_{\infty,\infty}^{1}}+\da_{T_\vartheta}(\Lambda^\frac{1}{2}\mathcal{H}\mathcal{N}(X_\vartheta)-\mathsf{M}(\partial_xX_\vartheta))+\|X_\vartheta\|_{T_\vartheta,*}\|\tilde{X}_\vartheta\|_{T_\vartheta,*}\right)\\
                    &\leq C_1\mathfrak{M}_0(\|\tilde X_{0,\vartheta}\|_{\dot B_{\infty,\infty}^{1}}+\mathfrak{M}_0\| (\tilde X_\vartheta,\Phi)\|_{T_\vartheta,*}^2 (1+\| X_\vartheta\|_{T_\vartheta})^{2(m+5)}+\|X_\vartheta\|_{T_\vartheta,*}\|\tilde{X}_\vartheta\|_{T_\vartheta,*}).
				\end{aligned}
			\end{equation}
            	\textbf{Step 4: Control of $\mathbf{\Theta}_{X_\vartheta}(T_\vartheta)$.}\\
                Now we control $\mathbf{\Theta}_{X_\vartheta}(T_{\vartheta})$ for $\dot B^1_{\infty,\infty}$ data satisfying \eqref{pesconbe}, which differs from the case of Lipschitz data addressed in estimate \eqref{wes}, due to the absence of smallness in the Lipschitz norm. We therefore employ the approach outlined in \cite{KN}. Define
			\begin{align*}
				Q_h(T)=\sup_{t\in[0,T]}\sup_{\alpha,s}\frac{|\alpha|^{\tilde\varepsilon}}{t^{\tilde\varepsilon}}\left|\frac{1}{|\Delta_\alpha h(t,x)|}-\frac{1}{|\Delta_\alpha h(0,x)|}\right|,
			\end{align*}
			where $0<\tilde \varepsilon<\frac{1}{10}$ is a fixed parameter. 
			By \eqref{eqpes}, we have 
			\begin{align*}
				&\left|\frac{1}{|\Delta_\alpha X_\vartheta(t,\cdot)(x)|}-\frac{1}{|\Delta_\alpha X_{0,\vartheta}(x)|}\right|=\int_0^t \left|\partial_t\left(\frac{1}{|\Delta_\alpha X_\vartheta(\tau,\cdot)(x)|}\right)\right|d\tau\\
				&\quad\quad\leq \mathbf{\Theta}^2_{X_\vartheta}(t) \int_0^t |\Delta_\alpha \mathcal{H} (\mathbf{T}(|\partial_x X_\vartheta|)\partial_xX_\vartheta)(\tau,\cdot)(x)|+|\Delta_\alpha \mathcal{N}(X_\vartheta(\tau,\cdot))(x)|d\tau.
			\end{align*}
			The integral on the right hand side can be controlled by $\| X_\vartheta\|_{T_\vartheta,*}$.
			It follows that 
            \begin{equation}\label{esQX}
			\begin{aligned}
				&Q_{X_\vartheta}(T_{\vartheta})\leq \mathbf{\Theta}_{X_\vartheta}^2 (T_{\vartheta})\sup_{t\in[0,T_{0,\vartheta}]} t^{1-\tilde\varepsilon}(\|\mathbf{T}(|\partial_x X_\vartheta|)\partial_xX_\vartheta(t)\|_{\dot C^{1-\tilde\varepsilon}}+\|\mathcal{N}(X_\vartheta(t))\|_{\dot C^{1-\tilde\varepsilon}})\\
				&\quad\quad\lesssim (1+\mathbf{\Theta}_{X_\vartheta}(T_{\vartheta}))^{10} (1+\|X_\vartheta\|_{T_{\vartheta}})^{10}\| X_\vartheta\|_{T_{\vartheta},*},
			\end{aligned}
             \end{equation}
        By \cite[Lemma 2.8]{KN}, we know that the smallness of $Q_{X_\vartheta}(T_{\vartheta})$ implies the boundedness of $\mathbf{\Theta}_{X_\vartheta}(T_{0,\vartheta})$.
			\textbf{Step 5: Existence.}\\
            We define the increasing function $\mathfrak{M}:[0,\infty)\to [1,\infty)$ by 
            \begin{align}\label{defMlam}
          \mathfrak{M}(\lambda)=\tilde C\left((\mathfrak{c}(\tilde C\lambda))^{-1}+\mathfrak{C}(\tilde C\lambda)+\lambda+1\right)^{10(m+1)^3}, 
            \end{align}
            where $\mathfrak{c},\mathfrak{C}$ are defined in  \eqref{conten}, and $\tilde C$ is a fixed constant which is large enough to dominate all absolute constants appeared in this proof. Take $\eps_0=(10+C_0+C_1)^{-10(m+5)}$ and $\eps\in(0,\eps_0)$.
			Denote
			\begin{equation*}
				M_1=4C_0(\eps+\|X_{0}\|_{\dot W^{1,\infty}}),\quad \eps_1=4C_1\mathfrak{M}_0\|\tilde X_{0}\|_{\dot B_{\infty,\infty}^1}.
			\end{equation*}
            With the definition of $\mathfrak{M}(\lambda)$ in \eqref{defMlam}, the condition \eqref{pesconbe} implies 
            \begin{align}\label{Meps}
            (1+M_1) \eps_1\leq \varepsilon.
            \end{align}
    Define
			\begin{equation*}
				T_{0,\vartheta}=\sup\{T:\|X_\vartheta\|_{L^\infty_{T}\dot W^{1,\infty}}\leq M_1,\ \|\tilde{X}_\vartheta\|_{T,*}\leq \eps_1,\ \mathbf{\Theta}_{X_\vartheta}(T)\leq 10\mathbf{\Theta}_0\}.
			\end{equation*}
			We claim that 
			\begin{equation}\label{lowbdPesT}
				T_{0,\vartheta}\geq T_1:=\left(\frac{\eps_1}{10+\|X_0\|_{\dot W^{1,\infty}}+\|\Phi\|_{C^{m+5}}}\right)^{1000(m+5)}.
			\end{equation}
			Note that since we fix $\Phi=X_{0,\ell}$ for some $\ell$, \eqref{lowbdPesT} gives a uniform lower bound of $T_{0,\vartheta}$ for any $\vartheta$. In fact, by contradiction, if $T_{0,\vartheta}<T_1$, then by \eqref{estpesL}, \eqref{esTs} and \eqref{Meps}, we have
			\begin{equation*}
				\begin{aligned}
					&\|X_\vartheta\|_{L^\infty_{T_{0,\vartheta}}\dot W^{1,\infty}}\leq \frac{M_1}{4}+\mathfrak{M}_0(\eps_1^2+T_{0,\vartheta}^{\frac{1}{10}}\|\Phi\|_{\dot C^{m+5}})\leq\frac{M_1}{2},\\
					&\|\tilde{X}_\vartheta\|_{T_{0,\vartheta},*}\leq \frac{\eps_1}{4}+\mathfrak{M}_0(\eps_1^2+T_{0,\vartheta}^{\frac{1}{10}}\|\Phi\|_{\dot C^{m+5}})\leq\frac{\eps_1}{2}.
				\end{aligned}
			\end{equation*}
          By \eqref{acmod}, \eqref{esQX} and \cite[Lemma 2.8]{KN}, we obtain 
            \begin{align*}\label{pesunibdgap}
				\mathbf{\Theta}_{X_\vartheta}(T_{0,\vartheta})\leq 2\mathbf{\Theta}_{0,\vartheta}\leq 4\mathbf{\Theta}_0.
			\end{align*}
			which contradicts the choice of $T_{0,\vartheta}$. Thus, the sequence $\{X_\vartheta\}_\vartheta$ have the following  uniform a priori bounds:
		\begin{equation*}\label{pesT1bound}
				\|X_\vartheta\|_{T_1}\leq M_1,\quad \|\tilde{X}_\vartheta\|_{T_1,*}\leq \eps_1,\quad \mathbf{\Theta}_{X_\vartheta}(T_1)\leq 10\mathbf{\Theta}_0.
			\end{equation*}
            Then following the existence argument in Proposition \ref{propb=0}, we can pass to the limit $\vartheta\to 0$ and the sequence $\{X_\vartheta\}_\vartheta$ will converge to a solution $X$ on $[0,T_1]$, which is a solution to \eqref{eqpes} with initial data $X_0$, and 
		\begin{align}\label{pesunibound}
			\|X\|_{T_1}\leq M_1,\quad \|\tilde{X}\|_{T_1,*}\leq \eps_1,\quad \mathbf{\Theta}_X(T_1)\leq 10\mathbf{\Theta}_0.
		\end{align}
			\textbf{Step 6: Stability and uniqueness.}\\
			establish a stability result in the Lipschitz space. Assume $X,Y$ are two solutions of \eqref{reeqPes2d} in $[0,T]$ with initial data $X_0,Y_0$ respectively, both satisfying \eqref{pesconbe}. Subtracting the corresponding equations, we obtain
			\begin{equation*}
				\begin{aligned}
					\partial_t\Lambda^\frac{1}{2}\mathcal{H}(X-Y)+\A(\partial_x\Phi)\Lambda\Lambda^\frac{1}{2}\mathcal{H}(X-Y)&=\Lambda^\frac{1}{2}\mathcal{H}(\mathcal{N}(X)-\mathcal{N}(Y))+F(\Phi,X)-F(\Phi,Y),\\
					\Lambda^\frac{1}{2}\mathcal{H}(X-Y)|_{t=0}&=\Lambda^\frac{1}{2}\mathcal{H}(X_0-Y_0).
				\end{aligned}
			\end{equation*}
			By  \eqref{main31}, \eqref{diffor} and \eqref{pesunibound}, we have
			\begin{equation*}
				\begin{aligned}
					\|X-Y\|_{T}&\lesssim \mathfrak{M}_0\|X_0-Y_0\|_{\dot W^{1,\infty}}+\da_T(\Lambda^\frac{1}{2}\mathcal{H}(\mathcal{N}(X)-\mathcal{N}(Y))+F(\Phi,X)-F(\Phi,Y))\\
                    &\lesssim \mathfrak{M}_0\|X_0-Y_0\|_{\dot W^{1,\infty}}+\mathfrak{M}_0^2\|X-Y\|_T\|(X,Y)\|_{T,*}(1+\|(X,T)\|_T)^{2(m+1)}\\
                    &\leq C_3\mathfrak{M}_0\|X_0-Y_0\|_{\dot W^{1,\infty}}+C_3\mathfrak{M}_0^2\eps_1(1+M_1)^{2(m+1)}\|X-Y\|_T.
				\end{aligned}
			\end{equation*}
 The condition \eqref{pesconbe} with $\eps<(1+C_3)^{-10}$ implies $C_3\mathfrak{M}_0^2\eps_1(1+M_1)^{2(m+1)}\leq \frac{1}{10}$.
         This yields
			\begin{equation}\label{stabpes}
				\|X-Y\|_{T}\lesssim \|X_0-Y_0\|_{\dot W^{1,\infty}}.
			\end{equation}
			Moreover, if $X,Y$ are two solutions of \eqref{reeqPes2d} with the same initial data $X_0$, then \eqref{stabpes} gives that $X=Y$, which ensures the uniqueness of the solution.
		\end{proof}\\
		\subsection{Global existence and asymptotic behavior}\label{secglo}
		Theorem \ref{thmPesB} establishes the local well-posedness for the 2D Peskin problem \eqref{eqpes}. We now outline how this result can be extended to a global existence framework.\vspace{0.3cm}\\
		It is known that the only stationary solutions to the 2D Peskin system \eqref{eqpesk} are uniformly parameterized circles forming a four-dimensional vector space (see \cite{LinTongSolvability2019}). When interpreting $X(x)$ as a complex-valued function $X=X_1+iX_2$, the stationary solution set can be expressed as $$\{X(x)=a_1+a_2e^{ix}:a_1,a_2\in\mathbb{C},a_2\neq 0\}:=\mathcal{V}\backslash\{0\}.$$
		Let $\mathcal{P}$ be the $L^2$ projection onto the space $\mathcal{V}$:
		\begin{align*}
			\mathcal{P}{Z}(t,x)=z_0(t)+z_1(t)e^{ix},\ \ \ \text{where}\ z_0(t)=\langle Z(t),1\rangle_{L^2(\mathbb{S})}, z_1(t)=\left\langle Z(t),e^{ix}\right\rangle_{L^2(\mathbb{S})}.
		\end{align*}
		Recently, a remarkable work by Garc\'{\i}a-Ju\'{a}rez and Haziot \cite{GHpeskin} (see also \cite{Rodenberg2018}) demonstrates that the linearized operator around the unit circle $Z=e^{ix}$
		\begin{align*}
			\mathcal{L}X=\left.\frac{d\N(Z+\eps X)}{d\eps}\right|_{\eps=0}
		\end{align*}
		satisfies the dissipative estimate
		\begin{equation*}
			\langle\mathcal{L}Y,Y\rangle_{L^2(\mathbb{S})} \leq - 2c_0\|Y\|_{L^2(\mathbb{S})}^2,\quad\quad~~\forall\ Y\in \mathcal{V}^\perp.
		\end{equation*}
		As a consequence, the Green's function $K(t,x,y)$ associated with $(\partial_t -\mathcal{L})$  satisfies the pointwise estimate:
		\begin{align}\label{esgre}
			|\partial_x^{l_1}\partial_y^{l_2}(\mathrm{Id}-\mathcal{P})K(t,x,y)|\lesssim _{l_1,l_2}\frac{e^{-\frac{3}{2}c_0t}}{t^{1+l_1+l_2}}\left\langle\frac{|x-y|}{t}\right\rangle^{-(l_1+l_2+3)},\ \ \quad\quad \forall\ l_1,l_2\in\mathbb{N}.
		\end{align}
		We denote by $e^{t\mathcal{L}}$ the semigroup generated by $\mathcal{L}$ and apply the freezing coefficient method to obtain short-time Schauder estimates:
		\begin{align}\label{sho}
			\|e^{t\mathcal{L}}\partial_xX(\tau)\|_{\dot C^{m}}\lesssim t^{-(m+1-\kappa)}\|X(\tau)\|_{\dot C^{\kappa}},\ \ \ \quad 0<t<T_0,
		\end{align}
		and 
		\begin{align}\label{exp2}
			\left\|\int_0^t e^{\tau \mathcal{L}}\partial_xX(\tau)d\tau\right\|_{\dot C^{m+\kappa}}\lesssim \sup_{\tau\in(0,t)}\|X(\tau)\|_{\dot C^{m+\kappa}},\ \ \ \quad 0<t<T_0,
		\end{align}
		for some $\kappa\in(0,1)$, with $T_0\in(0,\frac{1}{2})$ depending on $\mathbf{\Theta}_0$ and $\|X_0\|_{\dot W^{1,\infty}}$.
		Combining \eqref{sho} with \eqref{esgre}, we have 
		\begin{equation}\label{esexp}
			\begin{aligned}
				&\|e^{t\mathcal{L}}\partial_xX(\tau)\|_{\dot C^{m}}\lesssim e^{-c_0t}t^{-(m+1-\kappa)}\|X(\tau)\|_{\dot C^{\kappa}},\ \ \ \forall t>0,\ \ X\in\mathcal{V}^\perp.
			\end{aligned}
		\end{equation}
		Adjusted to the above estimates, we have the following Schauder-type estimate, which is an analog of Proposition \ref{propb=0}.
		\begin{lemma}\label{lemY}
			Consider the linear equation 
			\begin{align*}
				&\partial_t Y-\mathcal {L}Y=\partial_xF,\quad\quad \text{in}\ (0,+\infty)\times\mathbb{S},\\
				&Y|_{t=0}:=Y_0,
			\end{align*}
			with $Y_0, F\in \mathcal{V}^\perp$. 
			Let $\kappa\in(0,1)$. For any $T>0$, the following estimate holds:
			\begin{align*}
				\sup_{t\in[0,T]} e^{c_0t}(\|Y(t)\|_{L^\infty}+t^{m+\kappa}\|Y(t)\|_{\dot C^{m+\kappa}})\lesssim \|Y_0\|_{L^\infty}+\sup_{t\in[0,T]}e^{c_0t}(t^{\kappa}\|F(t)\|_{\dot C^{\kappa}}+ t^{m+\kappa}\|F(t)\|_{\dot C^{m+\kappa}}).
			\end{align*}
		\end{lemma}
		\begin{proof}
			The proof basically follows the proof of Proposition \ref{propb=0}, the only difference is the appearance of the exponential temporal weight for large time. It suffices to consider $t>1$. Firstly, we have the Duhamel principle:
			\begin{align*}
				Y(t)&=e^{t\mathcal{L}}Y_0+\int_0^t e^{(t-\tau)\mathcal{L}}\partial_x F(\tau)d\tau:=Y_L(t)+Y_F(t).
			\end{align*}
			Note that $Y=(\mathrm{Id}-\mathcal{P})Y$ provided $Y_0,F\in\mathcal{V}^\perp$. 
			From the estimate of \eqref{esgre}, it is straightforward to obtain 
			\begin{align*}
				&\sup_{t\in[0,T]} e^{c_0t}(\|Y_L(t)\|_{L^\infty}+t^{m+\kappa}\|Y_L(t)\|_{\dot C^{m+\kappa}
				})\lesssim \|Y_0\|_{L^\infty}.
			\end{align*}
			For the contribution of force term,  we have 
			\begin{align*}
				\|Y_F(t)\|_{L^\infty}	&\overset{\eqref{esexp}}\lesssim e^{-c_0t} \int_0^t (t-\tau)^{-(1-\kappa)}\tau^{-\kappa} d\tau \sup_{\sigma\in [0,T]}e^{c_0\sigma}\sigma^\kappa\|F(\sigma)\|_{\dot C^\kappa}\\
				&\lesssim e^{-c_0t} \sup_{\sigma\in [0,T]}e^{c_0\sigma}\sigma^\kappa\|F(\sigma)\|_{\dot C^\kappa}.
			\end{align*}
			Then we consider the higher order H\"{o}lder norm. By \eqref{exp2} and \eqref{esexp}, we obtain 
			\begin{align*}
				\|Y_F(t)\|_{\dot C^{m+\kappa}}&\lesssim \int_0^{T_0} \|e^{\tau\mathcal{L}}\partial_x F(t-\tau)\|_{\dot C^{m+\kappa}}d\tau+\int_{T_0}^t \|e^{\tau\mathcal{L}}\partial_x F(t-\tau)\|_{\dot C^{m+\kappa}}d\tau\\
				&\lesssim \sup_{\tau\in(0,T_0)}\|F(t-\tau)\|_{\dot C^{m+\kappa}}+\int_{T_0}^te^{-c_0\tau}\tau^{-(m+1)}\|F(t-\tau)\|_{\dot C^\kappa}d\tau\\ 
				&\lesssim e^{-c_0t }  t^{-(m+\kappa)}\sup_{\sigma\in(0,t)}e^{c_0\sigma } \sigma^{m+\kappa}\|F(\sigma)\|_{C^{m+\kappa}},
			\end{align*}
			where the implicit constant depends on $T_0$. 
			This completes the proof of the lemma.
		\end{proof}\vspace{0.3cm}\\
		We now consider the reformulated system for $(Y,Z)=((\mathrm{Id}-\mathcal{P})X, \mathcal{P}X)$:
		\begin{equation}\label{eqYZ}
			\begin{aligned}
				\partial_t Y -\mathcal{L}Y&=(\mathrm{Id}-\mathcal{P})\mathfrak{N}(X),\\
				\partial_t Z&
				=\mathcal{P}\mathfrak{N}(X).
			\end{aligned}
		\end{equation}
		where $\mathfrak{N}(X)=\N(X)-\mathcal{L}X$.\\
		For the nonlinear term $\mathfrak{N}$, we have 
		\begin{equation}\label{esNpes}
			\begin{aligned}
				\sup_{t\in[0,T]}e^{c_0t}&(t^\kappa\|\mathfrak{N}(X)(t)\|_{\dot C^\kappa}+t^{m+\kappa}\|\mathfrak{N}(X)(t)\|_{\dot C^{m+\kappa}})\\
				&\lesssim \|X\|_{\mathcal{G}_T}^2(1+\|X\|_{\mathcal{G}_T}+\mathbf{\Theta}_X(T)+\|\partial_xZ\|_{L^\infty_TL^\infty})^{m+3},
			\end{aligned}
		\end{equation}
		with the time-weighted norm $\|\cdot\|_{\mathcal{G}_T}$ defined as
		\begin{align*}
			\|X\|_{\mathcal{G}_T}:=\sup_{t\in [0,T]}e^\frac{c_0t}{4}(\|\partial_t\mathcal{P}X\|_{L^\infty}+\|(\mathrm{Id}-\mathcal{P})X\|_{\dot W^{1,\infty}}+t^{m+\kappa}\|(\mathrm{Id}-\mathcal{P})X\|_{\dot C^{m+1+\kappa}}).
		\end{align*}
		The proof follows similarly to Lemma \ref{lemnonpes} and Lemma \ref{lempesR}, we omit details here and refer to \cite{KN} for similar estimates. Remark that the solution (as well as the nonlinear term) generates the exponential decay property from the Green function \eqref{esgre}.\vspace{0.3cm}\\
		Consider the Cauchy problem \eqref{eqpesk} with initial data $X_0$ satisfying
		\begin{equation}\label{smbes}
			(1+\|\partial_xX_0\|_{L^\infty})^{10m}\|(\mathrm{Id}-\mathcal{P})X_0\|_{\dot B^1_{\infty,\infty}}\leq \varepsilon_0.
		\end{equation}
		Theorem \ref{thmPesB} guarantees the local existence of solution $X$ in $[0,T'_0]$ with the estimate 
		\begin{align*}
			&\|X\|_{T'_0}\lesssim (1+\|\partial_xX_0\|_{L^\infty})^{10m},\\
			&\|X-\mathcal{P}X_0\|_{T'_0,*}\leq C(1+\|\partial_xX_0\|_{L^\infty})^{10m}\|(\mathrm{Id}-\mathcal{P})X_0\|_{\dot B^1_{\infty,\infty}}.
		\end{align*}
		Furthermore, from \eqref{eqYZ} we obtain 
		\begin{align*}
			\|\mathcal{P}X(T'_0)-\mathcal{P}X_0\|_{L^\infty}\leq \int_0^{T'_0}\|\mathcal{P}\mathfrak{N}(X)(\tau)\|_{L^\infty}d\tau\lesssim T'_0(1+\|X\|_{T_0'})^{10m}.
		\end{align*}
	It then follows that
		\begin{align*}
			\|(\mathrm{Id}-\mathcal{P})X\|_{T'_0,*}\lesssim C(1+\|\partial_xX_0\|_{L^\infty})^{10m}\|(\mathrm{Id}-\mathcal{P})X_0\|_{\dot B^1_{\infty,\infty}},
		\end{align*}
		which implies 
		\begin{align*}
			\|(\mathrm{Id}-\mathcal{P})X(T'_0)\|_{\dot W^{1,\infty}}&\lesssim \|(\mathrm{Id}-\mathcal{P})X(T'_0)\|_{\dot C^{1+\eta}}\\
			&\lesssim (T'_0)^{-\eta}(1+\|\partial_xX_0\|_{L^\infty})^{10m}\|(\mathrm{Id}-\mathcal{P})X_0\|_{\dot B^1_{\infty,\infty}}.
		\end{align*}
		Note that $T'_0$ is independent of $\|(\mathrm{Id}-\mathcal{P})X_0\|_{\dot B^{1}_{\infty,\infty}}$. By \eqref{smbes}, we obtain \begin{align}\label{lipbes}
		\|(\mathrm{Id}-\mathcal{P})X(T'_0)\|_{\dot W^{1,\infty}}\leq C(T_0')^{-\eta}\varepsilon_0.
		\end{align} Thus, to obtain global well-posedness of solution with initial data satisfying \eqref{smbes}, it suffices to extend the solution globally starting at time $T'_0$, where the data already has a small Lipschitz norm.
		
		To obtain the global existence and asymptotic behavior of the solution, we separately estimate $Y$ and $Z$ in the system
		\eqref{eqYZ} with initial data $(Y_0,Z_0)=((\mathrm{Id}-\mathcal{P})X(T'_0),\mathcal{P}X(T_0))$. Denote $\tilde t=t-T'_0$. Applying Lemma \ref{lemY} for $\partial_xY$, by \eqref{esNpes} we derive for any $T>0$,
		\begin{align*}
			\sup_{\tilde t\in[0,T]} e^{c_0\tilde t}&(\|\partial_xY(t)\|_{L^\infty}+\tilde t^{m+\kappa}\|\partial_x^{m+1} Y(t)\|_{\dot C^\kappa})\\
			&\lesssim \|\partial_xY_0\|_{L^\infty}+\sup_{\tilde t\in[0,T]}e^{c_0\tilde t} (\tilde t^\kappa\|(\mathrm{Id}-\mathcal{P})\mathfrak{N}(X)(t)\|_{\dot C^\kappa}+\tilde t^{m+\kappa}\|(\mathrm{Id}-\mathcal{P})\mathfrak{N}(X)( t)\|_{\dot C^{m+\kappa}})\\
			&\lesssim \|\partial_xY_0\|_{L^\infty}+\|X\|_{\mathcal{G}_T}^2(1+\|X\|_{\mathcal{G}_T}+\mathbf{\Theta}_X(T)+\|\partial_xZ\|_{L^\infty_TL^\infty})^{m+3}.
		\end{align*}
		For $Z$, we directly obtain
		\begin{align*}
			\|\partial_t Z(t)\|_{L^\infty}=\|\mathcal{P}\mathfrak{N}(X)\|_{L^\infty}\lesssim e^{-c_0\tilde t}\|X\|_{\mathcal{G}_T}^2(1+\|X\|_{\mathcal{G}_T}+\mathbf{\Theta}_X(T)+\|\partial_xZ\|_{L^\infty_TL^\infty})^{m+3}. 
		\end{align*}
		Hence, we deduce that 
		\begin{align*}
			\|X\|_{\mathcal{G}_T} \leq C \|(\mathrm{Id}-\mathcal{P})X(T'_0)\|_{\dot W^{1,\infty}}+C\|X\|_{\mathcal{G}_T}^2(1+\|X\|_{\mathcal{G}_T}+\mathbf{\Theta}_X(T)+\|\partial_xZ\|_{L^\infty_TL^\infty})^{m+3}.
		\end{align*}
        Denote 
        \begin{align*}
        T_1=&\sup\left\{T: \|X\|_{\mathcal{G}_T}\leq 10C \|(\mathrm{Id}-\mathcal{P})X(T'_0)\|_{\dot W^{1,\infty}}, \mathbf{\Theta}_X(T)\leq 10C\mathbf{\Theta}_X(T'_0),\right. \\
        &\quad\quad\quad\quad\left.\|\partial_xZ\|_{L^\infty_TL^\infty}\leq 10 C\|\partial_xZ(T'_0)\|_{L^\infty}\right\}.
        \end{align*}
		By standard bootstrap argument, there exists $\varepsilon_1>0$ such that if $\|(\mathrm{Id}-\mathcal{P})X(T'_0)\|_{\dot W^{1,\infty}}<\varepsilon_1$, then $T_1=\infty$.
		Combining this with \eqref{lipbes}, we obtain the global existence of solution with initial data satisfying \eqref{smbes} with $\varepsilon_0\leq (1+C)^{-1}(T'_0)^\eta \varepsilon_1$. Moreover, the global solution converges
exponentially in time to a stationary circle solution. Specifically, the deviation $Y(t,x)=(\mathrm{Id}-\mathcal{P})X(t,x)$
		has the following decay estimate:  
		\begin{align*}
			&\|Y(t)\|_{L^\infty}+\|Y(t)\|_{\dot W^{1,\infty}}+t^{m+\kappa}\|Y(t)\|_{\dot C^{m+1+\kappa}}\\&\quad\quad\quad\leq C'(1+\|\partial_xX_0\|_{L^\infty})^{10m}\|(\mathrm{Id}-\mathcal{P})X_0\|_{\dot B^1_{\infty,\infty}} e^{-\frac{c_0t}{8}},
		\end{align*}
		for all $t>0.$
		\section{The Peskin problem in 3D}\label{secpes3d}
		Due to the non-trivial geometry in the 3D setting, we need to work in local charts. This makes the problem more complicated in view of the non-local character of the equation. However, we mention that the essential structure of the 3D problem is the same as that of the 2D problem.
		
		\subsection{Reformulation}
		In  this section, we briefly denote 
		$$\widetilde\nabla=\nabla_{\mathbb{S}^2},\quad \mathbf{T}(\lambda)= \frac{\mathcal{T}\left(\lambda\right)}{\lambda}.$$  
		
		By the standard stereographic projection $\mathcal{X}:\mathbb{R}^2\rightarrow\mathbb{S}^2$, see \cite[Definition 3.9]{3Dpeskin}, we can  
		transform the equation from $\mathbb{S}^2$ to $\mathbb{R}^2$. Define
		\begin{equation*}
			\mathcal{X}(\theta)=\left(\frac{2\theta_1}{1+|\theta|^2},\frac{2\theta_2}{1+|\theta|^2},\frac{-1+|\theta|^2}{1+|\theta|^2} \right).
		\end{equation*}
		We can see that for any compact set $K\Subset\mathbb{S}^2$ with $(0,0,1)\notin K$, $\mathcal{X}$ is a homeomorphism from $\mathcal{X}^{-1}(K)$ to $K$. By local chart, for $F\in C^1(\mathbb{S}^2)$ with $ (0,0,1)\notin \text{Supp}(F)$ and $\widehat{\mathbf{\boldsymbol{x}}}=\mathcal{X}(\theta)$, we can define $\tilde\nabla F$ as 
		\begin{equation*}
			\widetilde\nabla F(\widehat{\mathbf{\boldsymbol{x}}})=\left(\frac{1+|\theta|^2}{2}\right)^2\sum_i\frac{\partial(F\circ\mathcal{X})}{\partial\theta_i}\frac{\partial\mathcal{X}}{\partial\theta_i}(\theta),
		\end{equation*}
		which naturally gives
		\begin{equation*}
		|\widetilde\nabla F|(\widehat{\mathbf{\boldsymbol{x}}})=\frac{1+|\theta|^2}{2}\left|\nabla(F\circ\mathcal{X})\right|(\theta).
		\end{equation*}
		Throughout this section, we denote variables on $\mathbb{S}^2$ by $\widehat{\boldsymbol{x}},\widehat{\boldsymbol{y}}$, and their counterparts on $\mathbb{R}^2$ by $\theta,\eta$.  With a slight abuse of notation, $L^\infty$ will refer to either $L^\infty(\mathbb{S}^2)$ or $L^\infty(\mathbb{R}^2)$, depending on the domain of the variables under consideration. The same convention applies to $C^{\kappa}$ spaces.

		Denote $\widehat{\boldsymbol{x}}_*=(0,0,-1)$. For any $\widehat{\boldsymbol{x}}_0\in\mathbb{S}^2$, we define the rotation $\mathcal R_{\widehat{\boldsymbol{x}}_0}:\mathbb{S}^2\rightarrow\mathbb{S}^2$ such that $\mathcal R_{\widehat{\boldsymbol{x}}_0}(\widehat{\boldsymbol{x}}_*)=\widehat{\boldsymbol{x}}_0$, and define the smooth cut-off functions $0\leq\chi_{\widehat{\boldsymbol{x}}_0},\tilde{\chi}_{\widehat{\boldsymbol{x}}_0}\leq 1$ such that $\mathrm{Supp}(\chi_{\widehat{\boldsymbol{x}}_0})\subset \tilde\chi_{\widehat{\boldsymbol{x}}_0}^{-1}(1)$, and
		\begin{equation}\label{cutoffpara}
			\begin{aligned}
				&\text{Supp}(\chi_{\widehat{\boldsymbol{x}}_0}\circ\mathcal{R}_{\widehat{\boldsymbol{x}}_0})\subset B(\widehat{\boldsymbol{x}}_*,a_2)\cap\mathbb{S}^2,\quad \chi_{\widehat{\boldsymbol{x}}_0}\circ\mathcal{R}_{\widehat{\boldsymbol{x}}_0}(B(\widehat{\boldsymbol{x}}_*,a_1)\cap\mathbb{S}^2)=1,\\
				&\text{Supp}(\tilde\chi_{\widehat{\boldsymbol{x}}_0}\circ\mathcal{R}_{\widehat{\boldsymbol{x}}_0})\subset B(\widehat{\boldsymbol{x}}_*,a_4)\cap\mathbb{S}^2,\quad \tilde \chi_{\widehat{\boldsymbol{x}}_0}\circ\mathcal{R}_{x_0}(B(\widehat{\boldsymbol{x}}_*,a_3)\cap\mathbb{S}^2)=1,\\
				&B(\widehat{\boldsymbol{x}}_*,a_1)\cap\mathbb{S}^2\subset \mathcal{X}(B(0,R_0))\subset \mathcal{X}(B(0,2R_0))\subset B(\widehat{\boldsymbol{x}}_*,a_2)\cap\mathbb{S}^2,\\
				&B(\widehat{\boldsymbol{x}}_*,a_3)\cap\mathbb{S}^2\subset \mathcal{X}(B(0,3R_0))\subset \mathcal{X}(B(0,4R_0))\subset B(\widehat{\boldsymbol{x}}_*,a_4)\cap\mathbb{S}^2,
			\end{aligned}
		\end{equation}
		where $0<a_1<a_2<a_3<a_4$, $B(a,r)$ denotes the ball with center $a$ and radius $r$ on $\mathbb{R}^2$ or $\mathbb{R}^3$, and $R_0$ is a fixed constant in this section. We take finite $\{\widehat{\boldsymbol{x}}_i\}_{i=1}^n\subset \mathbb{S}^2$ and $\{\chi_{i}\}_{i=1}^n=\{\chi_{\widehat{\boldsymbol{x}}_i}\}_{i=1}^n$ such that $\cup_{i=1}^n\chi_{i}^{-1}(1)$ is a cover of $\mathbb{S}^2$, and denote $\mathcal{R}_i=\mathcal{R}_{\widehat{\boldsymbol{x}}_i}$. 
		
		Define the H\"{o}lder norm on $\mathbb{S}^2$ as 
		\begin{equation}\label{normsph}
			\begin{aligned}
				\|f\|_{ C^\alpha(\mathbb{S}^2)}:=\sum_{i=1}^n\|(f\chi_{i})\circ\mathcal{R}_{i}\circ\mathcal{X}\|_{ C^{\alpha}(\mathbb{R}^2)}.
			\end{aligned}
		\end{equation}
		Note that the definition is equivalent to the normal definition of inhomogeneous H\"{o}lder norm on sphere.

		Let $\Phi\in C^\infty(\mathbb{S}^2)$ close to $F_0$ in the sense of $C^1$ that will be fixed later, we rewrite the equation of $(F-\Phi)$ as
        \begin{equation}\label{3dpesb}
		\begin{aligned}
			&\partial_t(F-\Phi)(t,\widehat{\boldsymbol{x}})+\mathcal{L}_{\Phi}(F-\Phi)(t,\widehat{\boldsymbol{x}})=N(F,\Phi)(t,\widehat{\boldsymbol{x}}),\\
            &(F-\Phi)|_{t=0}=X_0-\Phi
		\end{aligned}
        \end{equation}
        We give the formula of the pseudo-differential operator $\mathcal{L}_{\Phi}$ and the nonlinear terms $N(F,\Phi)$ as the following.\\
		{\bf The Pseudo-differential operator.} The operator $\mathcal{L}_{\Phi}$  is defined by 
		\begin{equation}\label{3dpesdefl}
			\mathcal{L}_{\Phi}H(\hat{x})=\int_{\mathbb{S}^2}G(\Phi(\widehat{\boldsymbol{x}})-\Phi(\widehat{\boldsymbol{y}}))\widetilde\nabla_{\widehat{\boldsymbol{y}}}\cdot(J(\widetilde\nabla\Phi) \widetilde\nabla H)(\widehat{\boldsymbol{y}})d \mu_{\mathbb{S}^2}(\widehat{\boldsymbol{y}}),
		\end{equation}
		where $G$ is defined in \eqref{3dpesdefG}, and $J=J_1+J_2$ is a $3^{\otimes 4}$ tensor, which is defined by Einstein notation as 
		\begin{equation}\label{3dpesdefJ}
			\begin{aligned}
				&(J_1(\widetilde{\nabla} \Phi))^{ij}_{kl}=\mathbf{T}(|\widetilde{\nabla} \Phi|) \delta_{ki}\delta_{lj},\\
				&(J_2(\widetilde{\nabla} \Phi))^{ij}_{kl}=\mathbf{T}'(|\widetilde{\nabla} \Phi|) \frac{(\widetilde{\nabla} \Phi^k)_{l} (\widetilde{\nabla} \Phi^i)_{j}}{|\widetilde{\nabla} \Phi|},
			\end{aligned}
		\end{equation}
		where $\delta_{ij}=1$ if $i=j$ and $\delta_{ij}=0$ otherwise. Note that $(J_\sigma(\widetilde\nabla\Phi)\widetilde\nabla(F-\Phi))_{ij}=\sum_{k,l}J_\sigma(\widetilde\nabla\Phi)^{ij}_{kl}(\widetilde\nabla(F-\Phi)^k)_l$ for $\sigma=1,2$.\\
		{\bf Nonlinear terms.} The nonlinear term $N(F,\Phi)$ can be decomposed as $$N(F,\Phi)=(N_1+N_2+N_3)(F,\Phi)$$ with the following definitions
		\begin{equation}\label{3dpesnlt}
			\begin{aligned}
				&N_1(F,\Phi)(t,\widehat{\boldsymbol{x}})=-\int_{\mathbb{S}^2}\big(G(F(t,\widehat{\boldsymbol{x}})-F(t,\widehat{\boldsymbol{y}}))-G(\Phi(\widehat{\boldsymbol{x}})-\Phi(\widehat{\boldsymbol{y}}))\big)\widetilde\nabla\cdot(\mathbf{T}(|\widetilde\nabla F|) \widetilde\nabla F)(t,\widehat{\boldsymbol{y}})d \mu_{\mathbb{S}^2}(\widehat{\boldsymbol{y}}),\\
				&N_2(F,\Phi)(t,\widehat{\boldsymbol{x}})=-\int_{\mathbb{S}^2}G(\Phi(\widehat{\boldsymbol{x}})-\Phi(\widehat{\boldsymbol{y}}))\widetilde\nabla\cdot\big(\mathbf{T}(|\widetilde\nabla F|) \widetilde\nabla F-\mathbf{T}(|\widetilde\nabla\Phi|) \widetilde\nabla\Phi-J(\widetilde\nabla\Phi) \widetilde\nabla(F-\Phi)\big)(t,\widehat{\boldsymbol{y}})d \mu_{\mathbb{S}^2}(\widehat{\boldsymbol{y}}),\\
				&N_3(\Phi)(\widehat{\boldsymbol{x}})=-\int_{\mathbb{S}^2}G(\Phi(\widehat{\boldsymbol{x}})-\Phi(\widehat{\boldsymbol{y}}))\widetilde\nabla\cdot(\mathbf{T}(|\widetilde\nabla\Phi|) \widetilde\nabla\Phi)(\widehat{\boldsymbol{y}})d \mu_{\mathbb{S}^2}(\widehat{\boldsymbol{y}}).
			\end{aligned}
		\end{equation}
        Then \eqref{3dpesb} follows from definition above.
		Now we verify conditions \eqref{mfdcdl} and \eqref{mtmfcd1} for the operator $\mathcal{L}_{\Phi}$.
		\begin{lemma}\label{3dpesmlem}
			Let $\mathcal{L}_\Phi$ be the operator defined by \eqref{3dpesdefl}, with $\Phi$ satisfying \eqref{con3dpes}. There exists $\mathcal{L}_{\mathbb{R}^2}^i$ satisfying \eqref{defop} and \eqref{condop}, such that for any $l\in\mathbb{N}$,
			\begin{equation}\label{3dpesml}
				\|(\chi_i\mathcal{L}_{\Phi}H)\circ\mathcal{R}_i\circ\mathcal{X}-\mathcal{L}^i_{\mathbb{R}^2}((\chi_iH)\circ\mathcal{R}_i\circ\mathcal{X})\|_{\dot C^{l+\kappa}}\lesssim \|H\|_{C^{l+\kappa+1-\zeta_0}},
			\end{equation}
			for some $\zeta_0\in(0,1)$. 
		\end{lemma}
		\begin{proof}
       We divide the proof into two steps. In the first step, we reformulate $((\chi_i \mathcal{L}_{\Phi} H)\circ \mathcal{R}_i \circ \mathcal{X})$ and express it as the sum of the principal part $\mathcal{L}^i_{\mathbb{R}^2}((\chi_i H)\circ \mathcal{R}_i \circ \mathcal{X})$ and several error terms, where the principal operator satisfies \eqref{defop} and \eqref{condop}. In the second step, we show that the error terms satisfy \eqref{3dpesml}.\\
        \textbf{Step 1: Reformulation.}\\
			Recall that the initial data $X_0$ satisfies
		\begin{equation}\label{3dpesnc}
			\mathbf \Theta_0:=\sup_{\widehat{\boldsymbol{x}},\widehat{\boldsymbol{y}}\in\mathbb{S}^2}\frac{|\widehat{\boldsymbol{x}}-\widehat{\boldsymbol{y}}|}{|X_0(\widehat{\boldsymbol{x}})-X_0(\widehat{\boldsymbol{y}})|}<\infty.
		\end{equation}
		Since we choose $\|X_0-\Phi\|_{C^1(\mathbb{S}^2)}\leq \eps_0$ by \eqref{con3dpes}, by \eqref{3dpesnc} we obtain
			\begin{equation}\label{3dpesPc}
				\inf_{\widehat{\boldsymbol{x}},\widehat{\boldsymbol{y}}\in\mathbb{S}^2}\frac{|\Phi(\widehat{\boldsymbol{x}})-\Phi(\widehat{\boldsymbol{y}})|}{|\widehat{\boldsymbol{x}}-\widehat{\boldsymbol{y}}|}\geq \mathbf{\Theta}_0^{-1}-2C\|X_0-\Phi\|_{C^1(\mathbb{S}^2)}\geq\frac{4\mathbf{\Theta}_0^{-1}}{5},
			\end{equation}
			provided $\eps_0<\frac{1}{100(C+10)\mathbf{\Theta}_0}$. 
			Recall that we can split the integration region of \eqref{3dpesdefl} into the union of finite balls $\cup_{i=1}^nB(\widehat{\boldsymbol{x}}_i,a_1)$. Consider a small neighborhood of $\widehat{\boldsymbol{x}}_i$, and define
			\begin{equation*}
				\mathcal{L}_{\Phi,i}H=\int_{\mathbb{S}^2\cap B(\hat{\boldsymbol{x}}_i,a_2)}\widetilde\nabla_{\widehat{\boldsymbol{y}}}G(\Phi(\widehat{\boldsymbol{x}})-\Phi(\widehat{\boldsymbol{y}}))\cdot(J(\widetilde\nabla\Phi) \widetilde\nabla H)(t,\widehat{\boldsymbol{y}})d \mu_{\mathbb{S}^2}(\widehat{\boldsymbol{y}}).
			\end{equation*}
			Denote $H_i=H\chi_i$, $h_i=H_i\circ\mathcal{R}_i\circ\mathcal{X}$, $\phi_i=(\tilde{\chi}_i\Phi)\circ\mathcal{R}_i\circ\mathcal{X}$. Now we approximate $\chi_i(\mathcal{L}_\Phi H)$ by $\tilde\chi_i\mathcal{L}_{\Phi,i}H_i$, and denote the error term as
			\begin{equation}\label{apN4}
				N_4=\chi_i(\mathcal{L}_\Phi H)-\tilde\chi_i\mathcal{L}_{\Phi,i}H_i,\quad \tilde{N}_4=N_4\circ\mathcal{R}_i\circ\mathcal{X}.
			\end{equation}
			We consider $\tilde\chi_i\mathcal{L}_{\Phi,i}H_i$ in the following. We consider $\left(\widetilde\nabla G\cdot(J(\widetilde\nabla\Phi)\widetilde\nabla(H\chi_i))\right)\circ\mathcal{R}_i\circ\mathcal{X}$. By changing variable formula, generally for any $F:\mathbb{S}^2\rightarrow \mathbb{R}^3$, and $f=F\circ\mathcal{X}:\mathbb{R}^2\rightarrow\mathbb{R}^3$, we have
			\begin{equation}\label{3dpescvar}
				\int_{\mathbb{S}^2}F(\widehat{\boldsymbol{y}})d \mu_{\mathbb{S}^2}(\widehat{\boldsymbol{y}})=\int_{\mathbb{R}^2}\left(\frac{2}{1+|\eta|^2}\right)^2f(\eta)d\eta.
			\end{equation}
			And by classical calculus, for any $F,G:\mathbb{S}^2\rightarrow\mathbb{R}$,
			\begin{equation}\label{3dpescvar2}
				\begin{aligned}
					(\tilde{\nabla}F\cdot\tilde{\nabla }G)\circ \mathcal{R}_i\circ\mathcal{X}&=\left(\frac{1+|\theta|^2}{2}\right)^4\sum_{j,l,m}\frac{\partial(F\circ\mathcal{R}_i\circ\mathcal{X})}{\partial\theta_l}\frac{\partial\mathcal{X}^j}{\partial\theta_l}\frac{\partial(G\circ\mathcal{R}_i\circ\mathcal{X})}{\partial\theta_m}\frac{\partial\mathcal{X}^j}{\partial\theta_m}\\
					&=\left(\frac{1+|\theta|^2}{2}\right)^2\nabla(F\circ\mathcal{R}_i\circ\mathcal{X})\cdot\nabla(G\circ\mathcal{R}_i\circ\mathcal{X}).
				\end{aligned}
			\end{equation}
			Define 
			\begin{equation*}\label{def3dpestilJ}
				\begin{aligned}
					\tilde{\mathbf{T}}(\theta,\beta)=\mathbf{T}(\rho(\theta)\beta)=\frac{\mathcal{T}(\rho(\theta)\beta)}{\rho(\theta)\beta},\ \ \ \theta \in\mathbb{R}^2, \beta\in\mathbb{R}.
				\end{aligned}
			\end{equation*}
			for some smooth $\rho$ bounded from below and above, such that $\rho(\theta)=\frac{1+|\theta|^2}{2}$ for $|\theta|\leq 4R_0$ with $R_0$ in \eqref{cutoffpara}. Moreover, corresponding to the tensor $J$ in \eqref{3dpesdefJ}, we define the $3\times3\times 2\times 2$ tensor $\tilde{J}(\eta,\nabla\phi_i)$ as 
			\begin{equation*}
				\begin{aligned}
					&\left(\tilde{J}(\eta,\nabla\phi_i)\right)_{kn}^{jl}=\sum_{\sigma=1,2}\left(\tilde{J}_\sigma(\eta,\nabla\phi_i(\eta))\right)_{kn}^{jl},\quad l,n=1,2,\quad j,k=1,2,3,\\
					&\left(\tilde{J}_1(\eta,\nabla\phi_i(\eta))\right)_{kn}^{jl}=\tilde{\mathbf{T}}(\eta,\nabla\phi_i(\eta))\delta_{ln}\delta_{jk},\\
                    &\left(\tilde{J}_2(\eta,\nabla\phi_i(\eta))\right)_{kn}^{jl}=(\partial_\beta\tilde{\mathbf{T}}(\eta,\beta))\Big|_{\beta=|\nabla\phi_i(\eta)|}\frac{(\nabla\phi_i(\eta))_{jl}(\nabla\phi_i(\eta))_{kn}}{|\nabla\phi_i(\eta)|}.
				\end{aligned}
			\end{equation*}
			By \eqref{3dpescvar} and \eqref{3dpescvar2} we can write 
			\begin{equation*}
				\begin{aligned}
					&\left(\int_{\mathbb{S}^2}\widetilde\nabla_{\widehat{\boldsymbol{y}}} G(\Phi_i(\widehat{\boldsymbol{x}})-\Phi_i(\widehat{\boldsymbol{y}}))\cdot(J((\widetilde\nabla\Phi))\widetilde\nabla H_i)(t,\widehat{\boldsymbol{y}})d\mu_{\mathbb{S}^2}(\widehat{\boldsymbol{y}})\right)\circ\mathcal{R}_{i}\circ\mathcal{X}\\
                    &\quad\quad\quad\quad=\int_{\mathbb{R}^2} G(\phi_i(\theta)-\phi_i(\eta))\nabla\cdot \left(\tilde{J}(\eta,\nabla\phi_i)\nabla h_i(t,\eta)\right)d\eta.
				\end{aligned}
			\end{equation*}
			We approximate $G(\phi_i(\theta)-\phi_i(\eta))$ by $G(\nabla\phi_i(\theta)(\theta-\eta))$, so we have
			\begin{equation}\label{3dpesdefn5}
				\begin{aligned}
					&\left(\tilde\chi_i\mathcal{L}_{\Phi,i}H_i\right)\circ\mathcal{R}_i\circ\mathcal{X}(t,\theta)=\left(\tilde\chi_i\circ\mathcal{R}_i\circ\mathcal{X}\right)(\theta)\int_{\mathbb{R}^2} G(\phi_i(\theta)-\phi_i(\eta))\nabla\cdot \left(\tilde{J}(\eta,\nabla\phi_i)\nabla h_i(t,\eta)\right)d\eta\\
					&\quad=\left(\tilde\chi_i\circ\mathcal{R}_i\circ\mathcal{X}\right)(\theta)\int_{\mathbb{R}^2} G(\nabla\phi_i(\theta)(\theta-\eta))\nabla\cdot \left(\tilde{J}(\eta,\nabla\phi_i)\nabla h_i(t,\eta)\right)d\eta+\tilde{N}_5(t,\theta),
				\end{aligned}
			\end{equation}
			with error term $\tilde{N}_5$ defined as
			\begin{equation*}
				\tilde{N}_5(t,\theta)=(\tilde\chi_i\circ\mathcal{R}_i\circ\mathcal{X})(\theta)\int_{\mathbb{R}^2}\left(G(\nabla\phi_i(\theta)(\theta-\eta))-G(\phi_i(\theta)-\phi_i(\eta))\right)\nabla\cdot(\tilde{J}(\eta,\nabla\phi_i)\nabla h_i)(t,\eta)d\eta.
			\end{equation*}
			We proceed by analyzing the first term on the right-hand side of equation \eqref{3dpesdefn5} using the methodology established in \cite{3Dpeskin}. For $G(\nabla\phi_i(\theta)(\theta-\eta))$, denote
			\begin{align*}
				A(\theta)=\nabla\phi_i(\theta)\in \mathbb{M}_{3\times 2}(\mathbb{R}^2),\quad B(\theta)=\sqrt{A^\top(\theta)A(\theta)}\in \mathbb{M}_{2\times 2}(\mathbb{R}^2),\quad Q(\theta)=A(\theta)B^{-1}(\theta)\in \mathbb{M}_{3\times 2}(\mathbb{R}^2).
			\end{align*}
			where $\mathbb{M}_{m\times n}$ denotes the set of $m\times n$ matrices. By fundamental calculus, we know that 
			\begin{equation*}
				\begin{aligned}
					&\left|A\eta \right|^2=(A\eta)^\top A\eta=\eta^\top (A^\top A)\eta=\left|B\eta \right|^2,\\
					&(A\eta)\otimes (A\eta)=(A\eta)(A\eta)^\top=A\eta\eta^\top A^\top =QB\eta\eta^\top B^\top Q^\top =Q(B\eta)\otimes(B\eta)Q^\top,
				\end{aligned}
			\end{equation*}
			and
			\begin{equation*}
				\mathcal{F}\left(\frac{1}{|\theta|}\right)(\xi)=\frac{2\pi}{|\xi|},\quad\mathcal{F}\left(\frac{\theta_i\theta_j}{|\theta|^3}\right)(\xi)=2\pi\left(\frac{\delta_{ij}}{|\xi|}-\frac{\xi_i\xi_j}{|\xi|^3}\right).
			\end{equation*}
			By changing variable, for any fixed $\theta$, we write $\mathcal{F}(G(\nabla\phi_i(\theta)\cdot))(\xi)$ as
			\begin{equation}\label{3dpesopG}
				\begin{aligned}
					\frac{2\pi\mathrm{Id}_3}{|\det B(\theta)||B^{-1}\xi|}+\frac{2\pi}{|\det B(\theta)|} Q\left(\frac{\mathrm{Id}_2}{|B^{-1}\xi|}-\frac{(B^{-1}\xi)\otimes( B^{-1}\xi)}{|B^{-1}\xi|^3}\right) Q^\top=2\pi\frac{\mathrm{Id}_3+v(\theta,\xi)\otimes v(\theta,\xi)}{|\det B(\theta)||B^{-1}\xi|}
				\end{aligned}
			\end{equation}
			with $v(\theta,\xi)=Q\begin{pmatrix}
				0&1\\
				-1&0
			\end{pmatrix}\frac{B^{-1}\xi}{|B^{-1}\xi|}$. By \eqref{3dpesPc}, $A$ is non-degenerate and $B$ is positive definite, with $$0<(1+{\mathbf{\Theta}_0})C^{-1}\leq|\det B|^{\frac{1}{2}}\leq C\|\Phi\|_{W^{1,\infty}},$$ where $C$ is a universal constant. Consequently, the matrix defined in \eqref{3dpesopG} is positive definite. More precisely, denote $C_{B,\xi}=\frac{1}{\det(B)|B^{-1}\xi|}$, then $C_{B,\xi}|\xi|$ is uniformly bounded from above and below, and the eigenvalues of the matrix in \eqref{3dpesopG} are positive.\\
            We define $\tilde \phi_i\in C^\infty(\mathbb{R}^2;\mathbb{R}^{3\times 2})$ as an extension of $\nabla\phi_i$ bounded up and below. Precisely, it satisfies
			\begin{align*}
				&\tilde \phi_i(\theta)\equiv \nabla\phi_i(\theta),\quad \text{for}\ |\theta|\leq 3R_0,\\
				&0<C_1<|\tilde \phi_i|<C_2<\infty,\\
				&\| {\tilde{\phi}_i}\|_{C^{m+2}}\leq C_3\|\Phi\|_{C^{m+2}}<\infty.
			\end{align*}
			Based on \eqref{3dpesopG} and following the calculations in Section 4 of \cite{3Dpeskin}, we define
			\begin{equation}
            \label{3dpesdefop}
		\begin{aligned}
		    \mathcal{L}_{\mathbb{R}^2}^if(t,\theta)
            &=\int_{\mathbb{R}^2} G(\tilde\phi_i(\theta)(\theta-\eta))\nabla\cdot \left(\tilde{J}(\theta,\tilde \phi_i)\nabla f(t,\eta)\right)d\eta
           \\&=\int_{\mathbb{R}^2}\frac{(\mathrm{Id}_3+\tilde v(\theta,\xi)\otimes \tilde v(\theta,\xi))\big(\mathcal{ J}(\tilde\phi_i)(\theta,\xi)\big)}{|\det B(\theta)||B^{-1}\xi|}\hat{f}(t,\xi)e^{i\theta\cdot\xi}d\xi,
		\end{aligned}		
			\end{equation}
			with
			\begin{equation}\label{deftJ}
				\begin{aligned}
					&\mathcal{J}(\tilde\phi_i)(\theta,\xi)=\frac{\mathcal{T}(\rho(\theta)|\tilde A|)}{\rho(\theta)|\tilde A|}\left(|\xi|^2\mathrm{Id}_3-\frac{(\tilde A\cdot\xi)\otimes (\tilde A\cdot\xi)}{|\tilde A|^2}\right) +\mathcal{T}'(\rho(\theta)|\tilde A|)\frac{(\tilde A\cdot\xi)\otimes (\tilde A\cdot\xi)}{|\tilde A|^2},\\     
					&\tilde{A}=\tilde A(\theta)=\tilde{\phi}_i(\theta),\quad \tilde B(\theta)=\sqrt{\tilde A^\top(\theta)\tilde A(\theta)},\quad \tilde Q(\theta)=\tilde A(\theta)\tilde B^{-1}(\theta), \quad \tilde v(\theta,\xi)=\tilde Q(\theta)\begin{pmatrix}
						0&1\\
						-1&0
					\end{pmatrix}\frac{\tilde B^{-1}\xi}{|\tilde B^{-1}\xi|}.
				\end{aligned}
			\end{equation}
			By invoking the result from \cite{3Dpeskin} and using the fact that $\tilde A$ has full rank while $\tilde B$ is positive definite, along with condition \eqref{conten} and positivity of \eqref{3dpesopG}, we deduce that the operator in \eqref{3dpesdefop} satisfies the lower bound $C\mathfrak{c}(R)|\xi|$ where $C>0$, $R = 2\|X_0\|_{W^{1,\infty}}$ and $\mathfrak{c}_R$ is defined in \eqref{conten}. Furthermore, by the smoothness of $\tilde\phi_i$, an upper bound for the derivatives is easily verified, revealing \eqref{defop} holds. We refer to Section 4, \cite{3Dpeskin} for the proof of coercive condition \eqref{condop} of the operator $\mathcal{L}_{\mathbb{R}^2}^i$. So we have
			\begin{equation}\label{ap3}
				\left(\tilde\chi_i\circ\mathcal{R}_i\circ\mathcal{X}\right)(\theta)\int_{\mathbb{R}^2} G(\nabla\phi_i(\theta)(\theta-\eta))\nabla\cdot \left(\tilde{J}(\eta,\nabla\phi_i)\nabla h_i(t,\eta)\right)d\eta=\mathcal{L}_{\mathbb{R}^2}^if(t,\theta)+\tilde{N}_6(t,\theta),
			\end{equation}
			with the dominate operator $\mathcal{L}_{\mathbb{R}^2}^i$ satisfying \eqref{defop} and \eqref{condop}, and the error term $\tilde N_6$
			\begin{equation*}
				\tilde{N}_6(t,\theta)=(\tilde\chi_i\circ\mathcal{R}_i\circ\mathcal{X})(\theta)\int_{\mathbb{R}^2}G(\nabla\phi_i(\theta)(\theta-\eta))\nabla\cdot(\tilde{J}(\eta,\nabla\phi_i)\nabla h_i)(t,\eta)d\eta-\mathcal{L}_{\mathbb{R}^2}^ih_i(t,\theta).
			\end{equation*}
			Combining \eqref{ap3} with \eqref{3dpesdefn5} and \eqref{apN4}, we obtain 
			\begin{align*}
				(\chi_i\mathcal{L}_{\Phi}H)\circ\mathcal{R}_i\circ\mathcal{X}-\mathcal{L}^i_{\mathbb{R}^2}((\chi_iH)\circ\mathcal{R}_i\circ\mathcal{X})=\tilde N_4+\tilde N_5+\tilde N_6.
			\end{align*}
            \textbf{Step 2: Error estimates.}\\
			To prove  \eqref{3dpesml}, it suffices to prove 
			\begin{equation}\label{NH}
				\sum_{j=4,5,6}\|\tilde{N}_j(t)\|_{ C^{l+\kappa}}\lesssim\|H(t)\|_{C^{l+\kappa+1-\zeta_0}},
			\end{equation}
			for some small $\zeta_0\in(0,1)$ and any $l\in\mathbb{N}$. In the following of the proof, we will shortly denote 
			\begin{equation}\label{Pesconphi}
				M_{\Phi}:=(1+\mathbf{\Theta}_0)\|\Phi\|_{C^{10m}},\quad m_{\Phi}:=\|\Phi\|_{C^1}.
			\end{equation}
			By Lemma \ref{maininterpo}, we only need to estimate $\|\tilde N_j\|_{C^{k}}$ for any $k\leq m+1$.\\
            For $\tilde{N}_4$, by definition of H\"{o}lder norm on $\mathbb{S}^2$ and interpolation, we can write
			\begin{equation*}
				\begin{aligned}
					N_4(t,\widehat{\boldsymbol{x}})&=\tilde{\chi}_i(\widehat{\boldsymbol{x}})\int_{\mathbb{S}^2}(\chi_i(\widehat{\boldsymbol{x}})-\chi_i(\widehat{\boldsymbol{y}}))\widetilde\nabla_{\widehat{\boldsymbol{y}}} (G(\Phi(\widehat{\boldsymbol{x}})-\Phi(\widehat{\boldsymbol{y}})))\cdot(J(\widetilde\nabla\Phi) \widetilde\nabla H)(t,\widehat{\boldsymbol{y}})d \mu_{\mathbb{S}^2}(\widehat{\boldsymbol{y}})\\
					&\quad\quad\quad\quad\quad-\tilde{\chi}_i(\widehat{\boldsymbol{x}})\int_{\mathbb{S}^2}\widetilde\nabla_{\widehat{\boldsymbol{y}}}  G(\Phi(\widehat{\boldsymbol{x}})-\Phi(\widehat{\boldsymbol{y}}))\cdot(J(\widetilde\nabla \Phi) (\widetilde\nabla\chi_i\otimes H))(t,\widehat{\boldsymbol{y}})d \mu_{\mathbb{S}^2}(\widehat{\boldsymbol{y}})\\
					&:=N_{41}(t,\widehat{\boldsymbol{x}})+N_{42}(t,\widehat{\boldsymbol{x}}).
				\end{aligned}
			\end{equation*}
			By \eqref{3dpesPc}, if we define $r=|\widehat{\boldsymbol{x}}-\widehat{\boldsymbol{y}}|$ and the kernel $	\mathbf{G}_{\Phi}(\widehat{\boldsymbol{x}}, \widehat{\boldsymbol{y}})=G(\Phi(\widehat{\boldsymbol{x}})-\Phi(\widehat{\boldsymbol{y}}))$, then
			\begin{equation*}
				r |(\widetilde\nabla_{\widehat{\boldsymbol{x}}}+\widetilde\nabla_{\widehat{\boldsymbol{y}}}) \mathbf{G}_{\Phi}(\widehat{\boldsymbol{x}}, \widehat{\boldsymbol{y}})|+r^2|\widetilde\nabla_{\widehat{\boldsymbol{x}},\widehat{\boldsymbol{y}}} \mathbf{G}_{\Phi}(\widehat{\boldsymbol{x}}, \widehat{\boldsymbol{y}})|+r^3|\widetilde\nabla_{\widehat{\boldsymbol{x}}}\widetilde\nabla_{\widehat{\boldsymbol{y}}} \mathbf{G}_{\Phi}(\widehat{\boldsymbol{x}}, \widehat{\boldsymbol{y}})|\lesssim M_{\Phi}(1+M_{\Phi})^2.
			\end{equation*}
			Integrating by parts, we obtain that
			\begin{equation*}
				\|N_4(t)\|_{L^\infty}\lesssim M_{\Phi}(1+M_{\Phi})\|H(t)\|_{C^1}.
			\end{equation*}
			To estimate the H\"{o}lder norms, we employ the binomial identity 
            \begin{equation*}
                a^k=(a+b-b)^k=\sum_{l=0}^k C(k,l)(a+b)^lb^{k-l},
            \end{equation*}
            where $C(k,l)=(-1)^{k-l}\frac{k!}{l!(k-l)!}$ is the refined binomial coefficient. This identity allows us to systematically transform derivatives from $\widehat{\boldsymbol{x}}$-coordinates to the $\widehat{\boldsymbol{y}}$-coordinates. Specifically, we apply this decomposition on both $\mathbb{S}^2$ and $\mathbb{R}^2$ to facilitate the norm estimates.
			\begin{equation}\label{3dpeskerd}
				\begin{aligned}
					&\widetilde\nabla_{\widehat{\boldsymbol{x}}}^k=\sum_{l=0}^kC(k,l)(\widetilde\nabla_{\widehat{\boldsymbol{x}}}+\widetilde\nabla_{\widehat{\boldsymbol{y}}})^l\tilde\nabla_{\widehat{\boldsymbol{y}}}^{k-l},\\
					&\nabla_{\theta}^k=\sum_{l=0}^kC(k,l)(\nabla_{\theta}+\nabla_{\eta})^l\nabla_{\eta}^{k-l}.
				\end{aligned}
			\end{equation}
			Then by using \eqref{3dpesPc}, one can see that for any $l\in\mathbb{N}$ and  $r=|\widehat{\boldsymbol{x}}-\widehat{\boldsymbol{y}}|$,
			\begin{align*}
				r|(\widetilde\nabla_{\widehat{\boldsymbol{x}}}+\widetilde\nabla_{\widehat{\boldsymbol{y}}})^l \mathbf{G}_{\Phi}(\widehat{\boldsymbol{x}}, \widehat{\boldsymbol{y}})|&+r^2|\widetilde\nabla_{\widehat{\boldsymbol{y}}}(\widetilde\nabla_{\widehat{\boldsymbol{x}}}+\widetilde\nabla_{\widehat{\boldsymbol{y}}})^l \mathbf{G}_{\Phi}(\widehat{\boldsymbol{x}}, \widehat{\boldsymbol{y}})|\\&+r^3|\widetilde\nabla_{\widehat{\boldsymbol{x}}}\widetilde\nabla_{\widehat{\boldsymbol{y}}}(\widetilde\nabla_{\widehat{\boldsymbol{x}}}+\widetilde\nabla_{\widehat{\boldsymbol{y}}})^l \mathbf{G}_{\Phi}(\widehat{\boldsymbol{x}}, \widehat{\boldsymbol{y}})|\lesssim M_{\Phi}(1+M_{\Phi})^{l+2}.
			\end{align*}
			Denote $\tilde{G}(\widehat{\boldsymbol{x}},\widehat{\boldsymbol{y}})=\left(\chi_i(\widehat{\boldsymbol{x}})-\chi_i(\widehat{\boldsymbol{y}})\right)\widetilde{\nabla}_{\widehat{\boldsymbol{y}}}G(\Phi(\widehat{\boldsymbol{x}})-\Phi(\widehat{\boldsymbol{y}}))$. For any $l\in\mathbb{N}$, we have 
			\begin{equation*}
				\begin{aligned}
					& r|(\widetilde\nabla_{\widehat{\boldsymbol{x}}}+\widetilde\nabla_{\widehat{\boldsymbol{y}}})^l\tilde{G}(\widehat{\boldsymbol{x}},\widehat{\boldsymbol{y}})|+r^2|\widetilde\nabla_{\widehat{\boldsymbol{y}}}(\widetilde\nabla_{\widehat{\boldsymbol{x}}}+\widetilde\nabla_{\widehat{\boldsymbol{y}}})^l\tilde{G}(\widehat{\boldsymbol{x}},\widehat{\boldsymbol{y}})|\lesssim M_{\Phi}(1+M_{\Phi})^{l+2},
				\end{aligned}
			\end{equation*}
			which together with \eqref{3dpeskerd} implies
			\begin{equation*}
				\begin{aligned}
					\|\widetilde\nabla^{k}N_{41}(t)\|_{L^\infty}
					&\lesssim \sum_{m_1\leq k}\Big\|\widetilde\nabla^{m_1}\tilde\chi_i(\widehat{\boldsymbol{x}})\int_{\mathbb{S}^2}\widetilde\nabla_{\widehat{\boldsymbol{x}}}^{k-m_1}(\tilde{G}(\widehat{\boldsymbol{x}},\widehat{\boldsymbol{y}}))\cdot(J(\widetilde\nabla\Phi) \widetilde\nabla H)(t,\widehat{\boldsymbol{y}})d \mu_{\mathbb{S}^2}(\widehat{\boldsymbol{y}})\Big\|_{L^\infty_{\widehat{\boldsymbol{x}}}}\\
					&\lesssim \sum_{m_1\leq k}\Big\|\widetilde\nabla^{m_1}\tilde\chi_i(\widehat{\boldsymbol{x}})\int_{\mathbb{S}^2}(\widetilde\nabla_{\widehat{\boldsymbol{x}}}+\widetilde\nabla_{\widehat{\boldsymbol{y}}})^{k-m_1}(\tilde{G}(\widehat{\boldsymbol{x}},\widehat{\boldsymbol{y}}))\cdot(J(\widetilde\nabla\Phi) \widetilde\nabla H)(t,\widehat{\boldsymbol{y}})d \mu_{\mathbb{S}^2}(\widehat{\boldsymbol{y}})\Big\|_{L^\infty_{\widehat{\boldsymbol{x}}}}\\
					&\quad+\sum_{m_1\leq k-1}\Big\|\widetilde\nabla^{m_1}\tilde\chi_i(\widehat{\boldsymbol{x}})\sum_{1\leq m_2\leq k-m_1}\int_{\mathbb{S}^2}\widetilde\nabla_{\widehat{\boldsymbol{y}}}(\widetilde\nabla_{\widehat{\boldsymbol{x}}}+\widetilde\nabla_{\widehat{\boldsymbol{y}}})^{k-m_1-m_2}(\tilde{G}(\widehat{\boldsymbol{x}},\widehat{\boldsymbol{y}}))\\
					&\quad\quad\times\left(\widetilde{\nabla}^{m_2-1}(J(\widetilde\nabla\Phi) \widetilde\nabla H)(t,\widehat{\boldsymbol{y}})-\widetilde{\nabla}^{m_2-1}(J(\widetilde\nabla\Phi) \widetilde\nabla H)(t,\widehat{\boldsymbol{x}})\right)d \mu_{\mathbb{S}^2}(\widehat{\boldsymbol{y}})\Big\|_{L^\infty_{\widehat{\boldsymbol{x}}}}\\
					&\lesssim\|\tilde\chi_i\|_{C^{k}}(1+\|\tilde\chi_i\|_{C^{k}})M_{\Phi}(1+M_{\Phi})^{k}\|H(t)\|_{C^{k+\eps}},
				\end{aligned}
			\end{equation*}
			for some $\eps\in(0,1)$. Note that the critical norm here should be $\|H(t)\|_{C^{k+1}}$, which implies that $N_{41}$ is a lower order term. For $\widetilde\nabla^kN_{42}$, we have
			\begin{equation*}
				\begin{aligned}
					\|\widetilde\nabla^kN_{42}(t)\|_{L^\infty}&\lesssim
					\sum_{\substack{m_1\leq k\\
							m_2\leq k-m_1}}\left\|\widetilde\nabla^{m_1}\tilde{\chi}_i(\widehat{\boldsymbol{x}})\int_{\mathbb{S}^2}(\widetilde\nabla_{ \widehat{\boldsymbol{x}}}+\widetilde\nabla_{ \widehat{\boldsymbol{y}}})^{k-m_1-m_2}\widetilde\nabla_{ \widehat{\boldsymbol{y}}}G(\Phi(\widehat{\boldsymbol{x}})-\Phi(\widehat{\boldsymbol{y}}))\right.\\
					&\quad\quad \times\left(\widetilde\nabla^{m_2}((J(\widetilde\nabla\Phi)( \widetilde\nabla\chi_i\otimes H)))(t,\widehat{\boldsymbol{y}})-\widetilde\nabla^{m_2}((J(\widetilde\nabla\Phi) (\widetilde\nabla\chi_i\otimes H)))(t,\widehat{\boldsymbol{x}})\right)d \mu_{\mathbb{S}^2}(\widehat{\boldsymbol{y}})\Big\|_{L^\infty}\\
					&\lesssim \|\tilde\chi_i\|_{C^{k+1}}(1+\|\tilde\chi_i\|_{C^{k+1}})M_{\Phi}(1+M_{\Phi})^{k+1}\|H(t)\|_{C^{k+\eps}}.
				\end{aligned}
			\end{equation*}
			From the above estimates, for any $0\leq k\leq m+1$, we obtain
			\begin{equation}\label{3dpesn4h}
				\begin{aligned}
					\|N_4(t)\|_{C^{k}}\lesssim M_{\Phi}(1+M_{\Phi})^{k+1}\|\tilde\chi_i\|_{C^{k+1}}(1+\|\tilde\chi_i\|_{C^{k+1}})\|H(t)\|_{C^{k+\eps}}.
				\end{aligned}
			\end{equation}
			Now we estimate $\tilde{N}_5$. For $L^\infty$ norm, it is straightforward to check that
			\begin{equation*}
				\|\tilde{N}_5(t)\|_{L^\infty}\lesssim M_{\Phi}\|\nabla h_i(t)\|_{L^\infty}.
			\end{equation*}
			For H\"{o}lder norms $\dot C^{k+1}$, we shortly denote $\chi_i^R(\theta)=(\chi_i\circ\mathcal{R}_i\circ\mathcal{X})(\theta)$, $\tilde\chi_i^R(\theta)=(\tilde\chi_i\circ\mathcal{R}_i\circ\mathcal{X})(\theta)$ and $\tilde{G}(\theta,\eta)=G(\phi_i(\theta)-\phi_i(\eta))-G(\nabla\phi_i(\theta)(\theta-\eta))$. By \eqref{3dpesPc}, we can see that 
			\begin{align*}
				|\tilde\chi_i^R(\theta)\tilde\chi_i^R(\eta)\nabla_{\theta}\nabla_{\eta}\tilde{G}(\theta,\eta)|\lesssim \frac{M_{\Phi}}{|\theta-\eta|^2}.
			\end{align*}
			Furthermore, we notice that for any $l\in\mathbb{N}$, 
			\begin{align*}
				|\tilde\chi_i^R(\theta)\tilde\chi_i^R(\eta)\nabla_{\theta}\nabla_{\eta}(\nabla_{\theta}+\nabla_{\eta})^l\tilde{G}(\theta,\eta)|\lesssim \frac{M_{\Phi}(1+M_{\Phi})^{l+1}}{|\theta-\eta|^2}.
			\end{align*}
			Applying \eqref{3dpeskerd} and integration by parts, similar to \eqref{3dpesn4h}, we deduce
			\begin{equation*}
				\begin{aligned}
					&\|\nabla^k\tilde{N}_5(t)\|_{L^\infty}\lesssim \|\tilde\chi_i\|_{C^{k+1}}(1+\|\tilde\chi_i\|_{C^{k+1}})^{k+1}M_{\Phi}\|h_i(t)\|_{C^1}\\
					&\quad+\sum_{\substack{2\leq l\leq k\\
							m\leq l-1}}\|\tilde\chi_i\|_{C^{k-l}}\Big\|\int_{\mathbb{R}^2}\nabla_{\theta}\nabla_\eta(\nabla_\theta+\nabla_\eta)^{l-m-1}\tilde{G}(\theta,\eta)\\
                            &\quad\quad\quad\quad\quad\quad\times\left(\nabla^{m}(\tilde{J}(\eta,\nabla\phi_i)\nabla h_i)(t,\theta)-\nabla^{m}(\tilde{J}(\eta,\nabla\phi_i)\nabla h_i)(t,\eta)\right)d\eta\Big\|_{L^\infty_\theta}\\
					&\quad+\Big\|\int_{\mathbb{R}^2}\nabla_{\theta}\nabla_\eta\tilde{G}(\theta,\eta)\cdot\left(\nabla^{k-1}(\tilde{J}(\eta,\nabla\phi_i)\nabla h_i)(t,\theta)-\nabla^{k-1}(\tilde{J}(\eta,\nabla\phi_i)\nabla h_i)(t,\eta)\right)\tilde\chi_i(\eta)d\eta\Big\|_{L^\infty_\theta}\\
					&\quad+\Big\|\int_{\mathbb{R}^2}\nabla_{\theta}\tilde{G}(\theta,\eta)\cdot\nabla_\eta\tilde\chi_i(\eta)\nabla^{k-1}(\tilde{J}(\eta,\nabla\phi_i)\nabla h_i)(t,\theta)d\eta\Big\|_{L^\infty_\theta}\\
					&\lesssim \|\tilde\chi_i\|_{C^{k+1}}(1+\|\tilde\chi_i\|_{C^{k+1}})^{k+1}M_{\Phi}(1+M_{\Phi})^{k+2}\|h_i(t)\|_{C^{k+\eps}}.
				\end{aligned}
			\end{equation*}
			The estimate above gives for any $0\leq k\leq m+1$, 
			\begin{equation}\label{3dpesn5h}
				\begin{aligned}
					\|\tilde{N}_5(t)\|_{ C^{k}}\lesssim \|\tilde\chi_i\|_{C^{k+1}}M_{\Phi}(1+M_{\Phi})^{k+1}\|h_i(t)\|_{C^{k+\eps}}.
				\end{aligned}
			\end{equation}
			We write $\tilde N_6$ as
			\begin{equation*}
				\begin{aligned}
					\tilde{N}_6(t,\theta)=&- \int_{\mathbb{R}^2}(G(\tilde \phi_i(\theta)(\theta-\eta)))\nabla\cdot\left(\big(\tilde{J}(\eta,\tilde\phi_i)-\tilde{J}(\theta,\tilde\phi_i)\big)\nabla h_i(t,\eta)\right)d\eta\\
					&-(1-\tilde{\chi}_i\circ\mathcal{R}_i\circ\mathcal{X})(\theta)\int_{\mathbb{R}^2}G(\tilde \phi_i(\theta)(\theta-\eta))\nabla\cdot\left(\tilde{J}(\eta,\tilde\phi_i)\nabla h_i(t,\eta)\right)d\eta\\
					=&\tilde{N}_{61}(t,\theta)+\tilde{N}_{62}(t,\theta).
				\end{aligned}
			\end{equation*}
			For $\tilde{N}_{61}$, we use similar methods as $\tilde{N}_4$ and $\tilde{N}_5$
			to get
			\begin{equation*}
				\begin{aligned}
					\|\tilde{N}_{61}(t)\|_{C^k}\lesssim\|\tilde\chi_i\|_{C^{k+1}}(1+\|\tilde\chi_i\|_{C^{k+1}})^{k+1}M_{\Phi}(1+M_{\Phi})^{k+2}\|h_i(t)\|_{C^{k+\eps}}.
				\end{aligned}
			\end{equation*}
			For $\tilde{N}_{62}$, we notice that $\text{Supp}(1-\tilde{\chi}_i\circ\mathcal{R}_i\circ\mathcal{X})\subset \{|\theta|\geq 3R_0\}$, and $\text{Supp}(\nabla h_i)\subset \{|\eta|\leq 2R_0\}$, so $\tilde{N}_{62}$ has no singularity. Precisely, we have
			\begin{equation*}
				\left|\nabla_{\theta}^l\left(G(\tilde \phi_i(\theta)(\theta-\eta))\right)\right|\lesssim(1+M_{\Phi})^{l+2} R_0^{-l},\quad \forall |\theta|\geq 3R_0,\ |\eta|\leq 2R_0.
			\end{equation*}
			Then it follows
			\begin{equation*}
				\|\tilde{N}_{62}(t)\|_{C^k}\lesssim \|\tilde\chi_i\|_{C^{k+1}}(1+\|\tilde\chi_i\|_{C^{k+1}})^{k+1}M_{\Phi}(1+M_{\Phi})^{k+2}R_0^{-l}\|h_i(t)\|_{C^1}.
			\end{equation*}
			We conclude that for any $0\leq k\leq m+1$, 
			\begin{equation}\label{3dpesn6}
				\begin{aligned}
					\|\tilde N_6(t)\|_{C^{k}}\lesssim M_{\Phi}(1+M_{\Phi})^{k+2}\|h_i(t)\|_{C^{k+\eps}}.
				\end{aligned}
			\end{equation}
			Combining \eqref{3dpesn4h}, \eqref{3dpesn5h} and \eqref{3dpesn6}, we obtain \eqref{NH}. This completes the proof of the lemma.
		\end{proof}
		\subsection{Proof of Theorem \ref{thmPes3d}}
		In this section, we apply Theorem \ref{thmmani} (Schauder-type estimates on manifolds) to the reformulated system \eqref{3dpesb}, along with the nonlinear estimates (Lemmas \ref{3dpeslemn1}–\ref{3dpeslemn3}) and the error estimate (Lemma \ref{3dpesmlem}) to establish well-posedness results.
		
		Let $m\in\mathbb{N}, \kappa\in(0,1)$ and $1-\kappa \ll 1$. Define the following norm and set,
		\begin{equation}\label{def3dpeszt}
			\begin{aligned}
				&\|X\|_{Z_T}=\sup_{t\in[0,T]}\left(\|X(t)\|_{C^1}+t^{m+\kappa}\|\tilde\nabla X(t)\|_{ C^{m+\kappa}}\right),\\
				&\mathcal{Z}_{T,\Phi}^\sigma=\left\{X:\mathbf\Theta_X(T)<2\mathbf\Theta_0,\ \|X-\Phi\|_{Z_T}\leq\sigma \right\}.
			\end{aligned}
		\end{equation}
		Similarly to \eqref{defM0}, we define 
        \begin{equation}\label{defM1}
            \begin{aligned}
                &\mathfrak{M}'_T(X)=\left(C(1+\mathbf{\Theta}_0)^{m+1}(\mathfrak{c}(\|X\|_{L_T^\infty W^{1,\infty}})^{-1}+\mathfrak{C}(\|X_0\|_{L_T^\infty W^{1,\infty}}))\right)^{m+2},\\
                &\mathfrak{M}_1=\left(C(1+\mathbf{\Theta}_0)^{m+1}(\mathfrak{c}(4\|X_0\|_{W^{1,\infty}})^{-1}+\mathfrak{C}(4\|X_0\|_{W^{1,\infty}}))\right)^{m+2},
            \end{aligned}
        \end{equation}
       where $C$ is a universal constant and $\mathfrak{c},\mathfrak{C}$ defined in \eqref{conten}. Similar to \eqref{controlM}, we obtain 
       \begin{equation*}
           \mathfrak{M}'_T(X)\leq \mathfrak{M}_1,\quad \forall X\in\mathcal{Z}_{T,\Phi}^\sigma,
       \end{equation*}
       provided $\sigma,\varepsilon_0<\min\{\frac{1}{100\mathbf{\Theta}_0},\frac{1}{2}\|X_0\|_{W^{1,\infty}}\}$.
       Define the map $\mathcal{S}(Q)=F$ to be the solution of 
		\begin{equation}\label{3dpesm}
			\begin{aligned}
				&\partial_t(F-\Phi)(t,\widehat{\boldsymbol{x}})+\mathcal{L}_{\Phi}(F-\Phi)(t,\widehat{\boldsymbol{x}})=N(Q,\Phi)(t,\widehat{\boldsymbol{x}}),\\
				&F(0,\widehat{\boldsymbol{x}})=F_0(\widehat{\boldsymbol{x}}),
			\end{aligned}
		\end{equation}
		with $C$ a universal constant and $F_0=X_0$. We prove $\mathcal{S}$ is a contraction map on $\mathcal{Z}_{T,\Phi}^\sigma$ for some small $\sigma$ and $T$.
		For brevity, denote  $H=F-\Phi$ and $H_0=F_0-\Phi$. 
		With Lemma \ref{3dpesmlem}, we are able to apply \eqref{pehmfd} in Theorem \ref{thmmani} to \eqref{3dpesm} with $\mathcal{L}_{\mathcal{M}}^{s}=\mathcal{L}_\Phi$, $\mathcal{L}_{\mathbb{R}^d}^{s,j}$ as \eqref{3dpesdefop}, $G=N(Q,\Phi)$, $n=1$. There exists $T>0$ such that 
		\begin{equation*}
			\|H\|_{Z_T}\leq C_0\mathfrak{M}_1\left(\|H_0\|_{W^{1,\infty}}+\da^{1,m,\kappa}_{T,\mathbb{S}^2}(N(Q,\Phi))\right),
		\end{equation*}
		where $\da^{1,m,\kappa}_{T,\mathbb{S}^2}(N(Q,\Phi))$ is defined in \eqref{pehmfd}. Applying Lemma \ref{3dpeslemn1}, Lemma \ref{3dpeslemn2} and Lemma \ref{3dpeslemn3}, we obtain
		\begin{equation*}\label{3dpeshn}
			\begin{aligned}
				&\da^{1,m,\kappa}_{T,\mathbb{S}^2}(N(Q,\Phi))\leq C_1\mathfrak{M}_1(\|Q-\Phi\|_{Z_T}+T^{\kappa}M_{\Phi})^2(1+\|Q-\Phi\|_{Z_T}+TM_{\Phi})^{m+1}(1+\|Q-\Phi\|_{Z_T}+m_{\Phi})^3,
			\end{aligned}
		\end{equation*}
        where $M_{\Phi}, m_{\Phi}$ are defined in \eqref{Pesconphi}.
		Furthermore, by \eqref{3dpesPc} one has
		\begin{equation*}
			\begin{aligned}
				\inf_{\substack{t<T\\ \widehat{\boldsymbol{x}},\widehat{\boldsymbol{y}}\in\mathbb{S}^2}}\frac{|F(t,\widehat{\boldsymbol{x}})-F(t,\widehat{\boldsymbol{y}})|}{|\widehat{\boldsymbol{x}}-\widehat{\boldsymbol{y}}|}&\geq \inf_{ \widehat{\boldsymbol{x}},\widehat{\boldsymbol{y}}\in\mathbb{S}^2}\frac{|\Phi(\widehat{\boldsymbol{x}})-\Phi(\widehat{\boldsymbol{y}})|}{|\widehat{\boldsymbol{x}}-\widehat{\boldsymbol{y}}|}-2\sup_{t\in[0,T]}\|\nabla_{\mathbb{S}^2}(F-\Phi)(t)\|_{L^\infty}\\
				&\geq \frac{4}{3\mathbf{\Theta}_0}-2C_2\sigma.
			\end{aligned}
		\end{equation*}
		Hence, for any $Q\in\mathcal{Z}_{T,\Phi}^\sigma$, we can take $\eps_0<2^{-m-5}\mathfrak{M}_1^{-2}(1+C_0C_1+C_2\mathbf{\Theta}_0+m_{\Phi})^{-3}$, $\sigma=2C_0\mathfrak{M}_1\eps_0$ and $T<{2^{-m-5}(C_1+M_{\Phi})^{-m-1}}\sigma$ small enough such that 
		\begin{align*}
			\|H\|_{Z_T}\leq C_0\mathfrak{M}_1\eps_0+ \frac{\sigma}{4}\leq\sigma,
		\end{align*}
		and
		\begin{align*}
			\sup_{\substack{t<T\\ \widehat{\boldsymbol{x}},\widehat{\boldsymbol{y}}\in\mathbb{S}^2}}\frac{|\widehat{\boldsymbol{x}}-\widehat{\boldsymbol{y}}|}{|F(t,\widehat{\boldsymbol{x}})-F(t,\widehat{\boldsymbol{y}})|} &\leq \left(\frac{4}{3\mathbf{\Theta}_0}-2C_2\sigma\right)^{-1}\leq 2\mathbf{\Theta}_0.
		\end{align*}
		So we have proved that $\mathcal{S}:\mathcal{Z}_{T,\Phi}^\sigma\rightarrow\mathcal{Z}_{T,\Phi}^\sigma$. Now we only need to prove that $\mathcal{S}$ has the contraction property. For $Q_1,Q_2\in\mathcal Z_{T,\Phi}^\sigma$,  denote $H_i=\mathcal{S}Q_i$, and $\mathbf{H}=H_1-H_2$, $\mathbf{Q}=Q_1-Q_2$, $\mathbf{N}=N(Q_1,\Phi)-N(Q_2,\Phi)$ and $\vec Q=(Q_1,Q_2)$. We write the equation of $\mathbf{H}$ as 
		\begin{equation*}
			\begin{aligned}
				&\partial_t\mathbf{H}(t,\widehat{\boldsymbol{x}})+\mathcal{L}_{\Phi}\mathbf{H}=\mathbf{N}(t,\widehat{\boldsymbol{x}}),\\
				&\mathbf{H}(0,\widehat{\boldsymbol{x}})=0.
			\end{aligned}
		\end{equation*}
		Similarly, we have
		\begin{equation*}
			\begin{aligned}
				\sup_{t\in[0,T]}(\|\mathbf{H}(t)\|_{C^1}+t^{m+\kappa}\|\mathbf{H}(t)\|_{C^{m+1+\kappa}})\lesssim \mathfrak{M}_1\da^{1,m,\kappa}_{T,\mathbb{S}^2}(\mathbf{N}).
			\end{aligned}
		\end{equation*}
		By Lemma \ref{3dpeslemn1} and Lemma \ref{3dpeslemn2}, we have
		\begin{equation*}\label{3dpesnh}
			\begin{aligned}
				&\da^{1,m,\kappa}_{T,\mathbb{S}^2}(\mathbf{N})\\
				&\leq  C_3\mathfrak{M}_1\|\mathbf{Q}\|_{Z_T}(\|\vec Q-\Phi\|_{Z_T}+T^{\kappa}M_{\Phi})(1+\|\vec Q-\Phi\|_{Z_T}+T^{\kappa}M_{\Phi})^{m}(1+\|\vec Q-\Phi\|_{Z_T}+m_{\Phi})^3.
			\end{aligned}
		\end{equation*}
	Taking $\varepsilon_0<2^{-10m}\mathfrak{M}_1^{-2}(1+C_2\mathbf{\Theta}_0+C_3+m_{\Phi})^{-3}$ and $T<\frac{1}{100}(1+M_{\Phi})^{-10m}\sigma$ be small enough, we obtain
		\begin{equation*}
			\|\mathcal{S}(Q_1)-\mathcal{S}(Q_2)\|_{Z_T}\leq\frac{1}{2}\|\mathbf{Q}\|_{Z_T}.
		\end{equation*}
	By contraction mapping theorem, there exists a unique $X\in \mathcal{Z}_{T,\Phi}^\sigma$ such that $\mathcal{S}X=X$, which is a solution to \eqref{eqpes3d}. This completes the proof of Theorem \ref{thmPes3d}.
    
   \section*{Acknowledgement} 
       Q.-H. Nguyen is supported by  CAS Project for Young Scientists in Basic Research, Grant No. YSBR-031; and the National Natural Science Foundation of China (No. 12288201). K. Chen gratefully acknowledges the hospitality of AMSS, CAS during her visit, when this work was completed.
		\section{Appendix}
		\subsection{Besov space and some interpolation inequalities}
		We first recall the definition of homogeneous Besov spaces.
		\begin{definition}
			Let $\varkappa$ be a real number. The homogeneous Besov space $\dot B^{\varkappa}_{\infty,\infty}$ consists of those distributions $u$  such that 
			\begin{align*}
				\|u\|_{\dot B^\varkappa_{\infty,\infty}}:=\sup_{j\in\mathbb{Z}}2^{j\varkappa}\|\dot \Delta_ju\|_{L^\infty}<\infty.
			\end{align*}
			Here $\{\dot \Delta_j\}_{j\in\mathbb{Z}}$ are the standard Littlewood-Paley decomposition blocks: $\dot \Delta_j f=\mathcal{F}^{-1}(\phi_j(\xi)\hat f(\xi))$, where \begin{align*}\label{phidecom}
				\phi_j(\xi)=\phi(2^{-j}\xi),\ \  \operatorname{supp}(\phi)\subset \{\xi: \frac{3}{4}\leq|\xi|\leq \frac{8}{3}\}, \ \ \sum_{j\in\mathbb{Z}}\phi(2^{-j}\xi)\equiv 1.
			\end{align*}
		\end{definition}
		
		\begin{remark}
			\begin{itemize}
				\item For any $s\in \mathbb{R}$,  we have the equivalence: $\|\nabla u\|_{\dot B^{s}_{\infty,\infty}}\sim \|u\|_{\dot B^{s+1}_{\infty,\infty}}$.
				\item  For \( s \in \mathbb{R}^+ \backslash \mathbb{N} \), the Besov space \( \dot{B}^s_{\infty,\infty} \) coincides with the homogeneous Hölder space \( \dot{C}^s \). This equivalence fails when \( s \in \mathbb{N} \).
				
				\item The \( \dot{B}^1_{\infty,\infty} \) norm admits the following equivalent characterization (see \cite{Triebel}):
				\[
				\|u\|_{\dot{B}^1_{\infty,\infty}} \sim \sup_{x,y} \frac{|2u(x) - u(x+y) - u(x-y)|}{|y|}.
				\]\end{itemize}
		\end{remark}
		\begin{remark}\cite{Triebel}
			Let  $s,\alpha \in\mathbb{R}$. Then $(-\Delta)^\frac{\alpha}{2}:\dot B^s_{\infty,\infty}\to \dot B^{s-\alpha}_{\infty,\infty}$ is an
			isomorphism.
		\end{remark}
		In the following proposition, we present a characterization of Besov space $\dot B^0_{\infty,\infty}$ that relies on the the fundamental solution defined by \eqref{defk0}.
		\begin{proposition}\label{normequ}
			For any $a>0$, $x_0\in\mathbb{R}^d$,  denote $\mathcal{K}(f)(t,x)=\K_{x_0}(t,0)\ast f$, where $\K_{x_0}(t,0,x)$ is defined by \eqref{defk0}.
			Then there holds 
			\begin{align*}\label{rrr}
				\sup_{t>0} t^\frac{a}{s}\|\mathcal{K}(f)(t)\|_{\dot C^a}\lesssim \|f\|_{\dot B^0_{\infty,\infty}}.
			\end{align*}
		\end{proposition}
			The proof of Proposition \ref{normequ} follows from Littlewood-Paley theory, see \cite{Fourierbook} for detail.

		We list some elementary inequalities that used frequently throughout the paper. 
		\begin{lemma}\label{maininterpo}~~\\
			1. For any $0<\gamma_1<\gamma<\gamma_2$, there holds
			\begin{align*}
				\|f\|_{\dot C^\gamma}\lesssim \|f\|_{\dot C^{\gamma_1}}^\frac{\gamma_2-\gamma}{\gamma_2-\gamma_1} \|f\|_{\dot C^{\gamma_2}}^\frac{\gamma-\gamma_1}{\gamma_2-\gamma_1}.
			\end{align*}
			2. For any $j,j_1,j_2\in\mathbb{N}$ with $0\leq j_1\leq j\leq j_2$, $p\in[1,\infty]$, there holds
			\begin{align*}
				\|\nabla^jf\|_{L^p}\lesssim\|\nabla^{j_1}f\|_{L^p}^\frac{j_2-j}{j_2-j_1}\|\nabla^{j_2}f\|_{L^p}^{\frac{j-j_1}{j_2-j_1}}.
			\end{align*}
			3. Let $a\in(0,1)$, for any $f,g\in C^a$, there holds
			\begin{align*}
				\|fg\|_{\dot C^a}\lesssim \|f\|_{L^\infty}\|g\|_{\dot C^a}+\|f\|_{\dot C^a}\|g\|_{L^\infty}.
			\end{align*}
			4. (\cite[Lemma 2.2]{KN}) Let $\theta_1\in(0,1)$. For any function $f$ and any $0<\varepsilon_0<\frac{1}{2}\min\{\theta_1,1-\theta_1\}$, there hold
			\begin{equation*}\label{interpfrac}
				\|\Lambda^{\theta_1}f\|_{L^\infty}\lesssim (\|f\|_{\dot C^{\theta_1-\varepsilon_0}}\|f\|_{\dot C^{\theta_1+\varepsilon_0}})^{\frac{1}{2}},
			\end{equation*}	
			\begin{equation*}\label{interpfrac2}
				\|f\|_{\dot C^{\theta_1}}\lesssim \|\Lambda^{\theta_1}f\|_{L^\infty}.
			\end{equation*}	
		\end{lemma}
		We have the following interpolation lemma, which is classical in singular integral theory.
		\begin{lemma}\label{intp}
			i) Let $m(\xi)$ be a Fourier multiplier of order $\sigma$ with $\sigma>-d$, and
			\begin{align}\label{conlem6.6}
				| \nabla_\xi^km(\xi)|\leq |\xi|^{\sigma-k}.
			\end{align}
			Then for $\mathcal{T}f=\mathcal{F}^{-1}(m\hat f)$, it holds
			\begin{align}\label{estlem6.6}
				\|\mathcal{T}f\|_{L^\infty}\lesssim_\varepsilon \|\Lambda^{\sigma+\varepsilon}f\|_{L^\infty}^\frac{1}{2}\|\Lambda^
				{\sigma-\varepsilon}f\|_{L^\infty}^\frac{1}{2},
			\end{align}
			for any $\varepsilon$ small enough.\\
            ii) Let $m(x,\xi)$ be a Fourier multiplier of order $\sigma$ with $\sigma>-d$, with the corresponding operator $\mathcal{T}$ defined by 
            \begin{equation*}
                \mathcal{T}f=(2\pi)^{-\frac{d}{2}}\int_{\mathbb{R}^d}\hat{f}(\xi)m(x,\xi)e^{ix\cdot\xi}d\xi,
            \end{equation*}
            and satisfies \eqref{conlem6.6} for any $x$, then \eqref{estlem6.6} also holds.
		\end{lemma}
		\begin{proof}
			i) Denote $\chi_\ell$ as a dilation of a non-negative smooth cut-off function $\chi_\ell(\xi)=\chi(\frac{\xi}{\ell})$, $\mathbf{1}_{|\xi|<1}\leq \chi\leq \mathbf{1}_{|\xi|>2}$. Then we have the decomposition
			\begin{equation*}
				\mathcal{T}f=\mathcal{F}^{-1}\left(\chi_\ell\frac{m(\cdot)}{|\cdot|^{\sigma-\eps}}\widehat{\Lambda^{\sigma-\eps}f}\right)+\mathcal{F}^{-1}\left((1-\chi_\ell)\frac{m(\cdot)}{|\cdot|^{\sigma+\eps}}\widehat{\Lambda^{\sigma+\eps}f}\right).
			\end{equation*}
			By classical result in singular integral theory, one can see that
			\begin{equation*}
				\left\|\mathcal{F}^{-1}\left(\chi_\ell\frac{m(\cdot)}{|\cdot|^{\sigma-\eps}}\right)\right\|_{L^1}\lesssim \ell^\eps,\quad \left\|\mathcal{F}^{-1}\left(\chi_\ell\frac{m(\cdot)}{|\cdot|^{\sigma+\eps}}\right)\right\|_{L^1}\lesssim \ell^{-\eps}.
			\end{equation*}
			We can take $\ell^\varepsilon=\|\Lambda^{\sigma+\eps}f\|_{L^\infty}\|\Lambda^{\sigma-\eps}f\|_{L^\infty}^{-1}$ to get the result.\\
            ii) Define $\mathcal{T}_yf(x)=\mathcal{F}^{-1}(m(y,\cdot)\hat{u})(x)$. Then for any $x\in\mathbb{R}^d$, there holds
            \begin{equation*}
                \left|\mathcal{T}f(x)\right|\lesssim \sup_{x\in\mathbb{R}^d}\left\|\mathcal{T}_xf(y) \right\|_{L^\infty_y}\lesssim \|\Lambda^{\sigma+\varepsilon}f\|_{L^\infty}^\frac{1}{2}\|\Lambda^
				{\sigma-\varepsilon}f\|_{L^\infty}^\frac{1}{2},
            \end{equation*}
            where we use \eqref{estlem6.6} for the last inequality.
		\end{proof}
		
		\begin{lemma}\label{douint}
			Let $\kappa_1,\kappa_2\in (0,1)$, for any function $f:\mathbb{R}^d\to \mathbb{R}^N$, there holds
			\begin{equation*}
				\sup_{\alpha\in\mathbb{R}^d}  \frac{\|\delta_\alpha f\|_{L^\infty}}{|\alpha|^{\kappa_1}}\lesssim\|f\|_{\dot C^{\kappa_1}},\quad\quad\quad\quad  \sup_{\alpha,\beta\in\mathbb{R}^d}\frac{\|\delta_\alpha\delta_\beta f\|_{L^\infty}}{|\alpha|^{\kappa_1}|\beta|^{\kappa_2}}\lesssim\|f\|_{\dot C^{\kappa_1+\kappa_2}}.
			\end{equation*}
			Denote $\mathcal{O}_\alpha g(x)=\frac{\delta_\alpha\delta_{-\alpha}h(x)}{|\alpha|}$. Then  for any $\nu\in(0,\frac{1}{2})$, $\beta\in\mathbb{R}^d$, and $a\in(0,1)$, there hold
			\begin{align*}
				&\int_{\mathbb{R}^d}\|\mathcal{O}_\alpha g\|_{L^\infty}\frac{d\alpha}{|\alpha|^{d}}\lesssim \|g\|_{\dot C^{1+\nu}}^\frac{1}{2} \|g\|_{\dot C^{1-\nu}}^\frac{1}{2},\\
                &\int_{\mathbb{R}^d}\|\delta_\alpha f\|_{L^\infty}\|\delta_\alpha g\|_{L^\infty}\frac{d\alpha}{|\alpha|^{d+1}}\lesssim \|g\|_{\dot C^{\frac{1}{2}}}\|f\|_{\dot C^{\frac{1}{2}+\nu}}^\frac{1}{2} \|f\|_{\dot C^{\frac{1}{2}-\nu}}^\frac{1}{2}.
			\end{align*}
			Moreover, for $\E^\alpha, \E_\alpha$ defined in \eqref{defEal}, $\tilde{\E}^\alpha$  defined in \eqref{notapes}, and $h:\mathbb{S}\to \mathbb{R}^N$, there hold
			\begin{align*}
				\sup_{\alpha\in\mathbb{R}^d}\frac{\|\E^\alpha f\|_{L^\infty}+\|\E_\alpha f\|_{L^\infty}}{|\alpha|^{\kappa_1}}\lesssim \|f\|_{\dot C^{1+\kappa_1}},\quad \sup_{\alpha\in\mathbb{S}} \frac{  \|\tilde\E^\alpha h\|_{L^\infty}}{|\alpha|^{\kappa_1}}\lesssim \|h\|_{\dot C^{1+\kappa_1}}.
			\end{align*}
		\end{lemma}
		\begin{lemma}\label{lemcom}
			Let $m\in\mathbb{N}$, $\alpha\in(0,1)$. Consider $g,g_1,g_2:\mathbb{R}^d\to\mathbb{R}$, and $f:\mathbb{R}\to \mathbb{R}$ satisfying 
			\begin{align*}
				\sum_{k=1}^{m+2}\|f^{(k)}\|_{L^\infty}\lesssim 1.
			\end{align*}
			Then 
			\begin{equation}\label{mainint11}
				\|\nabla^m(f\circ g)\|_{L^\infty}\lesssim  \|g\|_{\dot C^1}^m+\| g\|_{\dot C^m},
			\end{equation}
			\begin{equation}\label{mainint12}
				\|\nabla^m(f\circ g)\|_{\dot C^\alpha}\lesssim  \| g\|_{\dot C^\alpha}^\frac{m+\alpha}{\alpha}+\|g\|_{\dot C^{m+\alpha}},
			\end{equation}
			\begin{equation}\label{mainint21}
				\|\nabla^m(f\circ g_1-f\circ g_2)\|_{L^\infty}\lesssim\sum_{n=0}^m \|g_1-g_2\|_{\dot C^n}(\|(g_1,g_2)\|_{\dot C^1}^{m-n}+\|(g_1,g_2)\|_{\dot C^{m-n}}),
			\end{equation}
			\begin{equation}\label{mainint22}
				\begin{aligned}
					\|\nabla^m(f\circ g_1-f\circ g_2)\|_{\dot C^\alpha}\lesssim &
					\sum_{n=0}^m\left\{\|g_1-g_2\|_{\dot C^{n+\alpha}}(\|(g_1,g_2)\|_{\dot C^1}^{m-n}+\|(g_1,g_2)\|_{\dot C^{m-n}})\right.\\
					&\left.\quad\quad+\|g_1-g_2\|_{\dot C^n}(\|(g_1,g_2)\|_{\dot C^\alpha}^\frac{m-n+\alpha}{\alpha}+\|(g_1,g_2)\|_{\dot C^{m-n+\alpha}})\right\}.
				\end{aligned}
			\end{equation}
			\begin{equation}\label{mainint31}
				\begin{aligned}
					&\left\|\nabla^m\big((f\circ g_1-f\circ g_2)-(f\circ g_3-f\circ g_4)\big)\right\|_{L^\infty}\\
					&\lesssim \sum_{n=0}^m\|(g_1-g_2)-(g_3-g_4)\|_{\dot C^n}(\|(g_1,g_2)\|_{\dot C^1}^{m-n}+\|(g_1,g_2)\|_{\dot C^{m-n}})\\
					&\quad\quad\quad+\sum_{n_1+n_2+n_3=n}\|g_3-g_4\|_{\dot C^{n_1}}\|(g_1-g_3,g_2-g_4)\|_{\dot C^{n_2}}\left(\sum_{k=1}^4(\|g_k\|_{\dot C^1}^{n_3}+\|g_k\|_{\dot C^{n_3}})\right),
				\end{aligned}
			\end{equation}
			
		\end{lemma}
		\begin{proof}
			The first two inequalities follow from Lemma \ref{maininterpo}, and we omit the proof.
			For \eqref{mainint21} and \eqref{mainint22}, since
			\begin{equation*}
				f\circ g_1-f\circ g_2=\int_0^1(g_1-g_2)\cdot\nabla f(g_1-\lambda(g_1-g_2))d\lambda.
			\end{equation*}
			We can get the result by \eqref{mainint11},  \eqref{mainint12} and Lemma \ref{maininterpo},.\\
			For the last two inequalities, similarly,
			\begin{align*}
				&(f\circ g_1-f\circ g_2)-(f\circ g_3-f\circ g_4)\\
				&=\int_0^1(g_1-g_2)\cdot\nabla f(g_1-\lambda(g_1-g_2))d\lambda-\int_0^1(g_3-g_4)\cdot\nabla f(g_3-\lambda(g_3-g_4))d\lambda\\
				&=\int_0^1\int_0^1(g_3-g_4)(g_3-\lambda(g_3-g_4))\nabla^2f\left(g_3-\lambda(g_3-g_4)-\mu\big((g_1-g_3)-\lambda((g_1-g_3)-(g_2-g_4))\big)\right)\\
				&\quad+\int_0^1\left((g_1-g_2)-(g_3-g_4)\right)\cdot\nabla f(g_1-\lambda(g_1-g_2))d\lambda.
			\end{align*}
			Then we obtain  \eqref{mainint31} by \eqref{mainint11}, \eqref{mainint12} and Lemma \ref{maininterpo}. This completes the proof of the lemma.
		\end{proof}
		\subsection{Estimates of nonlinear terms in the Muskat equation}
        In this section, we estimate nonlinear terms in the equation \eqref{eqmusre}.
   Denote $$\Upsilon(f)=\frac{1}{\langle\partial_xf\rangle^3},\quad\quad\quad \tilde \Upsilon(f_1,f_2)=\Upsilon(f_1)-\Upsilon(f_2).$$  
   We have the following estimate of $\Upsilon(f)$ and $\tilde\Upsilon(f_1,f_2)$.
   \begin{lemma}\label{lemUps}
    For any $0\leq l\leq m+2$, it holds 
    \begin{align*}
    &\|\Upsilon(f)(t)\|_{\dot C^l}\lesssim t^{-\frac{l}{3}}(1+\|f\|_T)^l,\\
    &\|\tilde \Upsilon(f_1,f_2)(t)\|_{\dot C^l}\lesssim t^{-\frac{l}{3}}\|f_1-f_2\|_T(1+\|(f_1,f_2)\|_T)^l.
    \end{align*}
    for any $t\in (0,T].$
   \end{lemma}
   The proof follows directly from Lemma \ref{lemcom}
	and the definition of  $\|\cdot\|_T$ in \eqref{normmst}, we omit details here.

        	For simplicity, we denote the quantity
		\begin{align*}
			\mathbf{C}_n(f_1,f_2):=\left(\|f_1-f_2\|_{T,*}\|(f_1,f_2)\|_{T,*}+\|f_1-f_2\|_{T}\|(f_1,f_2)\|_{T,*}^2\right)(1+\|(f_1,f_2)\|_T)^{2n+5}.
		\end{align*}
Recalling the definition of $\|\cdot \|_T,\|\cdot\|_{T,*}$ and $\|\cdot\|_{X_T}$ in   \eqref{normmst},      we have the following estimate of nonlinear term $\N[f]$.
		\begin{lemma}\label{nonst}
			Let $\N[f](t,x)$  be as defined in \eqref{defNst} with the constant $\varrho_0$ defined in \eqref{defvr0}. For any $T>0$, and any $f,f_1,f_2$ with $\|f\|_{X_T}+\|f_1\|_{X_T}+\|f_2\|_{X_T}<\infty$,  the following estimates hold:
			\begin{itemize}
				\item \textbf{(i) Nonlinear estimate:}
				\begin{equation}\label{NNN}
					\begin{aligned}
						&		\sum_{j=0}^m\sup_{t\in[0,T]}t^\frac{j+\kappa}{3}\| \N[f](t)\|_{\dot C^{j+\kappa-2}}\lesssim \mathbf{C}_m(f,0)+|\varrho_0|(T^\frac{2}{3}+T^\frac{1}{10})\|f\|_{X_T}(1+\|f\|_{X_T})^{m+5},
					\end{aligned}
				\end{equation}
				\item \textbf{(ii) Lipschitz type continuity:}
				\begin{equation}\label{Ndi}
					\begin{aligned}
								\sum_{j=0}^m\sup_{t\in[0,T]}t^\frac{j+\kappa}{3}&\| \N[f_1](t)- \N[f_2](t)\|_{\dot C^{j+\kappa-2}}\\
						&\lesssim \mathbf{C}_m(f_1,f_2)+|\varrho_0|(T^\frac{2}{3}+T^\frac{1}{10})||f_1-f_2||_{X_T}(1+||(f_1,f_2)||_{X_T})^{m+4}.
					\end{aligned}
				\end{equation}	
			\end{itemize}
		\end{lemma}
		\begin{proof}
			To prove \eqref{NNN} and \eqref{Ndi}, it suffices to show that, for any $n\in\mathbb{N}$, $n\leq m+1$,
			\begin{equation}\label{eee1}
				\begin{aligned}
					&\sup_{t\in[0,T]}t^{\frac{n+2}{3}}\|\N[f_1](t)-\N[f_2](t)\|_{\dot C^{n}}\\
					&\ \ \quad\quad\quad\quad\lesssim   \mathbf{C}_n(f_1,f_2)+|\varrho_0|(T^\frac{2}{3}+T^\frac{1}{10})||f_1-f_2||_{X_T}(1+||(f_1,f_2)||_{X_T})^{n+4}.
				\end{aligned}
			\end{equation}
			Note that $\kappa-2\in(0,1)$, then by Lemma \ref{maininterpo}, the desired estimate \eqref{Ndi} follows from \eqref{eee1}, and \eqref{NNN} follows from taking $f_1=f$, $f_2\equiv 0$ in \eqref{eee1} since $\N[0]\equiv 0$.
			
			For simplicity,  fix $t\in[0,T]$ and omit the time variable in the proof. Denote 
			\begin{equation}\label{defEal}
				\begin{aligned}
					&	\mathsf E_\alpha f(x)=\partial_x f(x)-\Delta_\alpha f(x),\ \ \ \ \ \ \E^\alpha f(x)=\partial_x f(x-\alpha)-\Delta_\alpha f(x),\\
					&\B[f](x,\alpha)=\frac{ \Delta_{\alpha} f(x)\E_\alpha f(x)}{\alpha\left\langle\Delta_{\alpha} f(x)\right\rangle^{2}},\ \ \quad\quad \mathsf{M}[f](x)=\frac{\partial_{x}^{2} f(x)}{\left\langle\partial_{x} f(x)\right\rangle^{3}}.
				\end{aligned}
			\end{equation}
			Then we have $\mathsf{N}[f]=\N_1[f]+\N_2[f]+\varrho_0\N_3[f]$ with 
			\begin{align*}
				&\N_{1}[f](x)=\frac{1}{\pi} \int_{\mathbb{R}}\B[f](\alpha,x)\partial_{x}\M[f](x-\alpha) {d \alpha},\\
				&\N_2[f](x)=-\frac{1}{\pi} \int_{\mathbb{R}}\partial_x^2f(x-\alpha)\left(\frac{1}{\left\langle\partial_{x} f(x-\alpha)\right\rangle^{3}}-\frac{1}{\left\langle\partial_{x} f(x)\right\rangle^{3}}\right)\frac{d\alpha}{\alpha^2},\\
				&\N_3[f](x)=-\frac{1}{\pi}\int_{\mathbb{R}}\B[f](\alpha,x)\partial_x f(x-\alpha){d\alpha}-\Lambda f(x).
			\end{align*}
			We first consider $\N_1[f_1]-\N_1[f_2]$, 
			\begin{align*}
				\N_1[f_1](x)-
				\N_1[f_2](x)=&\frac{1}{\pi} \int_{\mathbb{R}}\B[f_1](\alpha,x) \partial_{x}\left(\M[f_1]-\M[f_2]\right)(x-\alpha) d \alpha\\
				&+\frac{1}{\pi} \int_{\mathbb{R}}(\B[f_1]-\B[f_2])(\alpha,x) \partial_{x}\M[f_2](x-\alpha) d \alpha\\
				:=&\mathrm{I}_1+\mathrm{I}_2.
			\end{align*}
			Integrating  by parts gives
			\begin{align*}
				\mathrm{I}_1&=-\frac{1}{\pi} \int_{\mathbb{R}}\B[f_1](\alpha,x) \partial_{\alpha}\left(\M[f_1]-\M[f_2]\right)(x-\alpha)d \alpha\\
				&=-\frac{1}{\pi} \int_{\mathbb{R}}\partial_{\alpha}\B[f_1](\alpha,x)\delta_\alpha\left(\M[f_1]-\M[f_2]\right)(x) {d \alpha}.
			\end{align*}
			Then, for $n\in\mathbb{N}$, $n\leq m+1$,
			\begin{align}\label{fori1}
				\partial_x^{n}\mathrm{I}_1=-\frac{1}{\pi}\sum_{k+l=n}\int_{\mathbb{R}} \partial_x^k\partial_{\alpha}\B[f_1](\alpha,x)\delta_\alpha\partial_x^l\left(\M[f_1]-\M[f_2]\right)(x)d\alpha.
			\end{align}
			Denote $\psi(r)=\frac{r}{\langle r\rangle^2}$, then $\B[f](x,\alpha)=\frac{\psi(\Delta_\alpha f(x))\E_\alpha f(x)}{\alpha}$. 	Note that $-\partial_\alpha(\E_\alpha f(x))=\partial_\alpha (\Delta_\alpha f(x))=\frac{\E^\alpha f(x)}{\alpha}$. Hence,  \begin{align*}
				\partial_\alpha \B[f_1](\alpha,x)=\frac{1}{\alpha^2}\left(\psi'(\Delta_\alpha f(x))\E^\alpha f(x)\E_\alpha f(x)-\psi(\Delta_\alpha f(x))(\E^\alpha f(x)+\E_\alpha f(x))\right).
			\end{align*}
			This implies 
			\begin{align*}
				| \partial_x^k  \partial_\alpha \B[f](\alpha,x)|\lesssim& \frac{1}{\alpha^2}\left(\sum_{k_1+k_2+k_3=k}\|\partial_x^{k_1}\psi'(\Delta_\alpha f(\cdot))\|_{L^\infty}\|\partial_x^{k_2}\E^\alpha f\|_{L^\infty}\|\partial_x^{k_3}\E_\alpha f\|_{L^\infty}\right.\\
				&\ \left.+\sum_{k_4+k_5=k}\|\partial_x^{k_4}\left(\psi(\Delta_\alpha f(\cdot))\right)\|_{L^\infty}(\|\partial_x^{k_5}\E^\alpha f\|_{L^\infty}+\|\partial_x^{k_5}\E_\alpha f\|_{L^\infty})\right).
			\end{align*}
			Recall the definition of $\|\cdot\|_T$ in \eqref{normmst}. By \eqref{mainint11} in Lemma \ref{lemcom}, we have
			\begin{align*}
				\|\partial_x^{k_1}\left(\psi(\Delta_\alpha f(\cdot))\right)\|_{L^\infty}&+\|\partial_x^{k_1}\left(\psi'(\Delta_\alpha f(\cdot))\right)\|_{L^\infty}\lesssim \|\Delta_\alpha f\|_{\dot C^1}^{k_1}+\|\Delta_\alpha f\|_{\dot C^{k_1}}\\
				&\lesssim \|\partial_x f\|_{\dot C^{k_1}}+\|\partial_x f\|_{\dot C^1}^{k_1}\lesssim t^{-\frac{k_1}{3}}\|f\|_T(1+\|f\|_T)^{k_1}.
			\end{align*}
			Moreover, we have 
			\begin{align*}
				\|\partial_x^{k_3}\E_\alpha f\|_{L^\infty}\lesssim \|\partial_x f\|_{\dot C^{k_3}}\lesssim t^{-\frac{k_3}{3}}\|f\|_T.
			\end{align*}
			This yields that 
			\begin{equation}\label{b1}
				\begin{aligned}
					&|\partial_{\alpha}\partial_x^k\B[f](\alpha,x)|
					\lesssim \frac{1}{\alpha^2}\|f\|_T(1+\|f\|_T)^{k+1}\sum_{k_1+k_2=k} t^{-\frac{k_1}{3}}(\|\E^\alpha \partial_x^{k_2}f\|_{L^\infty}+\|\E_\alpha \partial_x^{k_2}f\|_{L^\infty}).
				\end{aligned}
			\end{equation}
			Moreover, by Lemma \ref{douint}, 
			\begin{align*}
				\|\partial_x^{k_2}\E^\alpha f\|_{L^\infty}+\|\partial_x^{k_2}\E_\alpha f\|_{L^\infty}&\lesssim \min\{|\alpha|^\frac{3}{4}\|\partial_xf\|_{\dot C^{k_2+\frac{3}{4}}},|\alpha|^\frac{3}{4}\|\partial_xf\|_{\dot C^{k_2+\frac{1}{4}}}\}\\
				&\lesssim t^{-\frac{k_2}{3}}\|f\|_{T,*}\min\{(|\alpha|t^{-\frac{1}{3}})^\frac{3}{4},(|\alpha|t^{-\frac{1}{3}})^\frac{1}{4}\}.
			\end{align*}
			Combining this with \eqref{b1}, we have
			\begin{equation}\label{i1esb}
				\begin{aligned}
					&|\partial_{\alpha}\partial_x^k\B[f](\alpha,x)|\lesssim \frac{1}{\alpha^2}\|f\|_{T,*}\| f\|_T (1+\| f\|_T )^{k+1}t^{-k/3}\min\{(|\alpha|t^{-\frac{1}{3}})^\frac{3}{4},(|\alpha|t^{-\frac{1}{3}})^\frac{1}{4}\}.
				\end{aligned}
			\end{equation}
			Here and in the following we will use Lemma \ref{lemcom} without claim.
		For any $0\leq l\leq m+2$,	one has 
			\begin{equation*}
				\begin{aligned}
					\|\partial_x^l\left(\M[f_1]-\M[f_2]\right)\|_{L^\infty}
					&\lesssim \left\|\partial_x^l\left( {\partial_{x}^{2} (f_1-f_2)\Upsilon(f_2)}\right)\right\|_{L^\infty}+\left\|\partial_x^l\left(\partial_x^2f_1\tilde \Upsilon(f_1,f_2)\right)\right\|_{L^\infty}\\
					&\quad\lesssim t^{-\frac{l+1}{3}}(1+\|(f_1,f_2)\|_{T})^{l+1}(\|f_1-f_2\|_T\|(f_1,f_2)\|_{T,*}+\|f_1-f_2\|_{T,*}).
				\end{aligned}
			\end{equation*}
			Combining this with \eqref{i1esb} and \eqref{fori1}, by interpolation, we derive
			\begin{equation}\label{i1}
				\begin{aligned}
					\|\partial_x^n\mathrm{I}_1\|_{L^\infty}&\lesssim t^{-\frac{n+1}{3}}\mathbf{C}_n(f_1,f_2)\int_{\mathbb{R}}\min\{(|\alpha|t^{-\frac{1}{3}})^\frac{3}{2},(|\alpha|t^{-\frac{1}{3}})^\frac{1}{2}\}\frac{d\alpha}{\alpha^2}\\
					&\lesssim t^{-\frac{n+2}{3}}\mathbf{C}_n(f_1,f_2).
				\end{aligned}
			\end{equation}
			Similarly, integrating by parts, we obtain
			\begin{align*}
				\mathrm{I}_2=-\frac{1}{\pi} \int_{\mathbb{R}}\partial_\alpha (\B[f_1]-\B[f_2])(\alpha,x) \delta_\alpha \M[f_2](x) {d \alpha},
			\end{align*}
			and
			\begin{align*}
				\partial_x^n\mathrm{I}_2=-\frac{1}{\pi} \sum_{k+l=n}\int_{\mathbb{R}}\partial_\alpha \partial_x^k(\B[f_1]-\B[f_2])(\alpha,x) \delta_\alpha \partial_x^l\M[f_2](x) {d \alpha}.
			\end{align*}
			Note that 
			\begin{equation*}
				(\B[f_1]-\B[f_2])(\alpha,x)= \frac{\Delta_\alpha(f_1-f_2)\E_\alpha f_1}{\alpha\langle \Delta_\alpha f\rangle^2}+\frac{\Delta_\alpha f_2\E_\alpha (f_1-f_2)}{\alpha\langle \Delta_\alpha f\rangle^2}+\frac{\Delta_\alpha f_2\E_\alpha f_2}{\alpha}\left(\frac{1}{\langle\Delta_\alpha f_1\rangle^2}-\frac{1}{\langle\Delta_\alpha f_2\rangle^2}\right).
			\end{equation*}
			Then we can use similar estimates as \eqref{b1} to obtain
			\begin{equation}\label{b2}
				\begin{aligned}
					&\left|\partial_\alpha \partial_x^k(\B[f_1]-\B[f_2])(\alpha,x) \right|\\
					&\lesssim \frac{1}{\alpha^2}\left(\sum_{k_1+k_2+k_3=k}(\|(f_1,f_2)\|_{\dot C^{k_1+1}}+\|(f_1,f_2)\|_{\dot C^2}^{k_1})(\|(\E_{\alpha}f_1,\E_{\alpha}f_2)\|_{\dot C^{k_2}}+\|(\E_{\alpha}f_1,\E_{\alpha}f_2)\|_{\dot C^1}^{k_2})\right.\\
					&\quad\quad\quad\quad\quad\quad\times(\|f_1-f_2\|_{\dot C^{k_3+1}}+\|f_1-f_2\|_{\dot C^2}^{k_3})\\
					&\left.\quad\quad+\sum_{k_1+k_2=k}(\|(f_1,f_2)\|_{\dot C^{k_1+1}}+\|(f_1,f_2)\|_{\dot C^2}^{k_1})(\|(\E_{\alpha}(f_1-f_2)\|_{\dot C^{k_2}}+\|(\E_{\alpha}(f_1-f_2)\|_{\dot C^1}^{k_2})\right)
				\end{aligned}
			\end{equation}
			and by definition of $\|\cdot\|_T$, we have
			\begin{equation*}
				\begin{aligned}
					&\left|\partial_\alpha \partial_x^k(\B[f_1]-\B[f_2])(\alpha,x) \right|\\
					&\lesssim \frac{1}{t^{\frac{k}{3}}\alpha^2}(1+\|(f_1,f_2)\|_T)^{k+3}(\|f_1-f_2\|_{T}\|(f_1,f_2)\|_{T,*}+\|f_1-f_2\|_{T,*})\min\{(|\alpha|t^{-\frac{1}{3}})^\frac{3}{4},(|\alpha|t^{-\frac{1}{3}})^\frac{1}{4}\},
				\end{aligned}
			\end{equation*}
			By Lemma \ref{lemcom}, we have for any $0\leq l\leq m+2$,
			\begin{equation*}
				\left\|\partial_x^l\M[f_2]\right\|_{L^\infty}\lesssim t^{-\frac{l+1}{3}}\|f_2\|_{T,*}(1+\| f_2\|_T)^l.
			\end{equation*}
			Hence one has
			\begin{equation}\label{i2}
				\begin{aligned}
					\|\partial_x^n\mathrm{I}_2\|_{L^\infty}&\lesssim t^{-\frac{n+1}{3}}\mathbf{C}_n(f_1,f_2)\int_{\mathbb{R}}\min\{(|\alpha|t^{-\frac{1}{3}})^\frac{3}{2},(|\alpha|t^{-\frac{1}{3}})^\frac{1}{2}\}\frac{d\alpha}{|\alpha|^2}\\
					&\lesssim  t^{-\frac{n+2}{3}}\mathbf{C}_n(f_1,f_2).
				\end{aligned}
			\end{equation}
			We conclude from \eqref{i1} and \eqref{i2} that 
			\begin{equation}\label{N1}
				\begin{aligned}
					&\sup_{t\in[0,T]}t^\frac{n+2}{3}\|\partial_x^n(\N_1[f_1]-\N_1[f_2])(t)\|_{L^\infty}\lesssim \mathbf{C}_n(f_1,f_2),
				\end{aligned}
			\end{equation}
			for any $0\leq n\leq m+1$.
			
			Then we consider $\N_2[f_1]-\N_2[f_2]$, note that 
			\begin{align*}
			\N_2[f](x)=-\frac{1}{\pi} \int_{\mathbb{R}}\delta_\alpha \partial_x^2f(x)\delta_\alpha\Upsilon(f)(x)\frac{d\alpha}{\alpha^2}+\frac{1}{2\pi}\partial_x^2f(x)\int _{\mathbb{R}}\mathcal{O}_\alpha \Upsilon(f)(x)\frac{d\alpha}{|\alpha|},
			\end{align*}
			where the operator $\mathcal{O}_\alpha$ is defined in Lemma \ref{douint}.
			We split $\N_2[f_1]-\N_2[f_2]$ into 
			\begin{equation*}\label{splitN2}
				\begin{aligned}
					&(\N_2[f_1]-\N_2[f_2])(x)\\
					&\quad\quad=-\frac{1}{\pi} \int_{\mathbb{R}}\delta_\alpha \partial_x^2(f_1-f_2)(x)\delta_\alpha\Upsilon(f_1)(x)\frac{d\alpha}{\alpha^2} -\frac{1}{\pi} \int_{\mathbb{R}}\delta_\alpha \partial_x^2f_2(x)\delta_\alpha\tilde \Upsilon(f_1,f_2)(x)\frac{d\alpha}{\alpha^2}\\
					&\quad\quad\quad\quad\ +\frac{1}{2\pi}\partial_x^2(f_1-f_2)(x)\int_{\mathbb{R}} \mathcal{O}_\alpha\Upsilon(f_1)(x)\frac{d\alpha}{|\alpha|}+\frac{1}{2\pi}\partial_x^2f_2(x)\int_{\mathbb{R}} \mathcal{O}_\alpha \tilde \Upsilon(f_1,f_2)(x)\frac{d\alpha}{|\alpha|}\\
					&\quad\quad:=\sum_{j=1}^4\mathrm{II}_j(x).
				\end{aligned}  
			\end{equation*}
            By Lemma \ref{douint}, we obtain 
            \begin{align*}
          &  \|\partial_x^n\mathrm{II}_1 \|_{L^\infty}\lesssim \sum_{n_1=0}^n(\|\partial_x^2(f_1-f_2)\|_{\dot C^{n_1+\frac{1}{4}}}\|\partial_x^2(f_1-f_2)\|_{\dot C^{n_1+\frac{3}{4}}}\|\Upsilon(f_1)\|_{\dot C^{n-n_1+\frac{1}{4}}}\|\Upsilon(f_1)\|_{\dot C^{n-n_1+\frac{3}{4}}})^\frac{1}{2},\\
           & \|\partial_x^n\mathrm{II}_2\|_{L^\infty}\lesssim \sum_{n_1=0}^n(\|\partial_x^2f_2\|_{\dot C^{n_1+\frac{1}{4}}}\|\partial_x^2f_2\|_{\dot C^{n_1+\frac{3}{4}}}\|\tilde \Upsilon(f_1,f_2)\|_{\dot C^{n-n_1+\frac{1}{4}}}\|\tilde \Upsilon(f_1,f_2)\|_{\dot C^{n-n_1+\frac{3}{4}}})^\frac{1}{2},\\
           &\|\partial_x^n\mathrm{II}_2\|_{L^\infty}\lesssim \sum_{n_1=0}^n\|\partial_x^2(f_1-f_2)\|_{\dot C^{n_1}}\|\Upsilon(f_1)\|_{\dot C^{n-n_1+\frac{3}{4}}}^\frac{1}{2}\|\Upsilon(f_1)\|_{\dot C^{n-n_1+\frac{5}{4}}}^\frac{1}{2},\\
           &\|\partial_x^n\mathrm{II}_2\|_{L^\infty}\lesssim \sum_{n_1=0}^n\|\partial_x^2f_2\|_{\dot C^{n_1}}\|\tilde \Upsilon(f_1,f_2)\|_{\dot C^{n-n_1+\frac{3}{4}}}^\frac{1}{2}\|\tilde \Upsilon(f_1,f_2)\|_{\dot C^{n-n_1+\frac{5}{4}}}^\frac{1}{2}.
            \end{align*}
	Combining this with Lemma \ref{lemUps}, we obtain  
			\begin{equation}\label{N2}
				\begin{aligned}
					&\sup_{t\in[0,T]}t^\frac{n+2}{3}\|\partial_x^n(\N_2[f_1]-\N_2[f_2])(t)\|_{L^\infty}\lesssim \mathbf{C}_n(f_1,f_2).
				\end{aligned}
			\end{equation}
			Finally, we estimate lower order terms
			\begin{align*}
				\N_3[f_1]-
                \N_3[f_2]=&-\frac{1}{\pi}\int_{\mathbb{R}}\B[f_1](\alpha,x)\partial_{x}(f_1-f_2)(x-\alpha){d \alpha}-\frac{1}{\pi}\int_{\mathbb{R}}(\B[f_1]-\B[f_2])(\alpha,x)\partial_{x}f_2(x-\alpha){d \alpha}\\
				&-\Lambda (f_1-f_2)(x)\\
				:=&\mathrm{III}_1+\mathrm{III}_2+\mathrm{III}_3.
			\end{align*}
			Using integration by parts,
			\begin{align*}
				&|\partial_x^n\mathrm{III}_1|\lesssim \sum_{k+l=n}\int_{\mathbb{R}}|\partial_\alpha\partial_x^k\B[f_1](\alpha,x)||\delta_\alpha\partial_x^l (f_1-f_2)(x)|d\alpha,\\
				&|\partial_x^n\mathrm{III}_2|\lesssim \sum_{k+l=n}\int_{\mathbb{R}}|\partial_\alpha\partial_x^k(\B[f_1]-\B[f_2])(\alpha,x)||\delta_\alpha \partial_x^lf_2(x)|d\alpha.
			\end{align*}
			Then we use \eqref{b1}, \eqref{b2}, together with the facts that
			\begin{align*}
				|\delta_\alpha\partial_x^l f(x)|\lesssim |\alpha|^\frac{1}{2}\|f\|_{\dot C^{l+\frac{1}{2}}}\lesssim |\alpha|^\frac{1}{2}(\mathbf{1}_{l\geq 1}t^{-\frac{1}{3}(l-\frac{1}{2})}\|f\|_T+\mathbf{1}_{l=0}\|f\|_{X_T}),
			\end{align*}
			to obtain
			\begin{align*}
				|\partial_x^n\mathrm{III}_1|&\lesssim \sum_{k+l=n}\|f_1\|_{T,*}\| f_1\|_T ^2(1+\| f_1\|_T )^kt^{-\frac{k}{3}}\int_{\mathbb{R}}\min\{|\alpha|^\frac{3}{4}t^{-\frac{1}{4}},|\alpha|^\frac{1}{4}t^{-\frac{1}{12}}\}|\delta_\alpha\partial_x^l (f_1-f_2)(x)|\frac{d\alpha}{\alpha^2}\\
				&\lesssim \|f_1\|_{T,*}\| f_1\|_T ^2(1+\| f_1\|_T )^nt^{-\frac{n}{3}}(\|f_1-f_2\|_T+t^{-\frac{1}{6}}||f_1-f_2||_{X_T}).\\
				|\partial_x^n\mathrm{III}_2|&\lesssim \sum_{k+l=n}(1+\|(f_1,f_2)\|_T)^{k+4}\|f_1-f_2\|_{T}t^{-\frac{k}{3}}\int_{\mathbb{R}}\min\{|\alpha|^\frac{3}{4}t^{-\frac{1}{4}},|\alpha|^\frac{1}{4}t^{-\frac{1}{12}}\}|\delta_\alpha \partial_x^lf_2(x)|\frac{d\alpha}{\alpha^2}\\
				&\lesssim (1+\|(f_1,f_2)\|_T)^{n+4}\|f_1-f_2\|_{T}t^{-\frac{n}{3}}(\|f_2\|_T+t^{-\frac{1}{6}}||f_2||_{X_T}).
			\end{align*}
			For $\mathrm{III}_3$, there holds
			\begin{align*}
				\|\partial_x^n\mathrm{III}_3\|_{L^\infty}\lesssim t^{-\frac{n}{3}}||f_1-f_2||_{X_T}.
			\end{align*}
			We conclude that 
			\begin{align*}
				\sup_{t\in[0,T]}t^\frac{n}{3}\|\partial_x^n (\N_3[f_1]-\N_3[f_2])(t)\|_{L^\infty}\lesssim (T^\frac{2}{3}+T^\frac{1}{10})||f_1-f_2)||_{X_T}(1+||(f_1,f_2)||_{X_T})^{n+5}.
			\end{align*}
			Combining this with \eqref{N1} and \eqref{N2} yields \eqref{eee1}. This completes the proof of the lemma.
		\end{proof}

     		\begin{lemma}\label{GGG}
			Let $\G[f,\phi]$ be as defined in \eqref{defffg}. For any $T>0$, and any $f,f_1,f_2$ with $\|f\|_{X_T}+\|f_1\|_{X_T}+\|f_2\|_{X_T}<\infty$, we have the following estimates:
			\begin{itemize}
				\item \textbf{(i) Nonlinear estimate:}
				\begin{align*}
					\sum_{n=0}^m\sup_{t\in[0,T]}t^{\frac{n+\kappa}{3}}\|\G[f_1,f_2](t)\|_{\dot C^{n+\kappa-2}}\lesssim \|f_1-f_2\|_T\|
					f_1\|_{T,*}(1+\|(f_1,f_2)\|_T)^{m+5},
				\end{align*}
				\item \textbf{(ii) Lipschitz type continuity:}
				\begin{align*}
					&\sum_{n=0}^m\sup_{t\in[0,T]}t^{\frac{n+\kappa}{3}}\|(\G[f_1,\phi]-\G[f_2,\phi])(t)\|_{\dot C^{n+\kappa-2}}\\
					&\quad\quad\quad\quad\lesssim (\|f_1-f_2\|_T\|f_1\|_{T,*}+\|f_1-f_2\|_{T,*}\|f_2-\phi\|_T)(1+\|(f_1,f_2,\phi)\|_T)^{m+5}.
				\end{align*}
			\end{itemize}
		\end{lemma}
		\begin{proof}
	Note that 
			\begin{align*}
		\G[f_1,f_2]=\Upsilon(f_2,f_1)\Lambda^3f_1,\quad\quad\quad	\G[f_1,\phi]-\G[f_2,\phi]
				&=\G[f_1,f_2]+\tilde \Upsilon(\phi,f_2)\Lambda^3 (f_1-f_2).
			\end{align*}
	For any fixed $t\in[0,T]$, by Lemma \ref{lemUps}, we obtain 
			\begin{align*}
				\left\|\partial_x^n \G[f_1,f_2](t)\right\|_{\dot C^{\kappa-2}}\lesssim t^{-\frac{n+\kappa}{3}}\|f_1-f_2\|_T\|
				f_1\|_{T,*}(1+\|(f_1,f_2)\|_T)^{m+5},
			\end{align*}
		and
			\begin{align*}
				\left\|\partial_x^n \left(\tilde \Upsilon(\phi,f_2)\Lambda^3 (f_1-f_2)\right)(t)\right\|_{\dot C^{\kappa-2}}\lesssim t^{-\frac{n+\kappa}{3}}\|f_2-\phi\|_T\|f_1-f_2\|_{T,*}(1+\|(f_2,\phi)\|_T)^{m+5}.
			\end{align*}
			Then we complete the proof.	
		\end{proof}
        \begin{remark}\label{musLinfty}
           Following the  proof of Lemma \ref{GGG}, we obtain the following $L^\infty$ estimates,
            \begin{align*}
                &\sup_{t\in[0,T]}t^\frac{2}{3}\|\G[f_1,f_2](t)\|_{L^\infty}\lesssim \|f_1-f_2\|_{T}\|f_1\|_{T,*},\\
                &\sup_{t\in[0,T]}t^\frac{2}{3}\|(\G[f_1,\phi]-\G[f_2,\phi])(t)\|_{L^\infty}\lesssim \|f_1-f_2\|_{T}\|f_1\|_{T,*}+\|f_1-f_2\|_{T,*}\|f_2-\phi\|_{T}.
            \end{align*}
        \end{remark}
		\subsection{Estimates of nonlinear terms in 2D Peskin problem}
		The following part is devoted to estimate the nonlinear terms of the Peskin equation. In this subsection, we fix $m\in\mathbb{N}$, $\kappa\in(0,1)$ and $1-\kappa\ll 1$. Recall the definition of semi-norms
		\begin{equation*}
			\begin{aligned}
				&\|h\|_{T}=\sup_{t\in[0,T]}(\|\partial_xh(t)\|_{L^\infty}+t^{m+\kappa}\|\partial_xh(t)\|_{\dot C^{m+\kappa}}),\\
				&	\|h\|_{T,*}=\sup_{t\in[0,T]}(t^\frac{1}{10}\|\partial_xh(t)\|_{\dot C^\frac{1}{10}}+t^{m+\kappa}\|\partial_xh(t)\|_{\dot C^{m+\kappa}}),
			\end{aligned}
		\end{equation*}
and		the set $\mathcal{X}^\sigma_{T,\Phi}$ defined in \eqref{defset} with $\sigma\in(0,1)$. We have the following results.
		\begin{lemma}
			\label{lemnonpes}Let $\mathcal{N}(X)$ be as defined in \eqref{defnonpes}. For any $T\in(0,1)$, and any functions $X,Y,Z$ satisfying $\|X\|_T+\|Y\|_T+\|Z\|_T<\infty$, and
            \begin{align}\label{MT}
            \mathfrak{M}_T:= \mathfrak{M}_T(X)+\mathfrak{M}_T(Y)+\mathfrak{M}_T(Z)<\infty,
            \end{align}
         where $\mathfrak{M}_T(X)$ is defined in \eqref{defM0},    the following estimates hold:
			\begin{itemize}
				\item \textbf{(i) Nonlinear estimate:}
				\begin{align}\label{resn}
					&\sum_{j=0,m}\sup_{t\in[0,T]}t^{j+\kappa}	\|\mathcal{N}(X)(t)\|_{\dot C^{j+\kappa}}
					\lesssim \mathfrak{M}_T\|X\|_{T,*}^2(1+\|X\|_{T})^{2(m+1)}.
				\end{align}
				\item \textbf{(ii) Lipschitz type continuity:}
				\begin{equation}\label{rendiff}
					\begin{aligned}
						\sum_{j=0,m}\sup_{t\in[0,T]} t^{j+\kappa}	&\|\mathcal{N}(Y)(t)-\mathcal{N}(Z)(t)\|_{\dot C^{j+\kappa}}\\
                        &\lesssim \mathfrak{M}_T\|Y-Z\|_{T}(\|( Y,Z)\|_{T,*}+T^{\frac{1}{2}}\|( Y,Z)\|_{T})(1+\|(Y,Z)\|_T)^{2(m+1)}.
					\end{aligned}
				\end{equation}
			\end{itemize}
		\end{lemma}
		The following lemma shows the estimates for $\M(\partial_xX)$.
		\begin{lemma}\label{lempesR}
			Let  $\M(\partial_xX)$ be as defined in \eqref{defpeR}. For any $T>0$, and any $X,Y,Z$ satisfying $\|X\|_T+\|Y\|_T+\|Z\|_T<\infty$ and \eqref{MT}, the following estimates hold:
            \begin{itemize}
                \item \textbf{(i) Nonlinear estimate:}
                \begin{align*}
				\sum_{j=0,m}\sup_{t\in[0,T]} t^{j+\kappa}\|\M(\partial_xX)(t)\|_{\dot C^{j+\kappa-\frac{1}{2}}}\lesssim \mathfrak{M}_T\|X\|_{T,*}^2(1+\|X\|_T)^{m+5}.
			\end{align*}
            \item \textbf{(ii) Lipschitz type continuity:}
			\begin{align*}
				&\sum_{j=0,m}\sup_{t\in[0,T]} t^{j+\kappa}\|(\M(\partial_xY)-\M(\partial_xZ))(t)\|_{\dot C^{j+\kappa-\frac{1}{2}}}\\
				&\quad\quad\quad\quad\lesssim \mathfrak{M}_T(\|Y-Z\|_{T,*}+\|Y-Z\|_{T}\|(Y,Z)\|_{T,
					*})\|(Y, Z)\|_{T,*}(1+\|(Y,Z)\|_T)^{m+4}.
			\end{align*}
		\end{itemize}
            \end{lemma}
		Recalling the definition \eqref{notapes},	we first prove the following lemmas. 
		\begin{lemma}\cite[Lemma 2.4]{KN}\label{lempesibp}
			For any function $f,g:{\mathbb{R}}\rightarrow\mathbb{R}$, denote $\tilde g(\alpha)={\Delta}_\alpha g(0)$. Then for any $\sigma\in(0,1)$ and $0<\varepsilon< 10^{-3}\min\{1-\sigma,\sigma\}$, there holds
			$$
			\left|\int_{\mathbb{R}} f(\alpha)(\partial_\alpha \tilde g)(\alpha)d\alpha\right|\lesssim \| f\|_{\dot C^{\sigma+\varepsilon}}^\frac{1}{2} \| f\|_{\dot C^{\sigma-\varepsilon}}^\frac{1}{2}\|g\|_{\dot C^{1-\sigma+\varepsilon}}^\frac{1}{2}\|g\|_{\dot C^{1-\sigma-\varepsilon}}^\frac{1}{2}.$$
		\end{lemma}
		\begin{lemma}\label{dede}
			For any function $f:\mathbb{S}\to\mathbb{R}$, then for any $0<\gamma_1\leq\gamma<1$,
			\begin{align*}
				\label{es2}
				&	
				\sup_x\sup_{\alpha\neq h} \frac{|\tilde \Delta_\alpha f(x)-\tilde \Delta_h f(x)||\alpha|^{\gamma-\gamma_1}}{|\alpha-h|^\gamma}\lesssim {\|f'\|_{\dot C^{\gamma_1}}}.
			\end{align*}
		\end{lemma}
		\begin{proof} 
			Consider $x=0$ without loss of generality, and denote $\tilde f (\alpha)=\tilde \Delta_\alpha f(0)$.
			Thanks to the periodicity, it is enough to consider $\alpha\in (0,\pi)$ and $h\in[\alpha/2,\alpha/2+2\pi]$. 
			We have 
			\begin{align*}
				\tilde f(\alpha)-\tilde f(h)=&\left(\frac{1}{2}\cot\left(\frac{h}{2}\right)-\frac{1}{h}\right)\int_0^{-h}f'(\omega)d\omega-\left(\frac{1}{2}\cot\left(\frac{\alpha}{2}\right)-\frac{1}{\alpha}\right)\int_0^{-\alpha}f'(\omega)d\omega\\
				&+\frac{1}{h}\int_0^{-h}f'(\omega)d\omega-\frac{1}{\alpha}\int_0^{-\alpha}f'(\omega)d\omega\\
				=&\left(\frac{1}{2}\cot\left(\frac{\alpha}{2}\right)-\frac{1}{\alpha}\right)\int_{-\alpha}^{-h}f'(\omega)d\omega+\left(\frac{1}{2}\cot\left(\frac{h}{2}\right)-\frac{1}{h}-\left(\frac{1}{2}\cot\left(\frac{\alpha}{2}\right)-\frac{1}{\alpha}\right)\right)\int_0^{-h}f'(\omega)d\omega\\
				&+\frac{1}{h}\int_{-\alpha}^{-h}(f'(\omega)-f'(-\alpha))d\omega+\frac{\alpha-h}{\alpha h}\int_0^{-\alpha}(f'(\omega)-f'(-\alpha))d\omega.
			\end{align*}
			Here, with a slight abuse of notation, we denote $\int_x^yg(\omega)d\omega=-\int_y^xg(\omega)d\omega$ if $y<x$.
			Note that $$\left|\frac{1}{2}\cot\left(\frac{\alpha}{2}\right)-\frac{1}{\alpha}\right|\lesssim |\alpha|,\ \ \ \ \  \left|\frac{1}{2}\cot\left(\frac{\alpha}{2}\right)-\frac{1}{\alpha}-\frac{1}{2}\cot\left(\frac{h}{2}\right)+\frac{1}{h}\right|\lesssim |\alpha-h|.$$ Hence, 
			\begin{align*}
				|\tilde f(\alpha)-\tilde f(h)|\lesssim \left(|h|+|\alpha|+\frac{|h-\alpha|^{\gamma_1}}{|h|}+\frac{|\alpha|^{\gamma_1}}{|h|}\right) |h-\alpha|\|f'\|_{\dot C^{\gamma_1}}.
			\end{align*}
			We obtain 
			\begin{align*}
				\frac{|\tilde \Delta_\alpha f-\tilde \Delta_h f|}{|\alpha-h|^\gamma}\lesssim \left(|h|+|\alpha|+\frac{|h-\alpha|^{\gamma_1}}{|h|}+\frac{|\alpha|^{\gamma_1}}{|h|}\right) |h-\alpha|^{1-\gamma}\|f'\|_{\dot C^{\gamma_1}}\lesssim \frac{\|f'\|_{\dot C^{\gamma_1}}}{|h|^{\gamma-\gamma_1}}\lesssim \frac{\|f'\|_{\dot C^{\gamma_1}}}{|\alpha|^{\gamma-\gamma_1}}.
			\end{align*}
			This completes the proof of the lemma.

		\end{proof}

		Fix $m\in\mathbb{N}$, $\kappa\in(0,1)$ and $1-\kappa\ll 1$. For $f(t,x,\alpha): [0,T]\times \mathbb{S}\times \mathbb{S}\to \mathbb{R}$, we define 
		\begin{equation}\label{defTn}
			\begin{aligned}
				&	[f_1]_{T,*}:=	\sup_{\vartheta\in[\frac{1}{2},\kappa]}\sup_{t\in[0,T]}(t^{\vartheta-\frac{1}{4}}\|f_1(t)\|_{\vartheta,\frac{1}{4}}+t^{m+\vartheta-\frac{1}{4}}\|\partial_x^mf_1(t)\|_{\vartheta,\frac{1}{4}})\\
				&\quad\quad\quad\quad\quad+\sup_{t\in[0,T]}\sup_x(t^\frac{1}{10}\|f_1(t,x,\cdot)\|_{\dot C^\frac{1}{10}}+t^{m+\kappa}\|\partial_x^{m}f_1(t,x,\cdot)\|_{\dot C^\kappa}),\\
				&[f_1]_{T}:=\sup_{t\in[0,T]}\sup_\alpha(\|f_1(t,\cdot,\alpha)\|_{L^\infty}+t^{m+\kappa}\|f_1(t,\cdot,\alpha)\|_{\dot C^{m+\kappa}})+	[f_1]_{T,*},
			\end{aligned}
		\end{equation}
		where 
		\begin{align}
			\label{defnorab}
			\|f_1(t)\|_{a,b}:=	\sup_{\alpha,z\in\mathbb{S}}|\alpha|^b\left(\|f_1(t,\cdot,\alpha)\|_{\dot C^a}+\frac{\|f_1(t,\cdot,\alpha )-f_1(t,\cdot,\alpha -z)\|_{L^\infty}}{|z|^a}\right),\ \ \ \text{for}\ a,b\in[0,1].
		\end{align}
		Denote 
		\begin{align*}
			&\interleave h \interleave_T=\sup_{t\in[0,T]}(\|h(t)\|_{L^\infty}+t^{m+\kappa}\|h(t)\|_{\dot C^{m+\kappa}}),\\
			&	\interleave h \interleave_{T,*}=\sup_{t\in[0,T]}(t^\frac{1}{10}\|h(t)\|_{\dot C^\frac{1}{10}}+t^{m+\kappa}\|h(t)\|_{\dot C^{m+\kappa}}).
		\end{align*}
		Note that $\interleave \partial_xh\interleave_T=\|h\|_T$ and  $\interleave \partial_xh\interleave_{T,*}=\|h\|_{T,*}$, where  $\|\cdot \|_T$ and $\|\cdot\|_{T,*}$ are defined in \eqref{normstar}.
		We have the following lemma.
		\begin{lemma}\label{lemeee} 
			Let $f_1:[0,T]\times \mathbb{S}\times \mathbb{S}\to \mathbb{R}$, and  $f_2,f_3: [0,T]\times  \mathbb{S}\to \mathbb{R} $ satisfy
			\begin{equation*}
				\begin{aligned}
					& [f_1]_{T}+\|f_2\|_{T}+ \interleave  f_3\interleave_{T} <\infty
				\end{aligned}
			\end{equation*}
			with $0<T<1$.
			Then for
			\begin{align*}
				f(t,x)=\int_{\mathbb{S}} f_1(t,x,\alpha)\tilde{\E}^\alpha f_2(t,x) f_3(t,x-\alpha)\frac{d\alpha}{\tilde \alpha},
			\end{align*}
			where $\tilde \alpha$ and $\tilde{\E}^\alpha$ are  defined in \eqref{notapes},
			there holds
			\begin{equation}\label{res1}
				\begin{aligned}
					\sup_{t\in[0,T]}t^{m+\kappa }\|f(t,\cdot)\|_{\dot C^{m+\kappa }}
					\lesssim \|f_2\|_{T,*}([f_1]_{T,*}\interleave f_3\interleave_{T}+[f_1]_T\interleave f_3\interleave_{T,*}+T^\frac{1}{2}[f_1]_T\interleave f_3\interleave_{T}).
				\end{aligned}
			\end{equation}
			If we further have $\int_{\mathbb{S}} f_1(t,x,\alpha)\tilde{\E}^\alpha f_2(t,x)\frac{d\alpha}{\tilde \alpha}=0$, then 
			\begin{align}\label{res2}
				\sup_{t\in[0,T]}t^{m+\kappa }\|f(t,\cdot)\|_{\dot C^{m+\kappa }}\lesssim [f_1]_T\|f_2\|_{T,*}\interleave f_3\interleave_{T,*}.
			\end{align}
		\end{lemma}
		\begin{proof}
			We first prove \eqref{res1}. For simplicity, for any $t\in[0,T]$, we fix $t$ and drop the time variable $t$ in this proof, and denote $f_i^{m_i}=\partial_x ^{m_i}f_i$ for $i=1,2,3$, $m_i\in\mathbb{N}$. We have 
			\begin{align*}
				\partial_x^m f(x)=\sum_{m_1+m_2+m_3=m}\int_{\mathbb{S}}f_1^{m_1}(x,\alpha)\tilde{\E}^\alpha f_2^{m_2}(x) f_3^{m_3}(x-\alpha)\frac{d\alpha}{\tilde \alpha}.
			\end{align*}
			In the following,  we drop the summation for $m_1+m_2+m_3=m$ with a slight abuse of notation. \\
			For any $\beta \neq 0$, we can write 
			\begin{align*}
				\delta_\beta\partial_x^m f(x)=&\int_{\mathbb{S}}\delta_\beta f_1^{m_1}(x,\alpha)\tilde{\E}^\alpha f_2^{m_2}(x) f_3^{m_3}(x-\alpha)\frac{d\alpha}{\tilde \alpha}\\
				&+\int_{\mathbb{S}} f_1^{m_1}(x-\beta,\alpha)\delta_\beta\tilde{\E}^\alpha f_2^{m_2}(x) f_3^{m_3}(x-\alpha)\frac{d\alpha}{\tilde \alpha}\\
				&+\int_{\mathbb{S}} f_1^{m_1}(x-\beta,\alpha) \tilde{\E}^\alpha f_2^{m_2}(x-\beta)\delta_\beta f_3^{m_3}(x-\alpha)\frac{d\alpha}{\tilde \alpha}\\
				:=&P_1+P_2+P_3.
			\end{align*}
			By the definition of $\|\cdot\|_{a,b}$ in \eqref{defnorab}, we have
			\begin{align*}
				&|\alpha|^\frac{1}{4}|\delta_\beta f_1^{m_1}(x,\alpha)|\lesssim |\beta|^\kappa\|f_1^{m_1}\|_{\kappa,\frac{1}{4}}.
			\end{align*}
			Hence it is easy to check that 
			\begin{equation}\label{1P1}
				\begin{aligned}
					|P_1|&\lesssim |\beta|^\kappa t^{-(m_1+\kappa-\frac{1}{4})}[f_1]_{T,*}\|f_3^{m_3}\|_{L^\infty} \int_{\mathbb{S}}|\tilde{\E}^\alpha f_2^{m_2}|\frac{d\alpha}{|\alpha|^\frac{5}{4}}\\
					&\lesssim |\beta|^\kappa t^{-(m+\kappa)}[f_1]_{T,*}\|f_2\|_{T,*}\interleave f_3\interleave_{T}.
				\end{aligned}
			\end{equation}
			Then we deal with $P_{2}$. Observe that 
			\begin{align}\label{ealp}
				\tilde{\E}^\alpha f_2^{m_2}(x)=\alpha\partial_\alpha (\Delta_\alpha f_2^{m_2})(x)+\delta_\alpha f_2^{m_2}(x)\left(\frac{1}{\alpha}-\frac{1}{\tilde \alpha}\right).
			\end{align}
			Hence, by \eqref{1perio},
			\begin{align*}
				P_{2}
				=&\int_{\mathbb{R}} f_1^{m_1}(x-\beta,\alpha)\partial_\alpha (\Delta_\alpha \delta_\beta f_2^{m_2})(x) f_3^{m_3}(x-\alpha){d\alpha}\\
				&+\int_{\mathbb{R}} f_1^{m_1}(x-\beta,\alpha)(\delta_\alpha\delta_\beta f_2^{m_2})(x)  f_3^{m_3}(x-\alpha)\left(\frac{1}{\alpha}-\frac{1}{\tilde\alpha}\right)\frac{d\alpha}{\alpha}\\
				:=&P_{2,1}+P_{2,2}.
			\end{align*}
			Since $|\frac{1}{\alpha}(\frac{1}{\alpha}-\frac{1}{\tilde\alpha})|\lesssim\min\{1,|\alpha|^{-2}\}$ for $|\alpha|\leq \pi$, we remark that $P_{2,1}$ is the main term, and $P_{2,2}$ is the remainder error term. 
			Applying Lemma \ref{lempesibp} with $f(\alpha)=f_1^{m_1}(x,\alpha) f_3^{m_3}(x-\alpha)$, $g=\delta_\beta f_2^{m_2}$, one has 
			\begin{align*}
				|P_{2,1}|\lesssim\prod_{+,-} \left\{ \|\delta_\beta f_2^{m_2}\|_{\dot C^{\frac{1}{4}\pm\varepsilon}}^\frac{1}{2}\sup_x\|f_1^{m_1}(x,\alpha) f_3^{m_3}(x-\alpha)\|_{\dot C_\alpha^{\frac{3}{4}\pm\varepsilon}}^\frac{1}{2}\right\}.
			\end{align*}
			Here we denote $\dot C^{b}_\alpha$ the H\"{o}lder semi-norm in variable $\alpha$. By Lemma \ref{maininterpo}, it is easy to check  that for $\gamma=\frac{3}{4}\pm\varepsilon$, 
			\begin{align*}
				\|f_1^{m_1}(x,\alpha) f_3^{m_3}(x-\alpha)\|_{L^\infty_x\dot C_\alpha^{\gamma}}
				&\lesssim \|f_1^{m_1}\|_{L^\infty_x\dot C^\gamma_\alpha}\|f_3^{m_3}\|_{L^\infty}+\|f_1^{m_1}\|_{L^\infty_{x,\alpha}}\|f_3^{m_3}\|_{\dot C^\gamma}\\
				&\lesssim t^{-(m_1+m_3)-\gamma}([f_1]_{T,*}\interleave f_3\interleave_{T}+[f_1]_T\interleave f_3\interleave_{T,*}).
			\end{align*}
			Hence we obtain that 
			\begin{align}
				|P_{2,1}|&\lesssim |\beta|^\kappa  \prod_{+,-} \left\{t^{-(m_1+m_3)-\frac{3}{4}\mp\varepsilon}\|\partial_x f_2^{m_2}\|_{\dot C^{\kappa-\frac{3}{4}\pm\varepsilon}} \right\}^\frac{1}{2}([f_1]_{T,*}\interleave f_3\interleave_{T}+[f_1]_T\interleave f_3\interleave_{T,*})\nonumber\\
				&\lesssim t^{-\kappa-m} |\beta|^\kappa \|f_2\|_{T,*}([f_1]_{T,*}\interleave f_3\interleave_{T}+[f_1]_T\interleave f_3\interleave_{T,*}),\label{1lP21}
			\end{align}
			where $m=m_1+m_2+m_3$. 
			For $P_{2,2}$, by \eqref{1perio} we have 
			\begin{equation}\label{aaaaa}
				\begin{aligned}
					\left|	\int_{\mathbb{R}}h(\alpha)\left(\frac{1}{\alpha}-\frac{1}{\tilde\alpha}\right)\frac{d\alpha}{\alpha}\right|&=\left|\int_{\mathbb{R}} h(\alpha)\left(\frac{1}{\alpha^2}-\frac{\mathbf{1}_{|\alpha|\leq \pi}}{\tilde\alpha^2}\right)d\alpha\right|\\
					&\lesssim \int_{\mathbb{R}}\left| h(\alpha)\right|\min\{1,|\alpha|^{-2}\}d\alpha.
				\end{aligned}
			\end{equation}
			Then we obtain that 
			\begin{align*}
				|P_{2,2}|&\lesssim \|f_1^{m_1}\|_{L^\infty}\| f_3^{m_3}\|_{L^\infty}\int _{\mathbb{R}} \|\delta_\alpha\delta_\beta f_2^{m_2}\|_{L^\infty}\min\{1,|\alpha|^{-2}\}d\alpha \\
				&\lesssim |\beta|^\kappa\|f_1^{m_1}\|_{L^\infty}\| f_3^{m_3}\|_{L^\infty}\|\partial_xf_2^{m_2}\|_{\dot C^{\kappa-\frac{1}{2}}}\\
				&\lesssim t^{-\kappa-m+\frac{1}{2}} |\beta|^\kappa [f_1]_T\|f_2\|_{T,*}\interleave f_3\interleave_{T}.
			\end{align*}
			Combining this with \eqref{1lP21} to obtain that for any $t\in[0,T]$,
			\begin{align}\label{1P2}
				t^{m+\kappa } |\beta|^{-\kappa} |P_{2}|\lesssim \|f_2\|_{T,*}([f_1]_{T,*}\interleave f_3\interleave_{T}+[f_1]_T\interleave f_3\interleave_{T,*}+T^\frac{1}{2}[f_1]_T\interleave f_3\interleave_{T}).
			\end{align}
			Finally, we deal with $P_3$. If $m_1+m_2\neq 0$, \textit{i.e.} $m_3<m$, we directly have
			\begin{align*}
				|P_3|&\lesssim \|f_1^{m_1}\|_{L^\infty}\|\partial_x f_2\|_{\dot C^{m_2+\frac{3}{4}}}\|\delta_\beta f_3\|_{\dot C^{m_3}}\\
				&\lesssim |\beta|^\kappa t^{-(m+\kappa)}[f_1]_T\|f_2\|_{T,*}\interleave f_3\interleave_{T,*}.
			\end{align*}
			If $m_1+m_2=0$, we observe that 
			for any function $h_1,h_2$, by a change of variable, 
			\begin{align*}
				\int_{\mathbb{S}} h_1(\alpha)\delta_\beta h_2(x-\alpha)d\alpha&=\int_{\mathbb{S}} \left(h_1(\alpha)-h_1(\alpha-\beta)\right) h_2(x-\alpha)d\alpha\\
				&=-\int_{\mathbb{S}} \left(h_1(\alpha)-h_1(\alpha-\beta)\right) \delta_\alpha h_2(x)d\alpha.
			\end{align*}
			Applying this to $P_3$, it follows
			\begin{align*}
				P_3&=\int_{\mathbb{S}}\left( \frac{f_1(x-\beta,\alpha)\tilde{\E}^\alpha f_2(x-\beta)}{\alpha}-\frac{f_1(x-\beta,\alpha-\beta)\tilde{\E}^{\alpha-\beta} f_2(x-\beta)}{\alpha-\beta}\right)\delta_\alpha f_3^m(x)d\alpha\\
				&=\int_{\mathbb{S}}{ (f_1(x-\beta,\alpha)-f_1(x-\beta,\alpha-\beta))\tilde{\E}^\alpha f_2(x-\beta)}\delta_\alpha f_3^m(x) \frac{d\alpha}{\alpha}\\
				&\ \ \ \ +\int_{\mathbb{S}} f_1(x-\beta,\alpha-\beta)(\tilde{\E}^\alpha f_2-\tilde{\E}^{\alpha-\beta}f_2)(x-\beta)\delta_\alpha f_3^m(x)\frac{d\alpha}{\alpha}\\
				&\ \ \ \ +\int_{\mathbb{S}} f_1(x-\beta,\alpha-\beta)\tilde{\E}^{\alpha-\beta}f_2(x-\beta)\delta_\alpha f_3^m(x)\left(\frac{1 }{\alpha}-\frac{1}{\alpha-\beta}\right)d\alpha \\
				&:=P_{3,1}+P_{3,2}+P_{3,3}.
			\end{align*}
			For $P_{3,1}$, by \eqref{defnorab} and  Lemma \ref{douint} we have 
			\begin{align}
				&	|f_1(x-\beta,\alpha)-f_1(x-\beta,\alpha-\beta)|\lesssim|\beta|^\kappa  \min\left\{\sup_x\|f_1(x,\cdot)\|_{\dot C^\kappa},|\alpha|^{-\frac{1}{4}}\|f_1\|_{\kappa,\frac{1}{4}}\right\},\nonumber\\
				&	|\tilde{\E}^\alpha f_2(x,\beta)|\lesssim |\alpha|^a \|\partial_x f_2\|_{\dot C^{a}},\quad\quad |\delta_\alpha f_3^m(x)|\lesssim |\alpha |^a \|f_3\|_{\dot C^{m+a}}, \ \ \forall a\in(0,1).\label{eal111}
			\end{align}
			This implies
			\begin{equation}\label{1lP31}
				\begin{aligned}
					|P_{3,1}|	&\lesssim 	\int_{\mathbb{S}} | (f_1(x-\beta,\alpha)-f_1(x-\beta,\alpha-\beta))\tilde{\E}^\alpha f_2(x-\beta)\delta_\alpha f_3^m(x) |\frac{d\alpha}{|\alpha|}\\
					&\lesssim |\beta|^\kappa\int_{\mathbb{S}} \min\left\{\sup_x\|f_1(x,\cdot)\|_{\dot C^\kappa},|\alpha|^{-\frac{1}{4}}\|f_1\|_{\kappa,\frac{1}{4}}\right\}\frac{d\alpha}{|\alpha|^\frac{4}{5}}  \|\partial_x f_2\|_{\dot C^{\frac{1}{10}}}\|f_3\|_{\dot C^{\frac{1}{10}+m}}\\
					&\lesssim |\beta|^\kappa t^{-(m+\kappa)}[f_1]_{T,*}\|f_2\|_{T,*}\interleave f_3\interleave_{T,*}.
				\end{aligned}
			\end{equation}
			For $P_{3,2}$, observe that
			\begin{align*}
				({\tilde{\E}^\alpha  f_2}-{\tilde{\E}^{\alpha-\beta}  f_2})(x)&=\partial_x(\delta_{-\beta}f_2)(x-\alpha)-(\tilde \Delta_\alpha f_2(x)-\tilde \Delta_{\alpha-\beta} f_2(x))\\
				&=\partial_\alpha (\alpha \Delta_\alpha \delta_{-\beta}f_2)(x)-(\tilde \Delta_\alpha f_2(x)-\tilde \Delta_{\alpha-\beta} f_2(x))\\
				&=\alpha\partial_\alpha ( \Delta_\alpha \delta_{-\beta}f_2)(x)+\Delta_\alpha \delta_{-\beta}f_2(x)-(\tilde \Delta_\alpha f_2(x)-\tilde \Delta_{\alpha-\beta} f_2(x)).
			\end{align*}
			We further decompose $P_{3,2}$ into 
			\begin{align*}
				P_{3,2}=&\int_{\mathbb{S}}  f_1(x-\beta,\alpha-\beta)\partial_\alpha ( \Delta_\alpha \delta_{-\beta}f_2)(x) \delta_\alpha f_3^m(x){d\alpha}\\
				&+\int_{\mathbb{S}}  f_1(x-\beta,\alpha-\beta)( \Delta_\alpha \delta_{-\beta}f_2)(x) \delta_\alpha f_3^m(x)\frac{d\alpha}{\alpha}\\
				&+\int_{\mathbb{S}}  f_1(x-\beta,\alpha-\beta)(\tilde \Delta_\alpha f_2-\tilde \Delta_{\alpha-\beta} f_2)(x) \delta_\alpha f_3^m(x)\frac{d\alpha}{\alpha}\\
				=&P_{3,2,1}+P_{3,2,2}+P_{3,2,3}.
			\end{align*}
			We estimate $P_{3,2,1}$ similarly as $P_{2,1}$, which yields
			\begin{equation}\label{1lP321}
				\begin{aligned}
					&|P_{3,2,1}|
					\lesssim  t^{-\kappa-m} |\beta|^\kappa \|f_2\|_{T,*}([f_1]_{T,*}\interleave f_3\interleave_{T}+[f_1]_{T}\interleave f_3\interleave_{T,*}).
				\end{aligned}
			\end{equation}	
			Moreover, we have 
			\begin{align*}
				|P_{3,2,2}|&\lesssim [f_1]_{T}\int_{\mathbb{S}} |\delta_\alpha\delta_{-\beta} f_2(x)\delta_\alpha f_3 ^m(x)|\frac{d\alpha}{|\alpha|^2}.
			\end{align*}
			By Lemma \ref{douint}, we obtain 
			\begin{align*}
				&\|\delta_\alpha\delta_{-\beta} f_2\|_{L^\infty}\lesssim |\beta|^\kappa \min\left\{|\alpha|^{\frac{1}{4}}\|\partial_x f_2\|_{\dot C^{\kappa-\frac{3}{4}}},|\alpha|^\frac{3}{4}\|\partial_x f_2\|_{\dot C^{\kappa-\frac{1}{4}}}\right\},\\
				&\|\delta_\alpha f_3^m\|_{L^\infty}\lesssim \min\left\{|\alpha|^\frac{1}{4}\|f_3\|_{\dot C^{\frac{1}{4}+m}}, |\alpha|^\frac{3}{4}\|f_3\|_{\dot C^{\frac{3}{4}+m}}\right\}.
			\end{align*}
			It  follows
			\begin{equation}\label{1lP322}
				\begin{aligned}
					&|P_{3,2,2}|\\
					&\lesssim |\beta|^\kappa \|f_1\|_{L^\infty} \int_{\mathbb{S}}  \min\left\{\|\partial_x f_2\|_{\dot C^{\kappa-\frac{3}{4}}},|\alpha|^{\frac{1}{2}}\|\partial_x f_2\|_{\dot C^{\kappa-\frac{1}{4}}}\right\}\min\{\|f_3\|_{\dot C^{\frac{1}{4}+m}}, |\alpha|^\frac{1}{2}\|f_3\|_{\dot C^{\frac{3}{4}+m}}\}\frac{d\alpha}{|\alpha|^\frac{3}{2}}\\
					&\lesssim |\beta|^\kappa \|f_1\|_{L^\infty} (\|\partial_x f_2\|_{\dot C^{\kappa-\frac{3}{4}}}\|\partial_x f_2\|_{\dot C^{\kappa-\frac{1}{4}}}\|f_3\|_{\dot C^{\frac{1}{4}+m}}\|f_3\|_{\dot C^{\frac{3}{4}+m}})^\frac{1}{2}.
				\end{aligned}
			\end{equation}
			Finally, similar to $P_{3,1}$, by Lemma \ref{dede},
			\begin{align*}
				|\alpha|^\frac{1}{4}\frac{\|\tilde \Delta_\alpha f_2-\tilde \Delta_{\alpha-\beta} f_2\|_{L^\infty}}{|\beta|^\kappa}\lesssim  \|\partial_x f_2\|_{\dot C^{\kappa-\frac{1}{4}}}.
			\end{align*}
			Hence, 
			\begin{align*}
				|P_{3,2,3}|&\lesssim \int_{\mathbb{S}}| f_1(x-\beta,\alpha-\beta)(\tilde \Delta_\alpha f_2-\tilde \Delta_{\alpha-\beta} f_2)(x) \delta_\alpha f_3^m(x)|\frac{d\alpha}{|\tilde \alpha|}\\&\lesssim |\beta|^\kappa \|f_1\|_{L^\infty_{x,\alpha}} \|\partial_x f_2\|_{\dot C^{\kappa-\frac{1}{4}}}\int_{\mathbb{S}}|\delta_\alpha f_3^m|\frac{d\alpha}{|\tilde\alpha|^\frac{5}{4}}\\
				&\lesssim t^{-{m+\kappa }} |\beta|^\kappa [f_1]_T\|f_2\|_{T,*}\interleave f_3\interleave_{T,*}.
			\end{align*}
			Combining this with \eqref{1lP321} and \eqref{1lP322}, we deduce that 
			\begin{align}\label{1lP32}
				t^{m+\kappa}|\beta|^{-\kappa}|P_{3,2}|\lesssim [f_1]_T\|f_2\|_{T,*}\interleave f_3\interleave_{T,*}.
			\end{align}
			Then we estimate $$P_{3,3}=-\int_{\mathbb{S}}  f_1(x-\beta,\alpha-\beta)\tilde\E^{\alpha-\beta}f_2(x-\beta)\delta_\alpha f_3^m(x)\left(\frac{1 }{\alpha}-\frac{1}{\alpha-\beta}\right)d\alpha. $$
			By \eqref{eal111}, one has
			\begin{align*}
				|P_{3,3}|\lesssim \|f_1\|_{L^\infty_{x,\alpha}}\|\partial_x f_2\|_{\dot C^\frac{\kappa}{2}} \|f_3\|_{\dot C^{m+\frac{\kappa}{2}}}\int_{\mathbb{R}}|\alpha-\beta|^\frac{\kappa}{2}|\alpha|^\frac{\kappa}{2}\left|\frac{1 }{\alpha}-\frac{1}{\alpha-\beta}\right| d\alpha.
			\end{align*}
			Note that 
			\begin{align*}
				\int_{\mathbb{R}}|\alpha-\beta|^\frac{\kappa}{2}|\alpha|^\frac{\kappa}{2}\left|\frac{1 }{\alpha}-\frac{1}{\alpha-\beta}\right| d\alpha\lesssim |\beta|^\kappa\int_{\mathbb{R}}|\alpha-1|^\frac{\kappa}{2}|\alpha|^\frac{\kappa}{2}\left|\frac{1 }{\alpha}-\frac{1}{\alpha-1}\right| d\alpha\lesssim |\beta|^\kappa.
			\end{align*}
			This yields that 
			\begin{align}\label{1lP33}
				|P_{3,3}|\lesssim |\beta|^\kappa \|f_1\|_{L^\infty_{x,\alpha}}\|\partial_x f_2\|_{\dot C^\frac{\kappa}{2}} \|f_3\|_{\dot C^{m+\frac{\kappa}{2}}}\lesssim|\beta|^\kappa t^{-(m+\kappa)}[f_1]_T\|f_2\|_{T,*}\interleave f_3\interleave_{T,*}.
			\end{align}
			Combining \eqref{1lP31},  \eqref{1lP32}, and \eqref{1lP33}, we deduce 
			\begin{align}\label{1lP3}
				t^{m+\kappa}|\beta |^{-\kappa}|P_{3}|\lesssim [f_1]_T\|f_2\|_{T,*}\interleave f_3\interleave_{T,*}.
			\end{align}
			Hence we obtain from \eqref{1P1}, \eqref{1P2} and \eqref{1lP3} that 
			\begin{align*}
				&	\|\delta_\beta \partial_x f^m\|_{L^\infty}\lesssim 	|P_1|+|P_2|+|P_3|\\
				&\quad\quad\quad\lesssim |\beta|^\kappa t^{-(m+\kappa)}\|f_2\|_{T,*}([f_1]_{T,*}\interleave f_3\interleave_{T}+[f_1]_{T}\interleave f_3\interleave_{T,*}+T^\frac{1}{2}[f_1]_T\interleave f_3\interleave_{T}).
			\end{align*}
			This completes the proof of \eqref{res1}.\vspace{0.3cm}\\
			Then we prove \eqref{res2}. If $$
			\int_{\mathbb{S}} f_1(t,x,\alpha)\tilde{\E}^\alpha f_2(t,x)\frac{d\alpha}{\tilde \alpha}=0
			$$
			holds, then 
			\begin{align*}
				f(t,x)=-\int_{\mathbb{S}}f_1(t,x,\alpha)\tilde{\E}^\alpha f_2(t,x)\delta_\alpha f_3(t,x)\frac{d\alpha}{\tilde \alpha}.
			\end{align*}
			We note the finite difference operator in $f_3$ allows $f_3$ to share derivative, hence we avoid the endpoint norm $\|f_3\|_{L^\infty}$ in the estimate \eqref{res2}. The idea of the proof is similar to that of \eqref{res1}. For $m\in\mathbb{N}$, we write 
			\begin{align*}
				\partial_x^m f(x)=-\sum_{m_1+m_2+m_3=m}\int_{\mathbb{S}}f_1^{m_1}(x,\alpha)\tilde{\E}^\alpha f_2^{m_2}(x)\delta_\alpha f_3^{m_3}(x)\frac{d\alpha}{\tilde \alpha}.
			\end{align*}
			We drop the summation for $m_1+m_2+m_3=m$ with a slight abuse of notation. \\
			For any $\beta \neq 0$, we can write 
			\begin{align*}
				\delta_\beta\partial_x^m f(x)=&\int_{\mathbb{S}}\delta_\beta f_1^{m_1}(x,\alpha)\tilde{\E}^\alpha f_2^{m_2}(x)\delta_\alpha f_3^{m_3}(x)\frac{d\alpha}{\tilde \alpha}\\
				&+\int_{\mathbb{S}} f_1^{m_1}(x-\beta,\alpha)\delta_\beta \tilde{\E}^\alpha f_2^{m_2}(x)\delta_\alpha f_3^{m_3}(x)\frac{d\alpha}{\tilde \alpha}\\
				&+\int_{\mathbb{S}} f_1^{m_1}(x-\beta,\alpha) \tilde{\E}^\alpha f_2^{m_2}(x-\beta)\delta_\beta\delta_\alpha f_3^{m_3}(x)\frac{d\alpha}{\tilde \alpha}\\
				:=&\tilde P_1+\tilde P_2+\tilde P_3.
			\end{align*}
			Note that 
			\begin{align*}
				&|\alpha|^\frac{1}{4}|\delta_\beta f_1^{m_1}(x,\alpha)|\lesssim |\beta|^\kappa\|f_1^{m_1}\|_{\kappa,\frac{1}{4}}.
			\end{align*}
			Then it is straightforward to obtain
			\begin{equation}\label{P1}
				\begin{aligned}
					|\tilde P_1|&\lesssim |\beta|^\kappa t^{-(m_1+\kappa-\frac{1}{4})}[f_1]_{T,*} \|f_2^{m_2+1}\|_{\dot C^\frac{1}{8}}\int_{\mathbb{S}}|\delta_\alpha f_3^{m_3}|\frac{d\alpha}{|\alpha|^\frac{9}{8}}\\
					&\lesssim |\beta|^\kappa t^{-(m+\kappa)}[f_1]_{T,*}\|f_2\|_{T,*}\interleave f_3\interleave_{T,*}.
				\end{aligned}
			\end{equation}
			Then we deal with $P_{2}$. By \eqref{ealp}, we can write 
			\begin{align*}
				\tilde P_{2}=&\int f_1^{m_1}(x-\beta,\alpha)\partial_\alpha (\Delta_\alpha \delta_\beta f_2^{m_2})(x)\delta_\alpha f_3^{m_3}(x){d\alpha}\\
				&+\int f_1^{m_1}(x-\beta,\alpha)(\delta_\alpha\delta_\beta f_2^{m_2})(x) \delta_\alpha f_3^{m_3}(x)\left(\frac{1}{\alpha}-\frac{1}{\tilde\alpha}\right)\frac{d\alpha}{\alpha}\\
				:=&\tilde P_{2,1}+\tilde P_{2,2}.
			\end{align*}
			By Lemma \ref{lempesibp}, one has 
			\begin{align*}
				|\tilde P_{2,1}|\lesssim\prod_{+,-} \left\{ \|\delta_\beta f_2^{m_2}\|_{\dot C^{\frac{1}{4}\pm\varepsilon}}^\frac{1}{2}\sup_x\|f_1^{m_1}(x,\alpha)\delta_{\alpha} f_3^{m_3}(\tau)\|_{\dot C_\alpha^{\frac{3}{4}\pm\varepsilon}}^\frac{1}{2}\right\}.
			\end{align*}
			Here we use $\dot C^{a}_\alpha$ denote the $\dot C^a$ H\"{o}lder semi-norm in variable $\alpha$. By Lemma \ref{maininterpo}, we obtain that for $\gamma=\frac{3}{4}\pm\varepsilon$, 
			\begin{align*}
				&\frac{|	f_1^{m_1}(x,\alpha)\delta_{\alpha} f_3^{m_3}(x)-f_1^{m_1}(x,\alpha-z)\delta_{\alpha-z} f_3^{m_3}(x)|}{|z|^\gamma}
				\lesssim \|f_1^{m_1}\|_{\gamma,\frac{1}{4}}\|f_3^{m_3}\|_{\dot C^\frac{1}{4}}+\|f_1^{m_1}\|_{L^\infty}\|f_3^{m_3}\|_{\dot C^\gamma}.
			\end{align*}
			Hence we obtain that 
			\begin{align}
				|\tilde P_{2,1}|&\lesssim |\beta|^\kappa  \prod_{+,-} \left\{\|\partial_x f_2^{m_2}\|_{\dot C^{\kappa-\frac{3}{4}\pm\varepsilon}}\left(\|f_1^{m_1}\|_{\frac{3}{4}\pm\varepsilon,\frac{1}{4}}\|f_3^{m_3}\|_{\dot C^\frac{1}{4}}+\|f_1^{m_1}\|_{L^\infty}\|f_3^{m_3}\|_{\dot C^{\frac{3}{4}\pm\varepsilon}} \right) \right\}\nonumber\\
				&\lesssim t^{-\kappa-m} |\beta|^\kappa [f_1]_T\|f_2\|_{T,*}\interleave f_3\interleave_{T,*}.\label{lP21}
			\end{align}
			For $\tilde P_{2,2}$, by \eqref{aaaaa},
			\begin{align*}
				|\tilde P_{2,2}|&\lesssim \|f_1^{m_1}\|_{L^\infty}\| f_3^{m_3}\|_{\dot C^\frac{1}{4}}\int _{\mathbb{R}} \|\delta_\alpha\delta_\beta f_2^{m_2}\|_{L^\infty}|\alpha|^\frac{1}{4}\min\{1,|\alpha|^{-2}\}d\alpha \\
				&\lesssim |\beta|^\kappa\|f_1^{m_1}\|_{L^\infty}\| f_3^{m_3}\|_{\dot C^\frac{1}{4}}\|f_2^{m_2}\|_{\dot C^{\kappa+\frac{1}{4}}}\int_{\mathbb{R}}  \min\{|\alpha|^\frac{1}{2},|\alpha|^{-\frac{3}{2}}\}d\alpha\\
				&\lesssim |\beta|^\kappa\|f_1^{m_1}\|_{L^\infty}\| f_3^{m_3}\|_{\dot C^\frac{1}{4}}\|f_2^{m_2}\|_{\dot C^{\kappa+\frac{1}{4}}}\\
				&\lesssim t^{-\kappa-m} |\beta|^\kappa [f_1]_T\|f_2\|_{T,*}\interleave f_3\interleave_{T,*}.
			\end{align*}
			Combining this with \eqref{lP21}, we obtain that for any $t\in[0,T]$,
			\begin{align}\label{lP2}
				t^{m+\kappa } |\beta|^{-\kappa} |\tilde P_{2}|\lesssim [f_1]_T\|f_2\|_{T,*}\interleave f_3\interleave_{T,*}.
			\end{align}
			Finally, for $\tilde P_3$, we consider two cases.
			If $m_1+m_2\neq 0$, we directly have
			\begin{align*}
				|\tilde P_3|&\lesssim \|f_1^{m_1}\|_{L^\infty}\|\partial_x f_2\|_{\dot C^{m_2+\frac{3}{4}}}\|f_3\|_{\dot C^{m_3+\kappa+\frac{1}{4}}}\lesssim |\beta|^\kappa t^{-(m+\kappa)}[f_1]_T\|f_2\|_{T,*}\interleave f_3\interleave_{T,*}.
			\end{align*}
			When $m_1=m_2=0$, by the condition $\int f_1(x,\alpha)E^\alpha f_2(x)\frac{d\alpha}{\alpha}=0$, the estimate is the same as $P_3$. Hence,
			\begin{align}\label{lP3}
				t^{m+\kappa}|\beta |^{-\kappa}|\tilde P_{3}|\lesssim [f_1]_T\|f_2\|_{T,*}\interleave f_3\interleave_{T,*}.
			\end{align}
			Hence we obtain from \eqref{P1}, \eqref{lP2} and \eqref{lP3} that 
			\begin{align*}
				\|\delta_\beta \partial_x f^m\|_{L^\infty}\lesssim 	|\tilde P_1|+|\tilde P_2|+|\tilde P_3|\lesssim |\beta|^\kappa t^{-(m+\kappa)}[f_1]_T\|f_2\|_{T,*}\interleave f_3\interleave_{T,*}.
			\end{align*}
			This holds for any $\beta \neq 0$. Then we  complete the proof of \eqref{res2}. \vspace{0.5cm}
		\end{proof}\\
		\begin{proof}[Proof of Lemma \ref{lemnonpes}]
			In this proof, we consider for any fixed $t\in[0,T]$. We first prove \eqref{resn}.	Rewrite  \eqref{defnonpes} to the following form
			\begin{align*}
				\mathcal{N}(X)(x)=&-\sum \int _{\mathbb{R}}H(\tilde \Delta_\alpha X(x))\tilde{\E}^\alpha X_i(x) (\mathbf{T}(|\partial_x X|)\partial_xX_j)(x-\alpha)\frac{d\alpha}{\alpha}.
			\end{align*}
			where the sum is for some $i,j=1,2$ and $H(x)=C_{i_1,i_2,i_3}\frac{x_{i_1}x_{i_2}x_{i_3}}{|x|^4}$, $i_1,i_2,i_3=1,2$.  With a slight abuse of notation, we omit the summation symbol and subscript $i,j$ in our proof, and we do not distinguish $\partial_x f$ and $f'$.
			Note that $H(\tilde \Delta_\alpha X(x))\tilde{\E}^\alpha X_i(x)=\partial_\alpha G(\delta_\alpha X(x))-\frac{1}{2}\cot \frac{\alpha}{2}$, hence
			\begin{align*}
				\int_{\mathbb{S}}H(\tilde {\Delta}_\alpha X(t,x))\tilde \E^\alpha X_i \frac{d\alpha}{\tilde \alpha}=0.
			\end{align*}
			By \eqref{eqpesk}, and applying \eqref{res2} with 
			\begin{align*}
				f_1(t,x,\alpha)=H(\tilde {\Delta}_\alpha X(t,x)),\ \ \ \ f_2(t,x)=X_i(t,x),\ \ \ \ f_3(t,x)=\big(\mathbf{T}(|\partial_x X|)\partial_xX_j\big)(t,x),
			\end{align*}
			we obtain 
			\begin{align*}
				\sup_{t\in[0,T]}t^{m+\kappa }\|\mathcal{N}(X)(t)\|_{\dot C^{m+\kappa }}\lesssim [f_1]_T\|f_2\|_{T,*}\interleave f_3\interleave_{T,*}.
			\end{align*}
			Denote $\tilde{H}[X](t,x,\alpha)=H(\tilde \Delta_\alpha X)(t,x)$. We first consider 
			\begin{align*}
				[f_1]_T=[\tilde{H}[X](t)]_T,
			\end{align*}
			with $[\cdot]_T$ defined in \eqref{defTn}. For any $k\in\mathbb{N}$, $\gamma\in[\frac{1}{2},\kappa]$,
			\begin{align}\label{Tn1}
				\|\partial_x^k\tilde{H}[X](t))\|_{\gamma,\frac{1}{4}}=\sup_{\alpha,z\in\mathbb{S}}|\alpha|^\frac{1}{4}\left(\|\partial_x^kH(\tilde {\Delta}_\alpha X(t,\cdot))\|_{\dot C^\gamma}+\frac{\|\partial_x^kH(\tilde {\Delta}_\alpha X(t,.))-\partial_x^kH(\tilde {\Delta}_{\alpha-z} X(t,.))\|_{L^\infty}}{|z|^\gamma}\right).
			\end{align}
			By definition, it is easy to check that for any $0\leq k\leq m$, $0\leq t\leq T$,
			\begin{equation}\label{Tn2}
				\sup_{\alpha\in\mathbb{S}}|\alpha|^\frac{1}{4}\|\partial_x^kH(\tilde {\Delta}_\alpha X(t,\cdot))\|_{\dot C^\gamma}\lesssim  t^{-(k+\gamma-\frac{1}{4})}(1+\mathbf{\Theta}_X(T))^{m+2}\| X\|_{T,*}(1+\|X\|_T)^{m+1}.
			\end{equation}
			Moreover, by Lemma \ref{dede}, 
			\begin{equation}\label{Tn3}
				\begin{aligned}
					&\sup_{\alpha,z\in\mathbb{S}}|\alpha|^\frac{1}{4}\frac{\|\partial_x^kH(\tilde {\Delta}_\alpha X(t,x))-\partial_x^kH(\tilde {\Delta}_{\alpha-z} X(t,x))\|_{L^\infty}}{|z|^\gamma}\\
					&\quad\lesssim t^{-(k+\gamma-\frac{1}{4})}(1+\mathbf{\Theta}_X(T))^{m+2}\| X\|_{T,*}(1+\|X\|_T)^{m+1}.
				\end{aligned}
			\end{equation}
			Combine \eqref{Tn1}, \eqref{Tn2} with \eqref{Tn3}, we obtain 
			\begin{align}\label{Hs}
				[\tilde{H}[X](t)]_{T,*}\lesssim (1+\mathbf{\Theta}_X(T))^{m+2}\| X\|_{T,*}(1+\| X\|_T)^{m+1}.
			\end{align}
			Note that $\mathbf{T}(|\partial_x X|)\partial_xX_j=\mathcal{T}(|\partial_x X|)\frac{\partial_xX_j}{|\partial_x X|}$, then for any $k\in\mathbb{N}$, $a\in(0,1)$,
			\begin{align*}
				\|\mathbf{T}(|\partial_x X|)\partial_xX_j\|_{\dot C^{k+a}}\lesssim \mathfrak{C}(\|X\|_{L^\infty_T\dot W^{1,\infty}}) t^{-(k+a)} \mathbf{\Theta}_X(T) \| X\|_{T,*}(1+\|X\|_{T})^{k+1}.
			\end{align*}
			This yields 
			\begin{align*}\label{Ts}
				\interleave\mathbf{T}(|\partial_x X|)\partial_xX_j\interleave_{T,*}\lesssim \mathfrak{C}(\|X\|_{L^\infty_T\dot W^{1,\infty}})\mathbf{\Theta}_X(T) \| X\|_{T,*}(1+\|X\|_{T})^{m+1}.
			\end{align*}
			Then we obtain 
			\begin{align*}
				&\sup_{t\in[0,T]}(t^{\kappa}\|\mathcal{N}(X)(t)\|_{\dot C^{\kappa}}+	t^{m+\kappa }\|\mathcal{N}(X)(t)\|_{\dot C^{m+\kappa }})\lesssim \mathfrak{M}_T\| X\|_{T,*}^2(1+\| X\|_T)^{2(m+1)}.
			\end{align*}
			This completes the proof of \eqref{resn}. 
			
			Then we deal with \eqref{rendiff}.
			Note that 
			\begin{align*}
				(\mathcal{N}(Z)-\mathcal{N}(Y))(x)=& \int_{\mathbb{R}} (H(\tilde \Delta_\alpha Y)-H(\tilde \Delta_\alpha Z))(x)\tilde{\E}^\alpha Y(x) (\mathbf{T}(|\partial_x Y|)\partial_xY)(x-\alpha)\frac{d\alpha}{\alpha}\\
				&+ \int_{\mathbb{R}} H(\tilde \Delta_\alpha Z)(x)\tilde{\E}^\alpha (Y-Z)(x)  (\mathbf{T}(|\partial_x Y|)\partial_xY)(x-\alpha)\frac{d\alpha}{\alpha}\\
				&+\int_{\mathbb{R}} H(\tilde \Delta_\alpha Z)(x)\tilde{\E}^\alpha Z(x) (\mathbf{T}(|\partial_x Y|)\partial_xY_j-\mathbf{T}(|\partial_x Z|)\partial_xZ_j)(x-\alpha)\frac{d\alpha}{\alpha}\\
				:=&B_1+B_2+B_3.
			\end{align*}
			We apply \eqref{res1} with $(f_1,f_2,f_3)$ equal to $((H(\tilde \Delta_\alpha Y)-H(\tilde \Delta_\alpha Z))(x),Y(x),\mathbf{T}(|\partial_x Y|)\partial_xY(x))$, $(H(\tilde \Delta_\alpha Z)(x),(Y-Z)(x),\mathbf{T}(|\partial_x Y|)\partial_xY(x))$ and $(H(\tilde \Delta_\alpha Z)(x),Z(x),(\mathbf{T}(|\partial_x Y|)\partial_xY-\mathbf{T}(|\partial_x Z|)\partial_xZ)(x))$, respectively. Then we derive
			\begin{align*}
				&\sup_{t\in[0,T]}t^{m+\kappa}\|\partial_x^m B_1(t)\|_{\dot C^\kappa}\\
				&\quad\lesssim \|Y\|_{T,*}\left([(\tilde{H}[Y]-\tilde{H}[Z])]_{T,*}\interleave\mathbf{T}(|\partial_x Y|)\partial_xY\interleave_{T}+[(\tilde{H}[Y]-\tilde{H}[Z])]_T\interleave\mathbf{T}(|\partial_x Y|)\partial_xY\interleave_{T,*}\right.\\
				&\quad\quad\quad\quad\quad\quad\quad\quad\left.+T^{\frac{1}{2}}[(\tilde{H}[Y]-\tilde{H}[Z])]_T\interleave\mathbf{T}(|\partial_x Y|)\partial_xY\interleave_{T}\right),\\
				&\sup_{t\in[0,T]}t^{m+\kappa}	\|\partial_x^m B_2(t)\|_{\dot C^\kappa}\\
				&\quad\lesssim \| W\|_{T,*}\left([\tilde{H}[Z]]_{T,*}\interleave\mathbf{T}(|\partial_x Y|)\partial_xY\interleave_{T}+[\tilde{H}[Z]]_T\interleave\mathbf{T}(|\partial_x Y|)\partial_xY\interleave_{T,*}+T^{\frac{1}{2}}[\tilde{H}[Z]]_T\interleave\mathbf{T}(|\partial_x Y|)\partial_xY\interleave_{T}\right),\\
				&\sup_{t\in[0,T]}t^{m+\kappa}	\|\partial_x^m B_3(t)\|_{\dot C^\kappa}\\
				&\quad\lesssim  \| Z\|_{T,*}\left([\tilde{H}[Z]]_{T,*}\interleave(\mathbf{T}(|\partial_x Y|)\partial_xY-\mathbf{T}(|\partial_x Z|)\partial_xZ)\interleave_T+[\tilde{H}[Z]] _T\interleave(\mathbf{T}(|\partial_x Y|)\partial_xY-\mathbf{T}(|\partial_x Z|)\partial_xZ)\interleave_{T,*}\right.\\
				&\quad\quad\quad\quad\quad\quad\quad\quad\left.+T^{\frac{1}{2}}[\tilde{H}[Z]] _T\|Z\|_{T,*}\interleave(\mathbf{T}(|\partial_x Y|)\partial_xY-\mathbf{T}(|\partial_x Z|)\partial_xZ)\interleave_T\right).
			\end{align*}
			Denote $W=Y-Z$. We prove similarly as \eqref{Hs} that 
			\begin{align*}
				[(\tilde{H}[Y]-\tilde{H}[Z]) ]_{T,*}\lesssim (1+\mathbf{\Theta}_Y(T)+\mathbf{\Theta}_Z(T))^{m+1} \|W\|_{T}(1+\| (Y,Z)\|_T)^{m+1},
			\end{align*}
			and 
			\begin{align*}
				&\interleave\mathbf{T}(|\partial_x Y|)\partial_xY-\mathbf{T}(|\partial_x Z|)\partial_xZ\interleave_{T,*}\\
                &\quad\quad\quad\lesssim \mathfrak{C}(\|(Y,Z)\|_{L^\infty_T\dot W^{1,\infty}})(1+\mathbf{\Theta}_Y(T)+\mathbf{\Theta}_Z(T))\|W\|_{T}(1+\|(Y,Z)\|_T)^{m+1}.
			\end{align*}
			Combining the above estimates to obtain
			\begin{align*}
				\sup_{t\in[0,T]}t^{m+\kappa}&(\|\partial_x^m B_1(t)\|_{\dot C^\kappa}+\|\partial_x^m B_2(t)\|_{\dot C^\kappa}+\|\partial_x^m B_3(t)\|_{\dot C^\kappa})\\
				&\lesssim \mathfrak{M}_T\| W\|_{T}(\|(Y,Z)\|_{T,*}+T^{\frac{1}{2}}\|(Y,Z)\|_{T})(1+\|(Y,Z)\|_T)^{2(m+1)}.
			\end{align*} 
			This completes the proof of \eqref{rendiff}.
		\end{proof}
		~~\vspace{0.3cm}\\	
		\begin{proof}[Proof of Lemma \ref{lempesR}] The proof of \textbf{(i)} follows the proof of \textbf{(ii)}  with $(Y,Z)=(X,0)$, so we only need to prove \textbf{(ii)}. For simplicity, we only estimate the norm $\sup_{t\in[0,T]}t^\kappa\|(\M(\partial_x Y)-\M(\partial_x Z))(t)\|_{\dot C^{\kappa-\frac{1}{2}}}$. The estimates for higher order derivatives are parallel. We denote $W=Y-Z$, $\vec{V}=(Y,Z)$, $W'=\partial_x(Y-Z)$ and $\vec{V}'=(\partial_xY,\partial_xZ)$ in this proof. During the proof, we fix $t\in[0,T]$ and drop the time variable $t$. Note that
			\begin{align*}
				\M(\partial_xY)(x)-\M(\partial_xZ)(x)=-\frac{1}{\pi}\int_{\mathbb{R}} \mathfrak{R}(\alpha,s)\frac{d\alpha}{|\alpha|^{\frac{3}{2}}}.
			\end{align*}
			Here we denote 
			\begin{align*}
				&\mathfrak{R}(\alpha,s)= \delta_\alpha (\mathbf{T}(|\partial_x Y|)-\mathbf{T}(|\partial_x Z|))(x)\delta_\alpha \partial_x Y(x)+\delta_\alpha (\mathbf{T}(|\partial_x Z|))(x)\delta_\alpha \partial_x W(x)\\
				&\ \ \quad\quad\quad+\partial_xW(x)D_{\mathbf{T},\alpha}(\partial_x Y)(x)+\partial_x Z(x)(D_{\mathbf{T},\alpha}(\partial_x Y)-D_{\mathbf{T},\alpha}(\partial_x Z))(x)\end{align*}
			with 
			\begin{align*}
				&D_{\mathbf{T},\alpha}(A)=\delta_\alpha (\mathbf{T}(|A|))-\delta_\alpha A\cdot \nabla \big(\mathbf{T}(|A|)\big),\ \ \ \forall \ A:\mathbb{R}\to \mathbb{R}^2.
			\end{align*}
			By Lemma \ref{maininterpo},
			\begin{equation}\label{eR1}
				\begin{aligned}
					&\|\mathfrak{R}(\alpha,\cdot)\|_{\dot C^{\kappa-\frac{1}{2}}}\lesssim
					\|\mathbf{T}(|\partial_x Y|)-\mathbf{T}(|\partial_x Z|)\|_{\dot C^{\kappa-\frac{1}{2}}}\|\delta_\alpha \partial_x Y\|_{L^\infty}+\|\delta_\alpha (\mathbf{T}(|\partial_x Y|)-\mathbf{T}(|\partial_x Z|))\|_{L^\infty}\|\partial_x Y\|_{\dot C^{\kappa-\frac{1}{2}}}\\
					&\quad\quad+\| \mathbf{T}(|\partial_x Z|)\|_{\dot C^{\kappa-\frac{1}{2}}}\|\delta_\alpha \partial_x W\|_{L^\infty}+\|\delta_\alpha (\mathbf{T}(|\partial_x Z|))\|_{L^\infty}\|\partial_x W\|_{\dot C^{\kappa-\frac{1}{2}}}\\
					&\quad\quad+\|\partial_xW\|_{\dot C^{\kappa-\frac{1}{2}}}\|D_{\mathbf{T},\alpha}(\partial_x Y)\|_{L^\infty}+\|\partial_xW\|_{L^\infty}\|D_{\mathbf{T},\alpha}(\partial_x Y)\|_{\dot C^{\kappa-\frac{1}{2}}}\\
					&\quad\quad+\|\partial_x Z\|_{\dot C^{\kappa-\frac{1}{2}}}\|(D_{\mathbf{T},\alpha}(\partial_x Y)-D_{\mathbf{T},\alpha}(\partial_x Z))\|_{L^\infty}+\|\partial_x Z\|_{L^\infty}\|(D_{\mathbf{T},\alpha}(\partial_x Y)-D_{\mathbf{T},\alpha}(\partial_x Z))\|_{\dot C^{\kappa-\frac{1}{2}}},
				\end{aligned}
			\end{equation}
			where all $\dot C^{\kappa-\frac{1}{2}}$ denote H\"{o}lder semi-norm of $x$ variable. We remark that all the estimates related to the tension $\mathbf{T}$ and its derivatives will result in a factor $\mathfrak{M}_T$. We omit this factor in the following to avoid redundancy. 
			Lemma \ref{lemcom} implies that 
			\begin{align*}
				&\|\mathbf{T}(|\partial_x Y|)-\mathbf{T}(|\partial_x Z|)\|_{\dot C^{\kappa-\frac{1}{2}}}\lesssim \|W'\|_{\dot C^{\kappa-\frac{1}{2}}}+\| W'\|_{L^\infty}\|\vec{V}'\|_{\dot C^{\kappa-\frac{1}{2}}},\\
				&\| \mathbf{T}(|\partial_x Z|)\|_{\dot C^{\kappa-\frac{1}{2}}}\lesssim \| Z'\|_{\dot C^{\kappa-\frac{1}{2}}},\ \ \ \ \ \ \ \ \|D_{\mathbf{T},\alpha}(\partial_x Y)\|_{L^\infty}\lesssim \|Y'\|_{L^\infty}\|\delta_\alpha Y'\|_{L^\infty}^2,\\
				&\|D_{\mathbf{T},\alpha}(\partial_x Y)\|_{\dot C^{\kappa-\frac{1}{2}}}\lesssim \|Y'\|_{L^\infty}\|\delta_\alpha Y'\|_{\dot C^{\kappa-\frac{1}{2}}}\|\delta_\alpha Y'\|_{L^\infty}+\| Y'\|_{\dot C^{\kappa-\frac{1}{2}}}\|\delta_\alpha Y'\|_{L^\infty}^2,\\
				&\|(D_{\mathbf{T},\alpha}(\partial_x Y)-D_{\mathbf{T},\alpha}(\partial_x Z))\|_{L^\infty}\lesssim \|\vec{V}'\|_{L^\infty}\|\delta_\alpha W'\|_{L^\infty}\|\delta_\alpha \vec{V}'\|_{L^\infty}+\|\delta_\alpha \vec{V}'\|_{L^\infty}^2\|W'\|_{L^\infty},\\
				&\|(D_{\mathbf{T},\alpha}(\partial_x Y)-D_{\mathbf{T},\alpha}(\partial_x Z))\|_{\dot C^{\kappa-\frac{1}{2}}}\lesssim\|\vec{V}'\|_{L^\infty}\|\|\delta_\alpha W'\|_{\dot C^{\kappa-\frac{1}{2}}}\|\delta_\alpha \vec{V}'\|_{L^\infty}\\
				&\quad\quad +\|\vec{V}'\|_{L^\infty}\|\delta_\alpha W'\|_{L^\infty}\|\delta_\alpha \vec{V}'\|_{\dot C^{\kappa-\frac{1}{2}}}+\|\delta_\alpha W'\|_{L^\infty}\|\delta_\alpha \vec{V}'\|_{L^\infty}\|\vec{V}'\|_{\dot C^{\kappa-\frac{1}{2}}}\\
				&\quad\quad+\|\delta_\alpha \vec{V}'\|_{\dot C^{\kappa-\frac{1}{2}}}\|\delta_\alpha \vec{V}'\|_{L^\infty}\| W'\|_{L^\infty}+\|\delta_\alpha \vec{V}'\|_{L^\infty}^2(\|W'\|_{\dot C^{\kappa-\frac{1}{2}}}+\| W'\|_{L^\infty}\|\vec{V}'\|_{\dot C^{\kappa-\frac{1}{2}}}),
			\end{align*}
			where all the norms are defined with respect to $x$ variable. Hence, 
			\begin{equation*}
				\label{eR2}
				\begin{aligned}
					&\|D_{\mathbf{T},\alpha}(\partial_x Y)\|_{L^\infty}\lesssim \min\{|\alpha|^\frac{1}{5}\|Y'\|_{\dot C^\frac{1}{10}}^2,|\alpha|\|Y'\|_{\dot C^\frac{1}{2}}^2\} \lesssim \min\{|\alpha|^\frac{1}{5}t^{-\frac{1}{5}},|\alpha|t^{-1}\}\|Y\|_{T,*}^2\|Y\|_{T}.\\
					&\|D_{\mathbf{T},\alpha}(\partial_x Y)\|_{\dot C^{\kappa-\frac{1}{2}}}
					\lesssim  \min\{|\alpha|^{\frac{1}{5}}\|Y'\|_{\dot C^{\kappa-\frac{1}{2}}}\| Y'\|_{\dot C^{\frac{1}{5}}}(1+\| Y'\|_{L^\infty}),|\alpha|(\|Y'\|_{\dot C^{\kappa}}\|Y'\|_{\dot C^{\frac{1}{2}}}+\| Y'\|_{\dot C^{\kappa-\frac{1}{2}}}\|Y'\|_{\dot C^\frac{1}{2}}^2)\}\\
					&\quad\quad\quad\quad \quad\quad\quad\quad \lesssim 	t^{-(\kappa-\frac{1}{2})}\min\{|\alpha|^{\frac{1}{10}}t^{-\frac{1}{10}},|\alpha|t^{-1}\}\|Y\|_{T,*}^2(1+\|Y\|_T)^2,\\
					&\|(D_{\mathbf{T},\alpha}(\partial_x Y)-D_{\mathbf{T},\alpha}(\partial_x Z))\|_{L^\infty}\\
					&\quad\quad \lesssim\min\left\{|\alpha|^\frac{1}{5}\|W'\|_{L^\infty}\|\vec{V}'\|_{\dot C ^\frac{1}{10}}^2(1+\|\vec{V}'\|_{L^\infty}), |\alpha|(\|W'\|_{\dot C^\frac{1}{2}}\|\vec{V}'\|_{\dot C^\frac{1}{2}}+\|\vec{V}'\|_{\dot C^\frac{1}{2}}^2\|W'\|_{L^\infty})\right\} \\
					&\quad\quad\lesssim\min\{|\alpha|^\frac{1}{5}t^{-\frac{1}{5}},|\alpha|t^{-1}\}(\|\vec{V}\|_{T,*}\|W\|_T+\|W\|_{T,*})\|\vec{V}\|_{T,*}(1+\|\vec{V}\|_T).
				\end{aligned}
			\end{equation*}
			Similarly, we have both
			\begin{align*}
				&\|(D_{\mathbf{T},\alpha}(\partial_x Y)-D_{\mathbf{T},\alpha}(\partial_x Z))\|_{\dot C^{\kappa-\frac{1}{2}}}\\
				&\quad\lesssim |\alpha|^\frac{1}{5}\big(\|W'\|_{\dot C^{\kappa-\frac{1}{2}}}\|\vec{V}'\|_{\dot C^\frac{1}{5}}+\|W'\|_{\dot C^\frac{1}{5}}\|\vec{V}'\|_{\dot C^{\kappa-\frac{1}{2}}}\big)(1+\|\vec{V}'\|_{L^\infty})+|\alpha|^\frac{1}{5}\|W'\|_{L^\infty}\|\vec{V}'\|_{\dot C^{\kappa-\frac{1}{2}}}\|\vec{V}'\|_{\dot C^\frac{1}{5}}\\
				&\quad\lesssim|\alpha|^\frac{1}{5} t^{-(\kappa-\frac{1}{2}+\frac{1}{5})}(\|\vec{V}\|_{T,*}\|W\|_T+\|W\|_{T,*})\|\vec{V}\|_{T,*}(1+\|\vec{V}\|_T),
			\end{align*}
			and 
			\begin{align*}
				&\|(D_{\mathbf{T},\alpha}(\partial_x Y)-D_{\mathbf{T},\alpha}(\partial_x Z))\|_{\dot C^{\kappa-\frac{1}{2}}}\\
				&\quad\lesssim |\alpha|\left\{\|W'\|_{\dot C^{\kappa}}\|\vec{V}'\|_{\dot C^\frac{1}{2}}+\|W'\|_{\dot C^\frac{1}{2}}\|\vec{V}'\|_{\dot C^{\kappa}}(1+\|\vec{V}'\|_{L^\infty})\right.\\
				&\quad\quad\quad\left.+\|W'\|_{L^\infty}\|\vec{V}'\|_{\dot C^{\kappa}}\|\vec{V}'\|_{\dot C^\frac{1}{2}}+\|\vec{V}'\|_{\dot C^\frac{1}{2}}^2(\|W'\|_{\dot C^{\kappa-\frac{1}{2}}}+\| W'\|_{L^\infty}\|\vec{V}'\|_{\dot C^{\kappa-\frac{1}{2}}})\right\}\\
				&\quad\lesssim |\alpha|t^{-(\kappa+\frac{1}{2})}(\|\vec{V}\|_{T,*}\|W\|_T+\|W\|_{T,*})\|\vec{V}\|_{T,*}(1+\|\vec{V}\|_T)^2.
			\end{align*}
			These imply 
			\begin{align}\label{eR3}
				&\|(D_{\mathbf{T},\alpha}(\partial_x Y)-D_{\mathbf{T},\alpha}(\partial_x Z))\|_{\dot C^{\kappa-\frac{1}{2}}}\lesssim \frac{\min\{|\alpha|^{\frac{1}{5}}t^{-\frac{1}{5}},|\alpha|t^{-1}\}}{t^{\kappa-\frac{1}{2}}}(\|\vec{V}\|_{T,*}\|W\|_T+\|W\|_{T,*})\|\vec{V}\|_{T,*}(1+\|\vec{V}\|_T)^2.
			\end{align}
			Combining \eqref{eR1}-\eqref{eR3}, we obtain 
			\begin{align*}
				\|\mathfrak{R}(\alpha,\cdot)\|_{\dot C^{\kappa-\frac{1}{2}}}\lesssim \frac{\min\{|\alpha|^{\frac{1}{5}}t^{-\frac{1}{5}},|\alpha|t^{-1}\}}{t^{\kappa-\frac{1}{2}}}(\|\vec{V}\|_{T,*}\|W\|_T+\|W\|_{T,*})\|\vec{V}\|_{T,*}(1+\|\vec{V}\|_T)^2. 
			\end{align*}
			This yields 
			\begin{align*}
				\|\M(\partial_xY)-\M(\partial_xZ)\|_{\dot C^{\kappa-\frac{1}{2}}}&\lesssim\int_{\mathbb{R}}	\|\mathfrak{R}(\alpha,\cdot)\|_{\dot C^{\kappa-\frac{1}{2}}}\frac{d\alpha}{|\alpha|^{\frac{3}{2}}}\\
				&\lesssim t^{-\kappa }(\|\vec{V}\|_{T,*}\|W\|_T+\|W\|_{T,*})\|\vec{V}\|_{T,*}(1+\|\vec{V}\|_T)^2.
			\end{align*}
			This completes the proof of the lemma.
		\end{proof}\\
		\subsection{Estimates of nonlinear terms in 3D Peskin problem}
		The following lemma establishes the boundedness of non-convolution-type singular integral operators and will be useful for estimating the 3D Peskin problem. While the proof follows classical methods, we refer interested readers to \cite{Agmon1964} for a detailed exposition.
		\begin{lemma}
			\label{lemG11}
			Consider a function $\mathbf{G}:\mathbb{R}^2\to\mathbb{R}$.		Suppose that there exists $C_{\mathbf{G}}>0$ such that  \begin{align}\label{conG}|\mathbf{G}(\theta,\eta)|+|\theta-\eta||\nabla _{\theta,\eta} \mathbf{G}(\theta,\eta)|+|\theta-\eta|^2|\nabla _\theta\nabla _\eta \mathbf{G}(\theta,\eta)|\leq \frac{C_{\mathbf{G}}}{|\theta-\eta|},\ \ \ \forall \theta,\eta\in\mathbb{R}^2,\ \theta\neq \eta.\end{align}   Define 
			\begin{align*}
				\mathcal{M}f(\theta)=\int_{\mathbb{R}^2}\mathbf{G}(\theta,\eta)\nabla \cdot f(\eta)d\eta.
			\end{align*}
			Then for $\kappa\in(0,1)$, there holds 
			\begin{align*}
				\|\mathcal{M}f\|_{\dot C^\kappa}\lesssim C_{\mathbf{G}}\|f\|_{\dot C^\kappa}.
			\end{align*}
		\end{lemma}
		\begin{remark}\label{rmk3dpesk}
			The result of Lemma \ref{lemG11} on $\mathbb{R}^2$ can be extended parallel to $\mathbb{S}^2$. More precisely, for $\mathbf{G}:\mathbb{S}^2\to\mathbb{R}$, denote
			\begin{equation*}
				\tilde{\mathcal{M}} F(\widehat{\boldsymbol{x}})=\int_{\mathbb{S}^2}\mathbf{G}(\widehat{\boldsymbol{x}},\widehat{\boldsymbol{y}})\widetilde\nabla\cdot F(\widehat{\boldsymbol{y}})d \mu_{\mathbb{S}^2}(\widehat{\boldsymbol{y}}).
			\end{equation*}
			Suppose there exists $C_\mathbf{G}>0$ such that 
			\begin{equation*}
				|\mathbf{G}(\widehat{\boldsymbol{x}},\widehat{\boldsymbol{y}})|+|\widehat{\boldsymbol{x}}-\widehat{\boldsymbol{y}}||\widetilde\nabla_{ \widehat{\boldsymbol{x}},\widehat{\boldsymbol{y}}}\mathbf{G}(\widehat{\boldsymbol{x}},\widehat{\boldsymbol{y}})|+|\widehat{\boldsymbol{x}}-\widehat{\boldsymbol{y}}|^2|\widetilde\nabla_{ \widehat{\boldsymbol{x}}}\widetilde\nabla_{ \widehat{\boldsymbol{y}}}\mathbf{G}(\widehat{\boldsymbol{x}},\widehat{\boldsymbol{y}})|\leq \frac{C_{\mathbf G}}{|\widehat{\boldsymbol{x}}-\widehat{\boldsymbol{y}}|},
			\end{equation*}
			then for $\kappa\in(0,1)$,
			\begin{equation*}
				\|\tilde{\mathcal{M}}F\|_{C^{\kappa}(\mathbb{S}^2)}\lesssim C_{\mathbf{G}}\|F\|_{ C^{\kappa}(\mathbb{S}^2)}.
			\end{equation*}
		To prove this,	we can work on local charts of  $\mathbb{S}^2$ and reduce the problem onto $\mathbb{R}^2$, we omit the details of proof here.
		\end{remark}
		We now estimate the nonlinear term $N_1(F,\Phi)$, introduced in \eqref{3dpesnlt}, in appropriate time-weighted H\"{o}lder norms. This will be crucial for establishing the local well-posedness via a fixed-point argument in the space $\mathcal{Z}_{T,\Phi}^\sigma$. For simplicity, in the following of this section we will denote  
		\begin{equation*}
			\begin{aligned}
				&\mathcal{M}(X;\Phi)=\mathfrak{M}_1\|X-\Phi\|_{Z_T}(\|X-\Phi\|_{Z_T}+T^{\kappa}M_{\Phi})(1+\|X-\Phi\|_{Z_T}+m_{\Phi}+T^{\kappa}M_{\Phi})^{m+3},\\
				&\mathcal{M}(Y,Z;\Phi)=\mathfrak{M}_1\|Y-Z\|_{Z_T}(\|(Y-\Phi,Z-\Phi)\|_{Z_T}+T^{\eps}M_{\Phi})(1+\|(Y-\Phi,Z-\Phi)\|_{Z_T}+m_{\Phi}+T^{\kappa}M_{\Phi})^{m+3}.
                \end{aligned}
		\end{equation*}
		where $M_\Phi$ and $m_\Phi$ are defined in \eqref{Pesconphi}, $\|\cdot \|_{Z_T}$ is defined in \eqref{def3dpeszt}, and $\mathfrak{M}_1$ is defined in \eqref{defM1}. The following lemma provides both growth and Lipschitz-type bounds, with explicit dependence on the profile $\Phi$.
		
		\begin{lemma}\label{3dpeslemn1}
			Let $N_1(F, \Phi)$ be as defined in \eqref{3dpesnlt}. For any $T \in(0,1)$, and any functions $X, Y, Z \in \mathcal{Z}_{T,\Phi}^\sigma$, the following estimates hold:
			\begin{itemize}
				\item \textbf{(i) Nonlinear estimate:}
				\begin{equation*}
					\begin{aligned}
						\sum_{j=0,m}\sup_{t\in[0,T]}t^{j+\kappa}\|N_1(X,\Phi)(t)\|_{C^{j+\kappa}}\lesssim\mathcal{M}(X;\Phi),
					\end{aligned}
				\end{equation*}
				\item  \textbf{(ii) Lipschitz-type continuity:}
				\begin{equation*}
					\begin{aligned}
						\sum_{j=0,m}\sup_{t\in[0,T]}t^{j+\kappa}\|(N_1(Y,\Phi)-N_1(Z,\Phi))(t)\|_{C^{j+\kappa}}\lesssim\mathcal{M}(Y,Z;\Phi).
					\end{aligned}
				\end{equation*}
			\end{itemize}
		\end{lemma}
		\begin{proof}
			For simplicity, we fix $t\in[0,T]$ and omit the time variable in the proof. We start with the lowest order H\"{o}lder estimate. 
			Let
            $$
			\tilde{G}(\widehat{\boldsymbol{x}},\widehat{\boldsymbol{y}}) := G(X(\widehat{\boldsymbol{x}}) - X(\widehat{\boldsymbol{y}})) - G(\Phi(\widehat{\boldsymbol{x}}) - \Phi(\widehat{\boldsymbol{y}})).
			$$
			By \eqref{3dpesnc}, the condition \eqref{conG} holds. Let $C_{\tilde{G}} = C (1+\mathbf{\Theta}_0)^m\|\widetilde\nabla(X - \Phi)\|_{L^\infty} (1 + \|\widetilde\nabla(X - \Phi)\|_{L^\infty})$. Applying Remark \ref{rmk3dpesk}, we obtain:
			\begin{equation}\label{3dpesn1}
				\begin{aligned}
					&\|N_1\|_{C^{\kappa}}\lesssim \mathfrak{M}_1\|\widetilde\nabla(X-\Phi)\|_{L^\infty}(1+\|\widetilde\nabla(X-\Phi)\|_{L^\infty})\|(\mathbf{T}(|\widetilde\nabla X|)\widetilde\nabla X)\|_{C^\kappa}\lesssim t^{-\kappa}\mathcal{M}(X;\Phi).
				\end{aligned}
			\end{equation}
			For higher order derivatives of $N_1(X,\Phi)$, we have
			$$
			N_1(X,\Phi)(t,\widehat{\boldsymbol{x}})=-\int_{\mathbb{S}^2}\tilde G(\widehat{\boldsymbol{x}},\widehat{\boldsymbol{y}})\widetilde\nabla\cdot(\mathbf{T}(|\widetilde\nabla X|) \widetilde\nabla X)(t,\widehat{\boldsymbol{y}})d \mu_{\mathbb{S}^2}(\widehat{\boldsymbol{y}}),
			$$
			By \eqref{3dpeskerd}, it holds
			\begin{equation*}
				\tilde\nabla_{\widehat{\boldsymbol{x}}}^k\tilde G(\widehat{\boldsymbol{x}},\widehat{\boldsymbol{y}})=\sum_{l\leq k}C(k,l)\tilde\nabla_{\widehat{\boldsymbol{y}}}^{k-l}G_l(\widehat{\boldsymbol{x}},\widehat{\boldsymbol{y}}),
			\end{equation*}
            where 
			$
			G_l(\widehat{\boldsymbol{x}},\widehat{\boldsymbol{y}})=(\widetilde\nabla_{ \widehat{\boldsymbol{x}}}+\widetilde\nabla_{ \widehat{\boldsymbol{y}}})^l\tilde G(\widehat{\boldsymbol{x}},\widehat{\boldsymbol{y}})
			$.
			Using the fact that
			\begin{equation*}
				\begin{aligned}
					|G_l(\widehat{\boldsymbol{x}},\widehat{\boldsymbol{y}})|+|\widehat{\boldsymbol{x}}-\widehat{\boldsymbol{y}}||\widetilde\nabla_{ \widehat{\boldsymbol{x}},\widehat{\boldsymbol{y}}}G_l(\widehat{\boldsymbol{x}},\widehat{\boldsymbol{y}})|&+|\widehat{\boldsymbol{x}}-\widehat{\boldsymbol{y}}|^2|\widetilde\nabla_{ \widehat{\boldsymbol{x}}}\widetilde\nabla_{ \widehat{\boldsymbol{y}}}G_l(\widehat{\boldsymbol{x}},\widehat{\boldsymbol{y}})|\\
					&\lesssim t^{-l}(1+\mathbf{\Theta}_0)^l\frac{\|\widetilde\nabla(X-\Phi)\|_{T}(1+\|\widetilde\nabla(X-\Phi)\|_{T}+TM_{\Phi})^{m}}{|\widehat{\boldsymbol{x}}-\widehat{\boldsymbol{y}}|},
				\end{aligned}
			\end{equation*}
			and applying Remark \ref{rmk3dpesk} with corresponding $G_l$ and $C_{G_l}$, we obtain
			\begin{equation}\label{3dpesn1h}
				\begin{aligned}
					\|N_1\|_{ C^{m+\kappa}}\lesssim t^{-m-\kappa}\mathcal{M}(X;\Phi).
				\end{aligned}
			\end{equation}
		Then we estimate $N_1(Y,\Phi)-N_1(Z,\Phi)$. Note that
			\begin{equation*}
				\begin{aligned}
					&N_1(Y,\Phi)(t,\widehat{\boldsymbol{x}})-N_1(Z,\Phi)(t,\widehat{\boldsymbol{x}})\\
					&\quad=\int_{\mathbb{S}^2}\left(G(Y(t,\widehat{\boldsymbol{x}})-Y(t,\widehat{\boldsymbol{y}}))-G(Z(t,\widehat{\boldsymbol{x}})-Z(t,\widehat{\boldsymbol{y}}))\right)\widetilde\nabla\cdot(\mathbf{T}(|\widetilde\nabla Y|)\widetilde\nabla Y)(t,\widehat{\boldsymbol{y}})d \mu_{\mathbb{S}^2}(\widehat{\boldsymbol{y}})\\
					&\quad\quad+\int_{\mathbb{S}^2}\left(G(Z(t,\widehat{\boldsymbol{x}})-Z(t,\widehat{\boldsymbol{y}}))-G(\Phi(\widehat{\boldsymbol{x}})-\Phi(\widehat{\boldsymbol{y}}))\right)\widetilde\nabla\cdot(\mathbf{T}(|\widetilde\nabla Y|)\widetilde\nabla Y-\mathbf{T}(|\widetilde\nabla Z|)\widetilde\nabla Z)(t,\widehat{\boldsymbol{y}})d \mu_{\mathbb{S}^2}(\widehat{\boldsymbol{y}}).
				\end{aligned}
			\end{equation*}
			Applying Remark \ref{rmk3dpesk} with $\mathbf{G}=G(Y(t,\widehat{\boldsymbol{x}})-Y(t,\widehat{\boldsymbol{y}}))-G(Z(t,\widehat{\boldsymbol{x}})-Z(t,\widehat{\boldsymbol{y}}))$ and $\mathbf{G}=G(t,Z(t,\widehat{\boldsymbol{x}})-Z(t,\widehat{\boldsymbol{y}}))-G(\Phi(\widehat{\boldsymbol{x}})-\Phi(\widehat{\boldsymbol{y}}))$, respectively, and by similar methods of \eqref{3dpesn1} and \eqref{3dpesn1h}, we deduce
			\begin{equation*}\label{pesN1}
				\begin{aligned}
					\|(N_1(Y,\Phi)-N_1(Z,\Phi))\|_{ C^{j+\kappa}}\lesssim  t^{-j+\kappa}\mathcal{M}(Y,Z;\Phi),
				\end{aligned}
			\end{equation*}
			for $j=0,m$. This completes the proof.
		\end{proof}
		\begin{lemma}\label{3dpeslemn2}
			Let $N_2(F,\Phi)$ be defined in \eqref{3dpesnlt}. For any $T\in(0,1)$, and any $X,Y,Z\in\mathcal{Z}_{T,\Phi}^\sigma$, we have the following estimates:
			\begin{itemize}
				\item \textbf{1: Nonlinear estimate:}\\
				\begin{equation*}
					\begin{aligned}
						\sum_{j=0,m}\sup_{t\in[0,T]}t^{j+\kappa}\|N_2(X,\Phi)(t)\|_{ C^{j+\kappa}}\lesssim\mathcal{M}(X;\Phi),
					\end{aligned}
				\end{equation*}
				\item \textbf{2: Lipschitz type continuity:}\\
				\begin{equation*}
					\begin{aligned}
						\sum_{j=0,n}\sup_{t\in[0,T]}t^{j+\kappa}\|(N_2(Y,\Phi)-N_2(Z,\Phi))(t)\|_{ C^{j+\kappa}}\lesssim\mathcal{M}(Y,Z;\Phi).
					\end{aligned}
				\end{equation*} 
			\end{itemize}
		\end{lemma}
		\begin{proof}
			The proof is just similar with Lemma \ref{3dpeslemn1}. Fix $t\in[0,T]$ and for simplicity we drop $t$ in the proof. We denote 
			\begin{equation*}
				R(X,\Phi)=\mathbf{T}(|\widetilde\nabla X|) \widetilde\nabla X-\mathbf{T}(|\widetilde\nabla\Phi|) \widetilde\nabla\Phi-J(\widetilde\nabla\Phi) \widetilde\nabla(X-\Phi).
			\end{equation*}
			By definition, we have 
			\begin{equation*}
				\begin{aligned}
					R(X,\Phi)=R_1+R_2+R_3,\\
				\end{aligned}
			\end{equation*}
			with
			\begin{align*}
				&R_1=(\mathbf{T}(|\widetilde\nabla X|)-\mathbf{T}(|\widetilde\nabla\Phi|))\widetilde\nabla\Phi-\mathbf{T}'(|\widetilde\nabla\Phi|)(|\widetilde\nabla X|-|\widetilde\nabla\Phi|)\widetilde\nabla\Phi,\\
				&R_2=(\mathbf{T}(|\widetilde\nabla X|)-\mathbf{T}(|\widetilde\nabla\Phi|))\widetilde\nabla(X-\Phi),\\
				&R_3=\mathbf{T}'(|\widetilde\nabla\Phi|)(|\widetilde\nabla X|-|\widetilde\nabla\Phi|)\widetilde\nabla\Phi-J_2(\widetilde\nabla\Phi)\widetilde\nabla(X-\Phi),
			\end{align*}
			with $J_2$ defined in \eqref{3dpesdefJ}. For lowest order H\"{o}lder norm, let $\mathbf{G}=G(\Phi(\widehat{\boldsymbol{x}})-\Phi(\widehat{\boldsymbol{y}}))$ and $C_{\mathbf{G}}=m_{\Phi}(1+m_{\Phi})$. By \eqref{3dpesnc}, the condition \eqref{conG} is satisfied. Then we apply Remark \ref{rmk3dpesk} to obtain
			\begin{equation*}
				\|N_2(X,\Phi)\|_{\dot C^\kappa}\lesssim (1+m_{\Phi})\|R(X,\Phi)\|_{\dot C^{\kappa}}.
			\end{equation*}
			By Newton-Leibniz formula, we can rewrite the remainder terms as
			\begin{equation*}
				\begin{aligned}
					&R_1=\int_0^1\int_0^1\mathbf{T}''(|\widetilde\nabla\Phi|+\lambda\mu(|\widetilde\nabla X|-|\widetilde\nabla\Phi|))\lambda (|\widetilde\nabla X|-|\widetilde\nabla\Phi|)^2\widetilde\nabla\Phi d\mu d\lambda,\\
					&R_2=\int_0^1\mathbf{T}'(|\widetilde\nabla\Phi|+\lambda(|\widetilde\nabla X|-|\widetilde\nabla\Phi|))(|\widetilde\nabla X|-|\widetilde\nabla\Phi|)\widetilde\nabla(X-\Phi)d\lambda,\\
					&R_3=\mathbf{T}'(|\widetilde\nabla\Phi|)\left(\frac{\widetilde\nabla(X-\Phi):\widetilde\nabla(X-\Phi)}{|\widetilde\nabla X|+|\widetilde\nabla\Phi|}+\left(\frac{2\widetilde\nabla\Phi:\widetilde\nabla(X-\Phi)}{|\widetilde\nabla X|+|\widetilde\nabla\Phi|}-\frac{\widetilde\nabla\Phi:\widetilde\nabla(X-\Phi)}{|\widetilde\nabla\Phi|}\right)\right)\widetilde\nabla\Phi.
				\end{aligned}
			\end{equation*}
			By Lemma \ref{maininterpo} and Lemma \ref{lemcom}, we can see that
			\begin{equation*}
				\|R(X,\Phi)\|_{ C^{\kappa}}\lesssim t^{-\kappa}\mathfrak{M}_1\|X-\Phi\|_{Z_T}(\|X-\Phi\|_{Z_T}+T^{\kappa}M_{\Phi})(1+\|X-\Phi\|_{Z_T}+m_{\Phi}),
			\end{equation*}
			which infers
			\begin{equation*}
				\|N_2(X,\Phi)\|_{ C^\kappa}\lesssim t^{-\kappa}\mathcal{M}(X;\Phi).
			\end{equation*}
			For higher order derivatives of $N_2(X,\Phi)$, we only need to use \eqref{3dpeskerd}, integrating by parts and taking derivatives on $R$, then by similar way of \eqref{3dpesn1h}, we can obtain
			\begin{equation*}
				\begin{aligned}
					\|N_2(X,\Phi)\|_{ C^{m+\kappa}}\lesssim t^{-m-\kappa}\mathcal{M}(X;\Phi).
				\end{aligned}
			\end{equation*}
			For $N_2(Y,\Phi)-N_2(Z,\Phi)$, we will shortly denote $\vec{V}=(Y,Z)$ and $\vec{V}-\Phi=(Y-\Phi,Z-\Phi)$ in the following proof. We have
			\begin{equation*}
				\begin{aligned}
					\mathbf{T}(|\widetilde\nabla Y|)\widetilde\nabla Y-\mathbf{T}(|\widetilde\nabla Z|)\widetilde\nabla Z-J(\widetilde\nabla \Phi)\widetilde\nabla(Y-Z)=\mathbf{R}=\mathbf{R}_1+\mathbf{R}_2+\mathbf{R}_3+\mathbf{R}_4,
				\end{aligned}
			\end{equation*}
			with
			\begin{align*}
				&\mathbf{R}_1=(\mathbf{T}(|\widetilde\nabla Y|)-\mathbf{T}(\widetilde\nabla\Phi))\widetilde\nabla(Y-Z)+(\mathbf{T}(|\widetilde\nabla Y|)-\mathbf{T}(|\widetilde\nabla Z|))\widetilde\nabla(Z-\Phi),\\
				&\mathbf{R}_2=\big(\mathbf{T}(|\widetilde\nabla Y|)-\mathbf{T}(|\widetilde\nabla Z|)-\mathbf{T}'(|\widetilde\nabla\Phi|)(|\widetilde\nabla Y|-|\widetilde\nabla Z|)\big)\widetilde\nabla\Phi,\\
				&\mathbf{R}_3=(\mathbf{T}'(|\widetilde\nabla Y|-\mathbf{T}'(|\widetilde\nabla\Phi|))(|\widetilde\nabla Y|-|\widetilde\nabla Z|),\\
				&\mathbf{R}_4=\big(\mathbf{T}'(|\widetilde\nabla\Phi|)(|\widetilde\nabla Y|-|\widetilde\nabla Z|)\big)\widetilde\nabla\Phi-J(\widetilde\nabla\Phi)\widetilde\nabla(Y-Z).
			\end{align*}
			By applying Remark \ref{rmk3dpesk} with $\mathbf{G}=G(\Phi(\widehat{\boldsymbol{x}})-\Phi(\widehat{\boldsymbol{y}}))$, we have
			\begin{equation*}
				\|(N_2(Y,\Phi)-N_2(Z,\Phi))\|_{C^\kappa}\lesssim (1+m_{\Phi})\|\mathbf{R}\|_{\dot C^{\kappa}}.
			\end{equation*}
		 By Lemma \ref{maininterpo} we have
			\begin{equation*}
				\|\mathbf{R}\|_{ C^{\kappa}}\lesssim t^{-\kappa}\mathfrak{M}_1\|Y-Z\|_{Z_T}(\|\vec{V}-\Phi\|_{Z_T}+T^\kappa M_{\Phi})(1+\|\vec{V}-\Phi\|_{Z_T}+m_{\Phi}).
			\end{equation*}
			This yields that 
			\begin{equation*}
				\begin{aligned}
					&\|(N_2(Y,\Phi)-N_2(Z,\Phi))\|_{ C^\kappa}\\
					&\quad\quad\quad\quad\lesssim t^{-\kappa}\mathfrak{M}_1\|Y-Z\|_{Z_T}(\|\vec{V}-\Phi\|_{Z_T}+T^\kappa M_{\Phi})(1+\|\vec{V}-\Phi\|_{Z_T}+m_{\Phi})^2.
				\end{aligned}
			\end{equation*}
			For higher order derivatives, we use \eqref{3dpeskerd}, integrating by parts and taking derivatives on $\mathbf{R}$ to obtain
			\begin{equation*}
				\begin{aligned}
					&\|(N_2(Y,\Phi)-N_2(Z,\Phi))\|_{ C^{m+\kappa}}\\
					&\quad\lesssim t^{-m-\kappa}\mathfrak{M}_1\|Y-Z\|_{Z_T}(\|\vec{V}-\Phi\|_{Z_T}+T^\kappa M_{\Phi})(1+\|\vec{V}-\Phi\|_{Z_T}+m_{\Phi}+T^\kappa M_{\Phi})^{m+3}.
				\end{aligned}
			\end{equation*}
        This completes the proof of the lemma.
		\end{proof}\\
		Finally, for $N_3(\Phi)$, since $\Phi$ is smooth, it is straightforward to obtain the following estimate.
		\begin{lemma}\label{3dpeslemn3}
			Let $N_3(\Phi)$ be defined in \eqref{3dpesnlt}, it holds
			\begin{equation*}
				\begin{aligned}
					\|N_3(\Phi)\|_{C^{m+\kappa}}\lesssim \mathfrak{M}_1M_{\Phi}(1+M_{\Phi})^{m+2}.
				\end{aligned}
			\end{equation*}
		\end{lemma}


\begin{thebibliography}{99}
			\bibitem{Am}H. Abels and B.-V. Matioc. {\em Well-posedness of the Muskat problem in subcritical
				$L^p$-Sobolev spaces. }European Journal of Applied Mathematics, 1–43, 2021.  
			\bibitem{Agmon1964} S. Agmon, A. Douglis, L. Nirenberg. {\em Estimates near the boundary for solutions of elliptic partial differential equations satisfying general boundary conditions. II}, Comm. Pure Appl. Math. {\bf 17} (1964), 35--92.
			\bibitem{ABZ} T. Alazard, N. Burq, and C. Zuily. {\em On the Cauchy problem for gravity water waves.} Invent. math., 198:71–163, 2014.
			\bibitem{Alazard-Lazar}
			T. Alazard and O. Lazar.
			{\em Paralinearization of the {M}uskat equation and application to the
				Cauchy problem.}
			Arch. Ration. Mech. Anal., 237(2):545--583, 2020.
			
			\bibitem{AN2020endpoint}T. Alazard and Q.-H. Nguyen. {\em Endpoint Sobolev theory for the Muskat equation}.  Communications in Mathematical Physics, 397, 1043--1102, 2023.

            \bibitem{AN2021}T. Alazard and Q.-H. Nguyen.{\em On the Cauchy problem for the Muskat equation with non-Lipschitz initial data.} Comm. Partial Differential Equations, 46(11):2171–2212, 2021.
			
		
			\bibitem{Alazard2020}T. Alazard and Q.-H. Nguyen. {\em On the Cauchy problem for the Muskat equation. II: Critical initial data},  Ann. PDE , 7 (2021). https://doi.org/10.1007/s40818-021-00099-x.
			
			\bibitem{TH4}  T. Alazard and Q.-H. Nguyen. {\em Quasilinearization of the 3D Muskat equation, and applications to the critical Cauchy problem.} Adv Math,. 399, 108278, 2022.

\bibitem{RAkin} R. Alexandre. {\em Fractional order kinetic equations and hypoellipticity}, arXiv: 1102.2161.
            
			\bibitem{Amb04} D. M. Ambrose. {\em Well-posedness of two-phase Hele-Shaw flow without surface tension.}
			European J. Appl. Math., 15(5):597–607, 2004.
			\bibitem{Amb07} D. M. Ambrose. {\em Well-posedness of two-phase Darcy flow in 3D. }Quart. Appl. Math.,
			65(1):189–203, 2007.
			
			\bibitem{Fourierbook} H. Bahouri, J. Y. Chemin, R. Danchin. {\em Fourier Analysis and Nonlinear Partial Differential Equations.} Springer, 2011. DOI:10.1007/978-3-642-16830-7.

\bibitem{FB2002} F. Bouchut. {\em Hypoelliptic regularity in kinetic equations}, J. Math. Pure
				Appl. {\bf 81} (2002), 1135–1159.
            
			\bibitem{CCFG}A. Castro, D. C\'{o}rdoba, C. L. Fefferman and  F. Gancedo. {\em Breakdown of smoothness for the Muskat problem.} Arch. Ration. Mech. Anal., 208(3):805–909, 2013.	
			
			\bibitem{RT2012} A. Castro, D. C\'{o}rdoba, C. L. Fefferman, F. Gancedo and M. L\'{o}pez-Fern\'{a}ndez. {\em Rayleigh Taylor breakdown
				for the Muskat problem with applications to water waves.} Annals of Math, 175(2): 909–948, 2012.

\bibitem{CCFG2016}A. Castro,D. C\'{o}rdoba, C. Fefferman, and F. Gancedo. {\em  Splash singularities for the one-phase Muskat problem in stable regimes.} Arch. Ration. Mech. Anal., 222(1):213–243, 2016.
            
			\bibitem{Cameron2019} S. Cameron. {\em Global well-posedness for the two-dimensional Muskat problem with slope
				less than 1}. Anal. PDE, 12(4):997–1022, 2019.
			
			\bibitem{Cameron2020} S. Cameron. {\em Global well-posedness for the 3D Muskat problem with medium size slope}. arXiv:2002.00508.
			
			\bibitem{CCHNX}S. Cameron, K. Chen, R. Hu, Q.-H. Nguyen, Y. Xu.
{\em The Muskat problem with a large slope,}
Journal of Functional Analysis,
290, 4, 111257, 2026.

			
			\bibitem{Camepeskin} S. Cameron and R. M. Strain. 
			{\em Critical local well-posedness for the fully nonlinear Peskin problem}. Comm. Pure Appl. Math., 77(2):901--989, 2024.
			
			
			
			\bibitem{ChangKang} T. Chang, K. Kang. {\em Local regularity near boundary for the Stokes and Navier–Stokes Equations.} SIAM Journal on Mathematical Analysis, 55(5): 5051-5085, 2023.

            \bibitem{CHN-CNS-local}K. Chen, R. Hu, and Q.-H. Nguyen. {\em Local well-posedness of the 1d compressible Navier-Stokes system with rough data,} Calculus of Variations and Partial Differential Equations, 63(42), 2024.


            \bibitem{CHHN-CNS-global} K. Chen, L. K. Ha, R. Hu, Q.-H. Nguyen. {\em Global well-posedness of the 1d compressible Navier-Stokes system with rough data,} Journal de Math\'ematiques Pures et Appliqu\'ees, 179: 425-453, 2023.
			
			\bibitem{KHN2025} K. Chen, R. Hu and Q.-H. Nguyen. {\em Well-posedness of the Fractional Fokker-Planck Equation},	arXiv:2501.16251
			
			\bibitem{KN} K. Chen and  Q.-H. Nguyen. {\em The Peskin problem with $\dot B^1_{\infty,\infty}$ data}.  SIAM J. Math. Anal., 55(6): 6262-6304, 2023.
			\bibitem{KeC1} K. Chen, Q.-H. Nguyen and Y. Xu. {\em The Muskat problem with $C^1$  data}. Trans. Am. 
			Math. Soc. 375: 3039–3060, 2022.

            \bibitem{CNY} K. Chen, Q.-H. Nguyen, and T. Yang. {\em Well-posedness of the Boltzmann and Landau Equations in Critical Spaces.} arXiv: 2509.14845.
			
			
			
			
			\bibitem{CGS2016}	A. Cheng, R. Granero-Belinch\'{o}n and S. Shkoller. {\em Well-posedness of the Muskat problem with H2 initial data,} Adv. Math., 286 , 32-104, 2016.
			
			
			
			
			
			
			
			\bibitem{1Constantin2010}P. Constantin, D. C\'{o}rdoba, F. Gancedo, and R. M. Strain.  {\em On the global existence for the Muskat problem}. Journal of the European Mathematical Society, 15(1):201-227, 2013.
			
			
			
			\bibitem{CCG2016}P. Constantin, D. C\'{o}rdoba, F. Gancedo, L. Rodriguez-Piazza and R. M. Strain. {\em On the Muskat problem: global in time results in 2D and 3D.} Amer. J. Math 138, no. 6, 1455-1494, 2016.
			
			\bibitem{CMT}P. Constantin, A.J. Majda, and E. Tabak. {\em Formation of strong fronts in the 2-D quasigeostrophic thermal active scalar, }Nonlinearity, 7(6) (1994), 1495–1533.
			
			\bibitem{CGSV}P. Constantin, F. Gancedo, R. Shvydkoy, and V. Vicol. {\em Global regularity for
				2D Muskat equations with finite slope.} Ann. Inst. H. Poincar\'{e} Anal. Non Lineaire, 34(4):1041–
			1074, 2017.
			
			\bibitem{CI17} P. Constantin, M. Ignatova. {\em Remarks on the fractional Laplacian with Dirichlet boundary conditions and applications,} IMRN,
			2017, Issue 6, (2017), 1653-1673.
			
			\bibitem{CI16} P. Constantin, M. Ignatova. {\em Critical SQG in bounded domains,} Annals of PDE, 2 (2016), no 8.
			
			\bibitem{CI20}P. Constantin, M Ignatova. {\em Estimates near the boundary for critical SQG}, Annals of PDE, 6 (2020).
			
			\bibitem{Constantin2023}P. Constantin, M. Ignatova, Q-H. Nguyen. {\em Global regularity for critical SQG in bounded domains.} Commun. Pure Appl. Math. 78(1): 3--59, 2025.
			
			\bibitem{CLSTW} P. Constantin, M.-C. Lai, R. Sharma, Y.-H. Tseng, and J. Wu. {\em New numerical results for the surface quasi-geostrophic equation}, J. Sci. Comput., 50(1) (2012), 1–28.
			
			\bibitem{CCG2011}A. C\'{o}rdoba, D. C\'{o}rdoba, and F. Gancedo. {\em Interface evolution: the Hele-Shaw and Muskat problems.} Ann. of Math. (2), 173(1):477–542, 2011.

\bibitem{CCG2013} A. C\'{o}rdoba, D. C\'{o}rdoba, and F. Gancedo. {\em Porous media: the Muskat problem in three dimensions.} Anal. PDE, 6(2):447– 497, 2013
            
			\bibitem{Cor98}D. C\'{o}rdoba. {\em Nonexistence of simple hyperbolic blow-up for the quasi-geostrophic equation, }Ann. of Math. (2), 148(3) (1998) 1135–1152.
	
			
			\bibitem{CG07}D. C\'{o}rdoba and F. Gancedo. {\em Contour dynamics of incompressible 3-D fluids
				in a porous medium with different densities. }Communications in Mathematical Physics,
			273(2):445–471, 2007.
			\bibitem{CG09} D. C\'{o}rdoba and F. Gancedo. {\em A maximum principle for the Muskat problem for
				fluids with different densities.} Communications in Mathematical Physics, 286(2):681–696,
			2009.

\bibitem{CGZ2015}D. C\'{o}rdoba, J. G\'{o}mez-Serrano, and A. Zlato\v{s}. {\em A note on stability shifting for the Muskat problem.} Philosophical Transactions of the Royal Society of London A: Mathematical, Physical and Engineering Sciences, 373(2050):20140278, 10, 2015.

\bibitem{CGZ2017}D. C\'{o}rdoba, J. G\'{o}mez-Serrano, and A. Zlato\v{s}. {\em A note on stability shifting for the Muskat problem,} II: From stable to unstable and back to stable. Anal. PDE, 10(2):367–378, 2017.
            
			\bibitem{Cordoba-Lazar-H3/2}
			D. C{\'o}rdoba and O. Lazar. {\em  Global well-posedness for the 2d stable Muskat problem in
				${H}^{\frac{3}{2}}$}.
			To appear in Annales scientifiques de l'\'{E}cole normale sup\'{e}rieure, 2021.
			
			
			
			
			
			
			\bibitem{DF96}
			H. Darcy.
			{\em Les Fontaines publiques de la ville de Dijon. Exposition et
				application des principes {\`a} suivre et des formules {\`a} employer dans
				les questions de distribution d'eau, etc}.
			V. Dalamont, 1856.
			
			
			
			
			\bibitem{DengLei2017}F. Deng, Z. Lei, and F. H. Lin. {\em On the two-dimensional Muskat problem with monotone
				large initial data}. Comm. Pure Appl. Math., 70(6):1115–1145, 2017.
			\bibitem{DLN14}	J. J. Donaire, J. G. Llorente, and A. Nicolau. {\em Differentiability of functions in the Zygmund class}. Proc. Lond. Math. Soc. 108.1 (3): 133–158, 2014.
			
			
			
			\bibitem{Escher11}	 J. Escher and B.-V. Matioc. {\em On the parabolicity of the Muskat problem: well-posedness, fingering, and stability
				results.} Z. Anal. Anwend., 30(2):193–218, 2011.
			
			
			\bibitem{Escher12}J. Escher, A.-V. Matioc, B.-V. Matioc. {\em A generalized Rayleigh-Taylor condition for the Muskat 
				problem}. Nonlinearity 25, 73–92, 2012.
			\bibitem{Escher18}J. Escher, B.-V. Matioc,  C. Walker. {\em The domain of parabolicity for the Muskat problem.} Indiana Univ. Math. J. 67, 679–737, 2018.
			
			\bibitem{Evans98}L. C. Evans. {\em Partial Differential Equations,} Graduate Studies in Mathematics, AMS, Providence, 1998.
			



            
			\bibitem{1Gancedo2020Global}F. Gancedo and O. Lazar. {\em Global well-posedness for the 3d Muskat problem in the critical Sobolev space.} Arch. Ration. Mech. Anal., 246(1), 141-207, 2022.

\bibitem{GGPS}F. Gancedo, E. Garc\'{\i}a-Ju\'{a}rez, N. Patel, and R. M. Strain. {\em On the Muskat problem with viscosity jump: global in time
results.} Adv. Math., 345:552–597, 2019.
            
\bibitem{Gancedo:survey}F. Gancedo. {\em A survey for the Muskat problem and a new estimate.} SeMA J., 74(1):21–35, 2017.
            
			\bibitem{GancedoGlobal2020}F. Gancedo, R. Granero-Belinchón, S. Scrobogna, {\em Global existence in the Lipschitz class for the N-Peskin problem}. Indiana University Mathematics Journal, 72(2):553--602, 2023.

            
			\bibitem{GarciaViscosityContrast2020}E. Garc\'{\i}a-Ju\'{a}rez, Y. Mori and R. M. Strain. {\em The Peskin problem with viscosity contrast}. Analysis and PDE. 16(3):785--838, 2023.

            \bibitem{GGNP}E. Garc\'{\i}a-Ju\'{a}rez, J. G\'{o}mez-Serrano, Huy Q. Nguyen, B. Pausader.
{\em Self-similar solutions for the Muskat equation,}
Advances in Mathematics,
399, 108294,
2022.

\bibitem{GGHP} E. Garc\'{\i}a-Ju\'{a}rez, J. G\'{o}mez-Serrano, S. V. Haziot, B. Pausader. {\em Desingu-
larization of small moving corners for the Muskat equation,} Ann. PDE 10 (2024) 17.
            
			\bibitem{GHpeskin} E. Garc\'{\i}a-Ju\'{a}rez, S. V. Haziot. {\em Critical well-posedness for the  2D Peskin problem with general tension}. Advances in Mathematics, 460, 110047, 2025.
			
			\bibitem{3Dpeskin} E. Garc\'{\i}a-Ju\'{a}rez, P.-C. Kuo, Y. Mori and R. M. Strain.
			{\em Well-posedness of the 3D Peskin Problem}. arXiv:2311.10157.


			
			\bibitem{GBL2020}R. Granero-Belinch \'{o}n and O. Lazar. {\em Growth in the Muskat problem.} Math. Model. Nat. Phenom., 15:Paper No. 7, 23, 2020.
2020.
			
			
			\bibitem{QTLG14}Q. T. Le Gia, W. Mclean. {\em Solving the heat equation on the unit sphere
				via Laplace transforms and radial basis functions.} Adv Comput Math (2014) 40:353–375, DOI:10.1007/s10444-013-9311-6.
			
			
			
			\bibitem{GTbook}D. Gilbarg, N. S. Trudinger. {\em Elliptic Partial Differential Equations of Second Order}. Springer-Verlag Berlin Heidelberg, 1977.
			
			\bibitem{HZ2024} Z. Hao, M. R\"{o}ckner, X. Zhang, {\em Second order fractional mean-field SDEs with singular kernels and measure
					initial data}, Ann. Probab. 54(1): 1-62, 2026.
			
			\bibitem{Hormand}L. Hormander, {\em Hypoelliptic second order differential equations,} Acta Math., 119, 147–171, 1967.
			
			\bibitem{HouNumerical2012}G. Hou, J. Wang, and A. Layton. {\em Numerical methods for fluid-structure interaction: a review.} Communications in Computational Physics, 12(02):337–377, 2012.
			
			
			
			\bibitem{Ignatova}M. Ignatova, {\em Construction of solutions of the critical SQG equation in bounded domains}, Advances in Mathematics, 351
			(2019), 1000–1023.
			
			
			\bibitem{Lazar24} O. Lazar, {\em Global well-posedness of arbitrarily large Lipschitz solutions for the Muskat problem with surface tension,} arXiv: 2407.09444.

            \bibitem{NL2012} N. Lerner, Y. Morimoto, K. Pravda-Starov. {\em Hypoelliptic Estimates for a Linear Model of the Boltzmann Equation
					without Angular Cutoﬀ}, Communications in Partial Differential Equations, 
				{\bf 37}, (2012), 234-284.
			\bibitem{Li2020Stability}H. Li. {\em Stability of the Stokes Immersed Boundary problem with Bending and Stretching
				energy}. Journal of Functional Analysis. 281(9):109204, 2021.
			
			\bibitem{LinTongSolvability2019}F.-H. Lin and J.-J. Tong. {\em Solvability of the Stokes immersed boundary problem in two dimensions.} Comm. Pure Appl.
			Math., 72(1):159–226, 2019.
			
			
			
			\bibitem{CM12}C. Mantegazza, L. Martinazzi. {\em A note on quasilinear parabolic equations on manifolds.} Ann. Sc. Norm. Super. Pisa Cl. Sci. (5) Vol. XI (2012), 857-874.
			
			
			
			\bibitem{Matio2018} B.-V. Matioc. {\em Viscous displacement in porous media: the Muskat problem in 2D.} Trans. Am. 
			Math. Soc. 370, 7511–7556,  2018.
			\bibitem{1Matioc2018} B.-V. Matioc. {\em The Muskat problem in two dimensions: equivalence of formulations, well-posedness, and regularity
				results}. Analysis and PDE, 12(2):281–332, 2018.
			\bibitem{MATIOC2022} A.-V. Matioc, B.-V. Matioc. {\em A new reformulation of the Muskat problem with surface tension}. Journal of Differential Equations. 350:308--335, 2023.
			
			\bibitem{YMCX2007} Y. Morimoto, C.-J. Xu. {\em Hypoelliticity for a class of kinetic equations}, J. Math. Kyoto Univ. {\bf 47} (2007), 129–152.
				\bibitem{CV} C. Villani. A review of mathematical topics in collisional kinetic theory, in {\it Handbook of mathematical fluid dynamics, Vol. I}, 71--305, North-Holland, Amsterdam.
			
			
			
			\bibitem{MittalImmersed2005}R. Mittal and G. Iaccarino. {\em Immersed boundary methods}. Annu. Rev. Fluid Mech., 37:239–261, 2005.
			\bibitem{MoriWell2019}Y. Mori, A. Rodenberg, and D. Spirn. {\em Well-posedness and global behavior of the Peskin problem of an immersed elastic filament in Stokes flow}. Comm. Pure Appl.
			Math., 72(5):887-980, 2019.
			
			
			
			
			\bibitem{Mus34} M. Muskat. {\em Two fluid systems in porous media. The encroachment of water into an oil sand.} J. Appl. Phys., 5(9):250–264, 1934.
			
			
			\bibitem{HNguyen}H. Q. Nguyen. {\em On well-posedness of the Muskat problem with surface tension}, Adv. Math., Volume 374,
			18, 107344 (2020).
			
			
			\bibitem{1HuyNguyen2020}	H. Q. Nguyen and B. Pausader. {\em A paradifferential approach for well-posedness of the Muskat problem.} Archive for
			Rational Mechanics and Analysis, 237:35-100, 2020.
			
			\bibitem{HuyNguyen2021} H. Q. Nguyen. {\em Global solutions for the Muskat problem in the scaling invariant Besov space $\dot B^1_{\infty,1}$.} Advances in Mathematics. 394:108122, 2022.
			
			
			
			\bibitem{Peskin1972Thesis}C. S Peskin. {\em Flow patterns around heart valves: a digital computer method for solving the
				equations of motion}. PhD thesis, Sue Golding Graduate Division of Medical Sciences, Albert
			Einstein College of Medicine, Yeshiva University, 1972.
			
			\bibitem{PeskinFlow1972} C. S Peskin. {\em Flow patterns around heart valves: a numerical method.} Journal of Computational Physics, 10(2):252–271,
			1972.
			
			\bibitem{PeskinImmersed2002}C. S Peskin. {\em The immersed boundary method.} Acta Numerica, 11:479–517, 2002.
			
			\bibitem{Polden}A. Polden. {\em Curves and surfaces of least total curvature and fourth-order flows,} Ph.D. Thesis, Mathematisches Institut, Universitat Tubingen, 1996.
			
			
			
			\bibitem{Rodenberg2018} A. Rodenberg. {\em 2D Peskin Problems of an Immersed Elastic Filament in Stokes Flow.} PhD thesis, The University of Minnesota. 2018.
			
			\bibitem{sesv} G. Seregin and V. Sverak. {\em On a Bounded Shear Flow in a Half-Space. } Journal of Mathematical Sciences, 178(3): 353-356, 2011.

\bibitem{Shi:regularity-muskat} J. Shi. {\em Regularity of Solutions to the Muskat Equation.} Arch. Ration. Mech. Anal., 247(3):Paper No. 36, 2023. 

\bibitem{Shi2024}J. Shi, {\em The regularity of the solutions to the Muskat equation: the degenerate regularity near the turnover points}, Adv. Math. 454, 109850, 2024.
            
			\bibitem{SCH04}M. Siegel, R. E. Caflisch, and S. Howison. {\em Global existence, singular solutions, and ill-posedness for the Muskat problem.} Comm. Pure Appl. Math., 57(10):1374–1411, 2004.
			
			
			\bibitem{JJS}J. J. Sharples. {\em Linear and quasilinear parabolic equations in Sobolev space}. J.Differential Equations. 202, 111–142, 2004.
			
			
			
			\bibitem{Steinbook} E. M. Stein. {\em Singular Integrals and Differentiability Properties of Functions}. Princeton University Pre, 1970.
			Fluid Mech. 19, 581–611,  2017.
			\bibitem{TongRegularized2021}J. Tong. {\em Regularized Stokes Immersed Boundary Problems in Two Dimensions: Well-posedness, Singular Limit, and Error Estimates. } Comm. Pure Appl.
			Math., 74(2):366-449, 2021.
			
			\bibitem{Tong24}J. Tong. {\em Global solutions to the tangential Peskin problem in 2-D}. Nonlinearity (2024), 37(1), 015006.
			
			\bibitem{TongWei}J. Tong and D. Wei. {\em Geometric properties of the 2-D Peskin problem.} Ann. PDE. 10, 24, 2024.

            \bibitem{TW2026} J. Tong and D. Wei. {\em
The Immersed Boundary Problem in 2-D: the Navier-Stokes Case.} arXiv: 2511.16189.
            
			\bibitem{Triebel}
			H. Triebel. {\em Theory of function spaces}, Birkh\"{a}user Basel, 1983. 
			
			
			
			\bibitem{wz1}W. Wang and L. Zhang. {\em The $C^\alpha$ regularity of a class of non-homogeneous ultraparabolic equations.} Sci. China
			Ser. A, 52(8):1589–1606, 2009.
			\bibitem{wz2} W. Wang and L. Zhang. {\em The $C^\alpha$ regularity of weak solutions of ultraparabolic equations.} Discrete Contin. Dyn.
			Syst., 29(3):1261–1275, 2011.
			
			\bibitem{Yi}F. Yi. {\em Global classical solution of Muskat free boundary problem. } J. Math. Anal. Appl.,
			288(2):442–461, 2003.
			
			
			
			
			
			
		\end{thebibliography}
	\end{document}